%% file: StableSpec.tex
\documentclass[english]{smfbook}

\usepackage{etex}
\usepackage[francais,english]{babel}%
\usepackage{xspace}%
\usepackage{amsfonts,amsmath,amssymb,smfthm}%
\usepackage{graphicx,mathdots}%
\usepackage{enumerate}%
\usepackage[thinlines,thicklines]{easybmat}%
\usepackage{pinlabel}%
\usepackage{color}%

\input{xy}
\xyoption{all}
\objectmargin={3mm}

\theoremstyle{plain}
\newtheorem{proposition}[table]{Proposition}
\newtheorem{theorem}[table]{Theorem}
\newtheorem{corollary}[table]{Corollary}
\newtheorem{lemma}[table]{Lemma}
\newtheorem{fact}[table]{Fact}

\newtheorem{claim}[table]{Claim}
\newtheorem{conjecture}[table]{Conjecture}
\newtheorem{example}[table]{Example}
\newtheorem{def-lem}[table]{Definition-Lemma}

\newtheoremstyle{theoremwithref}{}{}{\itshape}{}{\bfseries}{.}{.5em}{#1 #2 #3}
\theoremstyle{theorem}

\theoremstyle{definition}
\newtheorem{definition}[table]{Definition}

\newtheorem{remark}[table]{Remark}
\newtheorem{remarks}[table]{Remarks}
\newtheorem{observation}[table]{Observation}

\newcommand{\C}{\mathbb{C}}
\newcommand{\R}{\mathbb{R}}
\newcommand{\Z}{\mathbb{Z}}
\newcommand{\N}{\mathbb{N}}
\newcommand{\GL}{\mathrm{GL}}
\newcommand{\SL}{\mathrm{SL}}
\newcommand{\PSL}{\mathrm{PSL}}
\newcommand{\OO}{\mathrm{O}}
\newcommand{\SO}{\mathrm{SO}}
\newcommand{\PSO}{\mathrm{PSO}}
\newcommand{\SU}{\mathrm{SU}}
\newcommand{\U}{\mathrm{U}}
\newcommand{\Sp}{\mathrm{Sp}}
\newcommand{\Spin}{\mathrm{Spin}}
\newcommand{\g}{\mathfrak{g}}
\newcommand{\h}{\mathfrak{h}}
\newcommand{\kk}{\mathfrak{k}}
\newcommand{\p}{\mathfrak{p}}
\newcommand{\q}{\mathfrak{q}}
\newcommand{\aaa}{\mathfrak{a}}
\newcommand{\bb}{\mathfrak{b}}
\newcommand{\jj}{\mathfrak{j}}
\newcommand{\n}{\mathfrak{n}}

\newcommand{\ttt}{\mathfrak{t}}
\newcommand{\ssl}{\mathfrak{sl}}
\newcommand{\so}{\mathfrak{so}}
\newcommand{\su}{\mathfrak{su}}
\newcommand{\ssp}{\mathfrak{sp}}
\newcommand{\e}{\mathfrak{e}}
\newcommand{\f}{\mathfrak{f}}
\newcommand{\Hom}{\mathrm{Hom}}
\newcommand{\End}{\operatorname{End}}
\newcommand{\D}{\mathbb{D}}
\newcommand{\Diag}{\mathrm{Diag}}
\newcommand{\M}{\mathcal{M}}
\newcommand{\Spec}{\mathrm{Spec}}
\newcommand{\Ad}{\operatorname{Ad}}
\newcommand{\ad}{\operatorname{ad}}
\newcommand{\HH}{\mathbb{H}}
\newcommand{\PP}{\mathbb{P}}
\newcommand{\rank}{\operatorname{rank}}
\newcommand{\V}{\mathcal{V}}
\newcommand{\ZZ}{\mathcal{Z}}
\newcommand{\vol}{\mathrm{vol}}
\newcommand{\AdS}{\mathrm{AdS}}
\newcommand{\ie}{\emph{i.e.}\ }
\newcommand{\resp}{resp.\ }

\newcommand*{\longhookrightarrow}{\ensuremath{\lhook\joinrel\relbar\joinrel\rightarrow}}
\newcommand*{\longhookleftarrow}{\ensuremath{\leftarrow\joinrel\relbar\joinrel\rhook}}
\newcommand*{\longtwoheadrightarrow}{\relbar\joinrel\twoheadrightarrow}
\newcommand*{\longlongrightarrow}{\ensuremath{\relbar\joinrel\relbar\joinrel\relbar\joinrel\longrightarrow}}

\newenvironment{changemargin}[2]{\begin{list}{}{%
\setlength{\topsep}{0pt}%
\setlength{\leftmargin}{0pt}%
\setlength{\rightmargin}{0pt}%
\setlength{\listparindent}{\parindent}%
\setlength{\itemindent}{\parindent}%
\setlength{\parsep}{0pt plus 1pt}%
\addtolength{\leftmargin}{#1}%
\addtolength{\rightmargin}{#2}%
}\item }{\end{list}}

\def\rightset#1{\Big\downarrow\vcenter{%
                 \rlap{$\scriptstyle #1$}}}

\title[Poincar\'e series for locally symmetric spaces]{Poincar\'e series for non-Riemannian locally symmetric spaces}
\alttitle{S\'eries de Poincar\'e pour les espaces localement sym\'etriques non riemanniens}

\author{Fanny Kassel}
\address{CNRS and Universit\'e Lille 1, Laboratoire Paul Painlev\'e, 59655 Villeneuve d'Ascq Cedex, France}
\email{fanny.kassel@math.univ-lille1.fr}

\author{Toshiyuki Kobayashi}
\address{Kavli IPMU and Graduate School of Mathematical Sciences,
The University of Tokyo, 3-8-1 Komaba, Meguro,
153-8914 Tokyo, Japan}
\email{toshi@ms.u-tokyo.ac.jp}

\begin{document}
\frontmatter

\numberwithin{equation}{chapter}
\numberwithin{table}{chapter}

\begin{abstract}
The discrete spectrum of the Laplacian has been extensively studied on reductive symmetric spaces and on Riemannian locally symmetric spaces.
Here we examine it for the first time in the general setting of non-Riemannian, reductive, locally symmetric spaces.

For any non-Riemannian, reductive symmetric space $X$ on which the discrete spectrum of the Laplacian is nonempty, and for any discrete group of isometries $\Gamma$ whose action on~$X$ is sufficiently proper, we construct $L^2$-eigenfunctions of the Laplacian on $X_{\Gamma}:=\Gamma\backslash X$ for an infinite set of eigenvalues.
These eigenfunctions are obtained as generalized Poincar\'e series, \ie as projections to~$X_{\Gamma}$ of sums, over the $\Gamma$-orbits, of eigenfunctions of the Laplacian on~$X$.

We prove that the Poincar\'e series we construct still converge, and define nonzero $L^2$-functions, after any small deformation of~$\Gamma$, for a large class of groups~$\Gamma$.
In other words, the infinite set of eigenvalues we construct is stable under small deformations.
This contrasts with the classical setting where the nonzero discrete spectrum varies on the Teichm\"uller space of a compact Riemann surface.

We actually construct joint $L^2$-eigenfunctions for the whole commutative algebra of invariant differential operators on~$X_{\Gamma}$.
\end{abstract}

\begin{altabstract}
Le spectre discret du laplacien a \'et\'e beaucoup \'etudi\'e sur les espaces sym\'etriques r\'eductifs, ainsi que sur les espaces localement sym\'etriques riemanniens.
Dans cet article, nous l'\'etudions pour la premi\`ere fois dans le cadre g\'en\'eral des espaces localement sym\'etriques r\'eductifs non riemanniens.

Pour tout espace sym\'etrique r\'eductif $X$ dont le spectre discret du laplacien est non vide, et pour tout groupe discret d'isom\'etries $\Gamma$ dont l'action sur~$X$ est suffisamment propre, nous construisons des fonctions propres $L^2$ du laplacien sur $X_{\Gamma}:=\Gamma\backslash X$ pour une infinit\'e de valeurs propres.
Ces fonctions propres sont obtenues comme s\'eries de Poincar\'e g\'en\'eralis\'ees, c'est-\`a-dire comme projections sur~$X_{\Gamma}$ de sommes, sur les $\Gamma$-orbites, de fonctions propres du laplacien sur~$X$.

Nous montrons que ces s\'eries de Poincar\'e continuent \`a converger et \`a d\'efinir des fonctions $L^2$ non nulles apr\`es n'importe quelle petite d\'eformation de~$\Gamma$, pour une classe importante de groupes~$\Gamma$.
En d'autres termes, l'ensemble infini de valeurs propres que nous construisons est stable par petites d\'eformations de~$\Gamma$.
Ceci contraste avec la situation riemannienne classique o\`u le spectre discret non nul d'une surface de Riemann compacte varie de mani\`ere non constante sur son espace de Teichm\"uller.

Les fonctions propres que nous construisons sont en fait communes \`a toute l'alg\`ebre commutative des op\'erateurs diff\'erentiels invariants sur~$X_{\Gamma}$.
\end{altabstract}


\altkeywords{Laplacien, op\'erateurs diff\'erentiels invariants, spectre discret, vari\'et\'e pseudo-riemannienne, espace sym\'etrique r\'eductif, forme de Clifford--Klein, espace localement sym\'etrique, action propre, s\'eries discr\`etes}

\thanks{T.K. was partially supported by Grant-in-Aid for Scientific Research (B) (22340026) of the JSPS}

\maketitle
\tableofcontents

\mainmatter

\include{StableSpec1-Intro}

\part{Precise description of the results}\label{part1}

\include{StableSpec2-ListsEx}
\include{StableSpec3-Results}

\part{Construction of generalized Poincar\'e series}\label{part2}

\include{StableSpec4-Counting}
\include{StableSpec5-FJ}
\include{StableSpec6-Conv}

\part{Nonvanishing of the generalized Poincar\'e series}\label{part3}

\include{StableSpec7-Analytic}
\include{StableSpec8-Nonzero}

\part{Detailed discussion of some examples}\label{part4}

\include{StableSpec9-AdS}
\include{StableSpec10-OtherEx}

\chapter*{Acknowledgements}

We warmly thank Michael Atiyah, Alex Eskin, Livio Flaminio, Fran\c cois Gu\'eritaud, Colin Guillarmou, Toshihiko Matsuki, Amir Mohammadi, Werner M\"uller, Peter Sarnak, Birgit Speh, Alexei Venkov, Nolan Wallach, and Joseph Wolf for enlightening discussions on various aspects of the paper.
We are grateful to the University of Tokyo for its support through the GCOE program, and to the Institut des Hautes \'Etudes Scientifiques (Bures-sur-Yvette), the Mathematical Sciences Research Institute (Berkeley), the Max Planck Institut f\"ur Mathematik (Bonn), and the departments of mathematics of the University of Chicago and Universit\'e Lille~1 for giving us opportunities to work together in very good conditions.

\backmatter

\include{StableSpec-Biblio}
\end{document}

%% file: StableSpec1-Intro.tex
\chapter{Introduction}\label{sec:intro}

The spectral properties of the Laplacian have been much investigated both on Riemannian locally symmetric spaces $\Gamma\backslash G/K$ and on reductive symmetric spaces~$G/H$.
These are all special cases of pseudo-Riemannian locally symmetric spaces $\Gamma\backslash G/H$, for which the Laplacian continues to exist and be worthy of study.
The aim of this paper is to set up a framework for spectral theory in this general setting and to prove the first results on the discrete spectrum of such spaces under a rank condition on $G/H$ (which makes them non-Riemannian if $G$ is noncompact).
In particular, we construct $L^2$-eigenfunctions for an infinite set of eigenvalues on a large class of spaces (not necessarily compact or of finite volume) and prove some deformation results that have no analogue in the classical Riemannian setting.
More precisely, we work not only with the Laplacian, but with the whole commutative algebra of ``intrinsic'' differential operators on $\Gamma\backslash G/H$, which includes the Laplacian.
Before describing our results in more detail, we first recall the definitions of the main objects.

\section{The main objects}\label{subsec:objects}

A \emph{pseudo-Riemannian metric} on a manifold~$M$ is a smooth, nondegenerate, symmetric bilinear tensor~$g$ of signature $(p,q)$ for some $p,q\in\N$.
As in the Riemannian case (\ie $q=0$), the metric~$g$ induces a second-order differential operator
\begin{equation}\label{eqn:defLaplacian}
\square_M =  \operatorname{div}\,\operatorname{grad}
\end{equation}
called the \emph{Laplacian} or \emph{Laplace--Beltrami operator}.
For instance, for
$$(M,g) = \R^{p,q} := \big(\R^{p+q},\mathrm{d}x_1^2+\dots+\mathrm{d}x_p^2-\mathrm{d}x_{p+1}^2-\dots-\mathrm{d}x_{p+q}^2\big)$$
the Laplacian is
$$\square_{\R^{p,q}} = \frac{\partial^2}{\partial x_1^2} + \dots + \frac{\partial^2}{\partial x_p^2} - \frac{\partial^2}{\partial x_{p+1}^2} - \dots - \frac{\partial^2}{\partial x_{p+q}^2}\,.$$
In general, $\square_M$ is elliptic if $g$ is Riemannian, hyperbolic if $g$ is Lorentzian (\ie $q=1$), and none of these otherwise.
The \emph{discrete spectrum} of~$\square_M$ is its set of eigenvalues corresponding to $L^2$-eigenfunctions:
\begin{equation}\label{eqn:speclaplacian}
\Spec_d(\square_M) := \left\{ t\in\C :\ \exists f\in L^2(M),\ f\neq 0,\ \square_M f=tf\right\} ,
\end{equation}
where $L^2(M)$ is the Hilbert space of square-integrable functions on~$M$ with respect to the Radon measure induced by the pseudo-Riemannian structure.

A \emph{reductive symmetric space} is a homogeneous space $X=G/H$ where $G$ is a real reductive Lie group and $H$ an open subgroup of the group of fixed points of~$G$ under some involutive automorphism~$\sigma$.
The manifold~$X$ naturally carries a pseudo-Riemannian metric, induced by the Killing form of the Lie algebra $\g$ of~$G$ when $G$ is semisimple; therefore, $X$ has a Laplacian~$\square_X$.
Alternatively, $\square_X$ is induced by the Casimir element of the enveloping algebra~$U(\g)$, acting on $C^{\infty}(X)$ by differentiation (see Section~\ref{subsec:Laplacian}).
Let $\D(X)$ be the $\C$-algebra of differential operators on~$X$ that are invariant under the natural $G$-action
$$g\cdot D = \ell_g^{\ast}\circ D\circ (\ell_g^{\ast})^{-1} = \left(f\longmapsto D\big(f^{g^{-1}}\big)^g\right),$$
where we set $\ell_g^{\ast}(f)=f^g:=f(g\,\cdot)$.
The Laplacian~$\square_X$ belongs to~$\D(X)$ and, since $X$ is a symmetric space, $\D(X)$ is commutative (see Section~\ref{subsec:DGH}); we shall consider eigenfunctions for~$\square_X$ that are in fact joint eigenfunctions for~$\D(X)$.

A \emph{locally symmetric space} is a quotient $X_{\Gamma}=\Gamma\backslash X$ of a reductive symmetric space $X=G/H$ by a discrete subgroup $\Gamma$ of~$G$ acting properly discontinuously and freely.
Such a quotient is also called a \emph{Clifford--Klein form of~$X$}.
The proper discontinuity of the action of~$\Gamma$ ensures that $X_{\Gamma}$ is Hausdorff, and it is in fact a manifold since the action is free.
It is locally modeled on~$X$ (it is a complete $(G,X)$-manifold in the sense of Ehresmann and Thurston), hence inherits a pseudo-Riemannian structure from~$X$ and has a Laplacian~$\square_{X_{\Gamma}}$.
Any operator $D\in\D(X)$ induces a differential operator~$D_{\Gamma}$ on~$X_{\Gamma}$ such that the following diagram commutes, where $p_{\Gamma} : X\rightarrow X_{\Gamma}$ is the natural projection.
$$\xymatrix{
C^\infty(X) \ar[r]^D & C^{\infty}(X)\\
C^\infty(X_{\Gamma}) \ar[u]^{p_{\Gamma}^{\ast}} \ar[r]^{D_{\Gamma}} & C^{\infty}(X_{\Gamma}) \ar[u]_{p_{\Gamma}^{\ast}}
}$$
In particular, note that
$$\square_{X_{\Gamma}} = (\square_X)_{_{\Gamma}}.$$
The \emph{discrete spectrum} $\Spec_d(X_{\Gamma})$ of~$X_{\Gamma}$ is defined to be the set of $\C$-algebra homomorphisms $\chi_{\lambda} : \D(X)\rightarrow\C$ such that the space $L^2(X_{\Gamma},\M_{\lambda})$ of weak solutions $f\in L^2(X_{\Gamma})$ to the system
$$D_{\Gamma} f = \chi_{\lambda}(D) f \quad\quad\mathrm{for\ all}\ D\in\D(X) \eqno{(\M_{\lambda})}$$
is nonzero.
(The notation~$\chi_{\lambda}$ will be explained in Section~\ref{subsec:DGH}.)
It~is~the~set of joint eigenvalues for the commutative algebra $\D(X_{\Gamma}):=\{ D_{\Gamma} : D\in\nolinebreak\D(X)\}$, which we think of as the algebra of ``intrinsic'' differential operators on~$X_{\Gamma}$.
The discrete spectrum $\Spec_d(X_{\Gamma})$ refines the discrete spectrum of the Laplacian~$\square_{X_{\Gamma}}$ from \eqref{eqn:speclaplacian} (see Remark~\ref{rem:L2Mlambda}).

\section{The main problems}\label{subsec:program}

Let $X_{\Gamma}=\Gamma\backslash X$ be a locally symmetric space.
We consider the following problems (see \cite{kk10}):

\smallskip\noindent
\textbf{Problem A:} To construct joint $L^2$-eigenfunctions on~$X_{\Gamma}$ corresponding to $\Spec_d(X_{\Gamma})$.

\smallskip\noindent
\textbf{Problem B:} To understand the behavior of $\Spec_d(X_{\Gamma})$ under small deformations of~$\Gamma$ inside~$G$.

\smallskip

By a small deformation we mean a homomorphism close enough to the natural inclusion in the compact-open topology on $\Hom(\Gamma, G)$.

Problems A and~B have been studied extensively in the following two~cases.
\begin{itemize}
  \item Assume that $H$ is compact.
  Then $X$ is Riemannian and the Laplacian~$\square_X$ is elliptic.
  If $X_{\Gamma}$ is compact, then the discrete spectrum of~$\square_{X_{\Gamma}}$ is infinite.
  If furthermore $\Gamma$ is irreducible, then Weil's local rigidity theorem \cite{wei62} states that nontrivial deformations exist only when $X$ is the hyperbolic plane $\HH^2=\SL_2(\R)/\SO(2)$, in which case compact Clifford--Klein forms have an interesting deformation space \emph{modulo} conjugation, namely their Teichm\"uller space.
Viewed as a ``function'' on the Teichm\"uller space, the discrete spectrum varies analytically \cite{bc90} and nonconstantly (Fact~\ref{fact:Teich} below).
On the other hand, for noncompact~$X_{\Gamma}$ the discrete spectrum $\Spec_d(X_{\Gamma})$ may be considerably different depending on whether $\Gamma$ is arithmetic or not (see Selberg \cite{sel54}, Phillips--Sarnak \cite{ps85a,ps85b}, Wolpert \cite{wol94}, etc.).
  \smallskip
  \item Assume that $\Gamma$ is trivial.
  Then the group~$G$ naturally acts on $L^2(X_{\Gamma})=L^2(X)$ and so representation-theoretic methods may be used.
  Spectral analysis on the reductive symmetric space~$X$ with respect to $\D(X)$ is essentially equivalent to finding a Plancherel-type theorem for the irreducible decomposition of the regular representation of~$G$ on $L^2(X)$: see van den Ban--Schlichtkrull \cite{bs05}, Delorme \cite{del98}, and Oshima \cite{osh88a}, as a far-reaching generalization of Harish-Chandra's earlier work \cite{har76} on the regular representation $L^2(G)$ for group manifolds.
  Flensted-Jensen~\cite{fle80} and Matsuki--Oshima~\cite{mo84} showed that $\Spec_d(X)\neq\emptyset$ if and only if the condition $\rank G/H=\rank K/K\cap H$ is satisfied (see Section~\ref{subsec:Lambda+}), in which case they gave an explicit description of $\Spec_d(X)$ (Fact~\ref{fact:decompVlambda}).
  The rest of the spectrum (tempered representations for~$X$, see~\cite{ber88}) is constructed from the discrete spectrum of smaller symmetric spaces by induction.
\end{itemize}

On the other hand, Problems A and~B have not been much studied when $H$ is noncompact, $\Gamma$ is nontrivial, and $\Gamma$ acts properly discontinuously on $X=G/H$, except in the group manifold case when $X_{\Gamma}$ identifies with ${}^\backprime{}\Gamma\backslash{}^\backprime{}G$ for some reductive Lie group~${}^\backprime{}G$ and some discrete subgroup ${}^\backprime{}\Gamma$.
Here we give the first results that do not restrict to this case.
The fact that $H$ is noncompact and $\Gamma$ nontrivial implies new difficulties from several perspectives:
\begin{enumerate}
  \item Analysis: the Laplacian on~$X_{\Gamma}$ is not an elliptic operator anymore;
  \item Geometry: an arbitrary discrete subgroup~$\Gamma$ of~$G$ does not necessarily act properly discontinuously on~$X$;
  \item Representation theory: a discrete subgroup~$\Gamma$ of~$G$ acting properly discontinuously on~$X$ always has infinite covolume in~$G$; moreover, $G$ does not act on $L^2(X_{\Gamma})$ and $L^2(X_{\Gamma})\neq L^2(\Gamma\backslash G)^H$ since $H$ is noncompact.
\end{enumerate}
In particular, point~(1) makes Problem~A nontrivial: we do not know \emph{a priori} whether or not $\Spec_d(X_{\Gamma})\neq\emptyset$, even for compact~$X_{\Gamma}$.

Point~(2) creates some underlying difficulty to Problem~B: we need to consider Clifford--Klein forms~$X_{\Gamma}$ for which the proper discontinuity of the action of~$\Gamma$ on~$X$ is preserved under small deformations of $\Gamma$ in~$G$.
Not all Clifford--Klein forms~$X_{\Gamma}$ have this property (see Example~\ref{ex:propernonstable}), but a large class does (see Example~\ref{ex:deform} and subsequent comments).
The study of small deformations of Clifford--Klein forms in the general setting of reductive homogeneous spaces was initiated in \cite{kob98}; we refer to \cite{con12} for a recent survey in the case of compact Clifford--Klein forms.
An interesting aspect of the case of noncompact~$H$ is that there are more examples where nontrivial deformations of compact Clifford--Klein forms exist than for compact~$H$ (see Sections \ref{subsec:exdeform} and~\ref{subsec:exinfinitevol}).

\section{One approach: constructing generalized Poincar\'e series}

In this paper we investigate Problems A and~B under the assumption \eqref{eqn:rank} that $X$ admits a \emph{maximal compact subsymmetric space of full rank}.
This case is somehow orthogonal to the case of Riemannian symmetric spaces of the noncompact type, where compact subsymmetric spaces are reduced to points.
Assuming that $G$ is noncompact, the group~$H$ is thus noncompact and $X$ non-Riemannian.

By \cite{fle80,mo84}, the assumption \eqref{eqn:rank} is equivalent to the fact that $\Spec_d(X)$ is nonempty.
Our idea is then to construct joint eigenfunctions on~$X_{\Gamma}$ as \emph{generalized Poincar\'e series}
\begin{equation}\label{eqn:phiGamma}
\varphi^{\Gamma} :\ \Gamma x \longmapsto \sum_{\gamma\in\Gamma}\ \varphi(\gamma\!\cdot\!x),
\end{equation}
where the $\varphi$ are well-behaved joint eigenfunctions on~$X$.
The convergence and nonvanishing of the series are nontrivial since the behavior of~$\varphi$ needs to be controlled \emph{in relation to the distribution of $\Gamma$-orbits in the non-Riemannian space~$X$}, for which not much is known since $\Gamma$ is \emph{not} a lattice in~$G$ (see Remark~\ref{rem:lattice}).
From a representation-theoretic viewpoint, we build on Flensted-Jensen's discrete series representations \cite{fle80} for~$X$, whose underlying $(\g,K)$-modules are isomorphic to certain Zuckerman--Vogan derived functor modules $A_{\q}(\lambda)$.
The summation process \eqref{eqn:phiGamma} is different from that of \cite{tw89}: see Remark~\ref{rem:TongWang}.

Our approach enables us to address Problem~A for a large class of Clifford--Klein forms $X_{\Gamma}$ of~$X$, constructing eigenfunctions on~$X_{\Gamma}$ for an explicit, \emph{infinite} set of joint eigenvalues contained in $\Spec_d(X)$.
In particular, this proves that the discrete spectrum $\Spec_d(X_{\Gamma})$ is nonempty.

We also address Problem~B for a large class of Clifford--Klein forms $X_{\Gamma}$.
We prove that the infinite subset of $\Spec_d(X_{\Gamma})$ that we construct is \emph{stable under any small deformation} of $\Gamma$ in~$G$, by establishing that the generalized Poincar\'e series \eqref{eqn:phiGamma} still converges after such a small deformation.
This is achieved by carefully controlling the analytic parameters and using recent results in the deformation theory of proper actions on homogeneous spaces.

One special example to which our results apply is the aforementioned classical quotients $\Gamma\backslash G$, regarded as $\Gamma\times\{ e\} \backslash (G\times G)/\Diag(G)$ where $\Diag(G)$ is the diagonal of $G\times G$.
Our geometric and analytic estimates in this case imply that all discrete series representations of~$G$ with sufficiently regular parameter appear in the regular representation $L^2(\Gamma\backslash G)$, \emph{without replacing $\Gamma$ by a deep enough finite-index subgroup} (Proposition~\ref{prop:intertwining}).
When $\Gamma$ is arithmetic, this improves the non-vanishing results of the classical Poincar\'e series that were known earlier from the asymptotic multiplicity formulas of DeGeorge--Wallach \cite{dw78}, Clozel \cite{clo86}, and Rohlfs--Speh \cite{rs87} or the theta-lifting (see Kazhdan \cite{kaz77}, Borel--Wallach \cite{bw00}, Li \cite{li92}) in automorphic forms; these results required passing to a congruence subgroup that depended on the discrete series representation.
Our approach does not depend on the Arthur--Selberg trace formula or the theta-lifting.
We refer to Remark~\ref{rem:discseries} for more details.

We introduce three main ingredients:
\begin{enumerate}
  \item Uniform analytic estimates for eigenfunctions on~$X$, including their asymptotic behavior at infinity (Proposition~\ref{prop:asym}) and the local behavior near the origin of specific eigenfunctions (Proposition~\ref{prop:psilambda});%
  \item A quantitative understanding of proper actions on reductive homogeneous spaces (notion of \emph{sharpness} --- Definition~\ref{def:sharp});
  \item Counting estimates for points of a given $\Gamma$-orbit in~$X$, both in large ``pseudo-balls'' (Lemma~\ref{lem:growthfornu}) and near the origin (Proposition~\ref{prop:sharp}).
\end{enumerate}
In~(1), our estimates are uniform in the spectral parameter and refine results of Flensted-Jensen \cite{fle80} and Matsuki--Oshima \cite{mo84}.
In~(2), the quantitative approach to properness that we develop builds on the qualitative interpretation of Benoist \cite{ben96} and Kobayashi \cite{kob89,kob96} in terms of a Cartan decomposition $G=KAK$.
In~(3), we relate the natural ``pseudo-distance from the origin'' in the non-Riemannian space~$X$ to the distance from the origin in the Riemannian symmetric space $G/K$ of~$G$ in order to use the growth rate of~$\Gamma$, the Kazhdan--Margulis lemma, and the sharpness constants of~(2).
  Our counting results may be compared to those obtained by Eskin--McMullen \cite{em93} in a different setting, where $\Gamma$ is a lattice in~$G$ (see Remark~\ref{rem:lattice}).

We now state precise results, not on our construction of joint eigenfunctions (for this we refer to Propositions \ref{prop:Vlambda} and~\ref{prop:nonzero}), but on the corresponding eigenvalues, \ie on the discrete spectrum of our locally symmetric spaces.
These results were partially announced in~\cite{kk10}.
Before we state them in full generality, we illustrate them with two simple examples of rank one (see Chapters \ref{sec:exAdS} and~\ref{sec:ex} for more details); in these two examples, the commutative $\C$-algebra $\D(X)$ is generated by the Laplacian~$\square_X$ and therefore $\Spec_d(X_{\Gamma})$ identifies with $\Spec_d(\square_{X_{\Gamma}})$ for any Clifford--Klein form~$X_{\Gamma}$.

\section{Two examples}\label{subsec:introex}

Our first example is the 3-dimensional \emph{anti-de Sitter space} $X=\AdS^3=\SO(2,2)_0/\SO(1,2)_0$, which can be realized as the quadric of $\R^4$ of equation $Q=1$, endowed with the Lorentzian metric induced by $-Q$, where
$$Q(x) := x_1^2 + x_2^2 - x_3^2 - x_4^2.$$
It is a Lorentzian analogue of the real hyperbolic space~$\HH^3$, being a model space for all Lorentzian $3$-manifolds of constant sectional curvature~$-1$ (or \emph{anti-de Sitter $3$-manifolds}).
The Laplacian~$\square_{\AdS^3}$ is a hyperbolic operator of signature $(++-)$; it is given explicitly by
$$\square_{\AdS^3} f \,=\, \square_{\R^{2,2}}\Big(x\longmapsto f\Big({\textstyle\frac{x}{\sqrt{Q(x)}}}\Big)\Big)$$
for all $f\in C^{\infty}(\AdS^3)$, where $f(x/\sqrt{Q(x)})$ is defined on the neighborhood $\{ Q(x)>0\}$ of the quadric $\AdS^3$ in~$\R^4$.
It is equal to $4$ times the Casimir operator of $\g=\so(2,2)$ with respect to the Killing form.
We construct eigenfunctions of the Laplacian on all compact anti-de Sitter $3$-manifolds, for an infinite set of eigenvalues, and prove that this infinite set of eigenvalues is stable under any small deformation of the anti-de Sitter structure.

\begin{theorem}\label{thm:SL2}
The discrete spectrum of any compact anti-de Sitter $3$-manifold is infinite.
Explicitly, if $M=\Gamma\backslash\AdS^3$ with $-\mathrm{I}\notin\Gamma$, then
\begin{equation}\label{eqn:AdS3}
\Spec_d(\square_M) \supset \big\{ \ell(\ell-2):\ \ell\in\N,\ \ell\geq\ell_0\big\}
\end{equation}
for some integer~$\ell_0$; moreover, \eqref{eqn:AdS3} still holds (with the same~$\ell_0$) after any small deformation of the anti-de Sitter structure on~$M$.
\end{theorem}

\smallskip
\noindent
Here $-\mathrm{I}\in\SO(2,2)_0$ is the nontrivial element of the center of $\SO(2,2)_0$, acting on $\AdS^3=\{ x\in\R^4 : Q(x)=1\}$ by the antipodal map $x\mapsto -x$. If $-\mathrm{I}\in\Gamma$, then half of the spectrum survives:
$$\Spec_d(\square_M) \supset \big\{ \ell(\ell-2):\ \ell\in 2\N,\ \ell\geq\ell_0\big\} $$
for some~$\ell_0$.
We actually prove that \eqref{eqn:AdS3} holds (for some explicit~$\ell_0$) for any complete anti-de Sitter $3$-manifold $M=\Gamma\backslash\AdS^3$ with $\Gamma$ finitely generated (Theorem~\ref{thm:SL2precise}).
The stability of eigenvalues under small deformations in Theorem~\ref{thm:SL2} contrasts with the situation in the Riemannian case:

\begin{fact}[\normalfont see {\cite[Th.\,5.14]{wol94}}]\label{fact:Teich}
No nonzero eigenvalue of the Laplacian on a compact Riemann surface is constant on its Teichm\"uller space.
\end{fact}

\noindent
As we shall recall in Chapter~\ref{sec:exAdS}, any compact anti-de Sitter $3$-manifold~$M$ is a circle bundle over some closed hyperbolic surface~$S$ (up to a finite covering); the deformation space of~$M$ contains the Teichm\"uller space of~$S$, and its dimension is actually twice as large.
We shall also prove the existence of an infinite stable spectrum for a large class of noncompact complete anti-de Sitter $3$-manifolds (Corollary~\ref{cor:AdS3deform}).

Our second example is the 3-dimensional complex manifold
$$X = \SU(2,2)/\U(1,2) \simeq \SO(2,4)_0/\U(1,2),$$
which can be realized as the open subset of~$\PP^3\C$ of equation $h>0$, where
$$h(z) := |z_1|^2 + |z_2|^2 - |z_3|^2 - |z_4|^2$$
on~$\C^4$.
The space~$X$ is naturally endowed with an indefinite Hermitian structure of signature $(2,1)$ induced by~$-h$.
The imaginary part of $-h$ endows $X$ with a symplectic structure, making~$X$ into an indefinite K\"ahler manifold.  
The real part of $-h$ gives rise to a pseudo-Riemannian metric of signature~$(4,2)$.
The Laplacian $\square_X$ has signature ($++++-\,-)$ and is given by the following commutative diagram:
$$\xymatrix{
C^\infty(\C^4_{_{h>0}}) \ar[d]_{2h\,\square_{\C^{2,2}}} & \ar[l]_{\quad\pi^{\ast}} C^{\infty}(X) \ar[d]^{\square_X}\\
C^\infty(\C^4_{_{h>0}}) & \ar[l]_{\quad\pi^{\ast}} C^{\infty}(X),
}$$
where
$$\C^4_{_{h>0}} := \{ z\in\C^4 : h(z)>0\} ,$$
where $\pi : \C^4_{_{h>0}}\!\!\rightarrow X$ is the natural projection, and where
$$\square_{\C^{2,2}} := - \frac{\partial^2}{\partial z_1\partial\overline{z}_1} - \frac{\partial^2}{\partial z_2\partial\overline{z}_2} + \frac{\partial^2}{\partial z_3\partial\overline{z}_3} + \frac{\partial^2}{\partial z_4\partial\overline{z}_4}$$
on~$\C^4$.
It is $8$ times the Casimir operator of $\g=\su(2,2)$ with respect to the Killing form.
A natural way to construct Clifford--Klein forms of~$X$ is to notice that $X$ fibers over the quaternionic hyperbolic space $\HH^1_{\mathbf{H}}=\Sp(1,1)/\Sp(1)\times\Sp(1)$, with compact fiber:
$$\begin{matrix}
   \{ z\in\C^4 : h(z) = 1\}
   & \underset{\text{fiber }\U(1)}{\overset{\pi}{\longlongrightarrow}}
   & X = \big\{ [z]\in\PP^3\C : h(z)>0\big\}
\\[.5ex]
   \text{\rotatebox{-90}{$\!\!\!\!\simeq\quad$}}
   & & \rightset{\text{fiber }\Sp(1)/\U(1)}
\\[2ex]
   \!\!\!\big\{ u\in\mathbf{H}^2 : |u_1|^2 - |u_2|^2 = 1 \big\} 
   & \underset{\text{fiber }\Sp(1)}{\longlongrightarrow}
   & \HH^1_{\mathbf{H}} = \big\{ [u]\in\PP^1\mathbf{H} : |u_1|^2 - |u_2|^2 > 0\big\} ,
\end{matrix}$$
where $\mathbf{H}$ is the ring of quaternions and $\PP^1\mathbf{H}$ the quotient of $\mathbf{H}^2\smallsetminus\{ 0\}$ by the diagonal action of $\mathbf{H}\smallsetminus\{ 0\}$ on the right.
The isometry group $\Sp(1,1)$ of the Riemannian symmetric space~$\HH^1_{\mathbf{H}}$ acts transitively on~$X$, and this action is proper since the fiber $\Sp(1)/\U(1)\simeq\mathbb{S}^2$ is compact.
Any torsion-free discrete subgroup~$\Gamma$ of $\Sp(1,1)$ therefore acts properly discontinuously and freely on~$X$; we say that the corresponding Clifford--Klein form~$X_{\Gamma}$ is \emph{standard} (see Definition~\ref{def:standard}).

\begin{theorem}\label{thm:P3}
The discrete spectrum of any standard Clifford--Klein form $X_{\Gamma}$ of $X=\SU(2,2)/\U(1,2)$ is infinite.
Explicitly, for $\Gamma\subset\Sp(1,1)$ there is an integer~$\ell_0$, independent of~$\Gamma$, such that
\begin{equation}\label{eqn:P3}
\Spec_d(\square_{X_{\Gamma}}) \supset \big\{ 2(\ell-2)(\ell+1):\ \ell\in\N,\, \ell\geq\ell_0\big\} ;
\end{equation}
moreover, \eqref{eqn:P3} still holds after any small deformation of~$\Gamma$ in~$\SU(2,2)$.
\end{theorem}

\smallskip
\noindent
We will see in Section~\ref{subsec:3cpx} that there exist interesting small deformations of standard Clifford--Klein forms of $X=\SU(2,2)/\U(1,2)$, both compact and noncompact.
We will compute explicit eigenfunctions.
We refer to \cite{kob09} for further global analysis on~$X$ in connection with branching laws of unitary representations with respect to the restriction $\SU(2,2)\downarrow\Sp(1,1)$.

\section{General results for standard Clifford--Klein forms}\label{subsec:introstand}

We now state our results in the general setting of reductive symmetric spaces $X=G/H$, as defined in Section~\ref{subsec:objects}.
For simplicity we shall assume $G$ to be linear throughout the paper.

An important class of Clifford--Klein forms~$X_{\Gamma}$ of~$X$ that we consider is the \emph{standard} ones.

\begin{definition}\label{def:standard}
A Clifford--Klein form $X_{\Gamma}$ of~$X$ is \emph{standard} if $\Gamma$ is contained in some reductive subgroup~$L$ of~$G$ acting properly on~$X$.
\end{definition}

This generalizes the notion introduced above for $X=\SU(2,2)/\U(1,2)$.
When $L$ acts cocompactly on~$X$, we can obtain compact (\resp finite-volume noncompact) standard Clifford--Klein forms~$X_{\Gamma}$ by taking~$\Gamma$ to be a uniform (\resp nonuniform) lattice in~$L$.
An open conjecture \cite[Conj.\,3.3.10]{ky05} states that any reductive homogeneous space $G/H$ admitting compact Clifford--Klein forms should admit standard ones.

Our first main result in this general setting is the existence of an \emph{infinite discrete spectrum} for all standard Clifford--Klein forms of~$X$ when $\Spec_d(X)\neq\emptyset$.

\begin{theorem}\label{thm:universal}
Let $X=G/H$ be a reductive symmetric space with $\Spec_d(X)\neq\emptyset$, and $L$ a reductive subgroup of~$G$ acting properly on~$X$.
Then $\#\Spec_d(X_{\Gamma})=+\infty$ for any standard Clifford--Klein form $X_{\Gamma}$ with $\Gamma\subset L$.
Moreover, if $L$ is simple ({\resp}semisimple), then there is an infinite subset of $\Spec_d(X)$ that is contained in $\Spec_d(X_{\Gamma})$ for any ({\resp}any torsion-free) $\Gamma\subset L$.
\end{theorem}

We wish to emphasize that when $L$ is semisimple, the infinite subset of the discrete spectrum that we find is \emph{universal}, in the sense that it does not depend on $\Gamma\subset L$.
A universal spectrum does not exist in the Riemannian case (see Fact~\ref{fact:Teich}).
Our proof is constructive; we shall explicitly describe an infinite subset of $\Spec_d(X_{\Gamma})\cap\Spec_d(X)$, independent of $\Gamma\subset L$, in terms of the geometry of~$X$ and of some quantitative estimate of the proper discontinuity of $L$ acting on~$X$ (see Theorem~\ref{thm:precise}).

For $\Gamma=\{ e\}$, the existence of an infinite discrete spectrum was established by Flensted-Jensen \cite{fle80}.
As mentioned above, by \cite{fle80,mo84}, the condition $\Spec_d(X)\neq\emptyset$ is equivalent to the condition $\rank G/H=\rank K/K\cap H$ (see Section~\ref{subsec:Lambda+}), or in other words to the existence of a maximal compact subsymmetric space of~$X$ of full rank.

Our second main result concerns the \emph{stability} of the discrete spectrum of standard compact Clifford--Klein forms~$X_{\Gamma}$ of~$X$ under small deformations of~$\Gamma$ in~$G$.
The set $\Hom(\Gamma,G)$ of group homomorphisms from~$\Gamma$ to~$G$ is endowed with the compact-open topology.
In the following definition, we assume that the group $\varphi(\Gamma)$ acts properly discontinuously and freely on~$X$ for all $\varphi\in\Hom(\Gamma,G)$ in some neighborhood $\mathcal{U}_0$ of the natural inclusion of $\Gamma$ in~$G$ (we shall call this property ``stability for proper discontinuity'').
Under this assumption, $X_{\varphi(\Gamma)}=\varphi(\Gamma)\backslash X$ is a manifold for all $\varphi\in\mathcal{U}_0$ and we can consider the discrete spectrum $\Spec_d(X_{\varphi(\Gamma)})$; recall that it is contained in the set of $\C$-algebra homomorphisms from $\D(X)$ to~$\C$.

\begin{definition}\label{def:stablespec}
We say that $\lambda\in\Spec_d(X_{\Gamma})$ is \emph{stable under small deformations} if there exists a neighborhood $\mathcal{U}\subset\mathcal{U}_0\subset\Hom(\Gamma,G)$ of the natural inclusion such that $\lambda\in\Spec_d(X_{\varphi(\Gamma)})$ for all $\varphi\in\mathcal{U}$.

We say that $X_{\Gamma}$ has an \emph{infinite stable discrete spectrum} if there exists an infinite subset of $\Spec_d(X_{\Gamma})$ that is contained in $\Spec_d(X_{\varphi(\Gamma)})$ for all $\varphi$ in some neighborhood $\mathcal{U}\subset\mathcal{U}_0\subset\Hom(\Gamma,G)$ of the natural inclusion.
\end{definition}

We address the existence of an infinite stable discrete spectrum for standard compact Clifford--Klein forms~$X_{\Gamma}$, where $\Gamma$ is a uniform lattice in some reductive subgroup $L$ of~$G$.
First observe that if $L$ has real rank $\geq 2$ and $\Gamma$ is irreducible, then $\Gamma$ is locally rigid in~$G$ by Margulis's superrigidity theorem \cite[Cor.\,IX.5.9]{mar91}, \textit{i.e.}\ all small deformations of~$\Gamma$ in~$G$ are obtained by conjugation; consequently $\Spec_d(X_{\varphi(\Gamma)})=\Spec_d(X_{\Gamma})$ for all small deformations~$\varphi$, and thus $X_{\Gamma}$ has an infinite stable discrete spectrum by Theorem~\ref{thm:universal}.
Consider the more interesting case when $L$ has real rank~$1$.
Then nontrivial deformations of $\Gamma$ inside~$G$ may exist (see Section~\ref{subsec:exdeform}).
By \cite{kas12}, all compact Clifford--Klein forms~$X_{\Gamma}$ with $\Gamma\subset L$ have the stability property for proper discontinuity; more generally, so do all Clifford--Klein forms~$X_{\Gamma}$ with $\Gamma$ convex cocompact in~$L$.
We prove the existence of an infinite stable discrete spectrum when $\Spec_d(X)\neq\emptyset$.

\begin{theorem}\label{thm:deform}
Let $X=G/H$ be a reductive symmetric space with $\Spec_d(X)\neq\emptyset$, and $L$ a reductive subgroup of~$G$ of real rank~$1$ acting properly on~$X$.
Then $X_{\Gamma}$ has an infinite stable discrete spectrum for any uniform lattice~$\Gamma$ of~$L$, and more generally for any convex cocompact subgroup~$\Gamma$ of~$L$.
\end{theorem}

We recall that a discrete subgroup~$\Gamma$ of~$L$ is said to be \emph{convex cocompact} if it acts cocompactly on some nonempty convex subset of the Riemannian symmetric space of~$L$.
Convex cocompact groups include uniform lattices, but also discrete groups of infinite covolume such as Schottky groups, or for instance quasi-Fuchsian embeddings of surface groups for $L=\PSL_2(\C)$.

Let us emphasize that the small deformations of~$\Gamma$ that we consider in Theorem~\ref{thm:deform} are arbitrary inside~$G$; in particular, in the interesting cases $\Gamma$ does not remain inside a conjugate of~$L$.
A description of an infinite stable discrete spectrum as in Theorem~\ref{thm:deform} will be given in Theorem~\ref{thm:precisedeform}.

In addition to this infinite stable discrete spectrum, standard Clifford--Klein forms~$X_{\Gamma}$ may also have infinitely many eigenvalues that vary under small deformations (see Remark~\ref{rem:negativespec}).
Note that an explicit description of the full discrete spectrum is not known even in the Riemannian case.

\section{General results for sharp Clifford--Klein forms}\label{subsec:resultssharp}

The class of \emph{standard} Clifford--Klein forms that we have just considered is itself contained in a larger class of Clifford--Klein forms, namely those that we call \emph{sharp}.
Let us define this notion (see Sections \ref{subsec:sharp} and~\ref{subsec:exsharp} for more details and examples).

Let $G=K\overline{A_+}K$ be a Cartan decomposition of~$G$, where $K$ is a maximal compact subgroup of~$G$ and $\overline{A_+}$ a closed Weyl chamber in a maximal split abelian subgroup of~$G$.
Any element $g\in G$ may be written as $g=k_1ak_2$ for some $k_1,k_2\in K$ and a unique $a\in\overline{A_+}$; setting $\mu(g)=\log a$ defines a continuous, proper, and surjective map $\mu : G\rightarrow\log\overline{A_+}\subset\aaa:=\mathrm{Lie}(A)$, called the \textit{Cartan projection} associated with the Cartan decomposition $G=K\overline{A_+}K$ (see Example~\ref{ex:SLn} for $G=\SL_n(\R)$).
Let $\Vert\cdot\Vert$ be a norm on~$\aaa$.
We say that a discrete subgroup~$\Gamma$ of~$G$ is \textit{sharp} for $X=G/H$ if there are constants $c>0$ and $C\geq 0$ such that
\begin{equation}\label{eqn:sharpnessintro}
d_{\aaa}(\mu(\gamma),\mu(H)) \geq c\,\Vert\mu(\gamma)\Vert - C
\end{equation}
for all $\gamma\in\Gamma$, where $d_{\aaa}$ is the metric on~$\aaa$ induced by the norm~$\Vert\cdot\nolinebreak\Vert$.
This means that the set~$\mu(\Gamma)$ ``goes away \emph{linearly} from~$\mu(H)$ at infinity''.
This notion does not depend on the choice of the Cartan decomposition $G=K\overline{A_+}K$ nor of the norm~$\Vert\cdot\Vert$.
By the properness criterion of Benoist \cite{ben96} and Kobayashi \cite{kob96}, any sharp discrete subgroup~$\Gamma$ of~$G$ acts properly discontinuously on~$X$ (see Section~\ref{subsec:exsharp}); sharpness should be thought of as a form of \emph{strong proper discontinuity}.
When $\Gamma$ is sharp, we say that the corresponding Clifford--Klein form~$X_{\Gamma}$ is sharp too.

Examples of sharp Clifford--Klein forms are plentiful, as explained in Section~\ref{subsec:exsharp}.
For instance, all \emph{standard} Clifford--Klein forms are sharp.
Also, all known examples of \emph{compact} Clifford--Klein forms of reductive homogeneous spaces are sharp, even when they are nonstandard.
We conjecture that all compact Clifford--Klein forms of reductive homogeneous spaces should be sharp (Conjecture~\ref{conj:sharp}).

We generalize Theorem~\ref{thm:universal} from the standard to the sharp case and prove the following.

\begin{theorem}\label{thm:infinitespec}
Let $X=G/H$ be a reductive symmetric space with $\Spec_d(X)\neq\emptyset$.
Then $\Spec_d(X_{\Gamma})$ is infinite for any sharp Clifford--Klein form~$X_{\Gamma}$ of~$X$.
\end{theorem}

We give an explicit infinite subset of $\Spec_d(X_{\Gamma})$ contained in $\Spec_d(X)$ (see Theorem~\ref{thm:precise}), in terms of the geometry of~$X$, of the ``sharpness constants'' $c,C$ from \eqref{eqn:sharpnessintro}, and of a ``pseudo-distance'' from the origin $x_0=eH$ of $X=G/H$ to the other points of its $\Gamma$-orbit in~$X$.

Recall that on a Riemannian symmetric space all eigenfunctions of the Laplacian are analytic by the elliptic regularity theorem (see \cite[Th.\,3.4.4]{kkk86} for instance).
Here $X$ is non-Riemannian, hence eigenfunctions are not automatically analytic.
We still obtain some regularity result (see Section~\ref{subsec:regular}).

\section{Another approach in certain standard cases}\label{subsec:kk2}

The approach described in this paper is based on the existence of discrete series representation for the reductive symmetric space~$X$ --- a phenomenon specific to the non-Riemannian case, and equivalent to the condition \eqref{eqn:rank}.
It is not the only possible approach for constructing joint eigenfunctions on Clifford--Klein forms $X_{\Gamma}$.
When $\Gamma$ is contained in some reductive subgroup~$L$ of~$G$ acting properly \emph{and transitively} on~$X$, it is possible to construct other eigenfunctions by using the spectral analysis of the Riemannian symmetric space of~$L$ and the restriction to~$L$ of irreducible unitary representations of~$G$ (branching laws for $G\downarrow L$).
More precisely, if $X$ is irreducible and spherical as an $L$-homogeneous space (but does not necessarily satisfy \eqref{eqn:rank}), then it is possible to show that $\Spec_d(X_{\Gamma})\smallsetminus\Spec_d(X)$ is infinite for any uniform lattice $\Gamma$ of~$L$: details will be given in \cite{kk12}.
The following issues are also treated there in some standard cases:
\begin{itemize}
  \item Extension of the Laplacian~$\square_{X_{\Gamma}}$ to a self-adjoint operator on $L^2(X_{\Gamma})$;
  \item Inclusion of analytic functions as a dense subspace of $L^2(X_{\Gamma},\M_{\lambda})$;
  \item Infinite multiplicity of joint eigenvalues for $\D(X_{\Gamma})$;
  \item Relations with branching laws of unitary representations.
\end{itemize}

\section{Organization of the paper}

The paper is divided into four parts.

Part~\ref{part1} is a complement to the introduction.
In Chapter~\ref{sec:listex} we give an overview of various types of examples that our main theorems cover.
In Chapter~\ref{sec:theorems} we introduce some basic notation and give more precise statements of the theorems by means of the Harish-Chandra isomorphism for the ring of invariant differential operators; in particular, we describe an explicit infinite set of eigenvalues, which in the standard case of Theorem~\ref{thm:deform} is both universal and stable under small deformations.

Part~\ref{part2} is devoted to the proof that for all $K$-finite $L^2$-eigenfunct\-ions $\varphi$ on~$X$ with sufficiently regular spectral parameter, the generalized Poincar\'e series \eqref{eqn:phiGamma} converges and yields an $L^2$-eigenfunction on~$X_{\Gamma}$.
The proof is carried out in Chapter~\ref{sec:averaging}, based on both geometric and analytic estimates.
The geometric estimates are established in Chapter~\ref{sec:geometry}, where we quantify proper discontinuity through the notion of sharpness and count points of $\Gamma$-orbits in the non-Riemannian symmetric space~$X$ when $\Gamma$ is a sharp discrete subgroup of~$G$.
The analytic estimates are given in Chapter~\ref{sec:Vlambda}, where we reinterpret some asymptotic estimates of Oshima in terms of the regularity of the spectral parameter and of a ``pseudo-distance from the origin'' in~$X$.

Part~\ref{part3} establishes that, as soon as the spectral parameter~$\lambda$ is regular enough and satisfies some integrality and positivity condition, the generalized Poincar\'e series \eqref{eqn:phiGamma} is nonzero for some good choice of~$\varphi$; this completes the proof of the results stated in Chapters \ref{sec:intro} to~\ref{sec:theorems}.
The functions~$\varphi$ that we consider are $G$-translates of some $K$-finite $L^2$-eigenfunctions $\psi_{\lambda}$ on~$X$ introduced by Flensted-Jensen.
The proof is given in Chapter~\ref{sec:nonzero}, and prepared in Chapter~\ref{sec:FJ}, where we give a finer analytic estimate for~$\psi_{\lambda}$ that controls its behavior, not only at infinity, but also near the origin $x_0:=eH$ of $X=G/H$.
To deduce the nonvanishing of the series \eqref{eqn:phiGamma}, it is then enough to control how the $\Gamma$-orbit through~$x_0$ approaches~$x_0$: this is done in Chapter~\ref{sec:nonzero}, after conjugating $\Gamma$ by some appropriate element of~$G$; for uniformity for standard~$\Gamma$, we use the Kazhdan--Margulis theorem.
We complete the proof of the main theorems in Section~\ref{subsec:proofs}.

Finally, Part~\ref{part4} provides a detailed discussion of some examples, designed to illustrate the general theory in a more concrete way.

\section*{Notation}

In the whole paper, we use the notation $\R_+=(0,+\infty)$ and $\R_{\geq 0}=[0,+\infty)$, as well as $\N_+=\Z\cap\R_+$ and $\N=\Z\cap\R_{\geq 0}$.

%% file: StableSpec2-ListsEx.tex
\chapter{Lists of examples to which the results apply}\label{sec:listex}

There is a variety of locally symmetric spaces $X_{\Gamma}=\Gamma\backslash G/H$ to which Theorems \ref{thm:universal}, \ref{thm:deform}, and~\ref{thm:infinitespec} can be applied.
The aim of this chapter is to provide a brief overview, with an emphasis on compact $X_{\Gamma}$ in the first three sections.
Some of the examples mentioned here will be analyzed in more detail in Chapters \ref{sec:exAdS} and~\ref{sec:ex}.

\section{Symmetric spaces with standard compact Clifford--Klein forms}\label{subsec:exstandard}

We recall the following general construction from \cite{kob89}.
Assume that there exists a reductive subgroup~$L$ of~$G$ acting properly and cocompactly on~$X$.
Then standard compact Clifford--Klein forms $X_{\Gamma}=\Gamma\backslash X$ can be obtained by taking $\Gamma$ to be any torsion-free uniform lattice in~$L$.
Likewise, standard Clifford--Klein forms $X_{\Gamma}$ that are noncompact but of finite volume can be obtained by taking $\Gamma$ to be any torsion-free nonuniform lattice in~$L$.
Uniform lattices of~$L$ always exist and nonuniform lattices exist for semisimple~$L$, by work of Borel--Harish-Chandra, Mostow--Tamagawa, and Borel \cite{bor63}; they all admit torsion-free subgroups of finite index by the Selberg lemma \cite[Lem.\,8]{sel60}.

Here is a list, taken from \cite[Cor.\,3.3.7]{ky05}, of some triples $(G,H,L)$ where $G$ is a simple Lie group,  $X=G/H$ is a reductive symmetric space, and $L$ is a reductive subgroup of~$G$ acting properly and cocompactly on~$X$, with the additional assumption here that $\Spec_d(X)\neq\emptyset$ (so that Theorem~\ref{thm:universal} applies).
We denote by $m$ and~$n$ any integers $\geq 1$ with $m$ even.

\medskip

\begin{center}
\begin{tabular}{|p{1cm}|p{2.5cm}|p{2.5cm}|p{2.5cm}|}
\hline
& \centering $G$ & \centering $H$ & \centering $L$\tabularnewline
\hline
\centering (i) & \centering $\SO(2,2n)$ & \centering $\SO(1,2n)$ & \centering$\U(1,n)$\tabularnewline
\centering (ii) & \centering $\SO(2,2m)$ & \centering $\U(1,m)$ & \centering $\SO(1,2m)$\tabularnewline
\centering (iii) & \centering $\SO(4,4n)$ & \centering $\SO(3,4n)$ & \centering $\Sp(1,n)$\tabularnewline
\centering (iv) & \centering $\SU(2,2n)$ & \centering $\U(1,2n)$ & \centering $\Sp(1,n)$\tabularnewline
\centering (v) & \centering $\SO(8,8)$ & \centering $\SO(7,8)$ & \centering $\Spin(1,8)$\tabularnewline
\hline
\end{tabular}
\end{center}

\smallskip

\begin{center}
\textsc{Table~2.1}
\end{center}

\section{Group manifolds with interesting standard compact Clifford--Klein forms}\label{subsec:listexgroup}

Any reductive group~${}^\backprime{}G$ may be regarded as a homogeneous space under the action of ${}^\backprime{}G\times\!{}^\backprime{}G$ by left and right multiplication; in this way, it identifies with the symmetric space $X=({}^{\backprime}G\times\!{}^{\backprime}G)/\Diag({}^{\backprime}G)$, where $\Diag({}^{\backprime}G)$ denotes the diagonal of ${}^{\backprime}G\times {}^{\backprime}G$.
The condition $\Spec_d(X)\neq\emptyset$, or in other words $\rank G/H=\rank K/K\cap H$ (see Section~\ref{subsec:Lambda+}), is equivalent to the condition
\begin{equation}\label{eqn:rankforgroups}
\rank{}^{\backprime}G=\rank{}^{\backprime}K,
\end{equation}
where ${}^{\backprime}K$ is any maximal compact subgroup of~${}^{\backprime}G$; for ${}^{\backprime}G$ simple, this condition is satisfied if and only if the Lie algebra of~${}^{\backprime}G$ belongs to the following list, where $n$, $p$, and~$q$ are any integers $\geq 1$:
\begin{eqnarray}\label{eqn:listrank}
& \so(p,2q),\, \su(p,q),\, \ssp(p,q),\, \ssp(n,\R),\, \so^{\ast}(2n),& \\
& \e_{6(2)},\, \e_{6(-14)},\, \e_{7(7)},\, \e_{7(-5)},\, \e_{7(-25)},\, \e_{8(-24)},\, \f_{4(4)},\, \f_{4(-20)},\, \g_{2(2)}.& \nonumber
\end{eqnarray}
Standard Clifford--Klein forms~$X_{\Gamma}$ of $X=({}^{\backprime}G\times\!{}^{\backprime}G)/\Diag({}^{\backprime}G)$ can always be obtained by taking $\Gamma$ of the form ${}^{\backprime}\Gamma\times\{ e\} $ or $\{ e\} \times{}^{\backprime}\Gamma$, where ${}^{\backprime}\Gamma$ is a discrete subgroup of~${}^{\backprime}G$.
Then $X_{\Gamma}$ identifies with a usual quotient ${}^{\backprime}\Gamma\backslash{}^{\backprime}G$ or ${}^{\backprime}G/{}^{\backprime}\Gamma$ of~${}^{\backprime}G$ by a discrete subgroup on one side; in particular, $X_{\Gamma}$ has finite volume (\resp is compact) if and only if ${}^{\backprime}\Gamma$ is a lattice (\resp a uniform lattice) in~${}^{\backprime}G$.
Theorem~\ref{thm:universal} applies to such~$X_{\Gamma}$.

It is worth noting that for certain specific groups~${}^{\backprime}G$ of real rank $\geq 2$, there is another (more general) type of standard compact Clifford--Klein forms of~$X$, namely double quotients ${}^{\backprime}\Gamma_1\backslash{}^{\backprime}G/{}^{\backprime}\Gamma_2$ where ${}^{\backprime}\Gamma_1$ and~${}^{\backprime}\Gamma_2$ are discrete subgroups of~${}^{\backprime}G$ \cite{kob93}.
This happens when there exist two reductive subgroups ${}^{\backprime}G_1$ and ${}^{\backprime}G_2$ of~${}^{\backprime}G$ such that ${}^{\backprime}G_1$ acts properly and cocompactly on ${}^{\backprime}G/{}^{\backprime}G_2$.
In this case, the group $L:={}^{\backprime}G_1\times{}^{\backprime}G_2$ acts properly and cocompactly on $X=({}^{\backprime}G\times\!{}^{\backprime}G)/\Diag({}^{\backprime}G)$, and standard Clifford--Klein forms~$X_{\Gamma}$ can be obtained by taking~$\Gamma$ of the form $\Gamma={}^{\backprime}\Gamma_1\times{}^{\backprime}\Gamma_2\subset L$, where ${}^{\backprime}\Gamma_i$ is a discrete subgroup of~${}^{\backprime}G_i$.
Such a Clifford--Klein form~$X_{\Gamma}$ identifies with the double quotient ${}^{\backprime}\Gamma_1\backslash{}^{\backprime}G/{}^{\backprime}\Gamma_2$; it has finite volume (\resp is compact) if and only if ${}^{\backprime}\Gamma_i$ is a lattice (\resp a uniform lattice) in~${}^{\backprime}G_i$ for all $i\in\{ 1,2\}$.
We would like to emphasize that this ``exotic'' $X_{\Gamma}$ is locally modeled on the group manifold~${}^{\backprime}G$ and not on the homogeneous space ${}^{\backprime}G/{}^{\backprime}G_2$.
The following table, obtained from \cite[Cor.\,3.3.7]{ky05}, gives some triples $({}^{\backprime}G,{}^{\backprime}G_1,{}^{\backprime}G_2)$ such that ${}^{\backprime}G$ satisfies the rank condition \eqref{eqn:rankforgroups} and ${}^{\backprime}G_1$ acts properly and cocompactly on ${}^{\backprime}G/{}^{\backprime}G_2$; Theorem~\ref{thm:universal} applies to the corresponding double quotients ${}^{\backprime}\Gamma_1\backslash{}^{\backprime}G/{}^{\backprime}\Gamma_2$.
Here $n$ is any integer $\geq 1$; it does not need to be even in Example~(ii), in contrast with Example~(ii) of Table~2.1.
We note that neither $({}^{\backprime}G,{}^{\backprime}G_1)$ nor $({}^{\backprime}G,{}^{\backprime}G_2)$ has to be a symmetric pair, and that ${}^{\backprime}G_1$ and~${}^{\backprime}G_2$ play symmetric roles.

\medskip

\begin{center}
\begin{tabular}{|p{1cm}|p{4.6cm}|p{2.7cm}|p{2.9cm}|}
\hline
& \centering ${}^{\backprime}G$ & \centering ${}^{\backprime}G_1$ & \centering ${}^{\backprime}G_2$\tabularnewline
\hline
\centering (i) & \centering ${}^{\backprime}G$ with Lie algebra in \eqref{eqn:listrank} & \centering ${}^{\backprime}G$ & \centering $\{ e\} $\tabularnewline
\centering (ii) & \centering $\SO(2,2n)$ & \centering $\SO(1,2n)$ & \centering $\U(1,n)$\tabularnewline
\centering (iii) & \centering $\SO(4,4n)$ & \centering $\SO(3,4n)$ & \centering $\Sp(1,n)$\tabularnewline
\centering (iv) & \centering $\SU(2,2n)$ & \centering $\U(1,2n)$ & \centering $\Sp(1,n)$\tabularnewline
\centering (v) & \centering $\SO(8,8)$ & \centering $\SO(7,8)$ & \centering $\Spin(1,8)$\tabularnewline
\centering (vi) & \centering $\SO(4,4)$ & \centering $\SO(4,3)$ & \centering $\Spin(4,1)$\tabularnewline
\centering (vii) & \centering $\SO(4,4)$ & \centering $\Spin(4,3)$ & \centering $\SO(4,1)\times\SO(3)$\tabularnewline
\centering (viii) & \centering $\SO(4,3)$ & \centering $\mathrm{G}_{2(2)}$ & \centering $\SO(4,1)\times\SO(2)$\tabularnewline
\centering (ix) & \centering $\SO^{\ast}(8)$ & \centering $\U(3,1)$ & \centering $\Spin(1,6)$\tabularnewline
\centering (x) & \centering $\SO^{\ast}(8)$ & \centering $\SO^{\ast}(6)\times\SO^{\ast}(2)$ & \centering $\Spin(1,6)$ \tabularnewline
\hline
\end{tabular}
\end{center}

\smallskip

\begin{center}
\textsc{Table~2.2}
\end{center}

\section[Nontrivial deformations of standard compact forms]{Symmetric spaces with nontrivial deformations of standard compact Clifford--Klein forms}\label{subsec:exdeform}

Theorem~\ref{thm:deform} applies to all the examples in Table~2.1.
However, this theorem is relevant only for standard Clifford--Klein forms~$X_{\Gamma}$ such that $\Gamma$ admits \emph{nontrivial} small deformations inside~$G$, \ie deformations that are not obtained by conjugation.
Such deformations do not always exist when $X_{\Gamma}$ is compact.
We now point out a few examples where they do exist.

Consider Example~(i) of Table~2.1, where $X=\SO(2,2n)/\SO(1,2n)$ is the $(2n+1)$-dimensional \emph{anti-de Sitter space} $\AdS^{2n+1}$.
The group $L=\U(1,n)$ has a nontrivial center $Z(L)$, isomorphic to $\U(1)$.
For certain uniform lattices $\Gamma$ of~$L$, small nontrivial deformations of $\Gamma$ inside $G=\SO(2,2n)$ can be obtained by considering homomorphisms of the form $\gamma\mapsto\gamma\psi(\gamma)$ with $\psi\in\nolinebreak\Hom(\Gamma,Z(L))$ (see \cite{kob98}).
By \cite{rag65} and~\cite{wei64}, any small deformation of $\Gamma$ inside~$G$ is actually of this form, up to conjugation.
The Clifford--Klein forms corresponding to these nontrivial deformations remain standard, but the existence of a stable discrete spectrum given by Theorem~\ref{thm:deform} is not obvious even in this case.
We examine this example in more detail in Section~\ref{subsec:AdS}.

Consider Example~(ii) of Table~2.1, where $X=\SO(2,2m)/\U(1,m)$ has the additional structure of an indefinite K\"ahler manifold (see Section~\ref{subsec:3cpx}).
Here it is actually possible to deform certain standard compact Clifford--Klein forms of~$X$ into \emph{nonstandard} ones.
Indeed, using a \emph{bending} construction due to Johnson--Millson \cite{jm87}, one can obtain small \emph{Zariski-dense} deformations inside $G=\SO(2,2m)$ of certain arithmetic uniform lattices~$\Gamma$ of $L=\SO(1,2m)$ (see \cite[\S\,6]{kas12}): this yields a continuous family of compact Clifford--Klein forms~$X_{\Gamma}$ with $\Gamma$ Zariski-dense in~$G$.
(Recall that a group is said to be Zariski-dense in~$G$ if it is not contained in any proper algebraic subgroup of~$G$.)
Here the $\C$-algebra $\D(X)$ is a polynomial ring in $[\frac{m+1}{2}]$ generators; we discuss the discrete spectrum of~$X_{\Gamma}$ in Section~\ref{subsec:3cpx}.

Finally, consider the ``exotic'' standard compact Clifford--Klein forms\linebreak ${}^{\backprime}\Gamma_1\backslash {}^{\backprime}G/{}^\backprime \Gamma_2$ discussed in Section~\ref{subsec:listexgroup}, for which some examples are given in Table~2.2.
Here is an analog of Theorem~\ref{thm:deform} in this setting (see Proposition~\ref{prop:deformexotic} below for noncompact Clifford--Klein forms): the novelty is the stability of the discrete spectrum, whereas the fact that the quotient remains a manifold under small deformations (\ie stability for proper discontinuity, in the sense of Section~\ref{subsec:introstand}) is a direct consequence of \cite{kas12}.
We refer to Section~\ref{subsec:proofs} for a proof.

\begin{proposition}\label{prop:deformexoticcompact}
Let ${}^{\backprime}G$ be a reductive linear Lie group and let ${}^{\backprime}G_1$ and~${}^{\backprime}G_2$ be two reductive subgroups of~${}^{\backprime}G$ such that ${}^{\backprime}G_1$ acts properly on ${}^{\backprime}G/{}^{\backprime}G_2$.
Any standard Clifford--Klein form
$${}^{\backprime}\Gamma_1\backslash{}^{\backprime}G/\,{}^{\backprime}\Gamma_2 \simeq ({}^{\backprime}\Gamma_1\!\times\!{}^{\backprime}\Gamma_2)\backslash ({}^{\backprime}G\times\!{}^{\backprime}G)/\Diag({}^{\backprime}G),$$
where ${}^{\backprime}\Gamma_i$ is an irreducible uniform lattice of~${}^{\backprime}G_i$ for all $i\in\{ 1,2\}$, remains a manifold after any small deformation of ${}^{\backprime}\Gamma_1\!\times\!{}^{\backprime}\Gamma_2$ inside ${}^{\backprime}G\times\!{}^{\backprime}G$, and it has an infinite stable discrete spectrum if \eqref{eqn:rankforgroups} is satisfied.
\end{proposition}

\noindent
In Examples (ii), (vii), and~(viii) of Table~2.2, certain standard compact Clifford--Klein forms ${}^{\backprime}\Gamma_1\backslash{}^{\backprime}G/{}^{\backprime}\Gamma_2$ admit small nonstandard deformations obtained by bending, similarly to Example~(ii) of Table~2.1 above.
In Example~(i) of Table~2.2, there exist standard compact Clifford--Klein forms ${}^{\backprime}\Gamma_1\backslash{}^{\backprime}G$ with nonstandard small deformations if and only if ${}^{\backprime}G$ has a simple factor that is locally isomorphic to $\SO(1,2n)$ or $\SU(1,n)$ \cite[Th.\,A]{kob98}.

\section{Clifford--Klein forms of infinite volume}\label{subsec:exinfinitevol}

Most examples of Clifford--Klein forms that we have given in Sections \ref{subsec:exstandard} to~\ref{subsec:exdeform} were compact.
However, Theorems \ref{thm:universal}, \ref{thm:deform}, and~\ref{thm:infinitespec} do not require any compactness assumption.
In particular, in Theorems \ref{thm:universal} and~\ref{thm:deform} on the existence of an infinite (universal or stable) spectrum for standard Clifford--Klein forms, we remark that
\begin{itemize}
  \item the reductive group~$L$ does not need to act cocompactly on~$X$ (it could be quite ``small'', for instance locally isomorphic to $\SL_2(\R)$),
  \item the discrete group~$\Gamma$ does not need to be cocompact (nor of finite covolume) in~$L$.
\end{itemize}
Also, in Theorem~\ref{thm:infinitespec}, the sharp Clifford--Klein form~$X_{\Gamma}$ does not need to be compact (nor of finite volume).
Therefore, our theorems apply to much wider settings than those of Tables 2.1 and~2.2; we now discuss some examples.

Firstly, as soon as $\rank_{\R}H<\rank_{\R}G$, there exist infinite \emph{cyclic} discrete subgroups~$\Gamma$ of~$G$ that are sharp for $X=G/H$ \cite{kob89}; Theorem~\ref{thm:infinitespec} applies to the corresponding Clifford--Klein forms~$X_{\Gamma}$.
Even in this case, the existence of an infinite discrete spectrum for~$X_{\Gamma}$ is new.

Secondly, for many~$X$ there exist discrete subgroups $\Gamma$ of~$G$ that are \emph{nonvirtually abelian} (\ie with no abelian subgroup of finite index) and sharp for~$X$; we can again apply Theorem~\ref{thm:infinitespec}.
This is for instance the case for $X=\SO(p+1,q)/\SO(p,q)$ whenever $0<p<q-1$ or $p=q-1$ is odd \cite{ben96}.
Recently, Okuda \cite{oku13} gave a complete list of reductive symmetric spaces $X=G/H$ with $G$ simple that admit Clifford--Klein forms~$X_{\Gamma}$ with $\Gamma$ nonvirtually abelian.
For such symmetric spaces, there always exist interesting sharp examples:
\begin{enumerate}
  \item on the one hand, sharp Clifford--Klein forms~$X_{\Gamma}$ such that $\Gamma$ is a free group, Zariski-dense in~$G$ \cite[Th.\,1.1]{ben96};
  \item on the other hand, standard Clifford--Klein forms~$X_{\Gamma}$ with $\Gamma\subset L$ for some subgroup $L$ of~$G$ isomorphic to $\SL_2(\R)$ or $\PSL_2(\R)$ \cite{oku13}.
\end{enumerate}
In case~(1), the group~$\Gamma$ is in some sense ``as large as possible'', in contrast with case~(2), where it is contained in a proper algebraic subgroup $L$ of~$G$.
In case~(2), we can take~$\Gamma$ to be a surface group embedded in~$L$, therefore admitting nontrivial deformations inside~$L$.
Theorem~\ref{thm:infinitespec} applies to case~(1) and Theorems \ref{thm:universal} and~\ref{thm:deform} to case~(2).

Thirdly, for group manifolds $X=({}^{\backprime}G\times\!{}^{\backprime}G)/\Diag({}^{\backprime}G)$ there are many examples of standard Clifford--Klein forms of infinite volume that admit nontrivial deformations.
As in Section~\ref{subsec:listexgroup}, we can take a pair of reductive subgroups ${}^{\backprime}G_1,{}^{\backprime}G_2$ of~${}^{\backprime}G$ such that ${}^{\backprime}G_1$ acts properly on ${}^{\backprime}G/{}^{\backprime}G_2$, but now we do not require anymore that this action be cocompact.
We consider $X_{\Gamma}={}^{\backprime}\Gamma_1\backslash{}^{\backprime}G/{}^{\backprime}\Gamma_2$ where ${}^{\backprime}\Gamma_i$ is a discrete subgroup of ${}^{\backprime}G_i$ (not necessarily cocompact) and we deform~${}^{\backprime}\Gamma$ inside ${}^{\backprime}G\times\!{}^{\backprime}G$.
Here is an analog of Theorem~\ref{thm:deform} that applies in this setting; we refer to Section~\ref{subsec:proofs} for a proof.

\begin{proposition}\label{prop:deformexotic}
Let ${}^{\backprime}G$ be a reductive linear Lie group satisfying \eqref{eqn:rankforgroups} and let ${}^{\backprime}G_1$ and~${}^{\backprime}G_2$ be two reductive subgroups of~${}^{\backprime}G$ such that ${}^{\backprime}G_1$ acts properly on ${}^{\backprime}G/{}^{\backprime}G_2$.
Consider a standard Clifford--Klein form
$${}^{\backprime}\Gamma_1\backslash{}^{\backprime}G/{}^{\backprime}\Gamma_2 \,\simeq\, ({}^{\backprime}\Gamma_1\!\times\!{}^{\backprime}\Gamma_2)\backslash ({}^{\backprime}G\times\!{}^{\backprime}G)/\Diag({}^{\backprime}G),$$
where ${}^{\backprime}\Gamma_i$ is a discrete subgroup of~${}^{\backprime}G_i$ for all~$i$.
\begin{enumerate}
  \item If ${}^{\backprime}G_1$ has real rank~$1$ and ${}^{\backprime}\Gamma_1$ is convex cocompact in~${}^{\backprime}G_1$, then there exists an infinite subset~$I$ of $\Spec_d({}^{\backprime}\Gamma_1\backslash{}^{\backprime}G/{}^{\backprime}\Gamma_2)$ and a neighborhood ${}^{\backprime}\mathcal{U}\subset\Hom({}^{\backprime}\Gamma_1,{}^{\backprime}G\times Z_{{}^{\backprime}\!G}({}^{\backprime}\Gamma_2))$ of the natural inclusion such that ${}^{\backprime}\varphi({}^{\backprime}\Gamma_1)\backslash{}^{\backprime}G/{}^{\backprime}\Gamma_2$ is a manifold and $I\subset\Spec_d({}^{\backprime}\varphi({}^{\backprime}\Gamma_1)\backslash{}^{\backprime}G/{}^{\backprime}\Gamma_2)$ for all ${}^{\backprime}\varphi\in{}^{\backprime}\mathcal{U}$.
  \item If ${}^{\backprime}G_i$ has real rank~$1$ and ${}^{\backprime}\Gamma_i$ is convex cocompact in~${}^{\backprime}G_i$ for all $i\in\nolinebreak\{ 1,2\}$, then the standard Clifford--Klein form ${}^{\backprime}\Gamma_1\backslash{}^{\backprime}G/{}^{\backprime}\Gamma_2$ remains a manifold after any small deformation of ${}^{\backprime}\Gamma_1\times{}^{\backprime}\Gamma_2$ inside ${}^{\backprime}G\times\!{}^{\backprime}G$ and it has an infinite stable discrete spectrum in the sense of Definition~\ref{def:stablespec}.
\end{enumerate}
\end{proposition}

%% file: StableSpec3-Results.tex
\chapter{Quantitative versions of the results}\label{sec:theorems}

In this chapter, we give some quantitative estimates of Theorems \ref{thm:universal}, \ref{thm:deform}, and \ref{thm:infinitespec} (Section~\ref{subsec:infiniteset}) and discuss the regularity of our eigenfunctions (Section~\ref{subsec:regular}).
We first fix some notation that will be used throughout the paper and recall some useful classical facts (Sections \ref{subsec:DGH} to~\ref{subsec:Lambda+}).

\section{Invariant differential operators}\label{subsec:DGH}

\emph{In the whole paper, $G$ denotes a real reductive linear Lie group and $H$ an open subgroup of the group of fixed points of~$G$ under some involutive automorphism~$\sigma$.}
We denote their respective Lie algebras by $\g$ and~$\h$.
Without loss of generality, we may and will assume that $G$ is connected; indeed, we only need to consider the discrete spectrum of one connected component of $X=G/H$.

In this paragraph, we recall some classical results on the structure of the algebra $\D(X)$ of $G$-invariant differential operators on~$X$.
We refer the reader to \cite[Ch.\,II]{hel00} for proofs and more details.

Let $U(\g_{\C})$ be the enveloping algebra of the complexified Lie algebra $\g_{\C}:=\g\otimes_{\R}\C$ and $U(\g_{\C})^H$ the subalgebra of $\Ad_G(H)$-invariant elements (it contains in particular the center $Z(\g_{\C})$ of $U(\g_{\C})$). 
Recall that $U(\g_{\C})$ acts on $C^{\infty}(G)$ by differentiation on the right, with
$$\big((Y_1\cdots Y_m)\cdot f\big)(g) = \frac{\mathrm{d}}{\mathrm{d}t_1}\Big|_{t_1=0}\,\cdots\ \frac{\mathrm{d}}{\mathrm{d}t_m}\Big|_{t_m=0}\ f\big(g\exp(t_1Y_1)\cdots\exp(t_mY_m)\big)$$
for all $Y_1,\dots,Y_m\in\g$, all $f\in C^{\infty}(G)$, and all $g\in G$.
This gives an isomorphism between $U(\g_{\C})$ and the ring of left-invariant differential operators on~$G$.
By identifying the set of smooth functions on~$X$ with the set of right-$H$-invariant smooth functions on~$G$, we obtain a $\C$-algebra homomorphism
$$p :\ U(\g_{\C})^H \,\longrightarrow\, \D(X).$$
This homomorphism is surjective, with kernel $U(\g_{\C})\h_{\C}\cap\nolinebreak U(\g_{\C})^H$ \cite[Ch.\,II, Th.\,4.6]{hel00}, hence it induces an algebra isomorphism
\begin{equation}\label{eqn:DX}
U(\g_{\C})^H/U(\g_{\C})\h_{\C} \cap U(\g_{\C})^H \,\overset{\sim}\longrightarrow\, \D(X).
\end{equation}

Let $\g=\h+\q$ be the decomposition of~$\g$ into eigenspaces of~$\mathrm{d}\sigma$, with respective eigenvalues $+1$ and~$-1$.
\emph{In the whole paper, we fix a maximal semisimple abelian subspace~$\jj$ of $\sqrt{-1}\,\q$.}
The integer
\begin{equation}\label{eqn:defrank}
\rank G/H := \dim_{\R} \jj
\end{equation}
does not depend on the choice of~$\jj$.  
Geometrically, if $x_0$ denotes the image of $H$ in $X=G/H$, then $\exp(\sqrt{-1}\,\jj)\cdot x_0$ is a maximal flat totally geodesic submanifold of~$X$, where ``flat'' means that the induced pseudo-Riemannian metric is nondegenerate and that the curvature tensor vanishes (see \cite[Ch.\,XI, \S\,4]{kn69}).
Let $W$ be the Weyl group of the restricted root system $\Sigma(\g_{\C},\jj_{\C})$ of~$\jj_{\C}$ in~$\g_{\C}$, and let $S(\jj_{\C})^W$ be the subalgebra of $W$-invariant elements in the symmetric algebra~$S(\jj_{\C})$ of~$\jj_{\C}$.
The important fact that we will use is the following.

\begin{fact}\label{fact:3.1}
The algebra~$\D(X)$ of $G$-invariant differential operators on~$X$ is a polynomial algebra in $r:=\rank G/H$ generators.
It naturally identifies with $S(\jj_{\C})^W$, and the set of $\C$-algebra homomorphisms from $\D(X)$ to~$\C$ identifies with $\jj_{\C}^{\ast}/W$, where $\jj_{\C}^{\ast}$ is the dual vector space of~$\jj_{\C}$.
\end{fact}

Let us explicit these identifications.
Let $\Sigma^+(\g_{\C},\jj_{\C})$ be a system of positive roots in $\Sigma(\g_{\C},\jj_{\C})$ and let
$$\n_{\C} = \bigoplus_{\alpha\in\Sigma^+(\g_{\C},\jj_{\C})} (\g_{\C})_{\alpha}$$
be the sum of the corresponding root spaces, where
$$(\g_{\C})_{\alpha} := \{ Y\in\g_{\C},\ [T,Y]=\alpha(T)Y\ \forall T\!\in\jj\} .$$
The complexified Iwasawa decomposition $\g_{\C}=\h_{\C}+\jj_{\C}+\n_{\C}$ holds, implying that $U(\g_{\C})$ is the direct sum of $U(\jj_{\C})\simeq S(\jj_{\C})$ and $\h_{\C} U(\g_{\C})+U(\g_{\C})\n_{\C}$.
Let $p' : U(\g_{\C})\rightarrow S(\jj_{\C})$ be the projection onto~$S(\jj_{\C})$ with respect to this direct sum and let $p'' : U(\g_{\C})\rightarrow S(\jj_{\C})$ be the ``shifted projection'' given by
$$\langle p''(u),\lambda\rangle = \langle p'(u),\lambda-\rho\rangle$$
for all $\lambda\in\jj_{\C}^{\ast}$, where
$$\rho := \frac{1}{2} \sum_{\alpha\in\Sigma^+(\g_{\C},\jj_{\C})} \dim_{\C} (\g_{\C})_\alpha\,\alpha \in \jj_{\C}^{\ast}$$
is half the sum of the elements of $\Sigma^+(\g_{\C},\jj_{\C})$, counted with root multiplicities.
The restriction of~$p''$ to $U(\g_{\C})^H$ is independent of the choice of $\Sigma^+(\g_{\C},\jj_{\C})$ and induces an isomorphism
$$U(\g_{\C})^{\h_{\C}}/U(\g_{\C})\h_{\C} \cap U(\g_{\C})^{\h_{\C}} \,\overset{\sim}\longrightarrow\, S(\jj_{\C})^W$$
\cite[Ch.\,II, Th.\,5.17]{hel00}.
If $H$ is connected, then $U(\g_{\C})^{\h_{\C}}=U(\g_{\C})^H$ and, using \eqref{eqn:DX} above, we obtain the following commutative diagram.
$$\xymatrix@C=1mm{
\D(X) & \ar[l]_p U(\g_{\C})^H \ar[d] \ar[r]^{p''} & S(\jj_{\C})^W\\
& \ar[ul]_{\sim} U(\g_{\C})^H/U(\g_{\C})\h_{\C} \cap U(\g_{\C})^H \ar[ur]^{\sim} &
}$$
Thus we have a $\C$-algebra isomorphism $\Psi: \D(X)\!\overset{\scriptscriptstyle\sim\,}{\rightarrow}S(\jj_{\C})^W$ (\emph{Harish-Chandra isomorphism}).
In the general case when $H$ is not necessarily connected, we still have an isomorphism $\Psi : \D(X)\!\overset{\scriptscriptstyle\sim\,}{\rightarrow}S(\jj_{\C})^W$ by the following remark.

\begin{remark}
The $\C$-algebra $\D(X)$ is isomorphic to $\D(G/H_0)$, where $H_0$ denotes the identity component of~$H$.
\end{remark}

\begin{proof}
There is a natural injective algebra homomorphism $\D(X)\hookrightarrow\D(G/H_0)$ induced by the natural projection $G/H_0\rightarrow X$.
To see that this homomorphism is surjective, it is sufficient to see that $H$ acts trivially on $\D(G/H_0)$.
This follows from the fact that the quotient field of $\D(G/H_0)$ is isomorphic to that of $p(Z(\g_{\C}))$ \cite[Ch.\,III, Th.\,3.16]{hel00} (where $p : U(\g_{\C})^{H_0}\rightarrow\D(G/H_0)$ is given by the diagram above for~$H_0$) and from the fact that $H$ acts trivially on~$Z(\g_{\C})$ and $p$ is $H$-equivariant.
\end{proof}

By the Harish-Chandra isomorphism $\Psi : \D(X)\!\overset{\scriptscriptstyle\sim\,}{\rightarrow}S(\jj_{\C})^W$, the $\C$-algebra $\D(X)$ is a commutative algebra generated by $r:=\dim_{\R}\jj=\rank G/H$ homogeneous, algebraically independent differential operators $D_1,\dots,D_r$.
If we identify $S(\jj_{\C})$ with the ring of polynomial functions on~$\jj_{\C}^{\ast}$, then any homomorphism from $\D(X)$ to~$\C$ is of the form 
$$\chi_{\lambda} :\, D \longmapsto \langle\Psi(D),\lambda\rangle$$
for some $\lambda\in\jj_{\C}^{\ast}$, and $\chi_{\lambda}=\chi_{\lambda'}$ if and only if $\lambda'\in W\cdot\lambda$.
By construction, any $D\in\D(X)$ acts on the constant functions on~$X$ by multiplication by the scalar~$\chi_{\rho}(D)$.
\emph{From now on, we identify the set of $\C$-algebra homomorphisms from $\D(X)$ to~$\C$ with $\jj_{\C}^{\ast}/W$; in particular, we see $\Spec_d(X)$ (or $\Spec_d(X_{\Gamma})$ for any Clifford--Klein form~$X_{\Gamma}$) as a subset of $\jj_{\C}^{\ast}/W$:}
$$\Spec_d(X_{\Gamma}) = \big\{ \lambda\in\jj_{\C}^{\ast}/W :\, L^2(X_{\Gamma},\M_{\lambda})\neq\{ 0\} \big\} ,$$
where $L^2(X_{\Gamma},\M_{\lambda})$ is the space of weak solutions $f\in L^2(X_{\Gamma})$ to the system
$$D_{\Gamma} f = \chi_{\lambda}(D)f \quad\quad\mathrm{for\ all}\ D\in\D(X)
\eqno{(\M_{\lambda})}.$$

\begin{remark}\label{rem:L2Mlambda}
When $r=\rank G/H>1$, the space $L^2(X_{\Gamma},\M_{\lambda})$ is in general strictly contained in the space of $L^2$-eigenfunctions of the Laplacian~$\square_{X_{\Gamma}}$ (details will be given in \cite{kk12}).
\end{remark}

\section{The Laplacian}\label{subsec:Laplacian}

\emph{In the whole paper, we fix a Cartan involution~$\theta$ of~$G$ commuting with~$\sigma$ and let $K=G^{\theta}$ be the corresponding maximal compact subgroup of~$G$, with Lie algebra~$\kk$.}
Let $\g=\kk+\p$ be the corresponding Cartan decomposition, \ie the decomposition of~$\g$ into eigenspaces of $\mathrm{d}\theta$ with respective eigenvalues $+1$ and $-1$.
We fix a $G$-invariant nondegenerate symmetric bilinear form $B$ on~$\g$ with the following properties: $B$ is positive definite on $\p$, negative definite on~$\kk$, and $\p$ and~$\kk$ are orthogonal for~$B$.
If $G$ is semisimple, we can take $B$ to be the Killing form $\kappa$ of~$\g$.

On the one hand, since the involution~$\sigma$ commutes with the Cartan involution~$\theta$, the form~$B$ is nondegenerate on $\h\times\h$, and induces an $H$-invariant nondegenerate symmetric bilinear form on $\g/\h$.
By identifying the tangent space $T_{x_0}(G/H)$ at $x_0=eH\in G/H$ with $\g/\h$ and using left translations, we obtain a $G$-invariant pseudo-Riemannian structure on $X=G/H$.
We then define the Laplacian $\square_X$ as in \eqref{eqn:defLaplacian} with respect to this pseudo-Riemannian structure.

On the other hand, the form~$B$ defines an isomorphism $\g^{\ast}\simeq\g$, yielding a canonical element in $(\g\otimes\g)^G$ corresponding to the identity under the isomorphism $(\g^{\ast}\otimes\g)^G\simeq\Hom_G(\g,\g)$.
This element projects to the Casimir element of $U(\g_{\C})$, which lies in the center $Z(\g_{\C})$.
It gives a differential operator of order two on~$X$, the \emph{Casimir operator}, whose actions by differentiation on the left and on the right coincide.
Since $X$ is a symmetric space, the Casimir operator on~$X$ coincides with~$\square_X$.
(We refer to \cite[Ch.\,II, Exer.\,A.4]{hel00} for the case when $H$ is a maximal compact subgroup of~$G$; a proof for the general case goes similarly.)

We now explicit the eigenvalues of~$\square_X$.
For this we note that $B$ is nondegenerate on any $\theta$-stable subspace of~$\g$.
In particular, if $\jj$ is $\theta$-stable (which will always be the case below), then $B$ induces a nondegenerate $W$-invariant bilinear form $(\cdot,\cdot)$ on~$\jj^{\ast}$, which we extend to a complex bilinear form $(\cdot,\cdot)$ on~$\jj_{\C}^{\ast}$.

\begin{fact}\label{fact:eigenvalues}
If $f\in C^{\infty}(X)$ satisfies $(\M_{\lambda})$ for some $\lambda\in\jj_{\C}^{\ast}$, then
$$\square_X f = \big((\lambda,\lambda) - (\rho,\rho)\big)\,f.$$
\end{fact}

\noindent
Indeed, this follows from the above description of the Harish-Chandra isomorphism; one can also use \cite[Ch.\,II, Cor.\,5.20]{hel00} and the fact that $\D(X)\simeq\D(X^d)$, where $X^d$ is a Riemannian symmetric space of the noncompact type with the same complexification as~$X$ (see Section~\ref{subsec:realforms}).

\section{Some further basic notation}\label{subsec:Lambda+}

We now fix some additional notation that will be used throughout the paper.

We first recall that the connected reductive group~$G$ is the almost product of its connected center $Z(G)_0$ and of its commutator subgroup~$G_s$, which is semisimple.
The group~$G_s$ itself is the almost product of finitely many (nontrivial) connected simple normal subgroups, called the \emph{simple factors} of~$G$.
The connected center $Z(G)_0$ is isomorphic to $\R^a\times(\mathbb{S}^1)^b$ for some integers $a,b\in\N$.
Recall that $G$ admits a unique maximal compact normal subgroup~$G_c$, which is generated by the compact simple factors of~$G$, by the center $Z(G_s)$ of~$G_s$, and by the compact part of $Z(G)_0$.
The group~$G$ is said to have \emph{no compact factor} if $G_c=Z(G_s)$.

Flensted-Jensen \cite{fle80} and Matsuki--Oshima \cite{mo84} proved that $\Spec_d(X)$ is nonempty if and only if
\begin{equation}\label{eqn:rank}
\rank G/H = \rank K/H\cap K,
\end{equation}
where the rank is defined as in~\eqref{eqn:defrank}.
This is equivalent to the fact that $X$ admits a maximal compact subsymmetric space of full rank, namely $K/H\cap\nolinebreak K$.
Under the rank condition~\eqref{eqn:rank}, we may and do assume that \emph{the maximal abelian subspace $\jj$ of Section~\ref{subsec:DGH} is contained in $\sqrt{-1}(\kk\cap\q)$}.
Then $\jj$ is $\theta$-stable, all restricted roots $\alpha\in\Sigma(\g_{\C},\jj_{\C})$ take real values on~$\jj$, and the $W$-invariant bilinear form $(\cdot,\cdot)$ on~$\jj^{\ast}$ from Section~\ref{subsec:Laplacian} is positive definite.

\emph{We fix once and for all a positive system $\Sigma^+(\kk_{\C},\jj_{\C})$ of restricted roots of~$\jj_{\C}$ in~$\kk_{\C}$, which we will keep until the end of the paper;} we denote by~$\rho_c$ half the sum of the elements of $\Sigma^+(\kk_{\C},\jj_{\C})$, counted with root multiplicities.
We now introduce some notation $\Lambda_+$, $\Lambda$, and~$\Lambda^J$ that will be used throughout the paper.
We start by extending $\jj$ to a maximal abelian subspace $\widetilde{\jj}$ of $\sqrt{-1}\,\kk$.
Let $\Delta^+(\kk_{\C},\widetilde{\jj}_{\C})$ be a positive system of roots of $\widetilde{\jj}_{\C}$ in~$\kk_{\C}$ such that the restriction map $\alpha\mapsto\alpha|_{\jj_{\C}}$ sends $\Delta^+(\kk_{\C},\widetilde{\jj}_{\C})$ to $\Sigma^+(\kk_{\C},\jj_{\C})\cup\{0\}$.
We identify the set of irreducible finite-dimensional representations of~$\kk_{\C}$ with the set of dominant integral weights with respect to the positive system $\Delta^+(\kk_{\C},\widetilde{\jj}_{\C})$.
As a subset, we denote by
\begin{equation}\label{eqn:Lambda+}
\Lambda_+ \equiv \Lambda_+(K/H\cap K)
\end{equation}
the set of irreducible representations of~$K$ with nonzero $(H\cap K)$-fixed vectors; it is the support of the regular representation of $K$ on $L^2(K/H\cap K)$ by Frobenius reciprocity.
 
\begin{remark}\label{rem:Lmdj}
By definition, $\Lambda_+$ is a set of dominant integral elements in the dual of $\widetilde{\jj}=\jj+(\widetilde{\jj}\cap\h_{\C})$.
However, we can regard it as a subset of~$\jj^{\ast}$ by the Cartan--Helgason theorem \cite[Th.\,3.3.1.1]{war72}.
\end{remark}

\noindent
We set
\begin{equation}\label{eqn:Lambda}
\Lambda := \Z\text{-span}(\Lambda_+)\ \subset \jj^{\ast}.
\end{equation}
For any finite subgroup $J$ of the center $Z(K)$ of~$K$, let $\widehat{K/J}$ be the set of (highest weights of) irreducible representations of~$K$ that factor through $K/J$ and let
\begin{equation}\label{eqn:LambdaJ}
\Lambda^J := \Z\text{-span}\big(\Lambda_+\cap\widehat{K/J}\big).
\end{equation}
We note that the $\Z$-module $\Lambda^J$ has finite index in~$\Lambda$.
Indeed, if $J$ has cardinal~$m$, then $\Lambda^J$ contains $m\Lambda=\{ m\lambda : \lambda\in\Lambda\}$ since $(m\lambda)(z)=\lambda(z^m)=\nolinebreak 1$ for all $\lambda\in\Lambda_+$ and $z\in J$.
If $J\subset J'$, then $\Lambda^J\supset\Lambda^{J'}$; in particular, for any discrete subgroup~$\Gamma$ of~$G$ we have
\begin{equation}\label{eqn:inclusionLambda+}
\Lambda \supset \Lambda^{\Gamma\cap Z(G_s)} \supset \Lambda^{Z(G_s)},
\end{equation}
where, as before, $Z(G_s)$ is the center of the commutator subgroup $G_s$ of~$G$.

\begin{remark}\label{rem:LambdaJ}
If $J\subset H$, then $\Lambda=\Lambda^J$.
In particular, if $Z(G_s)\subset H$, then $\Lambda^{\Gamma\cap Z(G_s)}=\Lambda$ for any subgroup $\Gamma$ of~$G$.
\end{remark}

\noindent
Indeed, if $J\subset H$, then $J$ acts trivially on $K/H\cap K$, hence the regular representation of~$K$ on $L^2(K/H\cap K)$ factors through $K/J$.

\medskip

Any choice of a positive system $\Sigma^+(\g_{\C},\jj_{\C})$ of restricted roots of~$\jj_{\C}$ in~$\g_{\C}$ containing $\Sigma^+(\kk_{\C},\jj_{\C})$ will determine:
\begin{enumerate}
  \item a basis $\{ \alpha_1,\dots,\alpha_r\} $ of $\Sigma(\g_{\C},\jj_{\C})$,
  \item a positive Weyl chamber
  $$\jj^{\ast}_+ := \big\{ \lambda \in \Hom_{\R}(\jj,\R) :\, (\lambda,\alpha) > 0\text{ for all $\alpha\in\Sigma^+(\g_{\C},\jj_{\C})$}\big\} ,$$
  with closure~$\overline{\jj^{\ast}_+}$ in~$\jj^{\ast}$,
  \item an element $\rho\in\jj^{\ast}_+$, defined as half the sum of the elements of $\Sigma^+(\g_{\C},\jj_{\C})$, counted with root multiplicities,
  \item a function $d : \overline{\jj^{\ast}_+}\rightarrow\R_+$ measuring the ``weighted distance'' from~$\lambda$ to the walls of~$\jj^{\ast}_+$, given by
$$d(\lambda) := \min_{1\leq i\leq r} \frac{(\lambda,\alpha_i)}{(\alpha_i,\alpha_i)} \geq 0.$$
\end{enumerate}
The function~$d$ does not depend on the choice of the $W$-invariant inner product $(\cdot,\cdot)$ that we made in Section~\ref{subsec:Laplacian}; we extend it as a $W$-invariant function on~$\jj^{\ast}$.
We note that any element of~$\jj^{\ast}$ enters the positive Weyl chamber~$\jj^{\ast}_+$ if we add $t\rho$ for some sufficiently large $t>0$; conversely, $d(\lambda)$ measures to which extent $\lambda-t\rho$ remains in~$\jj^{\ast}_+$ for $\lambda\in\jj^{\ast}_+$:
 
\begin{observation}\label{obs:dwall}
For all $\lambda\in\overline{\jj^{\ast}_+}$,
$$\lambda- \frac{d(\lambda)}{m_{\rho}}\,\rho \,\in\, \overline{\jj^{\ast}_+},$$
where we set
\begin{equation}\label{eqn:rhomax}
m_{\rho} := \max_{1\leq i\le r} \frac{(\rho,\alpha_i)}{(\alpha_i, \alpha_i)}.
\end{equation}
\end{observation}

\begin{proof}
For any simple root $\alpha_i$ ($1\leq i\leq r$),
$$\frac{\big(\lambda- \frac{d(\lambda)}{m_{\rho}}\rho, \alpha_i\big)}{(\alpha_i, \alpha_i)} \,\geq\, d(\lambda) -\frac{d(\lambda)}{m_{\rho}}\,m_{\rho} \,=\, 0.\qedhere$$
\end{proof}

\noindent
We note that if $\jj_{\C}$ is a Cartan subalgebra of~$\g_{\C}$, then $d(\rho)=m_{\rho}=1/2$.

\section{Precise statements of the main theorems}\label{subsec:infiniteset}

With the above notation, here is a more precise statement of Theorems \ref{thm:universal} and~\ref{thm:infinitespec} on the existence of an infinite discrete spectrum, which is ``universal'' for standard Clifford--Klein forms.
We choose a positive system $\Sigma^+(\g_{\C},\jj_{\C})$ containing the fixed positive system $\Sigma^+(\kk_{\C},\jj_{\C})$ of Section~\ref{subsec:Lambda+}; this determines a positive Weyl chamber~$\jj^{\ast}_+$ and an element $\rho\in\jj^{\ast}_+$.

\begin{theorem}\label{thm:precise}
Suppose that $G$ is connected, that $H$ does not contain any simple factor of~$G$, and that the rank condition \eqref{eqn:rank} holds.
\begin{enumerate}
  \item For any sharp Clifford--Klein form~$X_{\Gamma}$ with $\Gamma\cap G_c\subset Z(G_s)$, there is a constant $R\geq 0$ such that
  $$\left\{ \lambda\in\jj^{\ast}_+\cap\big(2\rho_c-\rho+\Lambda^{\Gamma\cap Z(G_s)}\big) :\ d(\lambda)>R\right\} \subset \Spec_d(X_{\Gamma}).$$
  \item The constant~$R$ can be taken uniformly for standard Clifford--Klein forms: given any reductive subgroup~$L$ of~$G$, with a compact center and acting properly on~$X$, there is a constant $R>0$ such that
  $$\left\{ \lambda\in\jj^{\ast}_+\cap\big(2\rho_c-\rho+\Lambda^{\Gamma\cap Z(G_s)}\big) :\ d(\lambda)>R\right\} \subset \Spec_d(X_{\Gamma})$$
  for all discrete subgroups~$\Gamma$ of~$L$ with $\Gamma\cap L_c\subset Z(G_s)$ (this includes all torsion-free discrete subgroups $\Gamma$ of~$L$); in particular, by \eqref{eqn:inclusionLambda+},
  $$\left\{ \lambda\in\jj^{\ast}_+\cap\big(2\rho_c-\rho+\Lambda^{Z(G_s)}\big) :\ d(\lambda)>R\right\} \subset \Spec_d(X_{\Gamma})$$
  for all such~$\Gamma$.
\end{enumerate}
\end{theorem}

As in Section~\ref{subsec:Lambda+}, we denote by $G_c$ (\resp by~$L_c$) the maximal compact normal subgroup of~$G$ (\resp of~$L$), and by $Z(G_s)$ the center of the semisimple part of~$G$.
The $\Z$-modules $\Lambda^{\Gamma\cap Z(G_s)}$ and $\Lambda^{Z(G_s)}$ have been defined in \eqref{eqn:LambdaJ} and the term ``sharp'' in Section~\ref{subsec:resultssharp}.

We note that the technical assumptions of Theorem~\ref{thm:precise} are not very restrictive:

\begin{remarks}\label{rem:conditionsGL}
\begin{enumerate}
  \item[(a)] The assumption $\Gamma\cap G_c\subset Z(G_s)$ is automatically satisfied if $G$ has no compact factor (\ie if $G_c=Z(G_s)$) or if $\Gamma$ is torsion-free.
  This assumption will be removed in Section~\ref{subsec:proofs} in order to prove the theorems and propositions of Chapters \ref{sec:intro} and~\ref{sec:listex}.
  \item[(b)] The assumption $\Gamma\cap L_c\subset Z(G_s)$ is automatically satisfied if $\Gamma$ is torsion-free, or if $L$ has no compact factor and $Z(L)\subset Z(G_s)$.
  We note that for $\Gamma\subset L$, the condition $\Gamma\cap L_c\subset Z(G_s)$ is stronger than $\Gamma\cap G_c\subset Z(G_s)$.
\end{enumerate}
\end{remarks}

Constants~$R$ as in Theorem~\ref{thm:precise}.(1) and~(2) can be expressed in terms of the geometry of~$X$, of the sharpness constants $(c,C)$ of~$\Gamma$, and of a ``pseudo-distance'' from the origin $x_0=eH$ of $X=G/H$ to the other points of its $\Gamma$-orbit in~$X$: see \eqref{eqn:finalcst}, \eqref{eqn:finalcstconj}, and \eqref{eqn:finalcstL}.

We note that our choice of a positive system $\Sigma^+(\g_{\C},\jj_{\C})$ containing $\Sigma^+(\kk_{\C},\jj_{\C})$ could affect the lattice condition $\lambda\in 2\rho_c-\rho+\Lambda^{\Gamma\cap Z(G_s)}$, since $\rho$ depends on this choice.
All elements~$\lambda$ satisfying one of these lattice conditions appear in the discrete spectrum.
We refer to \eqref{eqn:Z1to1} for a geometric meaning of the choice of $\Sigma^+(\g_{\C},\jj_{\C})$.

\begin{remark}
In Theorem~\ref{thm:precise}.(1), we can take $R=0$ if $\Gamma=\{e\}$.  
This is the ``$C=0$'' conjecture of \cite{fle80} on the precise condition of the parameter~$\lambda$ for the square integrability of certain joint eigenfunctions on~$X$; this conjecture was proved affirmatively in \cite{mo84}, and the main ingredient is Fact~\ref{fact:oshima} that we also use below.
\end{remark}

The following theorem gives a description of an infinite \emph{stable} discrete spectrum as in Theorem~\ref{thm:deform}: it states that the constant~$R$ of Theorem~\ref{thm:precise}.(2) is stable under small deformations.

\begin{theorem}\label{thm:precisedeform}
Assume that $G$ is connected, that $H$ does not contain any simple factor of~$G$, and that the rank condition \eqref{eqn:rank} holds.
For any reductive subgroup~$L$ of~$G$ of real rank~$1$ and any convex cocompact subgroup~$\Gamma$ of~$L$ (in particular, any uniform lattice~$\Gamma$ of~$L$) with $\Gamma\cap G_c\subset Z(G_s)$, there are a constant $R>0$ and a neighborhood $\mathcal{U}\subset\Hom(\Gamma,G)$ of the natural inclusion such that $X_{\varphi(\Gamma)}=\varphi(\Gamma)\backslash X$ is a Clifford--Klein form of~$X$ for all $\varphi\in\mathcal{U}$ and
$$\{ \lambda\in\jj^{\ast}_+\cap\big(2\rho_c-\rho+\Lambda^{\Gamma\cap Z(G_s)}\big) :\ d(\lambda)>R\big\} \subset \Spec_d(X_{\varphi(\Gamma)}).$$
In particular, for all $\varphi\in\mathcal{U}$,
$$\{ \lambda\in\jj^{\ast}_+\cap\big(2\rho_c-\rho+\Lambda^{Z(G_s)}\big) :\ d(\lambda)>R\big\} \subset \Spec_d(X_{\varphi(\Gamma)}).$$
If $\Gamma\cap L_c\subset Z(G_s)$ (for instance if $\Gamma$ is torsion-free or if $L$ is simple with $Z(L)\subset Z(G_s)$), then we may take the same~$R$ (independent of~$\Gamma$) as in Theorem~\ref{thm:precise}.(2), up to replacing $\mathcal{U}$ by some smaller neighborhood.
\end{theorem}

Theorems \ref{thm:precise} and~\ref{thm:precisedeform} will be proved in Chapter~\ref{sec:nonzero}.

\begin{remark}
Our proofs depend on the rank condition \eqref{eqn:rank}.
It is plausible that for a general locally symmetric space, no nonzero eigenvalue is stable under nontrivial small deformations unless \eqref{eqn:rank} is satisfied.
This is corroborated by Fact~\ref{fact:Teich} (in the Riemannian case, \eqref{eqn:rank} is not satisfied).
It is also plausible that there should be no ``universal spectrum'' as in Theorems \ref{thm:universal} and~\ref{thm:precise} unless \eqref{eqn:rank} is satisfied.
\end{remark}

\section{Regularity of the generalized Poincar\'e series}\label{subsec:regular}

As explained in the introduction, Theorems \ref{thm:precise} and~\ref{thm:precisedeform} are proved by constructing generalized Poincar\'e series.
Consider the action of $G$ on $L^2(X,\M_{\lambda})$ by left translation
\begin{equation}\label{eqn:gvarphi}
g\cdot\varphi := \varphi(g^{-1}\,\cdot\,)
\end{equation}
and let $L^2(X,\M_{\lambda})_K$ be the subspace of $K$-finite functions in $L^2(X,\M_{\lambda})$.
We prove that for any $\lambda\in\jj^{\ast}_+$ with $d(\lambda)$ large enough, the operator
$$S_{\Gamma} : L^2(X,\M_{\lambda})_K \longrightarrow L^2(X_{\Gamma},\M_{\lambda})$$
mapping $\varphi$ to
$$\varphi^{\Gamma} := \Big(\Gamma x \longmapsto \sum_{\gamma\in\Gamma} (\gamma\cdot\varphi)(x)\Big)$$
is well-defined (Proposition~\ref{prop:Vlambda}.(1)).
We actually prove that $S_{\Gamma}$ is well-defined on\linebreak $g\!\cdot\!L^2(X,\M_{\lambda})_K$ for any $g\in G$ and $\lambda\in\jj^{\ast}_+$ with $d(\lambda)$ large enough, and that there exists $g\in G$ such that for any $\lambda\in\jj^{\ast}_+\cap (2\rho_c-\rho+\Lambda^{\Gamma\cap Z(G_s)})$ with $d(\lambda)$ large enough, $S_{\Gamma}$ is nonzero on $g\!\cdot\!L^2(X,\M_{\lambda})_K$ (Proposition~\ref{prop:nonzero} and Remark~\ref{rem:translprecise}).

By using the fact that $L^2(X,\M_{\lambda})_K$ is stable under the action of $\g$ by differentiation, we obtain the following regularity result for the image of~$S_{\Gamma}$ (Proposition~\ref{prop:Vlambda}.(2)).

\begin{theorem}\label{thm:regular}
Assume that $G$ is connected and that the rank condition \eqref{eqn:rank} holds.
Let $X_{\Gamma}$ be a sharp Clifford--Klein form with $\Gamma\cap G_c\subset Z(G_s)$ and let $R>0$ be the corresponding constant given by Theorem~\ref{thm:precise}.
For any $\lambda\in\jj^{\ast}_+$ with $d(\lambda)>R$ and any $g\in G$, the image of $g\!\cdot\!L^2(X,\M_{\lambda})_K$ under the summation operator~$S_{\Gamma}$ is contained in $L^p(X_{\Gamma})$ for all $1\leq p\leq\infty$, and in $C^m(X_{\Gamma})$ whenever $d(\lambda)>(m+1)R$.
\end{theorem}

In particular, if we take~$m$ to be the maximum degree of the generators $D_1,\dots,D_r$ of the $\C$-algebra $\D(X)$, then for $f\in S_{\Gamma}(g\!\cdot\!L^2(X,\M_{\lambda})_K)$ we have
$$(D_j)_{\Gamma}\,f = \chi_\lambda(D_j)f$$
for all $1\leq j\leq r$ in the sense of functions, not only in the sense of distributions.
For certain standard Clifford--Klein forms~$X_{\Gamma}$, it is actually possible to prove that the image of $L^2(X,\M_{\lambda})_K$ under the summation operator~$S_{\Gamma}$ consists of analytic functions (see \cite{kk12}).

%% file: StableSpec4-Counting.tex
\chapter[Sharpness and counting]{Sharpness and counting in non-Riemannian symmetric spaces}\label{sec:geometry}

In this chapter we examine in detail the new notion of \emph{sharpness}, which we have introduced in Section~\ref{subsec:resultssharp}.
We then establish some counting results for the orbits of sharp discrete groups~$\Gamma$ in the non-Riemannian symmetric space $X=G/H$ (Lemma~\ref{lem:growthfornu} and Corollary~\ref{cor:growthinX}).
We note that these groups~$\Gamma$ can never be lattices of~$G$: they have to be much ``smaller'' (Remark~\ref{rem:lattice}).

Counting is developed here in the perspective of spectral theory: our results will be useful, together with the analytic estimates of Chapter~\ref{sec:Vlambda}, to prove the convergence of the generalized Poincar\'e series \eqref{eqn:phiGamma}.
However, the counting results we obtain might also have some interest of their own.

We first introduce some notation and briefly recall the notions of Cartan and polar projections for noncompact, reductive~$G$.

\section{Preliminaries: Cartan and polar projections}\label{subsec:munu}

We keep the notation of Chapter~\ref{sec:theorems}.
In particular, $\theta$ is the Cartan involution and $\g=\kk+\p$ the Cartan decomposition introduced in Section~\ref{subsec:Laplacian}.
Let $\aaa$ be a maximal abelian subspace of~$\p$ and let $A=\exp\aaa$ be the corresponding connected subgroup of~$G$.
We consider the logarithm $\log : A\overset{\scriptscriptstyle\sim\,}{\rightarrow}\aaa$, which is the inverse of $\exp : \aaa\overset{\scriptscriptstyle\sim\,}{\rightarrow} A$.
We choose a system $\Sigma^+(\g,\aaa)$ of positive restricted roots and let $\overline{\aaa_+}$ and $\overline{A_+}=\exp\overline{\aaa_+}$ denote the corresponding closed positive Weyl chambers in $\aaa$ and~$A$, respectively.
The Cartan decomposition $G=K\overline{A_+}K$ holds \cite{hel01}: any $g\in G$ may be written as $g=k_ga_gk'_g$ for some $k_g,k'_g\in K$ and a unique $a_g\in\overline{A_+}$.
Setting $\mu(g)=\log a_g$ defines a map
$$\mu : G \longrightarrow \overline{\aaa_+}:=\log\overline{A_+},$$
called the \textit{Cartan projection} associated with the Cartan decomposition $G=K\overline{A_+}K$.
This map is continuous, proper, surjective, and bi-$K$-invariant; we will still denote by~$\mu$ the induced map on the Riemannian symmetric space $G/K$ of~$G$.

\begin{example}\label{ex:SLn}
For $G=\SL_n(\R)$ and $\theta=(g\mapsto\,^t\!g^{-1})$, we have $K=\SO(n)$.
We can take $A$ to be the group of diagonal matrices in~$\SL_n(\R)$ with positive entries and its subset $\overline{A_+}$ to consist of matrices with entries in nonincreasing order; then the Cartan decomposition $G=K\overline{A_+}K$ follows from the polar decomposition in $\SL_n(\R)$ and from the reduction of symmetric matrices.
We have $\mu(g)=(\frac{1}{2}\log t_i)_{1\leq i\leq n}$ where $t_i$ is the $i$-th eigenvalue of~$^t\!gg$.
\end{example}

\noindent
The $G$-invariant symmetric bilinear form~$B$ of Section~\ref{subsec:Laplacian} restricts to a $K$-invariant inner product on~$\p$, which defines a Euclidean norm $\Vert\cdot\Vert$ on~$\aaa$ and a $G$-invariant Riemannian metric $d_{G/K}$ on $G/K$.
The norm of the Cartan projection~$\mu$ admits the following geometric interpretation in terms of distances in the Riemannian symmetric space $G/K$:
\begin{equation}\label{eqn:mudistance}
\Vert\mu(g)\Vert = d_{G/K}(y_0,g\cdot y_0)
\end{equation}
for all $g\in G$, where $y_0$ denotes the image of~$K$ in $G/K$.
Using the triangular inequality and the fact that $G$ acts by isometries on $G/K$, we obtain that
\begin{equation}\label{eqn:triangineq}
\Vert\mu(gg')\Vert \leq \Vert\mu(g)\Vert + \Vert\mu(g')\Vert
\end{equation}
for all $g,g'\in G$.
In fact, the following stronger inequalities hold, which can be proved in a geometric way (see \cite[Lem.\,2.3]{kas08}):
\begin{eqnarray}
\Vert\mu(gg')-\mu(g)\Vert & \leq & \Vert\mu(g')\Vert,\label{eqn:rightstrongtriangineq}\\
\Vert\mu(gg')-\mu(g')\Vert & \leq & \Vert\mu(g)\Vert\label{eqn:leftstrongtriangineq}.
\end{eqnarray}

On the other hand, recall that the group~$H$ is an open subgroup of the set of fixed points of~$G$ under the involution~$\sigma$.
Let $\g=\h+\q$ be the decomposition of~$\g$ into eigenspaces of $\mathrm{d}\sigma$ as in Section~\ref{subsec:DGH}.
Since $\theta$ commutes with~$\sigma$, the following decomposition holds:
$$\g = (\kk\cap\h) + (\kk\cap\q) + (\p\cap\h) + (\p\cap\q).$$
Let $\bb$ be a maximal abelian subspace of $\p\cap\q$ and let $B:=\exp(\bb)$.
We choose a system $\Sigma^+(\g^{\sigma\theta},\bb)$ of positive restricted roots of~$\bb$ in the subspace~$\g^{\sigma\theta}$ of fixed points of~$\g$ under~$\mathrm{d}(\sigma\theta)$, and let $\overline{\bb_+}$ be the corresponding closed positive Weyl chamber and $\overline{B_+}:=\exp\overline{\bb_+}$.
Then the \emph{polar decomposition} (or \emph{generalized Cartan decomposition}) $G=K\overline{B_+}H$ holds \cite[Prop.\,7.1.3]{sch84}: any $g\in G$ may be written as $g=k_gb_gh_g$ for some $k_g\in K$, $h_g\in H$, and a unique $b_g\in\overline{B_+}$.
We refer to Chapters \ref{sec:exAdS} and~\ref{sec:ex} for examples.
Since all maximal abelian subspaces of~$\p$ are conjugate under the adjoint action of~$K$, we may (and will) assume that $\aaa$ contains~$\bb$.
As above, we define a projection
\begin{equation}\label{eqn:nu}
\nu : G \longrightarrow \overline{\bb_+} \ \subset \aaa
\end{equation}
corresponding to the polar decomposition $G=K\overline{B_+}H$.
It is continuous, surjective, and right-$H$-invariant; we will still denote by~$\nu$ the induced map on~$X$.
Geometrically, $\Vert\nu(x)\Vert$ can be interpreted as some kind of ``pseudo-distance'' from the origin $x_0=eH$ of $X=G/H$ to $x\in X$: in order to go from $x_0$ to $x$ in~$X$, one can first travel along the flat sector $\overline{B_+}\!\cdot\!x_0$, then along some (compact) $K$-orbit; $\Vert\nu(x)\Vert$ measures how far one must go in $\overline{B_+}\!\cdot\!x_0$.
The set of points $x\in X$ such that $\nu(x)=0$ is the maximal compact subsymmetric space $X_c:=K\!\cdot\!x_0\simeq K/H\cap K$.

We note that for any $b\in B$ there is some $w\in W(G,A)$ such that $\mu(b)=w\cdot\nu(b)$, hence
\begin{equation}\label{eqn:munuB}
\Vert\mu(b)\Vert=\Vert\nu(b)\Vert.
\end{equation}

\section{Sharpness}\label{subsec:sharp}

We now turn to the new notion of \emph{sharpness}, which quantifies proper discontinuity.
We first recall that not all discrete subgroups~$\Gamma$ of~$G$ can act properly discontinuously on $X=G/H$ since $H$ is noncompact.
A criterion for proper discontinuity was established by Benoist \cite[Cor.\,5.2]{ben96} and Kobayashi \cite[Th.\,1.1]{kob96}, in terms of the Cartan projection~$\mu$.
This criterion states that a closed subgroup~$\Gamma$ of~$G$ acts properly on $X=G/H$ if and only if the set $\mu(\Gamma)\cap (\mu(H)+\mathcal{C})$ is bounded for any compact subset~$\mathcal{C}$ of~$\aaa$; equivalently, if and only if $\mu(\Gamma)$ ``goes away from~$\mu(H)$ at infinity''.

In this paper, we introduce the following stronger condition.

\begin{definition}\label{def:sharp}
A subgroup~$\Gamma$ of~$G$ is said to be \emph{sharp for~$X$} if there are constants $c\in (0,1]$ and $C\geq 0$ such that
\begin{equation}\label{eqn:sharp}
d_{\aaa}(\mu(\gamma),\mu(H)) \geq c\,\Vert\mu(\gamma)\Vert - C
\end{equation}
for all $\gamma\in\Gamma$, where $d_{\aaa}$ is the metric on~$\aaa$ induced by the Euclidean norm~$\Vert\cdot\nolinebreak\Vert$.
If \eqref{eqn:sharp} is satisfied, we say that $\Gamma$ is $(c,C)$\emph{-sharp}.
\end{definition}

\begin{figure}[ht!]
\labellist
\small\hair 2pt
\pinlabel $\frac{C}{1-c^2}$ [r] at 0 302
\pinlabel $\arcsin(c)$ [r] at 685 320
\pinlabel $\overline{\aaa_+}$ [r] at 540 535
\pinlabel $\color{red}\mu(H)$ [r] at 650 285
\pinlabel $\color{blue}\mu(\Gamma)$ [r] at 553 122

\pinlabel $\color{blue}\cdot$ [r] at 70 295
\pinlabel $\color{blue}\cdot$ [r] at 120 320
\pinlabel $\color{blue}\cdot$ [r] at 130 290
\pinlabel $\color{blue}\cdot$ [r] at 170 300
\pinlabel $\color{blue}\cdot$ [r] at 190 330
\pinlabel $\color{blue}\cdot$ [r] at 210 295
\pinlabel $\color{blue}\cdot$ [r] at 220 360
\pinlabel $\color{blue}\cdot$ [r] at 260 320
\pinlabel $\color{blue}\cdot$ [r] at 280 345
\pinlabel $\color{blue}\cdot$ [r] at 310 315
\pinlabel $\color{blue}\cdot$ [r] at 350 345
\pinlabel $\color{blue}\cdot$ [r] at 370 330
\pinlabel $\color{blue}\cdot$ [r] at 380 420
\pinlabel $\color{blue}\cdot$ [r] at 400 355
\pinlabel $\color{blue}\cdot$ [r] at 410 385
\pinlabel $\color{blue}\cdot$ [r] at 430 400
\pinlabel $\color{blue}\cdot$ [r] at 450 375
\pinlabel $\color{blue}\cdot$ [r] at 460 440
\pinlabel $\color{blue}\cdot$ [r] at 490 415
\pinlabel $\color{blue}\cdot$ [r] at 500 380
\pinlabel $\color{blue}\cdot$ [r] at 530 400
\pinlabel $\color{blue}\cdot$ [r] at 540 355
\pinlabel $\color{blue}\cdot$ [r] at 550 430

\pinlabel $\color{blue}\cdot$ [r] at 100 275
\pinlabel $\color{blue}\cdot$ [r] at 150 250
\pinlabel $\color{blue}\cdot$ [r] at 160 230
\pinlabel $\color{blue}\cdot$ [r] at 180 270
\pinlabel $\color{blue}\cdot$ [r] at 200 245
\pinlabel $\color{blue}\cdot$ [r] at 230 270
\pinlabel $\color{blue}\cdot$ [r] at 240 220
\pinlabel $\color{blue}\cdot$ [r] at 270 245
\pinlabel $\color{blue}\cdot$ [r] at 290 220
\pinlabel $\color{blue}\cdot$ [r] at 300 245
\pinlabel $\color{blue}\cdot$ [r] at 330 200
\pinlabel $\color{blue}\cdot$ [r] at 350 240
\pinlabel $\color{blue}\cdot$ [r] at 370 190
\pinlabel $\color{blue}\cdot$ [r] at 380 210
\pinlabel $\color{blue}\cdot$ [r] at 390 155
\pinlabel $\color{blue}\cdot$ [r] at 420 195
\pinlabel $\color{blue}\cdot$ [r] at 440 160
\pinlabel $\color{blue}\cdot$ [r] at 470 120
\pinlabel $\color{blue}\cdot$ [r] at 490 180
\pinlabel $\color{blue}\cdot$ [r] at 510 140
\pinlabel $\color{blue}\cdot$ [r] at 540 160
\endlabellist
\centering
\includegraphics[scale=0.35]{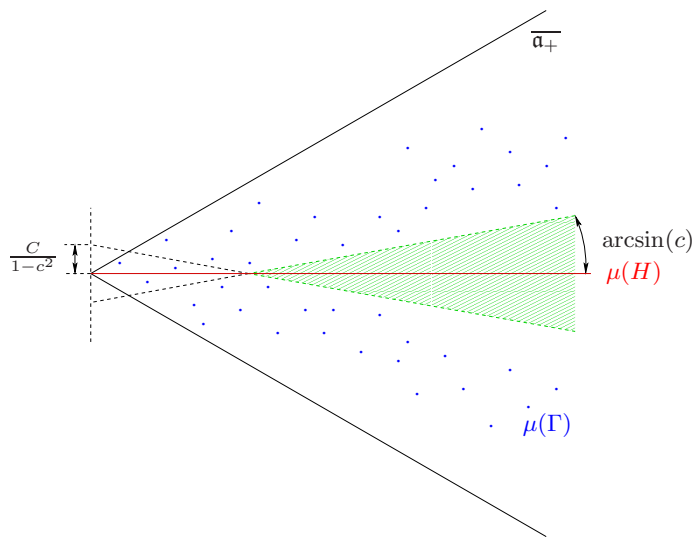}
\caption{The Cartan projection of a $(c,C)$-sharp group~$\Gamma$}
\label{fig1}
\end{figure}

We note that this definition makes sense in the more general context of a homogeneous space $X=G/H$ where $G$ is a reductive group and $H$ a closed subgroup of~$G$.

If $\Gamma$ is sharp for~$X$, then $\mu(\Gamma)$ ``goes away from~$\mu(H)$ at infinity'' with a speed that is \emph{at least linear}.
Indeed, consider the open cone
$$\mathfrak{C}(c) := \big\{ Y\in\overline{\aaa_+} : d_{\aaa}(Y,\mu(H)) < c\,\Vert Y\Vert\big\} $$
of angle $\arcsin(c)$ around~$\mu(H)$.
If $\Gamma$ is $(c,C)$-sharp with $c\in (0,1)$, then the set $\mu(\Gamma)$ is contained in the $\frac{C}{\sqrt{1-c^2}}$-neighborhood of $\overline{\aaa_+}\smallsetminus\mathfrak{C}(c)$; in other words, it does not meet the shaded region in Figure~1.

In particular, if $\Gamma$ is sharp for~$X$ and closed in~$G$, then the action of~$\Gamma$ on~$X$ is proper by the properness criterion.
The bigger $c$ is, the ``more proper'' the action is; the critical case is therefore when $c$ gets close to~$0$.
For $\Gamma$ discrete and sharp, we will equivalently say that the Clifford--Klein form $X_{\Gamma}=\Gamma\backslash X$ is sharp.

The following two properties will be useful.

\begin{proposition}\label{prop:sharpnessproperties}
\begin{enumerate}
  \item If a subgroup~$\Gamma$ of~$G$ is $(c,C)$-sharp for~$X$, then any conjugate of~$\Gamma$ is $(c,C')$-sharp for some $C'\geq 0$.
  \item Any reductive subgroup~$L$ of~$G$ acting properly on~$X$ admits a conjugate that is $(c,0)$-sharp for some $c>0$.
\end{enumerate}
\end{proposition}

Proposition~\ref{prop:sharpnessproperties}.(1) is an immediate consequence of the following inequality, which will be used several times in the paper.

\begin{lemma}\label{lem:prmu}
For any $g,g',g''\in G$,
$$d_{\aaa}\left(\mu(g'gg''),\mu(H)\right) \geq d_{\aaa}(\mu(g),\mu(H)) - \Vert\mu(g')\Vert - \Vert\mu(g'')\Vert.$$
\end{lemma}

\begin{proof}
For all $h\in H$, by \eqref{eqn:rightstrongtriangineq} and~\eqref{eqn:leftstrongtriangineq} we have
\begin{eqnarray*}
d_{\aaa}(\mu(g),\mu(H)) & \leq & \Vert\mu(g) - \mu(h)\Vert\\
& \leq & \Vert\mu(g)-\mu(g'gg'')\Vert + \Vert\mu(g'gg'')-\mu(h)\Vert\\
& \leq & \Vert\mu(g')\Vert + \Vert\mu(g'')\Vert + \Vert\mu(g'gg'')-\mu(h)\Vert.\qedhere
\end{eqnarray*}
\end{proof}

We will explain why Proposition~\ref{prop:sharpnessproperties}.(2) is true in Section~\ref{subsec:exsharp}.
We refer to Section~\ref{subsec:exsharp} for a list of examples of sharp Clifford--Klein forms and to Section~\ref{subsec:sharpdeform} for a discussion of how sharpness behaves under small deformations.

We note that $d_{\aaa}(\mu(\gamma),\mu(H))\leq\Vert\mu(\gamma)\Vert$ always holds, since $d_{\aaa}(\mu(\gamma),\mu(H))$ is the norm of the projection of $\mu(\gamma)$ to the orthogonal of~$\mu(H)$ in~$\aaa$; this is why we restrict to $c\leq 1$ in Definition~\ref{def:sharp}.

\section{Counting in the reductive symmetric space~$X$}\label{subsec:count}

In order to prove the convergence of the generalized Poincar\'e series \eqref{eqn:phiGamma}, we will need to understand the growth rate of~$\Gamma$ \emph{with respect to the norm of~$\nu$}.
Given the above geometric interpretation of~$\Vert\nu\Vert$ as a ``pseudo-distance from the origin'' in the reductive symmetric space~$X$, this means estimating the number of points of any given $\Gamma$-orbit in the ``pseudo-ball''
\begin{equation}\label{eqn:pseudoball}
B_X(R) := \{ x\in X : \Vert\nu(x)\Vert < R\}
\end{equation}
as $R$ tends to infinity.
We note that the closure of $B_X(R)$ is compact for all $R>0$, which implies the following (by definition of proper discontinuity).

\begin{remark}\label{rem:Gammapseudoball}
Let $\Gamma$ be a discrete subgroup of~$G$ acting properly discontinuously on~$X$.
For any $x\in X$, the set of elements $\gamma\in\Gamma$ with $\gamma\cdot x\in B_X(R)$ is finite.
\end{remark}

\noindent
In the case when $\Gamma$ is sharp for~$X$, we establish exponential bounds for the growth of $\Gamma$-orbits in~$X$: here are the precise estimates that we will need for our theorems (a proof will be given in Section~\ref{subsec:proofgrowth}).

\begin{lemma}\label{lem:growthfornu}
Let $c\in (0,1]$ and $C\geq 0$.
\begin{enumerate}
  \item For any discrete subgroup~$\Gamma$ of~$G$ that is $(c,C)$-sharp for~$X$ and any $\varepsilon>0$, there is a constant $c_{\varepsilon}(\Gamma)>0$ such that for any $R>0$ and any $x=g\cdot x_0\in X$ (where $g\in G$),
  $$\#\big\{ \gamma\in\Gamma :\ \Vert\nu(\gamma\!\cdot\!x)\Vert <R\big\} \leq c_{\varepsilon}(\Gamma)\,e^{(\delta_{\Gamma}+\varepsilon)(R+\Vert\mu(g)\Vert)/c}.$$
  \item (Removing the dependence in~$x$)\\
  For any discrete subgroup~$\Gamma$ of~$G$ that is $(c,C)$-sharp for~$X$ and any $\varepsilon>0$, there is a constant $c'_{\varepsilon}(\Gamma)>0$ such that for any $R>0$ and any $x\in X$,
  $$\#\big\{ \gamma\in\Gamma :\ \Vert\nu(\gamma\!\cdot\!x)\Vert<R\big\} \leq c'_{\varepsilon}(\Gamma)\,e^{2(\delta_{\Gamma}+\varepsilon)R/c}.$$
  \item (Controlling the dependence in~$\Gamma$, allowing for dependence in~$x$)\\
  There is a constant $c_G>0$ depending only on~$G$ such that for any discrete subgroup~$\Gamma$ of~$G$ that is $(c,C)$-sharp for~$X$, any $R>0$, and any $x=g\cdot x_0\in X$ (where $g\in G$),
  $$\#\big\{ \gamma\in\Gamma :\ \Vert\nu(\gamma\!\cdot\!x)\Vert<R\big\} \leq \#(\Gamma\cap K) \cdot c_G\,e^{2\Vert\rho_a\Vert (R+C+\Vert\mu(g)\Vert)/c}.$$
  \item (Controlling the dependence in~$\Gamma$ and removing the dependence in~$x$)\\
  There is a constant $c_G>0$ depending only on~$G$ such that for any discrete subgroup~$\Gamma$ of~$G$ that is $(c,C)$-sharp for~$X$, any $R>0$, and any $x\in X$,
  $$\#\big\{ \gamma\in\Gamma :\ \Vert\nu(\gamma\!\cdot\!x)\Vert<R\big\} \leq \#(\Gamma\cap K) \cdot c_G\,e^{4\Vert\rho_a\Vert (R+C)/c}.$$
\end{enumerate}
\end{lemma}

\smallskip

As before, $x_0$ is the image of~$H$ in $X=G/H$ and $\rho_a\in\aaa$ is half the sum of the elements of $\Sigma^+(\g,\aaa)$, counted with root multiplicities.
We denote by
\begin{equation}\label{eqn:critexp}
\delta_{\Gamma} := \limsup_{R\rightarrow +\infty} \left(\frac{1}{R} \log\#\big(\Gamma\!\cdot\!y_0\cap B_{G/K}(R)\big)\right)
\end{equation}
the \emph{critical exponent} of~$\Gamma$, which measures the growth rate of the $\Gamma$-orbits in the Riemannian symmetric space $G/K$ of~$G$.
Here
$$B_{G/K}(R) := \{ y\in G/K : \Vert\mu(y)\Vert < R\} $$
is the ball of radius~$R$ centered at $y_0=eK\in G/K$ for the Riemannian metric~$d_{G/K}$ (see \eqref{eqn:mudistance}).
Recall that the classical Poincar\'e series $\sum_{\gamma\in\Gamma} e^{-s\Vert\mu(\gamma)\Vert}$ converges for $s>\delta_{\Gamma}$ and diverges for $s<\delta_{\Gamma}$, and that if $G$ has real rank~$1$, then $\delta_{\Gamma}$ is the Hausdorff dimension of the limit set of~$\Gamma$ in the boundary at infinity of~$G/K$ \cite{pat76,sul79,cor90}.

In~$X$, consider the ``pseudo-ball'' $B_X(R)$ of radius~$R$ centered at~$x_0$, as in \eqref{eqn:pseudoball}.
For all $x=g\,\cdot\,x_0\in X$ (where $g\in G$), the stabilizer of~$x$ in~$\Gamma$ is $\Gamma\cap gHg^{-1}$, hence
\begin{equation}\label{eqn:linkcounting}
\#\big\{ \gamma\in\Gamma :\ \Vert\nu(\gamma\!\cdot\!x)\Vert<R\big\} = \#(\Gamma\cap gHg^{-1}) \cdot \#\big(\Gamma\!\cdot\!x\cap B_X(R)\big).
\end{equation}
Therefore, Lemma~\ref{lem:growthfornu} gives the following counting result for $\Gamma$-orbits in~$X$.

\begin{corollary}\label{cor:growthinX}
For any discrete subgroup~$\Gamma$ of~$G$ that is $(c,C)$-sharp for~$X$ and any $x\in X$,
$$\limsup_{R\rightarrow +\infty} \left(\frac{1}{R} \log\#\big(\Gamma\!\cdot\!x\cap B_X(R)\big)\right) \leq \frac{\delta_{\Gamma}}{c}\,;$$
if moreover $\Gamma\cap K=\{ e\}$ (for instance if $\Gamma$ is torsion-free), then
$$\#\big(\Gamma\!\cdot\!x_0\cap B_X(R)\big) \leq c_G\,e^{2\Vert\rho_a\Vert (R+C)/c}$$
and for all $x\in X$,
$$\#\big(\Gamma\!\cdot\!x\cap B_X(R)\big) \leq c_G\,e^{4\Vert\rho_a\Vert (R+C)/c}.$$
\end{corollary}

\begin{remark}\label{rem:lattice}
In our setting $\Gamma$ can never be a lattice in~$G$ because it acts properly discontinuously on $X=G/H$ and $H$ is noncompact.
(In fact $\Gamma$ has to be quite ``small'': the cohomological dimension of any torsion-free finite-index subgroup of~$\Gamma$ has to be $\leq\dim(G/K)-\dim(H/H\cap K)$, see \cite{kob89}.)
Corollary~\ref{cor:growthinX} can be compared with the following results on lattices of~$G$.
\begin{enumerate}[(a)]
  \item Let $\Gamma$ be an irreducible lattice of~$G$ such that $\Gamma\cap H$ is a lattice of~$H$.
  Here is a precise counting result, due to Eskin--McMullen \cite{em93}, for the $\Gamma$-orbit through the origin~$x_0$: for any sequence $(B_n)_{n\in\N}$ of ``well-rounded'' subsets of~$X$,
  $$\#\big(\Gamma\!\cdot\!x_0\cap B_n\big) \underset{\scriptscriptstyle n\rightarrow +\infty}{\sim} \frac{\vol((\Gamma\cap H)\backslash H)}{\vol(\Gamma\backslash G)} \cdot \vol_X(B_n).$$
  In particular (see Lemma~\ref{lem:triangineqnu} and \eqref{eqn:weightKBH}, \eqref{eqn:estimdelta} below), there is a constant $C>0$, independent of~$\Gamma$, such that
  $$\#\big(\Gamma\!\cdot\!x_0\cap B_X(R)\big) \underset{\scriptscriptstyle R\rightarrow +\infty}{\sim} C \cdot \frac{\vol((\Gamma\cap H)\backslash H)}{\vol(\Gamma\backslash G)} \cdot e^{2\Vert\rho_b\Vert R}.$$
  \item Let $\Gamma$ be a lattice of~$G$.
  The $\Gamma$-orbit through an arbitrary point $x\in X$ can be dense in~$X$, in which case $\#(\Gamma\!\cdot\!x\cap B_X(R))$ is infinite.
  For instance, this is generically the case for $X=\SL_3(\R)/\SO(2,1)$ and $\Gamma=\SL_3(\Z)$: see Margulis's proof \cite{mar87} of the Oppenheim conjecture.
\end{enumerate}
\end{remark}

Here we denote by $\Vert\rho_b\Vert$ the norm of half the sum of the elements of a positive system $\Sigma^+(\g,\bb)$ of restricted roots of~$\bb$ in~$\g$; this norm does not depend on the choice of $\Sigma^+(\g,\bb)$.
We note that $\Vert\rho_b\Vert\leq\Vert\rho_a\Vert$ (see Remark~\ref{rem:rhoab}).

It would be interesting to obtain a precise counting result in our setting, in terms of the sharpness constants and of the critical exponent of~$\Gamma$.
We observe that the following lower bound holds.

\begin{remark}
Let $\Gamma$ be a discrete subgroup of~$G$ whose Zariski closure in~$G$ is semisimple or contained in a semisimple group of real rank~$1$.
For any $\varepsilon>0$ there is a constant $c_{\varepsilon}(\Gamma)\in (0,1]$ such that for any $x=g\!\cdot\!x_0\in X$ (where $g\in G$) and any $R>0$,
$$\#\big(\Gamma\!\cdot\!x\cap B_X(R)\big) \geq \frac{c_{\varepsilon}(\Gamma)}{\#(\Gamma\cap gHg^{-1})}\ e^{(\delta_{\Gamma}-\varepsilon)(R-\Vert\mu(g)\Vert)}$$
(with the convention $1/\!+\!\infty=0$).
If $\Gamma$ is $(c,C)$-sharp, then
$$\#(\Gamma\cap gHg^{-1}) \leq c_{\varepsilon}(\Gamma)^{-1}\ e^{(\delta_{\Gamma}+\varepsilon)\frac{2\,\Vert\mu(g)\Vert+C}{c}} < +\infty.$$
\end{remark}

\noindent
Indeed, the first formula is a consequence of \eqref{eqn:linkcounting}, of the inequality $\Vert\nu\Vert\leq\Vert\mu\Vert$ (Lemma~\ref{lem:munu}), and of the fact that the critical exponent, defined as a limsup, is in fact a limit \cite{rob02,qui02}.
The bound on $\#(\Gamma\cap gHg^{-1})$ for sharp~$\Gamma$ comes from the fact that if $\gamma\in gHg^{-1}$, then $d_{\aaa}(\mu(\gamma),\mu(H))\leq 2\,\Vert\mu(g)\Vert$ by \eqref{eqn:rightstrongtriangineq} and \eqref{eqn:leftstrongtriangineq}, hence $\Vert\mu(\gamma)\Vert\leq\frac{2\,\Vert\mu(g)\Vert+C}{c}$ by $(c,C)$-sharpness.

\section{Examples of sharp groups}\label{subsec:exsharp}

Before we prove Lemma~\ref{lem:growthfornu} (in Section~\ref{subsec:proofgrowth}), we first give some examples of sharp Clifford--Klein forms to illustrate and motivate this notion.
We begin with an important example (which holds in the more general context of a homogeneous space $X=G/H$ where $G$ is a reductive group and $H$ a closed subgroup of~$G$).

\begin{example}\label{ex:standardsharp}
All standard Clifford--Klein forms of~$X$ are sharp.
\end{example}

The notion of ``standard'' was defined in the introduction (Definition~\ref{def:standard}).
To understand why Example~\ref{ex:standardsharp} is true, here is a more precise statement.

\begin{example}\label{ex:standardthetastable}
Let $L$ be a reductive subgroup of~$G$ acting properly on~$X$.
If $L$ is stable under the Cartan involution~$\theta$, then the set $\mu(L)$ is the intersection of $\overline{\aaa_+}$ with a finite union of subspaces of~$\aaa$, which meet $\mu(H)$ only in~$0$.
Let $c$ be the sine of the minimal angle between $\mu(L)$ and~$\mu(H)$.
Then any Clifford--Klein form~$X_{\Gamma}$ with $\Gamma\subset L$ is $(c,0)$-sharp.
\end{example}

\begin{proof}[Proof of Example~\ref{ex:standardthetastable}]
If $L$ is stable under the Cartan involution~$\theta$, then $K\cap L$ is a maximal compact subgroup of~$L$ and there is an element $k\in K$ such that $kAk^{-1}\cap L$ is a maximal split abelian subgroup of~$L$ and the Cartan decomposition
$$L = (K\cap L)(kAk^{-1}\cap L)(K\cap L)$$
holds.
The set $\mu(L)=\mu(A\cap k^{-1}Lk)=\overline{\aaa_+}\cap W\cdot(\aaa\cap\Ad(k^{-1})(\mathrm{Lie}(L)))$ is the intersection of $\overline{\aaa_+}$ with a finite union of subspaces of~$\aaa$; it meets~$\mu(H)$ only in~$0$ by the properness criterion \cite[Th.\,4.1]{kob89}.
By definition of sharpness, $L$ is $(c,0)$-sharp for~$X$, and so is any subgroup $\Gamma\subset L$.
\end{proof}

This explains why Proposition~\ref{prop:sharpnessproperties}.(2) is true.

\begin{proof}[Proof of Proposition~\ref{prop:sharpnessproperties}.(2)]
The fact that any reductive subgroup~$L$ of~$G$ acting properly on~$X$ admits a conjugate that is $(c,0)$-sharp for some $c>0$ follows from Example~\ref{ex:standardthetastable} and from the fact that any reductive subgroup~$L$ of~$G$ admits a conjugate in~$G$ that is $\theta$-stable.
\end{proof}

\begin{proof}[Proof of Example~\ref{ex:standardsharp}]
The fact that all standard Clifford--Klein forms of~$X$ are sharp follows from Proposition~\ref{prop:sharpnessproperties}.(1) and~(2).
\end{proof}

Additional evidence that sharpness is a fundamental concept is given by the fact that all known examples of compact Clifford--Klein forms of reductive homogeneous spaces are sharp, even when they are nonstandard.
We conjecture that they should all be.

\begin{conjecture}\label{conj:sharp}
Let $G$ be a reductive linear Lie group and $H$ a reductive subgroup of~$G$.
Any compact Clifford--Klein form of $X=G/H$ is sharp.
\end{conjecture}

The following particular case of Conjecture~\ref{conj:sharp} was proved in \cite{kas12}.

\begin{example}[{\cite[Th.\,1.1]{kas12}}]\label{ex:deform}
Let $X=G/H$, where $G$ is a reductive linear Lie group and $H$ a reductive subgroup of~$G$.
Let $\Gamma$ be a uniform lattice in some reductive subgroup~$L$ of~$G$ of real rank~$1$.
Any small deformation of the standard Clifford--Klein form~$X_{\Gamma}$ is sharp.
\end{example}

In other words, there exists a neighborhood $\mathcal{U}\subset\Hom(\Gamma,G)$ of the natural inclusion such that the group $\varphi(\Gamma)$ is discrete in~$G$ and sharp for~$X$ for all $\varphi\in\nolinebreak\mathcal{U}$.
More precisely, if $\Gamma$ is $(c,C)$-sharp, then for any $\varepsilon>0$ there is a neighborhood $\mathcal{U}_{\varepsilon}\subset\Hom(\Gamma,G)$ of the natural inclusion such that $\varphi(\Gamma)$ is $(c-\varepsilon,C+\varepsilon)$-sharp for all $\varphi\in\mathcal{U}_{\varepsilon}$ (and even $(c-\varepsilon,C)$-sharp if $C>0$ or $\Gamma\cap K=\{ e\}$, for instance if $\Gamma$ is torsion-free).
This holds more generally whenever $\Gamma$ is a convex cocompact subgroup of~$L$, \ie a discrete subgroup acting cocompactly on some nonempty convex subset of the Riemannian symmetric space of~$L$.

In the special case of $X=\AdS^3=\SO(2,2)_0/\SO(1,2)_0$, sharpness was proved in \cite{kasPhD} for \emph{all} compact Clifford--Klein forms, even for those that are not deformations of standard ones (such forms exist by \cite{sal00}).

\begin{example}[{\cite[Th.\,5.1.1]{kasPhD}}]\label{ex:AdS}
All compact Clifford--Klein forms of $X=\AdS^3$ are sharp.
\end{example}

As we will see in Section~\ref{subsec:groupmfd}, this is a special case of the following recent result.

\begin{example}[{\cite{ggkw}}]\label{ex:Grk1}
Let ${}^{\backprime}G$ be a real semisimple linear Lie group of real rank~$1$.
All compact Clifford--Klein forms of $X=({}^{\backprime}G\times\!{}^{\backprime}G)/\Diag({}^{\backprime}G)$ are sharp.
\end{example}

We note that there exist Clifford--Klein forms~$X_{\Gamma}$ with $\Gamma$ infinitely generated that are not sharp (see \cite{gk12}).
Also, not all sharp Clifford--Klein forms remain sharp under small deformations; it can happen that the action actually stops being properly discontinuous.

\begin{example}\label{ex:propernonstable}
Let $X=({}^{\backprime}G\times\!{}^{\backprime}G)/\Diag({}^{\backprime}G)$ and $\Gamma={}^{\backprime}\Gamma\times\{ e\}$, where ${}^{\backprime}G$ is a real semisimple linear Lie group of real rank~$1$ and ${}^{\backprime}\Gamma$ a discrete subgroup of~${}^{\backprime}G$ containing a nontrivial unipotent element ${}^{\backprime}\gamma_u$ (for instance a nonuniform lattice of~${}^{\backprime}G$).
For any neighborhood $\mathcal{U}\subset\Hom(\Gamma,{}^{\backprime}G\times\!{}^{\backprime}G)$, there is an element $\varphi\in\mathcal{U}$ such that the group $\varphi(\Gamma)$ does not act properly discontinuously on~$X$.
\end{example}

The idea is to obtain a contradiction with the properness criterion of Benoist and Kobayashi for some $\varphi$ such that the first projection of $\varphi({}^{\backprime}\gamma_u,e)$ to~${}^{\backprime}G$ is unipotent and the second projection is hyperbolic (see \cite{gk12}).

\section{Link between the Cartan and polar projections}\label{subsec:propertiesmunu}

In order to prove Lemma~\ref{lem:growthfornu}, we will use the following link between the Cartan projection~$\mu$ (on which the notion of sharpness is built) and the polar projection~$\nu$ (on which our counting is based).

\begin{lemma}\label{lem:munu}
For any $g\in G$,
$$d_{\aaa}(\mu(g),\mu(H)) \leq \Vert\nu(g)\Vert \leq \Vert\mu(g)\Vert\,.$$
\end{lemma}

\begin{proof}
For $g\in G$, write $g=kbh$, where $k\in K$, $b\in\overline{B_+}$, and $h\in H$.
Since $H$ is fixed by~$\sigma$, since $K$ is globally preserved by~$\sigma$ (because $\sigma$ and~$\theta$ commute), and since $\sigma(b)=b^{-1}\in B\subset A$, we have
$$\mu(g\sigma(g)^{-1}) = \mu(b\sigma(b)^{-1}) = \mu(b^2) = 2\,\mu(b).$$
Using~\eqref{eqn:triangineq} and the fact that $\Vert\mu(b)\Vert=\Vert\nu(b)\Vert=\Vert\nu(g)\Vert$ by \eqref{eqn:munuB}, we obtain
\begin{equation}\label{eqn:munusigma}
2\,\Vert\nu(g)\Vert = \Vert\mu(g\sigma(g)^{-1})\Vert \leq \Vert\mu(g)\Vert + \Vert\mu(\sigma(g)^{-1})\Vert.
\end{equation}
Since $\sigma(K)=K$ and $\sigma(A)=A$ (because $\aaa=(\aaa\cap\h)+\bb$), we have $\Vert\mu(\sigma(g)^{-1})\Vert=\Vert\mu(g)\Vert$, which implies $\Vert\nu(g)\Vert\leq\Vert\mu(g)\Vert$.
On the other hand, by \eqref{eqn:leftstrongtriangineq} and \eqref{eqn:munuB},
\begin{eqnarray*}
d_{\aaa}(\mu(g),\mu(H)) & \leq & \Vert\mu(g)-\mu(h)\Vert\\
& = & \Vert\mu(bh)-\mu(h)\Vert\\
& \leq & \Vert\mu(b)\Vert = \Vert\nu(b)\Vert = \Vert\nu(g)\Vert.\qedhere
\end{eqnarray*}
\end{proof}

The following lemma implies, together with \eqref{eqn:weightKBH} below, that for any sequence $(R_n)\in\R_+^{\N}$ tending to infinity, the sequence $(B_X(R_n))_{n\in\N}$ of ``pseudo-balls'' of radius~$R_n$ centered at the origin (see \eqref{eqn:pseudoball}) is ``well-rounded'' in the sense of Eskin--McMullen \cite{em93}: for any $\varepsilon>0$ there is a neighborhood $\mathcal{U}$ of $e$ in~$G$ such that
$$\vol_X\big(\mathcal{U}\cdot \partial B_X(R_n)\big) \leq \varepsilon \, \vol_X\big(B_X(R_n)\big).$$

\begin{lemma}\label{lem:triangineqnu}
For any $g,g'\in G$,
$$\Vert\nu(g')\Vert - \Vert\mu(g)\Vert \leq \Vert\nu(gg')\Vert \leq \Vert\nu(g')\Vert + \Vert\mu(g)\Vert.$$
\end{lemma}

\begin{proof}
Let $g,g'\in G$.
Write $g'=kbh$ with $k\in K$, $b\in\overline{B_+}$, and $h\in H$.
By Lemma~\ref{lem:munu} and \eqref{eqn:triangineq},
$$\Vert\nu(gg')\Vert = \Vert\nu(gkb)\Vert \leq \Vert\mu(gkb)\Vert \leq \Vert\mu(g)\Vert + \Vert\mu(kb)\Vert.$$
But $\Vert\mu(kb)\Vert=\Vert\nu(kb)\Vert=\Vert\nu(g)\Vert$ by \eqref{eqn:munuB}, hence $\Vert\nu(gg')\Vert\leq\Vert\nu(g')\Vert+\Vert\mu(g)\Vert$.
Applying this inequality to $(g^{-1},gg')$ instead of $(g,g')$, we obtain $\Vert\nu(gg')\Vert\geq\Vert\nu(g')\Vert-\Vert\mu(g)\Vert$.
\end{proof}

\section{Proof of the counting estimates}\label{subsec:proofgrowth}

We now use Lemmas \ref{lem:prmu} and~\ref{lem:munu}, together with the classical growth theory for discrete isometry groups in the Riemannian symmetric space $G/K$, to prove Lemma~\ref{lem:growthfornu}.

\begin{proof}[Proof of Lemma~\ref{lem:growthfornu}.(1)]
By Lemmas \ref{lem:prmu} and~\ref{lem:munu}, for all $g\in G$ and $\gamma\in\Gamma$ we have
$$\Vert\nu(\gamma g)\Vert \,\geq\, d_{\aaa}(\mu(\gamma g),\mu(H)) \,\geq\, d_{\aaa}(\mu(\gamma),\mu(H)) - \Vert\mu(g)\Vert.$$
Using the sharpness assumption, we obtain that for all $g\in G$,
\begin{equation}\label{eqn:ineqmunu}
\Vert\nu(\gamma g)\Vert \,\geq\, c\,\Vert\mu(\gamma)\Vert - C - \Vert\mu(g)\Vert,
\end{equation}
hence
$$\#\big\{ \gamma\in\Gamma : \Vert\nu(\gamma g)\Vert < R\big\} \leq \#\Big\{ \gamma\in\Gamma : \Vert\mu(\gamma)\Vert < \frac{R+C+\Vert\mu(g)\Vert}{c}\Big\} .$$
We conclude using the definition \eqref{eqn:critexp} of the critical exponent~$\delta_{\Gamma}$.
\end{proof}

The proof of Lemma~\ref{lem:growthfornu}.(3) follows rigorously the same idea, using the following classical observation (where $y_0=eK\in G/K$ as before).

\begin{observation}\label{obs:limsup}
There is a constant $c_G\geq 1$ depending only on~$G$ such that for any discrete subgroup~$\Gamma$ of~$G$ and any $R>0$,
$$\#\big(\Gamma\!\cdot\!y_0\cap B_{G/K}(R)\big) \,\leq\, c_G\,e^{2\,\Vert\rho_a\Vert R}.$$
In particular, $\delta_{\Gamma} \leq 2\,\Vert\rho_a\Vert$ and
$$\#\big\{ \gamma\in\Gamma : \Vert\mu(\gamma)\Vert <R\big\} \,\leq\, c_G\,e^{2\,\Vert\rho_a\Vert R} \cdot \#(\Gamma\cap K).$$
\end{observation}

\begin{proof}
Let
\begin{equation}\label{eqn:RiemDirichlet}
\mathcal{D}_{G/K} = \big\{ y\in G/K\ :\ d_{G/K}(y,y_0)\leq d_{G/K}(y,\gamma\cdot y_0) \quad\forall\gamma\in\Gamma\big\} 
\end{equation}
be the Dirichlet domain centered at~$y_0$, and let $t>0$ be the distance from~$y_0$ to the boundary of~$\mathcal{D}_{G/K}$.
For all $R>0$ and all $\gamma\in\Gamma$ with $\gamma\cdot y_0\in B_{G/K}(R)$,
$$\gamma\cdot B_{G/K}(t) \subset B_{G/K}(R+t)$$
since $G$ acts on $G/K$ by isometries.
Moreover, by definition of~$t$, the balls $\gamma\cdot B_{G/K}(t)$ and $\gamma'\cdot B_{G/K}(t)$ (for $\gamma,\gamma'\in\Gamma$) do not intersect if $\gamma\cdot y_0\neq\gamma'\cdot y_0$.
Therefore,
$$\#\big(\Gamma\!\cdot\!y_0\cap B_{G/K}(R)\big) \!\!\cdot \vol\,B_{G/K}(t) \,\leq\, \vol\,B_{G/K}(R+t).$$
Observation~\ref{obs:limsup} is then a consequence of the following volume estimate (see \cite[Ch.\,I, Th.\,5.8]{hel00}): there is a constant~$c'_G$ (depending only on~$G$) such that
$$\vol\,B_{G/K}(R') \underset{\scriptscriptstyle R'\rightarrow +\infty}{\sim} c'_G\,e^{2\Vert\rho_a\Vert R'}.\qedhere$$
\end{proof}

We now turn to Lemma~\ref{lem:growthfornu}.(2) and~(4).
It is sufficient to give a proof for $x$ in some fundamental domain of~$X$ for the action of~$\Gamma$.
We consider the following particular fundamental domain.

\begin{def-lem}[A pseudo-Riemannian Dirichlet domain]\label{def-lem:funddomain}
Let $\Gamma$ be a discrete subgroup of~$G$ acting properly discontinuously on~$X$.
The set
$$\mathcal{D}_X = \{ x\in X :\quad \Vert\nu(x)\Vert\leq\Vert\nu(\gamma\cdot x)\Vert \quad\forall\gamma\in\Gamma\} $$
is well-defined; it is a fundamental domain of~$X$ for the action of~$\Gamma$.
\end{def-lem}

\begin{proof}
By Remark~\ref{rem:Gammapseudoball}, for any given $x\in X$ there are only finitely many elements $\gamma\in\Gamma$ such that $\Vert\nu(\gamma\!\cdot\!x)\Vert\leq\Vert\nu(x)\Vert$; in particular, there is an element $\gamma_0\in\Gamma$ such that $\Vert\nu(\gamma_0\!\cdot\!x)\Vert\leq\Vert\nu(\gamma\!\cdot\!x)\Vert$ for all $\gamma\in\Gamma$.
Thus $\mathcal{D}_X$ is well-defined and $\Gamma\cdot\mathcal{D}_X=X$.
To see that $\mathcal{D}_X$ is actually a fundamental domain (which is not needed in our proof of Lemma~~\ref{lem:growthfornu}, where we only use $\Gamma\cdot\mathcal{D}_X=X$), it is sufficient to see that for any $\gamma$ in the countable group~$\Gamma$, the set
$$\mathcal{H}_{\gamma} := \{ x\in X :\quad \Vert\nu(x)\Vert = \Vert\nu(\gamma\cdot x)\Vert\} $$
has measure~$0$ in~$X$.
But \eqref{eqn:mudistance} and \eqref{eqn:munusigma} imply that for any $g\in G$,
$$2\,\Vert\nu(g)\Vert = \Vert\mu(g\sigma(g)^{-1})\Vert = d_{G/K}\big(y_0,g\sigma(g)^{-1}\cdot y_0\big).$$
Therefore the function $\Vert\nu\Vert^2$ is analytic on~$G$, hence on $X=G/H$.
Since $x\mapsto\Vert\nu(x)\Vert^2-\nolinebreak\Vert\nu(\gamma\cdot\nolinebreak x)\Vert^2$ is not constant on~$X$, the set~$\mathcal{H}_{\gamma}$ has measure~$0$.
\end{proof}

The fundamental domain~$\mathcal{D}_X$ is an analogue, in the pseudo-Riemannian space $X=G/H$, of the classical Dirichlet domain~$\mathcal{D}_{G/K}$ of \eqref{eqn:RiemDirichlet}.
Indeed, by \eqref{eqn:mudistance} and the $G$-invariance of the metric $d_{G/K}$,
$$\mathcal{D}_{G/K} = \big\{ y\in G/K\ :\ \Vert\mu(y)\Vert \leq \Vert\mu(\gamma\cdot y)\Vert \quad\forall\gamma\in\Gamma\big\} .$$
The distance to the origin $\Vert\mu\Vert$ in $G/K$ is replaced by the ``pseudo-distance to the origin'' $\Vert\nu\Vert$ in~$X$.

The proof of Lemma~\ref{lem:growthfornu}.(2) and~(4) is now similar to that of Lemma~\ref{lem:growthfornu}.(1) and~(3): we just replace \eqref{eqn:ineqmunu} by the following inequality.

\begin{lemma}\label{lem:uniformnu}
Let $\Gamma$ be a discrete subgroup of~$G$ that is $(c,C)$-sharp for~$X$.
For any $\gamma\in\Gamma$ and $x\in\mathcal{D}_X$,
$$\Vert\nu(\gamma\!\cdot\!x)\Vert \,\geq\, \frac{c}{2}\,\Vert\mu(\gamma)\Vert - C.$$
\end{lemma}

\begin{proof}
Let $\gamma\in\Gamma$ and $x\in\mathcal{D}_X$.
There is an element $g\in K\overline{B_+}\subset G$ such that $x=g\!\cdot\!x_0$.
If $\Vert\mu(g)\Vert\geq\frac{c}{2}\,\Vert\mu(\gamma)\Vert$, then, using the definition of~$\mathcal{D}_X$ and the fact that $g\in\nolinebreak K\overline{B_+}$, together with \eqref{eqn:munuB}, we have
$$\Vert\nu(\gamma g)\Vert \,\geq\, \Vert\nu(g)\Vert \,=\, \Vert\mu(g)\Vert \,\geq\, \frac{c}{2}\,\Vert\mu(\gamma)\Vert.$$
If $\Vert\mu(g)\Vert\leq\frac{c}{2}\,\Vert\mu(\gamma)\Vert$, then, using Lemmas \ref{lem:prmu} and~\ref{lem:munu} together with the sharpness of~$\Gamma$, we obtain
\begin{eqnarray*}
\Vert\nu(\gamma g)\Vert & \geq & d_{\aaa}(\mu(\gamma g),\mu(H))\\
& \geq & d_{\aaa}(\mu(\gamma),\mu(H)) - \Vert\mu(g)\Vert\\
& \geq & \frac{c}{2}\,\Vert\mu(\gamma)\Vert - C.\qedhere
\end{eqnarray*}
\end{proof}

\section{Sharpness and deformation}\label{subsec:sharpdeform}

We conclude this chapter by examining the behavior of the sharpness constants under small deformations in the standard case.
The two results below are easy corollaries of \cite[Th.\,1.4]{kas12} (see Example~\ref{ex:deform}).

\begin{lemma}\label{lem:munudeform}
Let $\Gamma$ be a convex cocompact subgroup (for instance a uniform lattice) of some reductive subgroup~$L$ of~$G$ of real rank~$1$ acting properly on the reductive symmetric space~$X$.
Assume that $\Gamma$ is $(c,C)$-sharp for~$X$ and that $\Vert\nu(\gamma)\Vert\geq r$ for all $\gamma\in\Gamma\smallsetminus Z(G_s)$.
For any $\varepsilon>0$ there is a neighborhood $\mathcal{U}_{\varepsilon}\subset\Hom(\Gamma,G)$ of the natural inclusion such that for any $\varphi\in\mathcal{U}_{\varepsilon}$, the group $\varphi(\Gamma)$ is discrete in~$G$ and $(c-\varepsilon,C+\varepsilon)$-sharp for~$X$, with $\Vert\nu(\varphi(\gamma))\Vert\geq r-\varepsilon$ for all $\gamma\in\Gamma\smallsetminus Z(G_s)$.
\end{lemma}

As in Section~\ref{subsec:Lambda+}, we denote by $Z(G_s)$ the center of the commutator subgroup of~$G$.

\begin{proof}
Fix $\varepsilon>0$ and let $\varepsilon'>0$ be small enough so that
$$\frac{c-\varepsilon'}{1+\varepsilon'} \,\geq\, c - \varepsilon \quad\quad\mathrm{and}\quad\quad \varepsilon' + \frac{\varepsilon'}{1+\varepsilon'} \,\leq\, \varepsilon.$$
By \cite[Th.\,1.4]{kas12}, there is a neighborhood $\mathcal{W}_{\varepsilon'}\subset\Hom(\Gamma,G)$ of the natural inclusion such that for any $\varphi\in\mathcal{W}_{\varepsilon'}$, the group $\varphi(\Gamma)$ is discrete in~$G$ and
$$\Vert\mu(\varphi(\gamma))-\mu(\gamma)\Vert \leq \varepsilon'\,\Vert\mu(\gamma)\Vert + \varepsilon'$$
for all $\gamma\in\Gamma$ (and even $\Vert\mu(\varphi(\gamma))-\mu(\gamma)\Vert\leq\varepsilon'\,\Vert\mu(\gamma)\Vert$ for all $\gamma\in\Gamma\smallsetminus K$).
By Lemma~\ref{lem:munu},
\begin{eqnarray*}
\Vert\nu(\varphi(\gamma))\Vert & \geq & d_{\aaa}(\mu(\varphi(\gamma)),\mu(H))\\
& \geq & d_{\aaa}(\mu(\gamma),\mu(H)) - \Vert\mu(\varphi(\gamma))-\mu(\gamma)\Vert\\
& \geq & (c-\varepsilon')\,\Vert\mu(\gamma)\Vert - (C+\varepsilon')\\
& \geq & \frac{c-\varepsilon'}{1+\varepsilon'}\,\Vert\mu(\varphi(\gamma))\Vert - \Big(C + \varepsilon' + \frac{\varepsilon'}{1+\varepsilon'}\Big)
\end{eqnarray*}
for all $\varphi\in\mathcal{W}_{\varepsilon'}$ and $\gamma\in\Gamma$; in particular, $\varphi(\Gamma)$ is $(c-\varepsilon,C+\varepsilon)$-sharp for~$X$.
Since $\Gamma$ is discrete in~$G$ and $\mu$ is a proper map, the set
$$F := \Big\{ \gamma\in\Gamma : \Vert\mu(\gamma)\Vert < \frac{r+C+\varepsilon'}{c-\varepsilon'}\Big\} $$
is finite.
For any $\varphi\in\mathcal{W}_{\varepsilon'}$ and $\gamma\in\Gamma\smallsetminus F$ we have
$$\Vert\nu(\varphi(\gamma))\Vert \,\geq\, (c-\varepsilon')\,\Vert\mu(\gamma)\Vert - (C+\varepsilon') \,\geq\, r.$$
Let $\mathcal{U}_{\varepsilon}$ be the set of elements $\varphi\in\mathcal{W}_{\varepsilon'}$ such that $\Vert\nu(\varphi(\gamma))\Vert\geq r-\varepsilon$ for all $\gamma\in F\smallsetminus Z(G_s)$.
Then $\mathcal{U}_{\varepsilon}$ is a neighborhood of the natural inclusion since $\nu$ is continuous and $F$ finite, and $\mathcal{U}_{\varepsilon}$ satisfies the conclusions of Lemma~\ref{lem:munudeform}.
\end{proof}

\begin{lemma}\label{lem:muexotic}
Suppose that $G={}^{\backprime}G\times\!{}^{\backprime}G$ for some reductive linear group~${}^{\backprime}G$ and let $X=({}^{\backprime}G\times\!{}^{\backprime}G)/\Diag({}^{\backprime}G)$.
Let ${}^{\backprime}G_1$ and~${}^{\backprime}G_2$ be reductive subgroups of~${}^{\backprime}G$ and let $\Gamma={}^{\backprime}\Gamma_1\times{}^{\backprime}\Gamma_2$ for some discrete subgroups ${}^{\backprime}\Gamma_1$ of~${}^{\backprime}G_1$ and ${}^{\backprime}\Gamma_2$ of~${}^{\backprime}G_2$.
Assume that $\Gamma$ is $(c,C)$-sharp for~$X$ and that $\Vert\nu(\gamma)\Vert\geq r$ for all $\gamma\in\Gamma\smallsetminus Z(G_s)$.
\begin{enumerate}
  \item Suppose that for all $i\in\{ 1,2\}$, the group~${}^{\backprime}\Gamma_i$ is
  \begin{itemize}
    \item either an irreducible uniform lattice of~${}^{\backprime}G_i$
    \item or, more generally, a convex cocompact subgroup of~${}^{\backprime}G_i$ if ${}^{\backprime}G_i$ has real rank~$1$.
  \end{itemize}
  Then for any $\varepsilon>\nolinebreak 0$ there is a neighborhood $\mathcal{U}_{\varepsilon}\subset\Hom(\Gamma,G)$ of the natural inclusion such that for any $\varphi\in\mathcal{U}_{\varepsilon}$, the group $\varphi(\Gamma)$ is discrete in~$G$ and $(c-\varepsilon,C+\varepsilon)$-sharp for~$X$, with $\Vert\nu(\varphi(\gamma))\Vert\geq r-\varepsilon$ for all $\gamma\in\Gamma\smallsetminus Z(G_s)$.
  \item Suppose that ${}^{\backprime}G_1$ has real rank~$1$ and that ${}^{\backprime}\Gamma_1$ is convex cocompact in~${}^{\backprime}G_1$.
  Then for any $\varepsilon>0$ there is a neighborhood ${}^{\backprime}\mathcal{U}_{\varepsilon}\subset\linebreak\Hom({}^{\backprime}\Gamma_1,{}^{\backprime}G\times Z_{{}^{\backprime}G}({}^{\backprime}\Gamma_2))$ of the natural inclusion such that for any ${}^{\backprime}\varphi\in{}^{\backprime}\mathcal{U}_{\varepsilon}$, the group ${}^{\backprime}\varphi({}^{\backprime}\Gamma_1){}^{\backprime}\Gamma_2$ is discrete in~$G$ and $(c-\varepsilon,C+\varepsilon)$-sharp for~$X$, with $\Vert\nu(\varphi(\gamma))\Vert\geq r-\varepsilon$ for all $\gamma\in\Gamma\smallsetminus Z(G_s)$.
\end{enumerate}
\end{lemma}

Here $Z_{{}^{\backprime}G}({}^{\backprime}\Gamma_2)$ denotes the centralizer of ${}^{\backprime}\Gamma_2$ in~${}^{\backprime}G$.

\begin{proof}
Fix $\varepsilon>0$ and let $\varepsilon'>0$ be small enough so that
$$\frac{c-2\varepsilon'}{1+2\varepsilon'} \,\geq\, c - \varepsilon \quad\quad\mathrm{and}\quad\quad 2\sqrt{2}\,\varepsilon' + \frac{2\sqrt{2}\,\varepsilon'}{1+2\varepsilon'} \,\leq\, \varepsilon.$$
By \cite[Th.\,1.4]{kas12}, if ${}^{\backprime}G_1$ (resp.~${}^{\backprime}G_2$) has real rank~$1$ and ${}^{\backprime}\Gamma_1$ (resp.~${}^{\backprime}\Gamma_2$) is convex cocompact in~${}^{\backprime}G_1$ (\resp in~${}^{\backprime}G_2$), then there is a neighborhood $\mathcal{W}_{1,\varepsilon'}\subset\Hom(\Gamma,G)$ (\resp $\mathcal{W}_{2,\varepsilon'}\subset\Hom(\Gamma,G)$) of the natural inclusion such that for any $\varphi\in\mathcal{W}_{1,\varepsilon'}$ (\resp $\varphi\in\mathcal{W}_{2,\varepsilon'}$), the group $\varphi({}^{\backprime}\Gamma_1\times\{ e\} )$ (\resp $\varphi(\{ e\} \times{}^{\backprime}\Gamma_2)$) is discrete in~$G$ and
\begin{equation}\label{eqn:muphi1}
\Vert\mu(\varphi({}^{\backprime}\gamma_1,e)) - \mu({}^{\backprime}\gamma_1,e)\Vert \leq \varepsilon'\,\Vert\mu({}^{\backprime}\gamma_1,e)\Vert + \varepsilon'
\end{equation}
for all ${}^{\backprime}\gamma_1\in{}^{\backprime}\Gamma_1$ (resp.
\begin{equation}\label{eqn:muphi2}
\Vert\mu(\varphi(e,{}^{\backprime}\gamma_2)) - \mu(e,{}^{\backprime}\gamma_2)\Vert \leq \varepsilon'\,\Vert\mu(e,{}^{\backprime}\gamma_2)\Vert + \varepsilon'
\end{equation}
for all ${}^{\backprime}\gamma_2\in{}^{\backprime}\Gamma_2$).
If ${}^{\backprime}G_1$ (resp.~${}^{\backprime}G_2$) has real rank $\geq 2$ and ${}^{\backprime}\Gamma_1$ (resp.~${}^{\backprime}\Gamma_2$) is an irreducible lattice in~${}^{\backprime}G_1$ (\resp in~${}^{\backprime}G_2$), then ${}^{\backprime}\Gamma_1$ (resp.~${}^{\backprime}\Gamma_2$) is locally rigid in~$G$ \cite{rag65,wei64}, and so a similar neighborhood $\mathcal{W}_{1,\varepsilon'}\subset\Hom(\Gamma,G)$ (\resp $\mathcal{W}_{2,\varepsilon'}\subset\Hom(\Gamma,G)$) of the natural inclusion exists by \eqref{eqn:rightstrongtriangineq} and \eqref{eqn:leftstrongtriangineq}.
Since $\Gamma$ is discrete in~$G$ and $\mu$ is a proper map, the set
$$F := \Big\{ \gamma\in\Gamma : \Vert\mu(\gamma)\Vert < \frac{r+C+2\sqrt{2}\,\varepsilon'}{c-2\varepsilon'}\Big\} $$
is finite.
In the setting of~\emph{(1)}, we let $\mathcal{U}_{\varepsilon}$ be the set of elements $\varphi\in\mathcal{W}_{1,\varepsilon'}\cap\mathcal{W}_{2,\varepsilon'}$ such that $\Vert\nu(\varphi(\gamma))\Vert\geq r-\varepsilon$ for all $\gamma\in F\smallsetminus Z(G_s)$; then $\mathcal{U}_{\varepsilon}\subset\Hom(\Gamma,G)$ is a neighborhood of the natural inclusion and any $\varphi\in\mathcal{U}_{\varepsilon}$ satisfies \eqref{eqn:muphi1} and~\eqref{eqn:muphi2}.
In the setting of~\emph{(2)}, we set
$${}^{\backprime}\mathcal{W}_{\varepsilon'} := \big\{ \varphi\circ i_1 :\ \varphi\in\mathcal{W}_{1,\varepsilon'},\ \varphi|_{\{ e\} \times{}^{\backprime}\Gamma_2}=\mathrm{id}_{\{ e\} \times{}^{\backprime}\Gamma_2}\big\},$$
where $i_1 : {}^{\backprime}\Gamma_1\hookrightarrow{}^{\backprime}\Gamma_1\times\{ e\} $ is the natural inclusion, and we let ${}^{\backprime}\mathcal{U}_{\varepsilon}$ be the set of elements ${}^{\backprime}\varphi\in{}^{\backprime}\mathcal{W}_{\varepsilon'}$ such that $\Vert\nu({}^{\backprime}\varphi({}^{\backprime}\gamma_1){}^{\backprime}\gamma_2)\Vert\geq r-\varepsilon$ for all $\gamma=({}^{\backprime}\gamma_1,{}^{\backprime}\gamma_2)\in F\smallsetminus\nolinebreak Z(G_s)$; then ${}^{\backprime}\mathcal{U}_{\varepsilon}\subset\Hom({}^{\backprime}\Gamma_1,{}^{\backprime}G\times Z_{{}^{\backprime}G}({}^{\backprime}\Gamma_2))$ is a neighborhood of the natural inclusion and for any ${}^{\backprime}\varphi\in{}^{\backprime}\mathcal{U}_{\varepsilon}$, the homomorphism $\varphi:=(({}^{\backprime}\gamma_1,{}^{\backprime}\gamma_2)\mapsto\nolinebreak{}^{\backprime}\varphi({}^{\backprime}\gamma_1){}^{\backprime}\gamma_2)$ satisfies \eqref{eqn:muphi1} and~\eqref{eqn:muphi2}.

We now consider $\varphi\in\Hom(\Gamma,G)$ satisfying \eqref{eqn:muphi1} and~\eqref{eqn:muphi2} and prove that the group~$\varphi(\Gamma)$ is discrete in~$G$ and $(c-\varepsilon,C+\varepsilon)$-sharp for~$X$, with $\Vert\nu(\varphi(\gamma))\Vert\geq r-\varepsilon$ for all $\gamma\in\Gamma\smallsetminus Z(G_s)$.
We note that $\aaa={}^{\backprime}\aaa+\!{}^{\backprime}\aaa$, where ${}^{\backprime}\aaa$ is a maximal split abelian subspace of~${}^{\backprime}\g$; for $i\in\{ 1,2\}$, let $\pi_i : \aaa\rightarrow{}^{\backprime}\aaa$ be the projection onto the $i$-th factor.
Then
\begin{eqnarray*}
& & \big\Vert\pi_1\big(\mu(\varphi({}^{\backprime}\gamma_1,{}^{\backprime}\gamma_2)) - \mu({}^{\backprime}\gamma_1,{}^{\backprime}\gamma_2)\big)\big\Vert\ =\ \big\Vert\pi_1\big(\mu(\varphi({}^{\backprime}\gamma_1,{}^{\backprime}\gamma_2)) - \mu({}^{\backprime}\gamma_1,e)\big)\big\Vert\\
& \leq & \big\Vert\pi_1\big(\mu(\varphi({}^{\backprime}\gamma_1,{}^{\backprime}\gamma_2)) - \mu(\varphi({}^{\backprime}\gamma_1,e))\big)\big\Vert + \big\Vert\pi_1\big(\mu(\varphi({}^{\backprime}\gamma_1,e)) - \mu({}^{\backprime}\gamma_1,e)\big)\big\Vert,
\end{eqnarray*}
where
\begin{eqnarray*}
\big\Vert\pi_1\big(\mu(\varphi({}^{\backprime}\gamma_1,{}^{\backprime}\gamma_2)) - \mu(\varphi({}^{\backprime}\gamma_1,e))\big)\big\Vert & \leq & \big\Vert\pi_1\big(\mu(\varphi(e,{}^{\backprime}\gamma_2)\big)\big\Vert\\
& = & \big\Vert\pi_1\big(\mu(\varphi(e,{}^{\backprime}\gamma_2)) - \mu(e,{}^{\backprime}\gamma_2)\big)\big\Vert\\
& \leq & \Vert\mu(\varphi(e,{}^{\backprime}\gamma_2)) - \mu(e,{}^{\backprime}\gamma_2)\Vert\\
& \leq & \varepsilon'\,\Vert\mu(e,{}^{\backprime}\gamma_2)\Vert + \varepsilon'
\end{eqnarray*}
(using \eqref{eqn:rightstrongtriangineq} applied to~${}^{\backprime}G$ and \eqref{eqn:muphi2}) and
\begin{eqnarray*}
\big\Vert\pi_1\big(\mu(\varphi({}^{\backprime}\gamma_1,e)) - \mu({}^{\backprime}\gamma_1,e)\big)\big\Vert & \leq & \Vert\mu(\varphi({}^{\backprime}\gamma_1,e)) - \mu({}^{\backprime}\gamma_1,e)\Vert\\
& \leq &\varepsilon'\,\Vert\mu({}^{\backprime}\gamma_1,e)\Vert + \varepsilon'
\end{eqnarray*}
(using \eqref{eqn:muphi1}).
Therefore,
\begin{eqnarray*}
\big\Vert\pi_1\big(\mu(\varphi({}^{\backprime}\gamma_1,{}^{\backprime}\gamma_2)) - \mu({}^{\backprime}\gamma_1,{}^{\backprime}\gamma_2)\big)\big\Vert & \leq & \varepsilon'\,\big(\Vert\mu({}^{\backprime}\gamma_1,e)\Vert + \Vert\mu(e,{}^{\backprime}\gamma_2)\Vert\big) + 2 \varepsilon'\\
& \leq & \sqrt{2}\varepsilon'\,\Vert\mu({}^{\backprime}\gamma_1,{}^{\backprime}\gamma_2)\Vert + 2 \varepsilon'.
\end{eqnarray*}
Similarly,
$$\big\Vert\pi_2\big(\mu(\varphi({}^{\backprime}\gamma_1,{}^{\backprime}\gamma_2)) - \mu({}^{\backprime}\gamma_1,{}^{\backprime}\gamma_2)\big)\big\Vert \leq \sqrt{2}\varepsilon'\,\Vert\mu({}^{\backprime}\gamma_1,{}^{\backprime}\gamma_2)\Vert + 2 \varepsilon'.$$
Thus
$$\Vert\mu(\varphi(\gamma)) - \mu(\gamma)\Vert \leq 2 \varepsilon'\,\Vert\mu(\gamma)\Vert + 2\sqrt{2}\varepsilon'$$
for all $\gamma\in\Gamma$.
Using the fact that $\Gamma$ is discrete in~$G$ and $\mu$ is a proper~map, we obtain that $\varphi(\Gamma)$ is discrete in~$G$.
We conclude as in the proof of Lemma~\ref{lem:munudeform}.
\end{proof}

%% file: StableSpec5-FJ.tex
\chapter[Asymptotic estimates for eigenfunctions on~$X$]{Asymptotic estimates for eigenfunctions on symmetric~spaces}\label{sec:Vlambda}

Under the rank condition \eqref{eqn:rank}, Flensted-Jensen \cite{fle80} proved that the space $L^2(X,\M_{\lambda})_K$ of $K$-finite elements in $L^2(X,\M_{\lambda})$ is nonzero for infinitely many joint eigenvalues~$\lambda$, by an explicit construction based on some duality principle and the Poisson transform.
Then, applying deep results of microlocal analysis and hyperfunction theory \cite{kkmoot78}, Oshima and Matsuki \cite{mo84,osh88b} gave a detailed analysis of the asymptotic behavior at infinity of these eigenfunctions.
In this chapter, we reformulate their estimates as follows, in terms of
\begin{itemize}
  \item the ``weighted distance'' $d(\lambda)$ of the spectral parameter~$\lambda$ to the walls of~$\jj^{\ast}$ (which measures the regularity of~$\lambda$),
  \item the ``pseudo-distance from the origin'' $\Vert\nu(x)\Vert$ of $x\in X$ (which measures how $x$ goes to infinity).
\end{itemize}

\begin{proposition}\label{prop:asym}
Under the rank condition \eqref{eqn:rank}, there is a constant $q>\nolinebreak 0$ such that for all $\lambda\in\jj^{\ast}$ and $\varphi\in L^2(X,\M_{\lambda})_K$, the function
$$x \longmapsto \varphi(x) \cdot e^{q\,d(\lambda)\Vert\nu(x)\Vert}$$
is bounded on~$X$; in particular, $\varphi\in L^1(X)$ if $d(\lambda)>2\Vert\rho_b\Vert/q$.
\end{proposition}

We refer to Section \ref{subsec:Lambda+} (\resp \ref{subsec:munu}) for the definition of $d : \jj^{\ast}\rightarrow\R_{\geq 0}$ (\resp $\nu : X\rightarrow\overline{\bb_+}$).
As in Remark~\ref{rem:lattice}, we denote by $\Vert\rho_b\Vert$ the norm of half the sum of the elements of a positive system $\Sigma^+(\g,\bb)$ of restricted roots of~$\bb$ in~$\g$; this norm does not depend on the choice~of~$\Sigma^+(\g,\bb)$.

As we shall see, the constant~$q$ is computable in terms of some root system (see \eqref{eqn:q} in the proof of Lemma~\ref{lem:Yd}).

The proof of Proposition~\ref{prop:asym} will be given in Section~\ref{subsec:asym}.
For the reader's convenience, we first give a brief review of the Poisson transform on Riemannian symmetric spaces of the noncompact type (Section~\ref{subsec:Poisson}), of the Flensted-Jensen duality (Section~\ref{subsec:realforms}), and of the construction of discrete series representations (Section~\ref{subsec:discreteseries}).
The material of these three sections is not new, but we will need it later.
Often analysis on reductive symmetric spaces requires a rather large amount of notation; here we try to keep it minimal for our purpose.

In the whole chapter, we denote by $\mathcal{A}$ the sheaf of real analytic functions and by $\mathcal{B}$ the sheaf of hyperfunctions; we refer to \cite{kkk86} for an introduction to hyperfunctions.

\section{Poisson transform in Riemannian symmetric spaces}\label{subsec:Poisson}

Let $X^d=G^d/K^d$ be a Riemannian symmetric space of the noncompact type, where $G^d$ is a connected reductive linear Lie group and $K^d$ a maximal compact subgroup of~$G^d$.
Let $P^d$ be a minimal parabolic subgroup of~$G^d$.
We give a brief overview of the theory of the Poisson transform as an intertwining operator between hyperfunctions on $G^d/P^d$ and eigenfunctions on~$X^d$ (see \cite{hel00, kkmoot78} for details).
The notation~$G^d$ is used to avoid confusion since the results of this paragraph will not be applied to~$G$ but to another real form of~$G_{\C}$.

Let $\jj$ be a maximal split abelian subalgebra of $\g^d:=\mathrm{Lie}(G^d)$ such that the Cartan decomposition $G^d=K^d(\exp\jj)K^d$ holds.
Since all minimal parabolic subgroups of~$G^d$ are conjugate, we may assume that $P^d$ contains $\exp\jj$ and has the Langlands decomposition $P^d=M^d(\exp\jj)N^d$, where $M^d=K^d\cap P^d$ is the centralizer of~$\exp\jj$ in~$K^d$ and $N^d$ is the unipotent radical of~$P^d$.
The Iwasawa decomposition $G^d=K^d(\exp\jj)N^d$ holds.
Let $\zeta : G^d\rightarrow\jj$ be the corresponding Iwasawa projection, defined by
$$g \in K^d (\exp\zeta(g)) N^d$$
for all $g\in G^d$.
For $\lambda\in\jj_{\C}^{\ast}$ we define functions $\xi_{\lambda},\xi_{\lambda}^{\vee}\in\mathcal{A}(G^d)$ by
\begin{equation}\label{eqn:xinu}
\xi_{\lambda}(g) := e^{\langle\lambda,\zeta(g)\rangle} \quad\quad\mathrm{and}\quad\quad \xi_{\lambda}^{\vee}(g) := \xi_{\lambda}(g^{-1})
\end{equation}
for $g\in G^d$.
Since $\xi_{\lambda}$ is left-$K^d$-invariant, $\xi_{\lambda}^{\vee}$ induces a function on~$X^d$, which we still denote by~$\xi_{\lambda}^{\vee}$.

We choose a positive system $\Sigma^+(\g_{\C},\jj_{\C})$, defining positive Weyl chambers $\jj_+$ in~$\jj$ and $\jj^{\ast}_+$ in~$\jj^{\ast}$.
Let $\rho$ be half the sum of the elements of $\Sigma^+(\g_{\C},\jj_{\C})$, counted with root multiplicities.
For $\lambda\in\jj_{\C}^{\ast}$, the function~$\xi_{\lambda}$ is a character of~$P^d$.
Let $\mathcal{B}(G^d/P^d,\mathcal{L}_\lambda)$ be the hyperfunction-valued normalized principal series representation of~$G^d$ associated with the character $\xi_{-\lambda}$ of~$P^d$: by definition, $\mathcal{B}(G^d/P^d,\mathcal{L}_{\lambda})$ is the set of hyperfunctions $f\in\mathcal{B}(G^d)$ such that
$$f(\,\cdot\,p) = \xi_{-\lambda+ \rho}(p^{-1})f \, (=f\,\xi_{\lambda-\rho}(p))$$
for all $p\in P^d$.
Here we use the character $\xi_{-\lambda}$ and not~$\xi_{\lambda}$, following the usual convention in harmonic analysis on symmetric spaces (see \cite{bs05,del98,fle80,hel00,mo84}) rather than in the representation theory of reductive groups (see \cite{kna86,war72}).
Setting
$$\mathcal{A}(G^d/P^d,\mathcal{L}_{-\lambda}) := \mathcal{A}(G^d) \cap \mathcal{B}(G^d/P^d,\mathcal{L}_{-\lambda}),$$
there is a natural $G^d$-invariant bilinear form
$$\langle\,\cdot\,,\cdot\,\rangle\ :\ \mathcal{B}(G^d/P^d,\mathcal{L}_\lambda) \times \mathcal{A}(G^d/P^d,\mathcal{L}_{-\lambda}) \longrightarrow \C$$
given by the integration over $G^d/P^d$.
We note that $\xi_{-\lambda-\rho}\in\mathcal{A}(G^d/P^d,\mathcal{L}_{-\lambda})$, hence the left translate~$\xi_{-\lambda-\rho}(g^{-1}\,\cdot\,)$ also belongs to $\mathcal{A}(G^d/P^d,\mathcal{L}_{-\lambda})$ for all $g\in\nolinebreak G^d$.
Since $\xi_{-\lambda-\rho}$ is left-$K^d$-invariant, we obtain a $G^d$-intertwining operator (\emph{Poisson transform})
$$\mathcal{P}_{\lambda} : \mathcal{B}(G^d/P^d,\mathcal{L}_\lambda) \longrightarrow \mathcal{A}(X^d)$$
given by
$$(\mathcal{P}_{\lambda} f)(g) := \langle f,\xi_{-\lambda-\rho}(g^{-1}\,\cdot\,)\rangle\,.$$
It follows directly from the definition of the Harish-Chandra isomorphism in Section~\ref{subsec:DGH} that for all $f\in\mathcal{B}(G^d/P^d,\mathcal{L}_\lambda)$, the function $\mathcal{P}_{\lambda}f\in\mathcal{A}(X^d)$ satisfies the system $(\M_{\lambda})$, defined similarly to Section~\ref{subsec:DGH}.
For $\operatorname{Re}\lambda\in\nolinebreak\overline{\jj^{\ast}_+}$, the Helgason conjecture (proved in \cite{kkmoot78}) asserts that the Poisson transform
$$\mathcal{P}_{\lambda} : \mathcal{B}(G^d/P^d,\mathcal{L}_{\lambda}) \longrightarrow \mathcal{A}(G^d/K^d,\M_{\lambda})$$
is actually a bijection.

\begin{example}
Assume that $G^d$ has real rank~$1$.
Then $G^d/P^d$ identifies with the boundary at infinity of~$X^d$.
The function $\xi_{\lambda}^{\vee}$ is the exponential of some multiple of the Busemann function associated with the geodesic ray $(\exp\jj_+) K^d$ in $X^d=G^d/K^d$; its level sets are the horospheres centered at $eP^d\in G^d/P^d$.
For $\lambda=\rho$, the Poisson operator~$\mathcal{P}_{\lambda}$ identifies the set of continuous functions on $G^d/P^d$ with the set of harmonic functions on~$X^d$ admitting a continuous extension to $\overline{X^d}=X^d\cup G^d/P^d$.
(See Section~\ref{subsec:PoissonH3} for the case $G^d=\SL_2(\C)$.)
\end{example}

\section{Real forms of~$G_{\C}/H_{\C}$ and the Flensted-Jensen duality}\label{subsec:realforms}

We now come back to the setting of Chapters \ref{sec:intro} to~\ref{sec:geometry}, where $G$ is a connected reductive linear Lie group and $H$ an open subgroup of the group of fixed points of~$G$ under some involutive automorphism~$\sigma$.
Let $G_{\C}$ be a connected Lie group containing~$G$ with Lie algebra $\g_{\C}:=\g\otimes_{\R}\C$, and let $H_{\C}$ be the connected subgroup of~$G_{\C}$ with Lie algebra $\h_{\C}:=\h\otimes_{\R}\C$.
We consider three different real forms of the complex symmetric space $X_{\C}=G_{\C}/H_{\C}$: our original pseudo-Riemannian symmetric space $X=G/H$, a Riemannian symmetric space $X_U=G_U/H_U$ of the compact type, and a Riemannian symmetric space $X^d=G^d/K^d$ of the noncompact type.
They are constructed as follows.
Let $\g=\h+\q$ be the decomposition of~$\g$ into eigenspaces of~$\mathrm{d}\sigma$ as in Section~\ref{subsec:DGH}, and let $\g=\kk+\p$ be the Cartan decomposition associated with the Cartan involution $\theta$ of~$G$ of Section~\ref{subsec:Laplacian}, which commutes with~$\sigma$.
The maps $\mathrm{d}\sigma$ and $\mathrm{d}\theta$ extend to automorphisms of the complex Lie algebra~$\g_{\C}$, for which we use the same letters.
We set
\begin{eqnarray*}
\g^d & := & \g^{\sigma\theta} + \sqrt{-1}\,\g^{-\sigma\theta} = (\h\cap\kk + \q\cap\p) + \sqrt{-1}\,(\h\cap\p + \q\cap\kk),\\
\kk^d = \h_U & := & \h\cap\kk + \sqrt{-1}\,(\h\cap\p),\\
\g_U & := & \kk + \sqrt{-1}\,\p,
\end{eqnarray*}
and let $G^d$ (\resp $K^d=H_U$, \resp $G_U$) be the connected subgroup of~$G_{\C}$ with Lie algebra $\g^d$ (\resp $\kk^d=\h_U$, resp.~$\g_U$).
We note that $K^d=H_U$ is the compact real form of~$H_{\C}$.
For instance, for the anti-de Sitter space $X=\AdS^{2n+1}=\SO(2,2n)_0/\SO(1,2n)_0$, we have $X_U=\SO(2n+2)/\SO(2n+1)=\mathbb{S}^{2n+1}$ and $X^d=\SO(1,2n+1)_0/\SO(2n+1)=\HH^{2n+1}$ (see Section~\ref{subsec:AdS}).

Let $H^d$ be the connected subgroup of~$G_{\C}$ with Lie algebra
$$\h^d := \h\cap\kk + \sqrt{-1}\,(\q\cap\kk).$$
We note that $K^d\cap H^d=(H\cap K)_0$ and that $H^d/K^d\cap H^d$ and $K/H\cap K$ are two Riemannian symmetric spaces with the same complexification --- the first one of the noncompact type, the second one of the compact type.
This will be used in Chapter~\ref{sec:FJ}.

For any $\h^d$-module~$V$ over~$\C$, the action of~$\h^d$ on~$V$ extends $\C$-linearly to an action of $\kk_{\C}=\h^d\otimes_{\R}\C$, and the set~$V_{\h^d}$ of $\h^d$-finite vectors is equal
 to the set~$V_{\kk_{\C}}$ of $\kk_{\C}$-finite vectors.
We define the set~$V_K$ of $K$-finite vectors of~$V$ to consist of vectors $v\in V_{\h^d}=V_{\kk_{\C}}$ such that the action of $\kk\subset\kk_{\C}$ on the $\C$-span of $\kk\cdot v$ lifts to an action of~$K$.
Then $V_K$ is a $K$-module contained in $V_{\h^d}$.

\begin{remark}\label{rem:Kfinite}
In the definition of~$V_K$, we do not assume that the group $K$ acts on~$V$.  
In the situation below, neither $V$ nor $V_{\h^d}=V_{\kk_{\C}}$ can be acted on by the group~$K$.  
\end{remark}

The Lie algebra~$\g^d$ (hence its subalgebra~$\h^d$) acts on $\mathcal{A}(X^d)$ by differentiation on the left:
\begin{equation}\label{eqn:gdiffleft}
(Y\cdot\varphi)(x) = \frac{\mathrm{d}}{\mathrm{d}t}\Big|_{t=0} \varphi\big(\exp(-tY)\cdot x\big)
\end{equation}
for all $Y\in\g^d$, all $\varphi\in\mathcal{A}(X^d)$, and all $x\in X^d$.
Since the system $(\M_{\lambda})$ is $G^d$-invariant, its space of solutions $\mathcal{A}(X^d,\M_{\lambda})$ is a $\g^d$-submodule of $\mathcal{A}(X^d)$ for $\lambda\in\jj_{\C}^{\ast}$; thus we can define $K$-modules $\mathcal{A}(X^d,\M_{\lambda})_K\subset\mathcal{A}(X^d)_K$.
By using holomorphic continuation, Flensted-Jensen \cite{fle80} constructed an injective homomorphism
\begin{alignat}{3}\label{eqn:FJisom}
\eta : \hphantom{i} &\quad\mathcal{A}(X)_K
        &&\ \longhookrightarrow\quad
        &&\;\;\;\mathcal{A}(X^d)_K
\\
        &\hphantom{iii}\quad\cup && &&\hphantom{iii}\quad\cup
\notag
\\
        &\mathcal{A}(X,\M_{\lambda})_K 
        &&\ \longhookrightarrow\quad
        &&\mathcal{A}(X^d,\M_{\lambda})_K
\notag
\end{alignat}
for all $\lambda\in\jj_{\C}^{\ast}$.  
For the reader's convenience, we now recall the construction of~$\eta$ in the case when $G_{\C}$ is simply connected.

Assume that $G_{\C}$ is simply connected.
Then the set of fixed points of~$G_{\C}$ under any involutive automorphism is connected \cite[Th.\,3.4]{bor61}.
We can extend $\sigma$ and~$\theta$ to holomorphic automorphisms of~$G_{\C}$, for which we use the same letters $\sigma$ and~$\theta$.  
The complex conjugation of $\g_{\C}=\g+\sqrt{-1}\,\g$ with respect to the real form~$\g$ lifts to an anti-holomorphic involution~$\tau$ of~$G_{\C}$, such that $G=G_{\C}^{\tau}$.
Since $\sigma$, $\theta$, and~$\tau$ commute, the composition of any of them gives involutive automorphisms of~$G_{\C}$.
We have
$$H_{\C} = G_{\C}^{\sigma}, \quad G^d = G_{\C}^{\tau\sigma\theta}, \quad K^d = H_U = H_{\C} \cap G^d, \quad\mathrm{and}\quad G_U = G_{\C}^{\tau\theta}.$$
Moreover, setting $K_{\C} = G_{\C}^{\theta}$, we have $H^d=(K_{\C} \cap G^d)_0$ and the following inclusions hold:
\begin{alignat}{5}\label{eqn:KGH}
     &\,K                                      
     &&\,\,\subset\,\, 
     &&\,G
     &&\,\supset\,\, 
     &&\,H\nonumber
\\
     &\,\,\text{\rotatebox{90}{$\supset$}}
     &&         
     &&\,\,\text{\rotatebox{90}{$\supset$}}
     && 
     &&\,\,\text{\rotatebox{90}{$\supset$}}\nonumber
\\
     &K_{\C}
     &&\,\,\subset\,\,
     &&G_{\C} 
     &&\,\supset \,\,
     &&H_{\C}
\\
     &\,\,\text{\rotatebox{90}{$\subset$}}
     && 
     &&\,\,\text{\rotatebox{90}{$\subset$}}
     && 
     &&\,\,\text{\rotatebox{90}{$\subset$}}\nonumber
\\ 
     &\,H^d  
     &&\,\, \subset\,\, 
     && \,G^d
     &&\,\supset\,\, 
     &&\,K^d.\nonumber
\end{alignat}
The restriction of~$\sigma$ to~$G^d$ is a Cartan involution of~$G^d$, and the corresponding Cartan decomposition $\g^d=\kk^d+\p^d$ is obtained as the intersection of~$\g^d$ with the direct sum decomposition $\g_{\C}=\h_{\C}+\q_{\C}$.
The restriction of~$\theta$ to~$G^d$ is an involution of~$G^d$, and the corresponding decomposition $\g^d=\h^d+\q^d$ of~$\g^d$ (into eigenspaces of~$\mathrm{d}\theta$ with respective eigenvalues $+1$ and~$-1$) is obtained as the intersection of~$\g^d$ with the complexified Cartan decomposition $\g_{\C}=\kk_{\C}+\p_{\C}$.
Let $\bb$ be the maximal semisimple abelian subspace of $\p\cap\q$ from Section~\ref{subsec:munu}.
Since $\p^d\cap\q^d=\p\cap\q$, we may regard $B=\exp\bb$ as a subgroup of~$G^d$, and the polar decomposition $G^d=H^d\overline{B_+}K^d$ holds similarly to the polar decomposition $G=K\overline{B_+}H$ of Section~\ref{subsec:munu}.
Any function $f\in\mathcal{A}(X)_K$ extends uniquely to a function $f_{\C} : K_{\C}\overline{B_+}H_{\C}/H_{\C}\rightarrow\C$ such that $k\mapsto f_{\C}(kbH_{\C})$ is holomorphic on~$K_{\C}$ for any $b\in\overline{B_+}$; by letting $\eta(f)$ be the restriction of~$f_{\C}$ to~$X^d$, we get the injective homomorphism \eqref{eqn:FJisom}, which is actually bijective.
The homomorphism~$\eta$ respects the left action of $U(\g_{\C})$ (\cite[Th.\,2.5]{fle80}).

We now return to the general case, where $G_{\C}$ is not necessarily simply connected.
Any $G$-invariant (\resp $G_U$-invariant, \resp $G^d$-invariant) differential operator on $X=G/H$ (\resp $X_U=G_U/H_U$, \resp $X^d=G^d/K^d$) extends holomorphically to $X_{\C}=G_{\C}/H_{\C}$, hence we have canonical $\C$-algebra isomorphisms
$$\D(X) \simeq \D(X_U) \simeq \D(X^d).$$
Therefore, for $\lambda\in\jj_{\C}^{\ast}$, a function $f\in\mathcal{A}(X)$ satisfies $(\M_{\lambda})$ if and only if $\eta(f)\in\mathcal{A}(X^d)$ does.

\section{Discrete series representations}\label{subsec:discreteseries}

We continue in the setting of Section~\ref{subsec:realforms} and now assume that the rank condition \eqref{eqn:rank} is satisfied.
In this section we summarize Flensted-Jensen's construction of discrete series representations $\V_{Z,\lambda}$ using his duality \eqref{eqn:FJisom}.
Recall that the regular representation of~$G$ on~$L^2(X)$ is unitary; an irreducible unitary representation~$\pi$ of~$G$ is said to be a \emph{discrete series representation} for~$X$ if there exists a nonzero continuous $G$-intertwining operator from $\pi$ to $L^2(X)$ or, equivalently, if $\pi$ can be realized as a closed $G$-invariant subspace of $L^2(X)$.
By a little abuse of notation, we shall also call the underlying $(\g,K)$-module~$\pi_K$ a discrete series representation.
It should be noted that discrete series representations for $X=G/H$ may be different from Harish-Chandra's discrete series representations for the group manifold~$G$ if $H$ is noncompact, because $L^2(X)\neq L^2(G)^H$.

We shall parameterize the discrete series representations for~$X$ by the spectral parameter~$\lambda$ and some finite set~$\ZZ$ defined as follows.
Let $\mathcal{P}^d$ be the set of minimal parabolic subalgebras of~$\g^d$, on which $G^d$ acts transitively by the adjoint action.
There are only finitely many $H^d$-orbits in~$\mathcal{P}^d$; a combinatorial description was given by Matsuki \cite{mat79}.
We set
\begin{equation}\label{eqn:Z}
\ZZ := \{ \text{closed $H^d$-orbits in $\mathcal{P}^d$}\}.
\end{equation}
Here is a description of the finite set~$\ZZ$.
Consider the maximal semisimple abelian subspace $\jj$ of $\sqrt{-1}(\q\cap\kk)$ from Chapter~\ref{sec:theorems}.
The rank condition~\eqref{eqn:rank} is equivalent to the fact that $\jj$ is maximal abelian in $\p^d=\q\cap\p+\sqrt{-1}(\q\cap\kk)$.
Thus $\jj$ is a maximal split abelian subalgebra of~$\g^d$ and the notation fits with that of Section~\ref{subsec:Poisson}.
All restricted roots of~$\jj$ in~$\g^d$ take real values on~$\jj$ and there is a natural bijection $\Sigma(\g^d,\jj)\simeq\Sigma(\g_{\C},\jj_{\C})$.
Note that $\jj$ is actually contained in~$\h^d$; there is a natural bijection $\Sigma(\h^d,\jj)\simeq\Sigma(\kk_{\C},\jj_{\C})$.
As in Section~\ref{subsec:DGH}, let $W$ be the Weyl group of the restricted root system $\Sigma(\g^d,\jj)$, and let $W_{H \cap K}$ be that of $\Sigma(\h^d, \jj)$.
Any choice of a positive system $\Sigma^+(\g^d,\jj)\simeq\Sigma^+(\g_{\C},\jj_{\C})$ defines a point in~$\mathcal{P}^d$ and the $H^d$-orbit through this point is closed.
Conversely, any closed $H^d$-orbit in~$\mathcal{P}^d$ is obtained in this way.
Recall that in Section~\ref{subsec:Lambda+} we have fixed once and for all a positive system $\Sigma^+(\kk_{\C},\jj_{\C})\simeq \Sigma^+(\h^d,\jj)$.
Since any two such positive systems are conjugate by $H^d$, we obtain a one-to-one correspondence
\begin{equation}\label{eqn:Z1to1}
\big\{ \text{positive systems $\Sigma^+(\g^d,\jj)$ containing $\Sigma^+(\h^d,\jj)$}\big\} \simeq \ZZ.
\end{equation}

Here is another description of the finite set~$\ZZ$.
We fix a positive system $\Sigma^+(\g^d,\jj)$ containing $\Sigma^+(\h^d,\jj)$; this defines a minimal parabolic subgroup~$P^d$ of~$G^d$.
The subspace~$\p^d$ in the Cartan decomposition $\g^d=\kk^d+\p^d$ should not be confused with the Lie algebra of~$P^d$.
The subset
\begin{equation}\label{eqn:Wcoset}
W(H^d,G^d):= \big\{w \in W :\ w(\Sigma^+(\g^d,\jj)) \cap \Sigma(\h^d,\jj) = \Sigma^+(\h^d,\jj)\big\}.  
\end{equation}
of the Weyl group~$W$ gives a complete set of representatives of the left coset space $W_{H \cap K}\backslash W$.
Clearly, $e\in W(H^d,G^d)$.  
We identify $\mathcal{P}^d$ with $G^d/P^d$.
Then, by \eqref{eqn:Z1to1}, the other closed $H^d$-orbits in $G^d/P^d$ are of the form 
\begin{equation}\label{eqn:Zw}
Z=H^d w P^d \quad \text{for $w \in W(H^d,G^d) \hphantom{i}(\simeq W_{H \cap K}\backslash W)$.}
\end{equation}
Thus we have a one-to-one correspondence
\begin{equation}\label{eqn:zw}
\ZZ \simeq W(H^d,G^d).  
\end{equation}

\begin{remark}
We have given two equivalent combinatorial descriptions of the finite set~$\ZZ$ in \eqref{eqn:Z1to1} and \eqref{eqn:zw}.  
The latter one \eqref{eqn:zw} depends on a fixed choice of a positive system $\Sigma^+(\g^d,\jj)$; it is convenient to treat different closed orbits~$Z$ simultaneously (\emph{e.g.}\ in Fact~\ref{fact:decompVlambda} below).
We shall use the former one \eqref{eqn:Z1to1} when we give an estimate of the asymptotic behavior of individual discrete series representations for a fixed $Z\in\ZZ$ (\emph{e.g.}\ in the proof of Proposition~\ref{prop:asym} in Section~\ref{subsec:asym}, or in Chapter~\ref{sec:FJ}).
\end{remark}

We now recall from \cite{fle80} how to construct, for any $Z\in\ZZ$ and infinitely many $\lambda\in\jj_{\C}^{\ast}$, a subspace $\V_{Z,\lambda}$ of $L^2(X,\M_{\lambda})_K$ that will be a discrete series representation for~$X$.
For $Z\in\ZZ$ and $\lambda\in\jj_{\C}^{\ast}$, we define a $\g^d$-submodule
$$\mathcal{B}_Z(G^d/P^d,\mathcal{L}_\lambda) := \big\{ f\in\mathcal{B}(G^d/P^d,\mathcal{L}_\lambda) : \operatorname{supp} f \subset Z \big\} $$
of the principal series representation $\mathcal{B}(G^d/P^d,\mathcal{L}_\lambda)$ of Section~\ref{subsec:Poisson}.
Similar\-ly to the definition of $\mathcal{A}(G^d/K^d,\M_{\lambda})_K$, we can define the set $\mathcal{B}_Z(G^d/P^d,\mathcal{L}_\lambda)_K$ of $K$-finite elements in $\mathcal{B}_Z(G^d/P^d,\mathcal{L}_\lambda)$ even though the group $K$ does not act on ${\mathcal{B}}_Z(G^d/P^d,{\mathcal{L}}_{\lambda})$ (see Remark~\ref{rem:Kfinite}).
For $\operatorname{Re}\lambda\in\overline{\jj^{\ast}_+}$, we then have the following commutative diagram, where $\mathcal{P}_{\lambda}$ is the Poisson transform of Section~\ref{subsec:Poisson}.
\begin{align*}
&
  \begin{matrix}
     \mathcal{B}(G^d/P^d,\mathcal{L}_\lambda)
     &\underset{\mathcal{P}_{\lambda}}{\overset{\sim}{\longrightarrow}}
     &\mathcal{A}(G^d/K^d,\M_\lambda)
   \\
     \cup && \cup
   \\[.5ex]
     \mathcal{B}_Z(G^d/P^d,\mathcal{L}_\lambda)_K
     &\longrightarrow
     &\mathcal{A}(G^d/K^d,\M_\lambda)_K
     &\overset{\eta}{\longhookleftarrow}
     &\mathcal{A}(X,\M_\lambda)_K.
  \end{matrix}
\end{align*}
We set
\begin{equation}\label{eqn:Vzlmd}
\V_{Z,\lambda} := \eta^{-1}\Big(\mathcal{P}_{\lambda}\big(\mathcal{B}_Z(G^d/P^d,\mathcal{L}_\lambda)_K\big)\Big).  
\end{equation}
Since $\mathcal{B}_Z(G^d/P^d,\mathcal{L}_{\lambda})_K$ is a $(\g,K)$-module, $\V_{Z,\lambda}$ is a $(\g,K)$-submodule of\ $\mathcal{A}(X,\M_{\lambda})_K$, where $\g$ acts by differentiation on the left, similarly to~\eqref{eqn:gdiffleft}.
We recall that the space $\V_{\lambda}:=L^2(X,\M_{\lambda})_K$ depends only on the image of~$\lambda$ in $\jj_{\C}^{\ast}/W$, hence we may assume $\operatorname{Re}\lambda\in\overline{\jj^{\ast}_+}$ without loss of generality.
The following fact (which includes the ``$C=0$'' conjecture \cite{fle80} and the irreducibility conjecture) is a consequence of the work of Flensted-Jensen \cite{fle80}, Matsuki--Oshima \cite{mo84}, and Vogan~\cite{vog88}.
See also \cite[Th.\,16.1]{bs05}.

\begin{fact}\label{fact:decompVlambda}
Let $\lambda\in\jj_{\C}^{\ast}$ satisfy $\operatorname{Re}\lambda\in\overline{\jj^{\ast}_+}$.
\begin{itemize}
  \item For any $Z\in\ZZ$, the space~$\V_{Z,\lambda}$ constructed above is contained in $\V_{\lambda}:=L^2(X,\M_{\lambda})_K$; it is either zero or irreducible as a $(\g,K)$-module.
  Moreover,
  $$\V_{\lambda} \,=\, \bigoplus_{Z\in\ZZ}\, \V_{Z,\lambda}.$$
  \item Let $Z\in\ZZ$ correspond to $w\in W(H^d,G^d)$ \emph{via} \eqref{eqn:Zw}.
  \begin{itemize}
    \item If $\V_{Z, \lambda}$ is nonzero, then $\lambda\in\jj^{\ast}_+$ and
    \begin{equation}\label{eqn:muwinLambda}
    \mu_{\lambda}^w := w(\lambda+\rho)-2\rho_c
    \end{equation}
    belongs to the $\Z$-module~$\Lambda$ defined in \eqref{eqn:Lambda}.
    \item Conversely, if $\lambda\in\jj^{\ast}_+$ and if the stronger integrality condition
    \begin{equation}\label{eqn:lmdint}
    \mu_{\lambda}^w \in \Lambda_+
    \end{equation}
    holds, where $\Lambda_+$ is defined in \eqref{eqn:Lambda+}, then $\V_{Z,\lambda}$ is nonzero.
  \end{itemize}
\end{itemize}
\end{fact}

Thus there are countably many discrete series representations for~$X$.
The discrete series representations $\V_{Z,\lambda}$ for $\lambda$ satisfying \eqref{eqn:lmdint} were constructed by Flensted-Jensen in \cite{fle80}; we will give more details in Section~\ref{subsec:FJ}.

We note that Fact~\ref{fact:decompVlambda} completely describes $\Spec_d(X)$ \emph{away from the walls of~$\jj^{\ast}_+$}: the following lemma states that any $\lambda\in\jj^{\ast}_+$ satisfying the weak condition $\mu_{\lambda}^w\in\Lambda$ but not the strong condition $\mu_{\lambda}^w\in\Lambda_+$ has a bounded ``weighted distance to the walls'' $d(\lambda)$.
On the other hand, the nonvanishing condition for $\V_{Z, \lambda}$ is combinatorially complicated for $\lambda$ near the walls of~$\jj^{\ast}_+$; it is still not completely settled in the literature.

\begin{lemma}\label{lem:regKdom}
Suppose that $\lambda\in\jj^{\ast}_+$ satisfies $d(\lambda)\geq m_{\rho}$, where $m_{\rho}$ is given by \eqref{eqn:rhomax}.  
For $w\in W(H^d,G^d)$, the following conditions on~$\lambda$ are equivalent:
\begin{enumerate}
  \item[{\rm{(i)}}] $\mu_{\lambda}^w\in\Lambda$,
  \item[{\rm{(ii)}}] $\mu_{\lambda}^w\in\Lambda_+$.  
\end{enumerate}
\end{lemma}

\begin{proof}
The implication (ii) $\Rightarrow$ (i) is obvious.  
Let us prove (i) $\Rightarrow$ (ii), namely that if $\mu_{\lambda}^w\in\Lambda$, then $\mu_{\lambda}^w$ is dominant with respect to $\Sigma^+(\h^d,\jj)=\Sigma^+(\kk_{\C},\jj_{\C})$.
Firstly, we note that $w \rho$ is half the sum of the elements in $w (\Sigma^+(\g^d,\jj))$ counted with root multiplicities, where $w (\Sigma^+(\g^d,\jj))$ is a positive system containing $\Sigma^+(\h^d,\jj)$ (by definition \eqref{eqn:Wcoset} of $W(H^d,G^d)$).  
By \cite{vz84}, $2 w \rho-2\rho_c$ is dominant with respect to $\Sigma^+(\h^d,\jj)$.  
(In fact, it occurs as the highest weight of a representation of $\h^d$ in $\Lambda^{\ast}\q^d$.)
Secondly, Observation~\ref{obs:dwall} and the inequality $d(\lambda)\geq m_{\rho}$ imply that
$$\lambda - \rho = \Big(\lambda - \frac{d(\lambda)}{m_{\rho}}\rho\Big) + \frac{d(\lambda)-m_{\rho}}{m_{\rho}}\,\rho \,\in\, \overline{\jj^{\ast}_+}\,;$$
therefore $w(\lambda-\rho)$ is dominant with respect to $\Sigma^+(\h^d,\jj)$ since $w\in W(H^d,G^d)$.
Thus $\mu_{\lambda}^w=2(w\rho-\rho_c)+w(\lambda-\rho)$ is dominant with respect to $\Sigma^+(\h^d,\jj)$.  
\end{proof}

\section{Asymptotic behavior of discrete series}\label{subsec:asym}

We can now complete the proof of Proposition~\ref{prop:asym}.

By Fact~\ref{fact:decompVlambda}, we may assume that $\varphi\in L^2(X,\M_{\lambda})_K$ belongs to $\V_{Z,\lambda}$ for some closed $H^d$-orbit $Z$ in $\mathcal{P}^d$.
We then use Oshima's theorem (\cite{osh88b}, see Fact~\ref{fact:oshima} below) that the asymptotic behavior of the eigenfunction~$\varphi$ is determined by~$Z$.
This theorem requires an unavoidable amount of notation.
Before entering into technical details, let us pin down the role of two positive systems involved:
\begin{changemargin}{-1.8cm}{0cm}
\begin{center}
\begin{tabular}{l c c}
\hspace{2.3cm} $\Sigma^+(\g^d,\jj)$ & $\overset{\sim}{\longleftrightarrow}$ & closed $H^d$-orbit~$Z$ in $\mathcal{P}^d$\\
{\tiny Cayley transform $\Ad(k)$} $\vdots$ & & $\vdots\ $\tiny{$\!^+ W(Z)$}\\
\hspace{2.2cm} $\Sigma^+(\g,\bb)$ & $\dots$ & asymptotic behavior of $\varphi\in\V_{Z,\lambda}$\\
& & at infinity in $X=G/H$
\end{tabular}
\end{center}
\end{changemargin}

\smallskip

\noindent
We now enter into details, retaining notation from Sections \ref{subsec:munu} and~\ref{subsec:discreteseries}.

We first recall that in Section~\ref{subsec:munu} we have chosen a positive system $\Sigma^+(\g^{\sigma\theta},\bb)$, determining a closed positive Weyl chamber $\overline{\bb_+}$ in~$\bb$, a polar decomposition $G=K(\exp\overline{\bb_+})H$, and a projection $\nu : G\rightarrow\overline{\bb_+}$.
Any choice of a positive system $\Sigma^+(\g,\bb)$ containing $\Sigma^+(\g^{\sigma\theta},\bb)$ gives rise to a closed positive Weyl chamber $\overline{\bb_{++}}\subset\overline{\bb_+}$, and $\overline{\bb_+}$ is the union of such Weyl chambers~$\overline{\bb_{++}}$ for the (finitely many) different choices of $\Sigma^+(\g,\bb)$.
On the other hand, by Fact~\ref{fact:decompVlambda}, the space $\V_{\lambda}=L^2(X,\M_{\lambda})_K$ is the direct sum of finitely many subspaces~$\V_{Z,\lambda}$, where $Z\in\ZZ$ is a closed $H^d$-orbit in $\mathcal{P}^d$.
Therefore, in the rest of the section, we may restrict to \emph{one} closed positive Weyl chamber~$\overline{\bb_{++}}$ (determined by some arbitrary positive system $\Sigma^+(\g,\bb)$ containing $\Sigma^+(\g^{\sigma\theta},\bb)$) and \emph{one} $H^d$-orbit $Z\in\ZZ$, and prove the existence of a constant $q>0$ such that for any $\lambda\in\jj^{\ast}$ and $\varphi\in\V_{Z,\lambda}$, the function
$$(k,Y) \,\longmapsto\, \varphi\big(k(\exp Y)\cdot x_0\big)\ e^{q\,d(\lambda)\Vert Y\Vert}$$
is bounded on $K\times\overline{\bb_{++}}$.
Since $\V_{Z,\lambda}$ and $d(\lambda)$ depend only on the image of $\lambda\in\jj^{\ast}$ \emph{modulo}~$W$, we will be able to take~$\lambda$ in any Weyl chamber $\jj^{\ast}_+$ of~$\jj^{\ast}$.

Fix $Z\in\ZZ$ and consider the positive Weyl chamber $\jj^{\ast}_+$ in~$\jj^{\ast}$ determined by~$Z$ \emph{via} \eqref{eqn:Z1to1}.
We introduce some additional notation.
Let
$$\overline{\!^+\jj} \,\equiv\, \overline{\!^+\jj(Z)} \,:=\, \big\{ \widetilde{Y}\in\jj :\ \langle\lambda,\widetilde{Y}\rangle\geq 0\quad\forall\lambda\in\jj^{\ast}_+\big\} $$
be the dual cone of~$\jj^{\ast}_+$ and let $\rho\in\jj^{\ast}_+$ be given as in Section~\ref{subsec:Lambda+}.
Since all maximally split abelian subspaces of~$\g^d$ are conjugate by~$K^d$, there exists $k\in K^d$ such that $\Ad(k)\bb\subset\jj$; the element $\Ad(k)$ may be thought of as an analog of a Cayley transform from the upper-half plane to the hyperbolic disk (see Section~\ref{subsec:IwasawaAdS3}).
We may assume that
$$(\Ad(k)^{\ast}\alpha)|_{\bb} \in \Sigma^+(\g,\bb) \cup \{0\}$$
for all $\alpha\in\Sigma^+(\g^d,\jj)$; in particular, $\Ad(k)(\overline{\bb_{++}})\subset\overline{\!^+\jj}$.
For $Y\in\bb$, we write
$$\widetilde{Y} := \Ad(k) Y \in \jj.$$
Let $\{ Y_1,\dots,Y_{\ell}\} $ be the basis of~$\bb$ that is dual to the set of simple roots in $\Sigma^+(\g,\bb)$.
For $t\in(\R_+)^{\ell}$, we set
$$Y_{\bb}(t) := -\sum_{j=1}^{\ell} (\log t_j) Y_j \in \bb,$$
so that $t\mapsto Y_{\bb}(t)$ is a bijection from $(\R_+)^{\ell}$ to~$\bb$, inducing a bijection between $(0,1]^{\ell}$ and $\overline{\bb_{++}}$.
For $w\in W$ and $\lambda\in\jj^{\ast}$, we set
$$\beta_w(\lambda) := \big(\langle\rho-w\lambda,\widetilde{Y}_1\rangle,\dots,\langle\rho-w\lambda,\widetilde{Y}_{\ell}\rangle\big) \,\in\R^{\ell}.$$
We recall that $W$ is the Weyl group of $\Sigma(\g^d,\jj)$.
We define
$$\!^+ W \equiv \!^+ W(Z) := \big\{ w\in W :\ -w^{-1}\cdot\Ad(k)(\overline{\bb_{++}}) \subset \overline{\!^+\jj}\big\} .$$
The set $\!^+W$ depends on the closed $H^d$-orbit $Z$ in $\mathcal{P}^d$.
If $\rank G/H=1$, then $\ell=1$ and $\!^+W=\{ w\}$, where $w$ is the unique nontrivial element of~$W$.

With this notation, here is the asymptotic behavior, due to Oshima, that we shall translate in terms of $\nu$ and~$d$ to obtain Proposition~\ref{prop:asym}.
We consider the partial order on~$\R^{\ell}$ given by
$$\beta\prec\beta' \quad\text{if and only if}\quad \beta_j \le \beta'_j\,\text{ for all $1\leq j\leq\ell$}.$$

\begin{fact}[\cite{osh88b}]\label{fact:oshima}
Let $\lambda\in\jj^{\ast}_+$ and let $I_{\lambda}$ be the set of minimal elements in the finite set $\{ \beta_w(\lambda) : w\in\!^+ W\} \subset\R^{\ell}$ for~$\prec$.
For any $\varphi\in\V_{Z,\lambda}$, there exist real analytic functions $a_{\beta}\in\mathcal{A}(K)$, for $\beta\in I_{\lambda}$, such that
$$\big|\varphi\big(k(\exp Y_{\bb}(t))H\big)\big| \,\leq\, \sum_{\beta\in I_{\lambda}} a_{\beta}(k)\,t^{\beta}$$
for all $k\in K$ and $t\in (0,1]^{\ell}$, where we write $t^{\beta}$ for $\prod_{j=1}^{\ell} {t_j}^{\beta_j}$.
\end{fact}

Let $\!^+W_{\lambda}:=\{ w\in\!^+W : \beta_w(\lambda)\in I_{\lambda}\}$.
Then Fact~\ref{fact:oshima} has the following immediate consequence: for any $\lambda\in\jj^{\ast}_+$ and $\varphi\in\V_{Z,\lambda}$, there is a constant $c_{\varphi}>0$ such that
\begin{equation}\label{eqn:oshima}
\big|\varphi\big(k(\exp Y)H\big)\big| \,\leq\, c_{\varphi}\,\sum_{w\in^+W_{\lambda}} e^{\langle w\lambda,\widetilde{Y}\rangle}
\end{equation}
for all $k\in K$ and $Y\in\overline{\bb_{++}}$.
Indeed, $K$ is compact, $I_{\lambda}$ is finite, and for all $w\in\!^+W_{\lambda}$ and $t\in (0,1]^{\ell}$,
$$t^{\beta_w(\lambda)} \,=\, e^{\langle w\lambda-\rho,{\scriptscriptstyle\widetilde{Y_{\bb}(t)}}\rangle} \,\leq\, e^{\langle w\lambda,{\scriptscriptstyle\widetilde{Y_{\bb}(t)}}\rangle}.$$

We now bound $\langle w\lambda,\widetilde{Y}\rangle$ in terms of the ``weighted distance to the walls''~$d(\lambda)$.

\begin{lemma}\label{lem:Yd}
There is a constant $q_Z>0$ such that
$$\langle w\lambda,\widetilde{Y}\rangle \leq -q_Z\,d(\lambda)\,\Vert Y\Vert$$
for all $w\in\!^+W$, all $\lambda\in\jj^{\ast}_+$, and all $Y\in\overline{\bb_{++}}$.
\end{lemma}

\begin{proof}
Let $\{ \alpha_1,\dots,\alpha_r\}$ be the basis of $\Sigma(\g^d,\jj)$ corresponding to~$\jj^{\ast}_+$.
Recall that for any $\lambda\in\jj^{\ast}_+$,
$$d(\lambda) = \min_{1\leq i\leq r} \frac{(\lambda,\alpha_i)}{(\alpha_i,\alpha_i)}.$$
Let $\Vert\cdot\Vert'$ be the norm on~$\bb$ defined by $\Vert\sum_{j=1}^{\ell} y_j Y_j\Vert' := \sum_{j=1}^{\ell} |y_j|$ for all $y_1,\dots,y_{\ell}\in\R$.
An elementary computation shows that we may take
\begin{equation}\label{eqn:q}
q_Z = \frac{q_1q_2}{m_{\rho}},
\end{equation}
where $m_{\rho}$ was defined in \eqref{eqn:rhomax} and 
\begin{align*}
& q_1 := \min \big\{ - \langle w\rho,\widetilde{Y}_j \rangle : w \in \!^+W,\  1\leq j\leq\ell\big\} ,\\
& q_2 := \min_{Y\in\bb\smallsetminus\{ 0\} } \frac{\Vert Y\Vert'}{\Vert Y\Vert}.\qedhere
\end{align*}
\end{proof}

\medskip

By \eqref{eqn:oshima} and Lemma~\ref{lem:Yd}, for any $\lambda\in\jj^{\ast}_+$ and $\varphi\in\V_{Z,\lambda}$ there is a constant $c'_{\varphi}>0$ such that
\begin{equation}\label{eqn:estimphib++}
\big|\varphi\big(k(\exp Y)H\big)\big| \,\leq\, c'_{\varphi}\,e^{-q_Z\,d(\lambda)\,\Vert Y\Vert}
\end{equation}
for all $k\in K$ and $Y\in\overline{\bb_{++}}$.
We now recall (see \cite[Th.\,2.6]{fle80} for instance) that the $G$-invariant Radon measure on $X=G/H$ is given (up to scaling)~by
\begin{equation}\label{eqn:weightKBH}
\mathrm{d}\big(k(\exp Y)H\big) = \delta(Y)\,\mathrm{d}k\,\mathrm{d}Y
\end{equation}
with respect to the polar decomposition $G=K(\exp\overline{\bb_+})H$, where the weight function~$\delta$ is given on~$\overline{\bb_{++}}$ by
$$\delta(Y) = \prod_{\alpha\in\Sigma^+(\g,\bb)} |\sinh\alpha(Y)|^{\dim\g^{\sigma\theta}_{\alpha}} \, |\cosh\alpha(Y)|^{\dim\g^{-\sigma\theta}_{\alpha}}.$$
When $Y\in\overline{\bb_{++}}$ tends to infinity,
$$\delta(Y) \sim e^{2\langle\rho_b,Y\rangle},$$
where $\rho_b\in\overline{\bb_{++}}$ is half the sum of the elements of $\Sigma^+(\g,\bb)$, counted with root multiplicities.
In particular, there is a constant $C>0$ such that
\begin{equation}\label{eqn:estimdelta}
|\delta(Y)| \leq C\,e^{2\langle\rho_b,Y\rangle} \leq C\,e^{2\Vert\rho_b\Vert\,\Vert Y\Vert}
\end{equation}
for all $Y\in\overline{\bb_{++}}$.
Proposition~\ref{prop:asym} follows from \eqref{eqn:estimphib++}, \eqref{eqn:weightKBH}, and~\eqref{eqn:estimdelta}, setting
$$q := \min_{Z\in\ZZ}\, q_Z.$$

%% file: StableSpec6-Conv.tex
\chapter[Convergence, square integrability, and regularity]{Convergence, square integrability, and regularity of the generalized Poincar\'e series}\label{sec:averaging}

As before, $X=G/H$ is a reductive symmetric space satisfying the rank condition~\eqref{eqn:rank}.
We use the notation from Chapters \ref{sec:theorems} to~\ref{sec:Vlambda}.
For any Clifford--Klein form $X_{\Gamma}=\Gamma\backslash X$ and any $p\geq 1$, we denote by $L^p(X_{\Gamma},\M_{\lambda})$ the subspace of $L^p(X_{\Gamma})$ consisting of the weak solutions to the system~$(\M_{\lambda})$.
The group~$G$ acts on $L^p(X,\M_{\lambda})$ by left translation: for $g\in G$ and $\varphi\in L^p(X,\M_{\lambda})$,
$$g\cdot\varphi := \varphi(g^{-1}\,\cdot\,) \in L^p(X,\M_{\lambda}).$$
The first key step in our construction of eigenfunctions on Clifford--Klein forms of~$X$ is the following (see Definition~\ref{def:sharp} for the notion of sharpness).

\begin{proposition}\label{prop:Vlambda}
There is a constant $R_X>0$ depending only on~$X$ such that for any $c,C>0$ and any discrete subgroup~$\Gamma$ of~$G$ that is $(c,C)$-sharp for~$X$,
\begin{enumerate}
  \item the function $\varphi^{\Gamma} : X_{\Gamma}\rightarrow\C$ given by
  $$\varphi^{\Gamma}(\Gamma x) := \sum_{\gamma\in\Gamma} (\gamma\cdot\varphi)(x) = \sum_{\gamma\in\Gamma} \varphi(\gamma^{-1}\!\cdot x)$$
  is well-defined and continuous for all $\varphi\in L^2(X,\M_{\lambda})_K$ with $\lambda\in\jj^{\ast}$ and $d(\lambda)>R_X/c$,
  \item furthermore, $\varphi\mapsto\varphi^{\Gamma}$ defines a linear operator
  $$S_{\Gamma}\ :\ L^2(X,\M_{\lambda})_K \quad\longrightarrow\quad C^m(X_{\Gamma})\cap\bigcap_{1\leq p\leq\infty} L^p(X_{\Gamma},\M_{\lambda})$$
  for all $\lambda\in\jj^{\ast}$ and $m\in\N$ with $d(\lambda)>(m+1)R_X/c$.
\end{enumerate}
\end{proposition}

The fact that the constant $R_X/c$ depends only on the first sharpness constant~$c$ explains why we obtain a \emph{universal} discrete spectrum in Theorem~\ref{thm:universal}, independent of the discrete subgroup~$\Gamma$ of~$L$ (see Proposition~\ref{prop:sharpnessproperties}).
Note that Proposition~\ref{prop:Vlambda}.(2) actually contains Theorem~\ref{thm:regular}.
We could obtain a slightly weaker condition than $d(\lambda)>(m+1)R_X/c$ by taking into account the critical exponent $\delta_{\Gamma}$ of~$\Gamma$ (see Section~\ref{subsec:bettercst}).

In Proposition~\ref{prop:Vlambda}, the function~$\varphi^{\Gamma}=S_{\Gamma}(\varphi)$ satisfies $(\M_{\lambda})$ (in the sense of distributions) because $\varphi$ does and any $D\in\D(X)$ is $G$-invariant, that is,
\begin{equation}\label{eqn:translate}
D(g\cdot\varphi) = g\cdot (D\varphi)
\end{equation}
for all $g\in G$.
Furthermore, Proposition~\ref{prop:Vlambda}.(2) ensures that $\varphi^{\Gamma}$ satisfies $(\M_{\lambda})$ in the sense of \emph{functions} if $\lambda$ is regular enough (\ie $d(\lambda)$ large enough).
More precisely, recall from Section~\ref{subsec:DGH} that $\D(X)$ is a polynomial algebra in $r:=\rank(G/H)$ generators $D_1,\dots,D_r$.
By Proposition~\ref{prop:Vlambda}.(2), if we take $m$ to be the maximum degree of $D_1,\dots,D_r$, then for any $\lambda\in\jj^{\ast}$ with $d(\lambda)>(m+1)R$ and any $\varphi\in L^2(X,\M_{\lambda})_K$, the function $\varphi^{\Gamma}=S_{\Gamma}(\varphi)$ satisfies
$$(D_j)_{\Gamma}\,\varphi^{\Gamma} = \chi_\lambda(D_j)\,\varphi^{\Gamma}$$
for all $1\leq j\leq r$ in the sense of functions.

We note that the image of $L^2(X,\M_{\lambda})_K$ under the summation operator~$S_{\Gamma}$ could be trivial.
In Chapter~\ref{sec:nonzero}, we will discuss a condition for the nonvanishing of~$S_{\Gamma}$ (Proposition~\ref{prop:nonzero}).
For this we will consider the summation operator $S_{\Gamma}$, not only on $L^2(X,\M_{\lambda})_K$, but also on some $G$-translates $g\!\cdot\!L^2(X,\M_{\lambda})_K$.

The rest of this chapter is devoted to the proof of Proposition~\ref{prop:Vlambda}, using the geometric estimates of Chapter~\ref{sec:geometry} (Lemma~\ref{lem:growthfornu}) and the analytic estimates of Chapter~\ref{sec:Vlambda} (Proposition~\ref{prop:asym}).
As a consequence of Proposition~\ref{prop:asym}, the series $\sum_{\gamma\in\Gamma} e^{-q\,d(\lambda)\Vert\nu(\gamma\cdot x)\Vert}$ will naturally appear in the proof of Proposition~\ref{prop:Vlambda}: it is a pseudo-Riemannian analogue of the classical Poincar\'e series
$$\sum_{\gamma\in\Gamma} e^{-q\,d(\lambda)\Vert\mu(\gamma\cdot y)\Vert} = \sum_{\gamma\in\Gamma} e^{-q\,d(\lambda)\,d_{G/K}(y_0,\gamma\cdot y)}$$
for $y\in G/K$.

\begin{remark}\label{rem:TongWang}
A summation process was used by Tong--Wang in \cite{tw89} to construct cohomology classes of \emph{Riemannian} locally symmetric spaces $\Gamma\backslash G/K$ with coefficients in a locally constant vector bundle.
The summation described here is different in two respects:
\begin{itemize}
  \item in the situation considered by Tong--Wang, $\Gamma$ was a lattice in~$G$ and $\Gamma\cap H$ a lattice in~$H$, whereas here $\Gamma$ can never be a lattice in~$G$ and $\Gamma\cap H$ is finite (see Remark~\ref{rem:lattice});
  \item Tong--Wang obtained a $(\g,K)$-homomorphism from $L^2(X,\M_{\lambda})_K$ to $C^{\infty}(\Gamma\backslash G)$, whereas we obtain a map from $L^2(X,\M_{\lambda})_K$ to $L^2(\Gamma\backslash G/H)$ (which cannot be a $(\g,K)$-homomorphism since $G$ does not act on $L^2(\Gamma\backslash G/H)$).
\end{itemize}
\end{remark}

\section{Convergence and boundedness}

Let us prove Proposition~\ref{prop:Vlambda}.(1).
We denote by $q>0$ the constant of Proposition~\ref{prop:asym}.

\begin{lemma}\label{lem:conv}
Let $\Gamma$ be a discrete subgroup of~$G$ that is $(c,C)$-sharp for~$X$.
\begin{enumerate}
  \item For any $\lambda\in\jj^{\ast}$ with $d(\lambda)>\delta_{\Gamma}/qc$ and any $\varphi\in L^2(X,\M_{\lambda})_K$, the function $\varphi^{\Gamma}$ is well-defined and continuous.
  \item For any $\lambda\in\jj^{\ast}$ with $d(\lambda)>2\delta_{\Gamma}/qc$ and any $\varphi\in L^2(X,\M_{\lambda})_K$, the function $\varphi^{\Gamma}$ is bounded.
\end{enumerate}
\end{lemma}

\begin{proof}
Fix $\lambda\in\jj^{\ast}$ with $d(\lambda)>\delta_{\Gamma}/qc$ and $\varphi\in L^2(X,\M_{\lambda})_K$.
We claim that
$$x \longmapsto \sum_{\gamma\in\Gamma} |\varphi(\gamma^{-1}\!\cdot x)|$$
converges uniformly on any compact subset of~$X$.
Indeed, by Proposition~\ref{prop:asym}, there is a constant $c_{\varphi}>0$ such that for all $x\in X$,
$$\sum_{\gamma\in\Gamma} |\varphi(\gamma^{-1}\!\cdot x)| \leq c_{\varphi}\, \sum_{\gamma\in\Gamma} e^{-q\,d(\lambda)\Vert\nu(\gamma^{-1}\cdot x)\Vert},$$
hence
$$\sum_{\gamma\in\Gamma} |\varphi(\gamma^{-1}\!\cdot x)| \leq c_{\varphi}\, \sum_{n\in\N} e^{-q\,d(\lambda)n} \cdot \#\{\gamma\in\Gamma : n\leq\Vert\nu(\gamma^{-1}\!\cdot x)\Vert<n+1\} .$$
Fix $\varepsilon>0$ such that $d(\lambda)>\frac{\delta_{\Gamma}+\varepsilon}{qc}$ and, as before, let $x_0$ be the image of $H$ in $X=G/H$.
By Lemma~\ref{lem:growthfornu}.(1), there is a constant $c_{\varepsilon}(\Gamma)>0$ such that for all $x=g\cdot x_0\in X$ (where $g\in G$) and all $n\in\N$,
\begin{equation}\label{eqn:growthnu}
\#\big\{ \gamma\in\Gamma : \Vert\nu(\gamma^{-1}\!\cdot x)\Vert<n+1\big\} \leq c_{\varepsilon}(\Gamma)\,e^{(\delta_{\Gamma}+\varepsilon)(n+1+\Vert\mu(g)\Vert)/c}.
\end{equation}
Therefore, for any compact subset~$\mathcal{C}$ of~$G$ and any $x\in\mathcal{C}\cdot x_0$,
$$\sum_{\gamma\in\Gamma} |\varphi(\gamma^{-1}\!\cdot x)| \leq c_{\varphi}\,c_{\varepsilon}(\Gamma)\,e^{(\delta_{\Gamma}+\varepsilon)(1+M)/c} \ \sum_{n\in\N} e^{-(q\,d(\lambda)-\frac{\delta_{\Gamma}+\varepsilon}{c})n},$$
where
$$M := C + \max_{g\in\mathcal{C}} \Vert\mu(g)\Vert.$$
This series converges since $d(\lambda)>\frac{\delta_{\Gamma}+\varepsilon}{qc}$, proving the claim and Lemma~\ref{lem:conv}.(1).

The proof of Lemma~\ref{lem:conv}.(2) is similar: we replace \eqref{eqn:growthnu} by the uniform (but slightly less good) estimate of Lemma~\ref{lem:growthfornu}.(2) in order to obtain a uniform convergence on the fundamental domain~$\mathcal{D}$ of Definition-Lemma~\ref{def-lem:funddomain}, and hence on the whole of~$X$.
\end{proof}

\section{Square integrability}\label{subsec:L1}

In order to see that the image of the summation operator~$S_{\Gamma}$ is contained in $L^2(X_{\Gamma})$, and more generally in $L^p(X_{\Gamma})$ for any $1\leq p\leq\infty$, it is enough to see that it is contained in both $L^1(X_{\Gamma})$ and $L^{\infty}(X_{\Gamma})$, by H\"older's inequality.
The case of $L^{\infty}(X_{\Gamma})$ has already been treated in Lemma~\ref{lem:conv}.
For $L^1(X_{\Gamma})$, we note that by Fubini's theorem,
$$\int_{\overline{x}\in X_{\Gamma}} \left|\varphi^{\Gamma}(\overline{x})\right|\,\mathrm{d}\overline{x} = \int_{x\in X} |\varphi(x)|\,\mathrm{d}x\,;$$
using Proposition~\ref{prop:asym}, we obtain the following.

\begin{lemma}\label{lem:L1}
For any discrete subgroup~$\Gamma$ of~$G$, any $\lambda\in\jj^{\ast}$ with $d(\lambda)>\nolinebreak 2\Vert\rho_b\Vert/q$, and any $\varphi\in\nolinebreak L^2(X,\M_{\lambda})_K$, we have $\varphi^{\Gamma}\in L^1(X_{\Gamma})$.
\end{lemma}

Here, as in Proposition~\ref{prop:asym}, we denote by $\Vert\rho_b\Vert$ the norm of half the sum of the elements of a positive system $\Sigma^+(\g,\bb)$ of restricted roots of~$\bb$ in~$\g$, and $q>0$ is again the constant of Proposition~\ref{prop:asym}.

H\"older's inequality then gives the following.

\begin{corollary}\label{cor:Lp}
Let $\Gamma$ be a discrete subgroup of~$G$ that is $(c,C)$-sharp for~$X$.
For any $\lambda\in\jj^{\ast}$ with
$$d(\lambda) > \frac{2}{q}\,\max\big(\delta_{\Gamma}/c,\Vert\rho_b\Vert\big)$$
and any $\varphi\in L^2(X,\M_{\lambda})_K$, we have $\varphi^{\Gamma}\in L^p(X_{\Gamma})$ for all $1\leq p\leq\infty$; in particular, $\varphi^{\Gamma}\in L^2(X_{\Gamma})$.
\end{corollary}

\section{Regularity}\label{subsec:regularity}

We now complete the proof of Proposition~\ref{prop:Vlambda}.(2) (hence Theorem~\ref{thm:regular}) by examining the regularity of the image of $S_{\Gamma}$.
We set
$$e_G := \max_{\alpha\in\Sigma(\g,\aaa)} \Vert\alpha\Vert.$$

\begin{lemma}\label{lem:Cm}
Let $\Gamma$ be a discrete subgroup of~$G$ that is $(c,C)$-sharp for~$X$.
For any $m\in\N$ and any $\lambda\in\jj^{\ast}$ with $d(\lambda)>(\delta_{\Gamma}+e_G\,m)/qc$,
$$S_{\Gamma}\left(L^2(X,\M_{\lambda})_K\right) \subset C^m(X_{\Gamma}).$$
\end{lemma}

The idea of the proof of Lemma~\ref{lem:Cm} is to control the decay at infinity of the derivatives of the elements of $L^2(X,\M_{\lambda})_K$ by using the action of the enveloping algebra $U(\g_{\C})$ by differentiation on the left, given by
\begin{equation}\label{eqn:gdiffleft2}
(Y\cdot\varphi)(x) = \frac{\mathrm{d}}{\mathrm{d}t}\Big|_{t=0} \varphi\big(\exp(-tY)\!\cdot\!x\big)
\end{equation}
for all $Y\in\g$, all $\varphi\in L^2(X,\M_{\lambda})_K$, and all $x\in X$.
This idea works as a consequence of Fact~\ref{fact:oshima} and of the following well-known fact.

\begin{fact}[See \cite{ban87}]\label{fact:stab}
For any $\lambda\in\jj_{\C}^{\ast}$, the subspace $L^2(X,\M_{\lambda})_K$ of $\mathcal{A}(X)$ is stable under the action of~$\g$ by differentiation.
\end{fact}

\begin{proof}[Proof of Lemma~\ref{lem:Cm}]
Consider $\lambda\in\jj^{\ast}$ with $d(\lambda)>\delta_{\Gamma}/qc$ and $\varphi\in L^2(X,\M_{\lambda})_K$.
Let $\{ U_m(\g_{\C})\}_{m\in\N}$ be the natural filtration of the enveloping algebra $U(\g_{\C})$.
Then any $u\in U_m(\g_{\C})$ gives rise to a differential operator on~$X$ of degree $\leq m$ by \eqref{eqn:gdiffleft2}.
Conversely, any differential operator on~$X$ of degree $\leq m$ is obtained as a linear combination of differential operators induced from $U_m(\g_{\C})$ with coefficients in $C^{\infty}(X)$.
Therefore, in order to prove that $\varphi^{\Gamma}$ is $C^m$, it is sufficient to show that for any differential operator $D$ on~$X$ that is induced from an element $u\in U_m(\g_{\C})$,
$$x \longmapsto \sum_{\gamma\in\Gamma} |D(\gamma\cdot\varphi)(x)|$$
converges uniformly on all compact subsets of~$X$.
As before, let $x_0$ be the image of $H$ in $X=G/H$.
In view of the formula
$$D(\gamma\cdot\varphi)(x) = \big(\Ad(\gamma^{-1})(u)\cdot\varphi\big)(\gamma^{-1}\!\cdot x),$$
we only need to prove the existence of a constant $R\geq 0$ such that for any integer $m\geq 1$, any $Y\in\g^{\otimes m}$, and any compact subset~$\mathcal{C}$ of~$G$,
$$x \longmapsto \sum_{\gamma\in\Gamma} \big|\big(\Ad(\gamma)(Y)\cdot\varphi\big)(\gamma\cdot x)\big|$$
converges uniformly on $\mathcal{C}\cdot x_0$ whenever $d(\lambda)>(m+1)R$.

We fix a $K$-invariant inner product on~$\g$, extend it to $\g^{\otimes m}$, and write the corresponding Euclidean norms as $\Vert\cdot\Vert_{\g}$ and $\Vert\cdot\Vert_{\g^{\otimes m}}$, respectively.
Let $\Vert\cdot\Vert_{\End(\g)}$ be the operator norm on~$\g$.
We observe that
$$\Vert T(Y)\Vert_{\g^{\otimes m}} \leq \Vert T\Vert_{\End(\g)}^m\,\Vert Y\Vert_{\g^{\otimes m}}$$
for all $T\in\End(\g)$ and $Y\in\g^{\otimes m}$, where $T$ acts on~$\g^{\otimes m}$ diagonally.
Moreover,
\begin{equation}\label{eqn:norm}
\log \Vert\Ad(g)\Vert_{\End(\g)} \leq e_G\,\Vert\mu(g)\Vert
\end{equation}
for all $g\in G$: indeed, the Cartan decomposition $G=KAK$ holds and the norm $\Vert\cdot\Vert_{\g}$ is $K$-invariant.
By Proposition~\ref{prop:asym} and Fact~\ref{fact:stab}, we may define a function $\ell : \g^{\otimes m}\rightarrow\R_{\geq 0}$~by
$$\ell(Y) = \sup_{x\in X} |(Y\cdot\varphi)(x)|\, e^{q\,d(\lambda) \Vert\nu(x)\Vert}.$$
It satisfies
$$\ell(tY+t'Y') \leq |t|\,\ell(Y) + |t'|\,\ell(Y')$$
for all $t,t'\in\C$ and $Y,Y'\in\g^{\otimes m}$.
Taking a (finite) basis of~$\g^{\otimes m}$, this implies the existence of a constant $c_m>0$ such that
$$\ell(Y) \leq c_m\,\Vert Y\Vert_{\g^{\otimes m}}$$
for all $Y\in\g^{\otimes m}$.
Then for any $\gamma\in\Gamma$, any $Y\in\g^{\otimes m}$, and any $x\in X$,
$$\big|\big(\Ad(\gamma)(Y)\cdot\varphi\big)(\gamma\cdot x)\big| \leq c_m\,\Vert\Ad(\gamma)\Vert_{\End(\g)}^m\,\Vert Y\Vert_{\g^{\otimes m}}\,e^{-q\,d(\lambda) \Vert\nu(\gamma\cdot x)\Vert}.$$
Therefore we only need to prove the existence of a constant $R\geq 0$ such that for any integer $m\in\N$ and any compact subset $\mathcal{C}$ of~$G$,
$$x \longmapsto \sum_{\gamma\in\Gamma} \Vert\Ad(\gamma)\Vert_{\End(\g)}^m\,e^{-q\,d(\lambda) \Vert\nu(\gamma\cdot x)\Vert}$$
converges uniformly on $\mathcal{C}\cdot x_0$ whenever $d(\lambda)>(m+1)R$.
Let us fix an integer $m\in\N$ and a compact subset~$\mathcal{C}$ of~$G$.
By \eqref{eqn:ineqmunu},
$$\Vert\nu(\gamma\cdot x)\Vert \geq c\,\Vert\mu(\gamma)\Vert - M$$
for all $\gamma\in\Gamma$ and $x\in\mathcal{C}\cdot x_0$, where
$$M = C + \max_{g\in\mathcal{C}} \Vert\mu(g)\Vert.$$
Using \eqref{eqn:norm}, we obtain that for all $\gamma\in\Gamma$ and $x\in\mathcal{C}\cdot x_0$,
$$\sum_{\gamma\in\Gamma} \Vert\Ad(\gamma)\Vert_{\g}^m\,e^{-q\,d(\lambda) \Vert\nu(\gamma\cdot x)\Vert} \,\leq\, e^{q\,d(\lambda)M} \sum_{\gamma\in\Gamma} e^{-(q\,d(\lambda) c - e_Gm)\,\Vert\mu(\gamma)\Vert}.$$
This series converges as soon as
$$d(\lambda) > \frac{\delta_{\Gamma}+e_G\,m}{qc}.\qedhere$$
\end{proof}

\section{The constant~$R_X$ in Proposition~\ref{prop:Vlambda}}\label{subsec:bettercst}

Lemma~\ref{lem:conv}, Corollary~\ref{cor:Lp}, and Lemma~\ref{lem:Cm} show that the summation operator
$$S_{\Gamma}\ :\ L^2(X,\M_{\lambda})_K \quad\longrightarrow\quad \bigcap_{1\leq p\leq\infty} L^p(X_{\Gamma},\M_{\lambda})$$
is well-defined and with values in $C^m(X_{\Gamma})$ as soon as
\begin{equation}\label{eqn:bettercst}
d(\lambda) > \frac{1}{q}\,\max\Big(\frac{2\delta_{\Gamma}}{c},\,2\Vert\rho_b\Vert,\,\frac{\delta_{\Gamma}+e_G\,m}{c}\Big).
\end{equation}
We note that
\begin{itemize}
  \item $\delta_{\Gamma}\leq 2\Vert\rho_a\Vert$ (Observation~\ref{obs:limsup}),
  \item $\Vert\rho_b\Vert\leq\Vert\rho_a\Vert/c$ by Remark~\ref{rem:rhoab} below and the fact that $c\leq 1$,
  \item $e_G\leq 2\Vert\rho_a\Vert$ by definition of~$e_G$.
\end{itemize}
Therefore \eqref{eqn:bettercst} is satisfied as soon as $d(\lambda)>(m+1)R_X/c$ for
\begin{equation}\label{eqn:explicitRX}
R_X := \frac{4\Vert\rho_a\Vert}{q}.
\end{equation}

\begin{remark}\label{rem:rhoab}
Suppose that the positive systems $\Sigma^+(\g,\aaa)$ defining~$\rho_a$ and $\Sigma^+(\g,\bb)$ defining~$\rho_b$ are compatible, in the sense that the restriction from $\aaa$ to~$\bb$ maps $\Sigma^+(\g,\aaa)$ to $\Sigma^+(\g,\bb)\cup\{ 0\}$.
Then $\rho_b$ is the restriction of $\rho_a$ to~$\bb$, \ie the orthogonal projection of $\rho_a$ to~$\bb^{\ast}$. Thus
$$\Vert\rho_b\Vert = \Vert\rho_a\Vert \cdot \cos(\Phi),$$
where $\Phi\in [0,\frac{\pi}{2})$ is the angle between $\rho_a$ and~$\rho_b$.
In particular $\Vert\rho_b\Vert\leq\Vert\rho_a\Vert$.
This inequality is true in general since the norms $\Vert\rho_a\Vert$ and $\Vert\rho_b\Vert$ do not depend on the choice of the positive systems.
\end{remark}

%% file: StableSpec7-Analytic.tex
\chapter{An estimate for certain eigenfunctions near the origin}\label{sec:FJ}

Let $\Gamma$ be a discrete subgroup of~$G$ that is sharp for the reductive symmetric space $X=G/H$ satisfying the rank condition \eqref{eqn:rank}.
In Proposition~\ref{prop:Vlambda}, we saw that the summation operator
$$S_{\Gamma}\ :\ L^2(X,\M_{\lambda})_K \quad\longrightarrow\quad \bigcap_{1\leq p\leq\infty} L^p(X_{\Gamma},\M_{\lambda})$$
mapping $\varphi$ to $\varphi^{\Gamma}=\big(\Gamma x\mapsto\sum_{\gamma\in\Gamma}\,(\gamma\cdot\varphi)(x)\big)$ is well-defined for all $\lambda\in\nolinebreak\jj^{\ast}$ with $d(\lambda)$ sufficiently large.
In Section~\ref{subsec:transl}, we are similarly going to define a summation operator~$S_{\Gamma}$ on any $G$-translate $g\!\cdot\!L^2(X,\M_{\lambda})_K$.
Our goal will be to show that $S_{\Gamma}$ is nonzero on some $G$-translate $g\!\cdot\!L^2(X,\M_{\lambda})_K$ for infinitely many joint eigenvalues $\lambda\in\jj^{\ast}$, namely for all
\begin{equation}\label{eqn:lambdainlattice}
\lambda \in \jj^{\ast}_+ \cap \big(2\rho_c-\rho+\Lambda^{\Gamma\cap Z(G_s)}\big)
\end{equation}
with $d(\lambda)$ large enough (Proposition~\ref{prop:nonzero}).
Here $\jj^{\ast}_+$ and~$\rho$ are defined with respect to some choice of a positive system $\Sigma^+(\g_{\C},\jj_{\C})$ containing the fixed positive system $\Sigma^+(\kk_{\C},\jj_{\C})$ of Section~\ref{subsec:Lambda+}; the set~$\Lambda^{\Gamma\cap Z(G_s)}$ is the $\Z$-submodule of~$\Lambda$ of finite index that was defined in \eqref{eqn:LambdaJ}.

A similar argument to the one used in Chapter~\ref{sec:averaging} for the convergence of~$\varphi^{\Gamma}$ would show that for a fixed $\lambda$ satisfying \eqref{eqn:lambdainlattice} with $d(\lambda)$ large enough, $S_{\Gamma'}$ is nonzero for any finite-index subgroup~$\Gamma'$ of~$\Gamma$ such that the index $[\Gamma:\Gamma']$ is large enough, where ``large enough'' depends on $\Gamma$ and~$\lambda$.
However, we wish to prove that $S_{\Gamma}$ is nonzero without passing to any subgroup; therefore we need to carry out some more delicate estimates in the summation process.

In preparation for Proposition~\ref{prop:nonzero}, the goal of the current chapter is to establish the following analytic estimate, where, as before, $x_0$ denotes the image of~$H$ in $X=G/H$.

\begin{proposition}\label{prop:psilambda}
Under the rank condition \eqref{eqn:rank}, there exists $q'>0$ with the following property: for any $\lambda\in\jj^{\ast}_+\cap (2\rho_c-\rho+\Lambda_+)$, there is a function $\psi_{\lambda}\in\V_{Z,\lambda}\subset L^2(X,\M_{\lambda})_K$ such that $\psi_{\lambda}(x_0)=1$, such that
\begin{equation}\label{eqn:FJglobal}
|\psi_{\lambda}(x)| \leq \cosh(q'\Vert\nu(x)\Vert)^{-d(\lambda+\rho)}
\end{equation}
for all $x\in X$, and such that for any finite subgroup $J$ of the center $Z(K)$~of~$K$ we have $\psi_{\lambda}(g\cdot\nolinebreak x_0)=1$ for all $g\in J$ if $\lambda\in 2\rho_c-\rho+\Lambda^J$.
\end{proposition}

Here $Z\in\ZZ$ denotes the closed $H^d$-orbit through the origin in the flag variety $\mathcal{P}^d\simeq G^d/P^d$, where $P^d$ is the minimal parabolic subgroup of~$G^d$ corresponding to the choice of the positive system $\Sigma^+(\g_{\C},\jj_{\C})$ defining $\jj^{\ast}_+$ and~$\rho$, using \eqref{eqn:Z1to1}.
We refer to Section~\ref{subsec:discreteseries} (and more precisely to \eqref{eqn:Vzlmd}) for the definition of~$\V_{Z,\lambda}$.

The decay at infinity (\ie when $\Vert\nu(x)\Vert\rightarrow +\infty$) of the elements of $L^2(X,\M_{\lambda})_K$ was already discussed in Chapter~\ref{sec:Vlambda}.
The point of Proposition~\ref{prop:psilambda} is to control the behavior of certain eigenfunctions~$\psi_{\lambda}$, not only at infinity, but also \emph{near the origin} $x_0\in X$.

We actually prove that the estimate \eqref{eqn:FJglobal} holds for the \emph{Flensted-Jensen eigenfunction} $\psi_{\lambda}=\psi_{\lambda, Z}$, given by \eqref{eqn:psilambda} below.
In Chapter~\ref{sec:nonzero} we shall consider some $G$-translates of $\psi_{\lambda,Z}$ and apply the analytic estimate of Proposition~\ref{prop:psilambda} in connection with some geometric estimates near the origin (Propositions \ref{prop:sharp} and~\ref{prop:KM}).

\section{Flensted-Jensen's eigenfunctions}\label{subsec:FJfunction}

Before we prove Proposition~\ref{prop:psilambda}, we recall the definition of the Flensted-Jensen eigenfunction $\psi_{\lambda}=\psi_{\lambda,Z}$, in the spirit of Chapter~\ref{sec:Vlambda}.
We note that we may assume that $H$ is connected, because otherwise the Flensted-Jensen function $\psi_{\lambda}\in L^2(G/H)(\subset L^2(G/H_0))$ is the average of finitely many Flensted-Jensen functions in $L^2(G/H_0)$.
\emph{We will assume that $H$ is connected for the rest of the chapter.}

We retain the notation of Chapters \ref{sec:theorems} and~\ref{sec:Vlambda}.
As explained above, in the whole chapter we fix a positive system $\Sigma^+(\g_{\C},\jj_{\C})\simeq\Sigma^+(\g^d,\jj)$ containing the fixed positive system $\Sigma^+(\kk_{\C},\jj_{\C})\simeq\Sigma^+(\h^d,\jj)$ of Section~\ref{subsec:Lambda+}; it determines a positive Weyl chamber~$\jj^{\ast}_+$ and an element $\rho\in\jj^{\ast}_+$.
Let $P^d$ be the corresponding minimal parabolic subgroup of~$G^d$.
We denote by $Z\in\ZZ$ the closed $H^d$-orbit through the origin in $G^d/P^d$.
For $\lambda\in\jj^{\ast}_+$, we set $\mu_{\lambda}:=\lambda+\rho-2\rho_c$.
The condition on $\lambda\in\jj^{\ast}_+$ that appears in Proposition~\ref{prop:psilambda} is $\mu_{\lambda}\in\Lambda_+$ (\ie \eqref{eqn:lmdint} with $w=e$).

Let $\delta_Z$ be the $(K^d\cap H^d)$-invariant probability measure supported on~$Z$.
For any $\lambda\in\jj_{\C}^{\ast}$, the $G^d$-equivariant line bundle $\mathcal{L}_{\lambda}=G^d\times_{P^d}\nolinebreak\xi_{\rho-\lambda}$ over $G^d/P^d$ is trivial as a $K^d$-equivariant line bundle over $K^d/K^d\cap P^d(\simeq G^d/P^d)$, because the restriction of~$\xi_{\rho-\lambda}$ to $K^d\cap P^d$ is trivial.
Thus $\delta_Z$ can be seen as an element of $\mathcal{B}(G^d/P^d,\mathcal{L}_{\lambda})$ \emph{via} the isomorphism $\mathcal{B}(K^d/K^d\cap P^d) \simeq \mathcal{B}(G^d/P^d,\mathcal{L}_{\lambda})$.
Flensted-Jensen \cite{fle80} proved that if $\lambda\in\jj^{\ast}_+$ satisfies $\mu_{\lambda}\in\nolinebreak\Lambda_+$, then $\delta_Z$ is $K$-finite (see Remark~\ref{rem:Kfinite}) and generates the irreducible representation of~$\h^d$ with highest weight~$\mu_{\lambda}$.
The Poisson transform $\mathcal{P}_{\lambda}(\delta_Z)$ is also $K$-finite and moreover, viewed as an element of $\mathcal{A}(G^d/K^d,\M_\lambda)_K$, it belongs to the image of the homomorphism~$\eta$ of \eqref{eqn:FJisom}.
He then set
\begin{equation}\label{eqn:psilambda}
\psi_{\lambda,Z} := \eta^{-1}\big(\mathcal{P}_{\lambda}(\delta_Z)\big) \in \mathcal{A}(X,\M_{\lambda})_K.
\end{equation}

We shall prove that this function $\psi_{\lambda}=\psi_{\lambda,Z}$ satisfies \eqref{eqn:FJglobal}.
We note that our estimate \eqref{eqn:FJglobal} is stronger, for this specific~$\psi_{\lambda}$, than what is given in the general theory of \cite{fle80,mo84,osh88b}, as it is both uniform on the spectral parameter~$\lambda$ and uniform on $x\in X$ near the origin.

\section{Spherical functions on compact symmetric spaces}\label{subsec:cptsymm}

We first recall some basic results concerning spherical functions on the compact symmetric space $X_U=G_U/H_U$ (see Section~\ref{subsec:realforms} for notation).
In Section~\ref{subsec:FJ}, some of these results will actually be used, not only for $X_U=G_U/H_U$, but also for the compact symmetric space $K/H\cap K$.

Let $\g_U=\h_U+\q_U$ be the decomposition of~$\g_U$ into eigenspaces of $\mathrm{d}\sigma$ with respective eigenvalues $+1$ and $-1$.
We note that $\jj$ is a maximal abelian subspace of~$\q_U$.
Similarly to \eqref{eqn:Lambda+}, let $\Lambda_+(G_U/H_U)$ be the set of highest weights of finite-dimensional irreducible representations of~$G_U$ with nonzero $H_U$-invariant vectors; we see it as a subset of~$\jj_{\C}^{\ast}$ by Remark~\ref{rem:Lmdj}.
We note that $X_U$ has the same complexification as the Riemannian symmetric space of the noncompact type $X^d=G^d/K^d$.
The Borel--Weil theorem (see \cite[Th.\,5.29]{kna86}) implies that
\begin{equation}\label{eqn:LmdG}
\Lambda_+(G_U/H_U) \simeq \{\lambda\in\jj_{\C}^{\ast} :\ \xi_{\lambda} \text{ extends holomorphically to $G_{\C}$}\} ,
\end{equation}
where $\xi_{\lambda} : G^d\rightarrow\C$ is defined by \eqref{eqn:xinu}.
If $\mathcal{O}_{\mathrm{alg}}(G_{\C}/H_{\C})$ denotes the ring of regular functions on $G_{\C}/H_{\C}$, endowed with the action of~$G_{\C}$ by left translation, then we have an isomorphism
$$\mathcal{O}_{\mathrm{alg}}(G_{\C}/H_{\C}) \simeq \bigoplus_{\lambda\in\Lambda_+(G_U/H_U)} V_{\lambda}$$
of $G_U$-modules, where $(\pi_{\lambda},V_{\lambda})$ is the finite-dimensional irreducible representation of~$G_U$ with highest weight~$\lambda$.
A highest weight vector of $(\pi_{\lambda},V_{\lambda})$
 is given
 by the holomorphic extension of $\xi_{\lambda}^{\vee}$
 to $G_{\C}$ (see Section~\ref{subsec:Poisson}), 
which is denoted
 by the same symbol~$\xi_{\lambda}^{\vee}$.
Let $\{ \alpha_1,\dots,\alpha_r\} $ be the basis of $\Sigma(\g_{\C},\jj_{\C})$ corresponding to our choice of $\Sigma^+(\g_{\C},\jj_{\C})$, and let $\omega_1,\dots,\omega_r\in\jj^{\ast}_+$ be defined by
\begin{equation}\label{eqn:omegaj}
\frac{(\alpha_i,\omega_j)}{(\alpha_i,\alpha_i)} = \delta_{i,j}
\end{equation}
for all $1\leq i,j\leq r$, so that
\begin{equation}\label{eqn:lambdaomegaj}
\lambda = \sum_{j=1}^r \frac{(\lambda,\alpha_j)}{(\alpha_j,\alpha_j)}\,\omega_j
\end{equation}
for all $\lambda\in\jj^{\ast}$; we note that $\omega_j$ is twice the usual fundamental weight associated with~$\alpha_j$.
If $G_{\C}$ is simply connected, then the Cartan--Helgason theorem (see \cite[Th.\,3.3.1.1]{war72}) shows that
\begin{equation}\label{eqn:Lambda+U}
\Lambda_+(G_U/H_U) = \bigoplus_{j=1}^r \N\,\omega_j.
\end{equation}
For any $\lambda\in\Lambda_+(G_U/H_U)$, we fix a $G_U$-invariant inner product $(\cdot,\cdot)$ on~$V_{\lambda}$ with $(\xi_{\lambda}^{\vee},\xi_{\lambda}^{\vee})=1$.
The following easy observation and lemma will be useful in the next section.

\begin{observation}\label{obs:xi2}
For any $g\in G^d$,
$$\xi_{\lambda}(g)^2 = \big(\pi_{\lambda}(g)\xi_{\lambda}^{\vee}, \pi_{\lambda}(g)\xi_{\lambda}^{\vee}\big).$$
\end{observation}

\begin{proof}
We consider the Iwasawa decomposition $G^d=K^d(\exp\jj)N^d$ of Section~\ref{subsec:Poisson}.
For any $g=k(\exp\zeta(g))n\in K^d(\exp\jj)N^d=G^d$,
$$\pi_{\lambda}(g)\xi_{\lambda}^{\vee} \,=\, e^{\langle\lambda,\zeta(g)\rangle}\ \pi_{\lambda}(k)\xi_{\lambda}^{\vee} \,=\, \xi_{\lambda}(g)\ \pi_{\lambda}(k)\xi_{\lambda}^{\vee}.$$
Since $K^d=H_U$ is contained in~$G_U$ and $(\cdot,\cdot)$ is $G_U$-invariant, we obtain
$$\big(\pi_{\lambda}(g)\xi_{\lambda}^{\vee}, \pi_{\lambda}(g)\xi_{\lambda}^{\vee}\big) = \xi_{\lambda}(g)^2.\qedhere$$
\end{proof}

\begin{lemma}\label{lem:spherical}
For $\lambda \in \Lambda_+(G_U/H_U)$, the function $\xi_{\lambda}\in\mathcal{O}(G_{\C})$ satisfies
$$|\xi_\lambda(g)| \leq 1 \quad\quad\quad\text{for all }g\in G_U.$$
\end{lemma}

\begin{proof}
By Observation~\ref{obs:xi2},
$$\xi_{\lambda}(g)^2 = \big(\pi_{\lambda}(\sigma(g)^{-1}g)\xi_{\lambda}^{\vee}, \xi_{\lambda}^{\vee}\big)
\quad\text{for all $g\in G^d$.}$$
Since both sides are holomorphic functions on~$G_{\C}$, this holds for all $g\in G_{\C}$.  
Applying the Cauchy--Schwarz inequality, we get $|\xi_{\lambda}(g)|\leq 1$ on~$G_U$.  
\end{proof}

\section{Proof of Proposition~7.1 for the Flensted-Jensen functions}\label{subsec:FJ}

We now go back to the setting of Section~\ref{subsec:FJfunction}.
When $\lambda\in\jj^{\ast}_+$ satisfies $\mu_{\lambda}\in\nolinebreak\Lambda_+$, the function $\psi_{\lambda}\in\V_{Z,\lambda}$ of \eqref{eqn:psilambda} is well-defined and extends uniquely to a right-$H_{\C}$-invariant function on $K_{\C}\overline{B_+}H_{\C}$ \cite{fle80}; we keep the notation~$\psi_{\lambda}$ for this extension.
Directly from the definition, we have 
\begin{equation}\label{eqn:psilambdaky}
\psi_\lambda(ky) = \int_{H\cap K} \xi_{\mu_{\lambda}}(k\ell) \, \xi_{-\lambda-\rho}(y^{-1}\ell) \, \mathrm{d}\ell
\end{equation}
for all $k\in K_{\C}$ and $y\in G^d$ \cite[(3.13)]{fle80}, where $\xi_{-\lambda-\rho} : G^d\rightarrow\C$ is given by \eqref{eqn:xinu} and $\xi_{\mu_{\lambda}} : K_{\C}\rightarrow\C$ is the holomorphic extension, given by \eqref{eqn:LmdG} for the compact symmetric space $K/K\cap H$ instead of $G_U/H_U$, of the function $\xi_{\mu_{\lambda}} : H^d\rightarrow\C$ given by \eqref{eqn:xinu} with respect to the Iwasawa decomposition
\begin{equation}\label{eqn:IwasawaHd}
H^d = (K^d \cap H^d)(\exp\jj)(N^d \cap H^d).
\end{equation}
We note that the restriction to~$H^d$ of any ``$\xi$'' function for~$G^d$ coincides with the corresponding ``$\xi$'' function for~$H^d$, which is why we use the same notation.
The fact that \eqref{eqn:IwasawaHd} is an Iwasawa decomposition of~$H^d$ relies on the rank condition \eqref{eqn:rank}.

In order to prove Proposition~\ref{prop:psilambda}, we first observe the following.

\begin{lemma}\label{lem:FJcenter}
Let $J$ be a finite subgroup of the center $Z(K)$ of~$K$.
For $\lambda\in\jj^{\ast}_+$ with $\mu_{\lambda}\in\nolinebreak\Lambda_+\cap\nolinebreak\Lambda^J$, the Flensted-Jensen function~$\psi_{\lambda}$ satisfies $\psi_{\lambda}(g\cdot x_0)=1$ for all $g\in J$.
\end{lemma}

\begin{proof}
As in Section~\ref{subsec:Lambda+}, we can see the highest weight of any irreducible representation of~$K$ with nonzero $(K\cap H)$-fixed vectors as an element of~$\overline{\jj^{\ast}_+}$ (see Remark~\ref{rem:Lmdj}).
Let $\lambda\in\jj^{\ast}_+$ satisfy $\mu_{\lambda}\in\Lambda_+$.
By construction, the highest weight of the $K$-span of $\psi_{\lambda}|_{K/K\cap H}\in L^2(K/K\cap H)$ is~$\mu_{\lambda}$; this can be seen directly on \eqref{eqn:psilambdaky}, using the fact that $[\jj_{\C},\h_{\C}\cap\kk_{\C}]\subset\h_{\C}\cap\kk_{\C}$.
If $\mu_{\lambda}\in\Lambda^J$, then by definition $g\cdot\psi_{\lambda}|_{K/K\cap H}=\psi_{\lambda}|_{K/K\cap H}$ for all $g\in J$ (where $g$ acts by left translation); in particular, $\psi_{\lambda}(g\cdot x_0)=\psi_{\lambda}(x_0)=1$ for all $g\in J$.
\end{proof}

Proposition~\ref{prop:psilambda} for the Flensted-Jensen function $\psi_{\lambda}\in\V_{Z,\lambda}$ is an immediate consequence of \eqref{eqn:psilambdaky}, of Lemma~\ref{lem:FJcenter}, and of the following lemma.

\begin{lemma}\label{lem:FJineq}
Let $\lambda\in\jj^{\ast}_+$ satisfy~\eqref{eqn:lmdint}.
Then
\begin{enumerate}
  \item $|\xi_{\mu_{\lambda}}(k)|\leq 1$ for all $k\in K$;
  \item there exists $q'>0$ such that for all $Y\in\bb$ and $\ell\in H\cap K$,
$$|\xi_{-\lambda -\rho} (\exp(-Y)\ell)| \leq \cosh(q'\Vert Y\Vert)^{-d(\lambda+\rho)}.$$
\end{enumerate}
\end{lemma}

\begin{proof}[Proof of Lemma~\ref{lem:FJineq}]
Lemma~\ref{lem:FJineq}.(1) follows immediately from Lemma~\ref{lem:spherical} applied to the compact symmetric space $K/H\cap K$ instead of $G_U/H_U$.

To prove Lemma~\ref{lem:FJineq}.(2), we may assume that $G_{\C}$ is simply connected, because the Iwasawa projection for~$G^d$ is compatible with that of any covering of~$G^d$.
Then $\omega_j\in\Lambda_+(G_U/H_U)$ for all $1\leq j\leq r$ by \eqref{eqn:Lambda+U}.
To simplify notation, we write $(\pi_j,V_j,\xi_j^{\vee})$ for $(\pi_{\omega_j},V_{\omega_j},\xi_{\omega_j}^{\vee})$ and $\Vert\cdot\Vert_j$ for the Euclidean norm on~$V_j$ corresponding to the $G_U$-invariant inner product $(\cdot,\cdot)$ of Section~\ref{subsec:cptsymm}.
Then \eqref{eqn:lambdaomegaj} and Observation~\ref{obs:xi2} imply that for all $\lambda\in\jj^{\ast}$ and $g\in G^d$,
$$|\xi_{-\lambda-\rho}(g)| = e^{-\langle\lambda+\rho,\zeta(g)\rangle} = \prod_{j=1}^r \Vert\pi_j(g)\xi_j^{\vee}\Vert_j^{-\frac{(\lambda+\rho,\alpha_j)}{(\alpha_j,\alpha_j)}} \leq \prod_{j=1}^r \Vert\pi_j(g)\xi_j^{\vee}\Vert_j^{-d(\lambda+\rho)}.$$
Therefore, in order to prove Lemma~\ref{lem:FJineq}.(2), we only need to prove the existence of a constant $q'>0$ such that
\begin{equation}\label{eqn:minpij}
\min_{1\leq j\leq r} \Vert\pi_j((\exp Y)\ell)\xi_j^{\vee}\Vert_j \geq 1
\end{equation}
and
\begin{equation}\label{eqn:maxpij}
\max_{1\leq j\leq r} \Vert\pi_j((\exp Y)\ell)\xi_j^{\vee}\Vert_j \geq \cosh(q'\Vert Y\Vert)
\end{equation}
for all $Y\in\bb$ and $\ell\in H\cap K$.
For any $1\leq j\leq r$, the Lie algebra~$\bb$ acts semisimply on~$V_j$ with real eigenvalues, hence there are an orthonormal basis $(v_{ij})_{1\leq i \leq\dim V_j}$ of~$V_j$ and linear forms $\beta_{ij}\in\bb^{\ast}$, $1\leq i \leq\dim V_j$, such that
$$\pi_j (\exp Y)\,v_{ij} = e^{\langle\beta_{ij},Y\rangle}\,v_{ij}$$
for all $Y\in\bb$ and $1\leq i \leq\dim V_j$.
Write the matrix coefficients $\{b_{ij}\}$ for the restriction $\pi_j|_{H\cap K}$ as
$$\pi_j(\ell)\,\xi_j^{\vee} = \sum_{i=1}^{\dim V_j} b_{ij}(\ell)\,v_{ij} \qquad (\ell\in H \cap K),$$
where $\sum_{i=1}^{\dim V_j} |b_{ij}(\ell)|^2 = 1$ since $\pi_j|_{H\cap K}$ is unitary.
By \cite[Lem.\,4.6]{fle80},
$$\Vert\pi_j ((\exp Y)\ell)\xi_j^{\vee}\Vert_j^2 = \sum_{i=1}^{\dim V_j} |b_{ij}(\ell)|^2 \cosh\langle 2\beta_{ij},Y\rangle$$
for all $1\leq j\leq r$, all $Y\in\bb$, and all $\ell\in H\cap K$, hence \eqref{eqn:minpij} holds.
Let us prove~\eqref{eqn:maxpij}.
By a compactness argument \cite[Th.\,4.8]{fle80}, there is a constant $\varepsilon>0$ with the following property: for any $Y\in\bb$ and $\ell\in H\cap K$, there exist $j\in\{1,\dots,r\} $ and $i_0\in\{ 1,\dots,\dim V_j\} $ such that
\begin{equation}\label{eqn:ij}
\langle\beta_{i_0j},Y\rangle \geq \varepsilon\Vert Y\Vert \quad\text{and}\quad |b_{i_0j}(\ell)| \geq \varepsilon.
\end{equation}
For $Y\in\bb$ and $\ell\in H\cap K$, let $(i_0,j)$ be as in~\eqref{eqn:ij}.
Then
\begin{align*}
\Vert\pi_j((\exp Y)\ell) \xi_j^{\vee}\Vert_j^2 & = \sum_{i=1}^{\dim V_j} |b_{ij}(\ell)|^2 \cosh\langle2\beta_{ij},Y\rangle\\
& \geq |b_{i_0j}(\ell)|^2 \cosh\langle2\beta_{i_0j},Y\rangle + \sum_{i\neq i_0} |b_{ij}(\ell)|^2\\
& \geq \varepsilon^2 \cosh (2\varepsilon \Vert Y\Vert) + (1-\varepsilon^2).
\end{align*}
By using the general inequality
$$t \cosh(x) + (1-t) \geq \Big(\!\!\cosh\frac{tx}{2}\Big)^2,$$
which holds for any $0<t\leq 1$ and $x\in\R$, we obtain
\begin{equation*}
\Vert\pi_j((\exp Y)\ell) \xi_j^{\vee}\Vert_j \geq \cosh(\varepsilon^3\Vert Y\Vert).
\end{equation*}
This proves \eqref{eqn:maxpij} for $q':=\varepsilon^3$ and completes the proof of Lemma~\ref{lem:FJineq}.
\end{proof}

%% file: StableSpec8-Nonzero.tex
\chapter[Nonvanishing of eigenfunctions]{Nonvanishing of eigenfunctions on locally symmetric spaces}\label{sec:nonzero}

As explained at the beginning of Chapter~\ref{sec:FJ}, our goal now is to complete the proof of the theorems and propositions of Chapters \ref{sec:intro} to~\ref{sec:theorems} by establishing the following key proposition.

As in Section~\ref{subsec:Lambda+}, we denote by $G_c$ (\resp $L_c$) the maximal compact normal subgroup of the reductive group~$G$ (\resp $L$) and by $Z(G_s)$ the center of the commutator subgroup of~$G$.
The $\Z$-module $\Lambda^{\Gamma\cap Z(G_s)}$ for $\Gamma\subset G$ has been defined in \eqref{eqn:LambdaJ}.
We choose a positive system $\Sigma^+(\g_{\C},\jj_{\C})$ containing the fixed positive system $\Sigma^+(\kk_{\C},\jj_{\C})$ of Section~\ref{subsec:Lambda+}; this defines a positive Weyl chamber $\jj^{\ast}_+$ and an element $\rho\in\jj^{\ast}_+$ as in Section~\ref{subsec:Lambda+}.

\begin{proposition}\label{prop:nonzero}
Suppose that $G$ is connected, that $H$ does not contain any simple factor of~$G$, and that the rank condition \eqref{eqn:rank} holds.
\begin{enumerate}
  \item (Sharp Clifford--Klein forms)\\
  For any sharp Clifford--Klein form~$X_{\Gamma}$ of~$X$ with $\Gamma\cap G_c\subset Z(G_s)$, there is a constant~$R\geq\nolinebreak 0$  such that for any $\lambda\in\jj^{\ast}_+\cap (2\rho_c-\nolinebreak\rho+\nolinebreak\Lambda^{\Gamma\cap Z(G_s)})$ with $d(\lambda)>R$, the summation operator~$S_{\Gamma}$ is well-defined and nonzero on $g\!\cdot\!L^2(X,\M_{\lambda})_K$ for some $g\in G$.
  \smallskip
  \item (Uniformity for standard Clifford--Klein forms)\\
  Let $L$ be a reductive subgroup of~$G$, with a compact center and acting properly on~$X$.
  There is a constant $R>0$ with the following property: for any discrete subgroup~$\Gamma$ of~$L$ with $\Gamma\cap L_c\subset Z(G_s)$ (in particular, for any torsion-free discrete subgroup $\Gamma$ of~$L$) and for any $\lambda\in\jj^{\ast}_+\cap (2\rho_c-\rho+\Lambda^{\Gamma\cap Z(G_s)})$ with $d(\lambda)>R$, the operator~$S_{\Gamma}$ is well-defined and nonzero on $g\!\cdot\!L^2(X,\M_{\lambda})_K$ for some $g\in G$.
  \smallskip
  \item (Stability under small deformations)\\
  Let $L$ be a reductive subgroup of~$G$ of real rank~$1$, acting properly on~$X$, and let $\Gamma$ be a convex cocompact subgroup of~$L$ (for instance a uniform lattice) with $\Gamma\cap G_c\subset Z(G_s)$.
  Then there are a constant $R>0$ and a neighborhood $\mathcal{U}\subset\Hom(\Gamma,G)$ of the natural inclusion such that for any $\varphi\in\mathcal{U}$, the group $\varphi(\Gamma)$ acts properly discontinuously on~$X$ and for any $\lambda\in\jj^{\ast}_+\cap (2\rho_c-\rho+\Lambda^{\Gamma\cap Z(G_s)})$ with $d(\lambda)>R$, the operator~$S_{\varphi(\Gamma)}$ is well-defined and nonzero on $g\!\cdot\!L^2(X,\M_{\lambda})_K$ for\linebreak some $g\in G$.\\
  If $\Gamma\cap L_c\subset Z(G_s)$ (for instance if $\Gamma$ is torsion-free or if $L$ is simple with $Z(L)\subset Z(G_s)$), then we may take the same~$R$ (independent of~$\Gamma$) as in~(2), up to replacing $\mathcal{U}$ by some smaller neighborhood.
\end{enumerate}
\end{proposition}

Recall that $L^2(X,\M_{\lambda})$ is the space of $L^2$-weak solutions to the system $(\M_{\lambda})$ of Section~\ref{subsec:Lambda+} and $L^2(X,\M_{\lambda})_K$ is the subspace of $K$-finite functions.
The group~$G$ acts on $L^2(X,\M_{\lambda})$ by left translation \eqref{eqn:gvarphi}.
We define a summation operator $S_{\Gamma}$ on any $G$-translate $g\!\cdot\!L^2(X,\M_{\lambda})_K$ by the same formula as in Proposition~\ref{prop:Vlambda}: see Section~\ref{subsec:transl} below.
The fact that we need to consider $G$-translates is linked to the geometric issue of distribution of $\Gamma$-orbits in~$X$ and in the Riemannian symmetric space $G/K$ (see Remark~\ref{rem:translconj}, together with Propositions \ref{prop:sharp} and~\ref{prop:KM}).

As we shall see in Section~\ref{subsec:proofnonzero} (Formulas \eqref{eqn:finalcst} and \eqref{eqn:finalcstconj}), the constant~$R$ of Proposition~\ref{prop:nonzero}.(1) can be expressed in terms of the sharpness constants $(c,C)$ of~$\Gamma$ and of the minimal nonzero value of $\Vert\nu\Vert$ on the $\Gamma$-orbit $\Gamma\!\cdot\!x_0$.
Recall that $\Vert\nu\Vert$ measures the ``pseudo-distance to the origin~$x_0$''.

We note that the technical assumptions of Proposition~\ref{prop:nonzero} are not very restrictive: Remarks~\ref{rem:conditionsGL} also apply in this context.

\begin{remark}\label{rem:translprecise}
We can make Proposition~\ref{prop:nonzero}.(1), (2), and~(3) more precise with respect to $G$-translation: we actually prove that
\begin{enumerate}
  \item[(a)] for $d(\lambda)>R$, the operator~$S_{\Gamma}$ is well-defined on $g\!\cdot\!L^2(X,\M_{\lambda})_K$ for \emph{all} $g\in G$;
  \item[(b)] there is an element $g\in G$ such that $S_{\Gamma}$ is nonzero on $g\!\cdot\!L^2(X,\M_{\lambda})_K$ for \emph{all} $\lambda$ with $d(\lambda)>R$.
\end{enumerate}
\end{remark}

\noindent
Statement~(a) follows from Proposition~\ref{prop:Vlambda} and from the fact that the first sharpness constant is invariant under conjugation (Proposition~\ref{prop:sharpnessproperties}), using Remark~\ref{rem:translconj} below.
For Statement~(b), we refer to Section~\ref{subsec:proofnonzero}.

\begin{remark}\label{rem:Zprecise}
We can make Proposition~\ref{prop:nonzero} more precise in terms of discrete series representations for~$X$.
Recall from Fact~\ref{fact:decompVlambda} that $L^2(X,\M_{\lambda})_K$ is the direct sum of finitely many irreducible $(\g,K)$-modules $\V_{Z,\lambda}$, where $Z\in\ZZ$.
We have given two combinatorial descriptions of the set~$\ZZ$.
\begin{itemize}
  \item In terms of positive systems: by \eqref{eqn:Z1to1}, any $Z\in\ZZ$ corresponds to a positive system $\Sigma^+(\g_{\C},\jj_{\C})$, which determines a positive Weyl chamber $\jj^{\ast}_+$ and an element $\rho\in\jj^{\ast}_+$.
  We prove that $S_{\Gamma}$ is well-defined and nonzero on $g\!\cdot\!\V_{Z,\lambda}\subset g\!\cdot\!L^2(X,M_{\lambda})_K$ for any $\lambda\in\jj^{\ast}_+$ with $d(\lambda)>R$ satisfying
  $$\mu_{\lambda} = \lambda + \rho - 2\rho_c \in \Lambda^{\Gamma\cap Z(G_s)}.$$
  \item In terms of Weyl group elements: fix a positive system $\Sigma^+(\g_{\C},\jj_{\C})$ containing the positive system $\Sigma^+(\kk_{\C},\jj_{\C})$ of Section~\ref{subsec:Lambda+}; this determines a positive Weyl chamber $\jj^{\ast}_+$ and an element $\rho\in\jj^{\ast}_+$.
  By \eqref{eqn:zw}, any $Z\in\ZZ$ corresponds to an element $w\in W(H^d,G^d)$, where $W(H^d,G^d)\subset W$ is a complete set of representative for the left coset space $W_{H\cap K}\backslash W$.
  We prove that $S_{\Gamma}$ is well-defined and nonzero on $g\cdot\V_{Z,\lambda}\subset g\!\cdot\!L^2(X,\M_{\lambda})_K$ for any $\lambda\in\jj^{\ast}_+$ with $d(\lambda)>R$ satisfying
  $$\mu_{\lambda}^w = w(\lambda+\rho) - 2\rho_c \in \Lambda^{\Gamma\cap Z(G_s)}.$$
\end{itemize}
Thus we get different integrality conditions on~$\lambda$ depending on the element $Z\in\ZZ$ we are considering.
These conditions might not be all equivalent; it is enough for $\lambda$ to satisfy one of them in order to belong to the discrete spectrum $\Spec_d(X_{\Gamma})$.
\end{remark}

\section{The summation operator $S_{\Gamma}$ on $G$-translates of $L^2(X,\M_{\lambda})_K$}\label{subsec:transl}

Let $X_{\Gamma}$ be a Clifford--Klein form of~$X$.
We define the summation operator~$S_{\Gamma}$ on any $G$-translate $g\!\cdot\!L^2(X,\M_{\lambda})_K$ as follows.

For $g\in G$, let $\ell_g : x\mapsto g\cdot x$ be the translation by~$g$ on~$X$.
The following diagram commutes, where $p_{\Gamma} : X\rightarrow X_{\Gamma}$ is the natural projection.
$$\xymatrix{
X \ar[d]_{p_{\Gamma}} \ar[r]^{\ell_g}_{\sim} & \!X \ar[d]^{p_{g\Gamma g^{-1}}} & x \ar@{|->}[d] \ar@{|->}[r] & g\cdot x \ar@{|->}[d]\\
X_{\Gamma} \ar[r]^{\sim\ } & \!X_{g\Gamma g^{-1}} & \Gamma x \ar@{|->}[r] & (g\Gamma g^{-1})\,(g\cdot x)
}$$
\noindent
Since $\D(X)$ consists of $G$-invariant differential operators, we obtain the following commutative diagram for smooth functions satisfying~$(\M_{\lambda})$.
$$\xymatrix{
C^{\infty}(X,\M_{\lambda}) & \ar[l]_{\ell_g^{\ast}}^{\sim} C^{\infty}(X,\M_{\lambda})\\
C^{\infty}(X_{\Gamma},\M_{\lambda}) \ar[u]^{p_{\Gamma}^{\ast}} & \ar[l]_{\sim} C^{\infty}(X_{g\Gamma g^{-1}},\M_{\lambda}) \ar[u]_{p_{g\Gamma g^{-1}}^{\ast}}
}$$
\noindent
The space $L^2(X,\M_{\lambda})_K$ is contained in $C^{\infty}(X,\M_{\lambda})$ (see Section~\ref{subsec:discreteseries}), and
\begin{equation}\label{eqn:Kfiniteconj}
\ell_g^{\ast}\,L^2(X,\M_{\lambda})_K = L^2(X,\M_{\lambda})_{g^{-1}Kg}.
\end{equation}
For $\varphi\in\ell_g^{\ast}\,L^2(X,\M_{\lambda})_K\subset C^{\infty}(X,\M_{\lambda})$, we set
$$S_{\Gamma}(\varphi) = \varphi^{\Gamma} := \bigg(\Gamma x \longmapsto \sum_{\gamma\in\Gamma} \varphi(\gamma^{-1}\!\cdot x)\bigg)\,;$$
this is the same formula as the one defining $S_{\Gamma}$ on $L^2(X,\M_{\lambda})_K$ in Proposition~\ref{prop:Vlambda}.
Then $S_{\Gamma}$ is well-defined on $\ell_g^{\ast}\,L^2(X,\M_{\lambda})_K$ if and only if $S_{g\Gamma g^{-1}}$ is well-defined on $L^2(X,\M_{\lambda})_K$, and in this case the following diagram commutes.
$$\xymatrix{
C^{\infty}(X,\M_{\lambda})\ \supset\ \ell_g^{\ast}\,L^2(X,\M_{\lambda})_K \ar[d]<8ex>_{S_{\Gamma}} & \ar[l]_{\ell_g^{\ast}}^{\sim} L^2(X,\M_{\lambda})_K \ar[d]<-8ex>^{S_{g\Gamma g^{-1}}}\ \subset\ C^{\infty}(X,\M_{\lambda})\\
\hspace{2.8cm} L^2(X_{\Gamma},\M_{\lambda}) & \ar[l]_{\sim} L^2(X_{g\Gamma g^{-1}},\M_{\lambda}) \hspace{2.5cm}
}$$
\noindent
We note that
\begin{equation}\label{eqn:transL2}
g\cdot L^2(X,\M_{\lambda})_K = \ell_{g^{-1}}^{\ast}\big(L^2(X,\M_{\lambda})_K\big).
\end{equation}
In particular, we will use the following.

\begin{remark}\label{rem:translconj}
The operator $S_{\Gamma}$ is nonzero on $g\cdot L^2(X,\M_{\lambda})_K$ if and only if the operator $S_{g^{-1}\Gamma g}$ is nonzero on $L^2(X,\M_{\lambda})_K$.
\end{remark}

The reason why we consider $G$-translates $g\!\cdot\!L^2(X,\M_{\lambda})_K$ to construct nonzero eigenfunctions on~$X_{\Gamma}$ is precisely that we want to allow ourselves to replace the groups~$\Gamma$ by conjugates $g^{-1}\Gamma g$ (see Propositions \ref{prop:sharp} and~\ref{prop:KM}).

\section{Nonvanishing on sharp Clifford--Klein forms}\label{subsec:sharpnonzero}

We adopt the first point of view described in Remark~\ref{rem:Zprecise}: for the whole chapter we choose a positive system $\Sigma^+(\g_{\C},\jj_{\C})$ containing the fixed positive system $\Sigma^+(\kk_{\C},\jj_{\C})$ of Section~\ref{subsec:Lambda+}; this defines a positive Weyl chamber $\jj^{\ast}_+$ and an element $\rho\in\jj^{\ast}_+$ as in Section~\ref{subsec:Lambda+}, as well as an element $Z\in\ZZ$ by \eqref{eqn:Z1to1}.
The key ingredient in the proof of Proposition~\ref{prop:nonzero} is the following lemma.

\begin{lemma}\label{lem:nonzeroprecise}
Assume that the rank condition \eqref{eqn:rank} holds.
For $c,C,r>0$, let $\Gamma$ be a discrete subgroup of~$G$ such that:
\begin{enumerate}
  \item $\Gamma$ is $(c,C)$-sharp for~$X$,
  \item $\inf\{ \Vert\nu(x)\Vert : x\in\Gamma\!\cdot\!x_0\text{ and }x\notin X_c\} \geq r$,
  \item $\Gamma\!\cdot\!x_0\cap X_c\subset Z(G_s)\!\cdot\!x_0$.
\end{enumerate}
For any $\lambda\in\jj^{\ast}_+\cap (2\rho_c-\rho+\Lambda^{\Gamma\cap Z(G_s)})$ with $d(\lambda)>\max(m_{\rho},R_X/c)$ and
$$d(\lambda+\rho) > \frac{4\Vert\rho_a\Vert (r+C) + \log\big(2c_G\,\#(\Gamma\cap K)\big)}{c\,\log\cosh(q'r)},$$
the operator $S_{\Gamma} : L^2(X,\M_{\lambda})_K\rightarrow L^2(X_{\Gamma},\M_{\lambda})$ is well-defined
and any function $\psi_{\lambda}\in\V_{Z,\lambda}\subset L^2(X,\M_{\lambda})_K$ as in Proposition~\ref{prop:psilambda} satisfies $S_{\Gamma}(\psi_{\lambda})(x_0)\neq\nolinebreak 0$.
\end{lemma}

Let us recall earlier notation: $\rho_a\in\aaa$ is half the sum of the elements of $\Sigma^+(\g,\aaa)$, counted with root multiplicities, and $m_{\rho}$, $c_G$, $R_X$, and~$q'$ are the constants of \eqref{eqn:rhomax}, Observation~\ref{obs:limsup}, Proposition~\ref{prop:Vlambda}, and Proposition~\ref{prop:psilambda} respectively.
We denote by $x_0$ the image of~$H$ in $X=G/H$ and keep the same notation for its image in $X_{\Gamma}=\Gamma\backslash X$ for any Clifford--Klein form~$X_{\Gamma}$.
The set $X_c=K\!\cdot\!x_0$ consists of the points $x$ in~$X$ whose ``pseudo-distance to the origin'' $\Vert\nu(x)\Vert$ is zero; it is a maximal compact subsymmetric subspace of~$X$, and identifies with $K/K\cap H$.
Remark~\ref{rem:Gammapseudoball} implies the following.

\begin{remark}\label{rem:infnu}
For any discrete subgroup $\Gamma$ of~$G$ acting properly discontinuously on~$X$,
$$\inf\big\{ \Vert\nu(x)\Vert : x\in\Gamma\!\cdot\!x_0\text{ and }x\notin X_c\big\} >0.$$
\end{remark}

\begin{remark}\label{rem:dlambdarho}
For any $\lambda\in\jj^{\ast}_+$ we have $d(\lambda+\rho)\geq d(\lambda)$, hence for $R'>0$ the condition $d(\lambda+\rho)>R'$ is satisfied as soon as $d(\lambda)>R'$.
\end{remark}

\begin{proof}[Proof of Lemma~\ref{lem:nonzeroprecise}]
Let $\lambda\in\jj^{\ast}_+\cap (2\rho_c-\rho+\Lambda^{\Gamma\cap Z(G_s)})$.
Assume that $d(\lambda)>\max(m_{\rho},R_X/c)$; then the summation operator
$$S_{\Gamma} : L^2(X,\M_{\lambda})_K \longrightarrow L^2(X_{\Gamma},\M_{\lambda})$$
is well-defined by Proposition~\ref{prop:Vlambda}.
Assume moreover that $d(\lambda)\geq m_{\rho}$; then $\lambda\in 2\rho_c-\rho+\Lambda_+$ by Lemma~\ref{lem:regKdom} and we can apply Proposition~\ref{prop:psilambda}.
The function~$\psi_{\lambda}$ of Proposition~\ref{prop:psilambda} has module $<1$ outside of~$X_c$.
In order to prove that $\psi_{\lambda}^{\Gamma}(x_0)\neq 0$, we naturally split the sum into two: on the one hand the sum over the elements $\gamma\in\Gamma$ with $\gamma\cdot x_0\in X_c$, on the other hand the sum over the elements $\gamma\in\Gamma$ with $\gamma\cdot x_0\notin X_c$.
We control the first summand by using the assumption~(3) that the $\Gamma$-orbit of $\Gamma\!\cdot\!x_0$ meets~$X_c$ only inside the finite set $Z(G_s)\!\cdot\!x_0$, where $\psi_{\lambda}$ takes value~$1$: by Lemma~\ref{lem:FJcenter},
$$\left|\sum_{\gamma\in\Gamma,\ \gamma\cdot x_0\in X_c} \psi_{\lambda}(\gamma\!\cdot\!x_0)\right| \,=\, \#\{ \gamma\in\Gamma : \gamma\!\cdot\!x_0\in X_c\} \,\geq\, 1.$$
Therefore, in order to prove that $\psi_{\lambda}^{\Gamma}(x_0)\neq 0$, it is sufficient to prove that
$$\sum_{\gamma\in\Gamma,\ \gamma\cdot x_0\notin X_c} |\psi_{\lambda}(\gamma\!\cdot\!x_0)| < 1.$$
The estimate \eqref{eqn:FJglobal} and the assumption (2) on the ``pseudo-distance to the origin''~$\Vert\nu\Vert$ imply
\begin{align*}
& \sum_{\gamma\in\Gamma,\ \gamma\cdot x_0\notin X_c} |\psi_{\lambda}(\gamma\!\cdot\!x_0)|\\
& \quad \leq \sum_{n=1}^{+\infty} \cosh(q'rn)^{-d(\lambda+\rho)} \cdot \#\{ \gamma\in\Gamma : nr\leq\Vert\nu(\gamma)\Vert<(n+1)r\} ,
\end{align*}
where the constant $q'>0$ of Proposition~\ref{prop:psilambda} depends only on~$X$.
We now use the assumption~(1) that $\Gamma$ is $(c,C)$-sharp.
By Lemma~\ref{lem:growthfornu}.(3),
$$\#\{ \gamma\in\Gamma : \Vert\nu(\gamma)\Vert<(n+1)r\} \leq \#(\Gamma\cap K) \cdot c_G\,e^{2\Vert\rho_a\Vert\frac{(n+1)r+C}{c}},$$
where the constant $c_G>0$ of Observation~\ref{obs:limsup} depends only on~$G$.
Thus
\begin{align*}
& \sum_{\gamma\in\Gamma,\ \gamma\cdot x_0\in X_c} |\psi_{\lambda}(\gamma\!\cdot\!x_0)|\\
& \quad \leq \#(\Gamma\cap K) \cdot c_G e^{\frac{2\Vert\rho_a\Vert(r+C)}{c}} \cdot \sum_{n=1}^{+\infty} \cosh(q'rn)^{-d(\lambda+\rho)} \cdot e^{(\frac{2\Vert\rho_a\Vert r}{c})n},
\end{align*}
and we conclude using the following lemma.
\end{proof}

\smallskip

\begin{lemma}\label{lem:coshsum}
For any $S,T,U>0$ with $S\geq 1$,
$$S\,\sum_{n=1}^{+\infty} \cosh(Tn)^{-d}\,e^{Un} < 1$$
for all $d>R:=\frac{\log(2S)+U}{\log\cosh T}$.
\end{lemma}

\begin{proof}
It is sufficient to prove that for all $d>R$ and all $n\geq 1$,
$$S\,\cosh(Tn)^{-d}\,e^{Un} < 2^{-n},$$
or equivalently
$$d > \frac{\log S + n\,(\log 2 + U)}{\log\cosh(Tn)}.$$
One easily checks that for all $n\geq 1$,
$$\log S + n\,(\log 2 + U) \leq n\,(\log(2S) + U)$$
and
$$\log\cosh(Tn) \geq n\,\log\cosh T.\qedhere$$
\end{proof}

\section{Points near the origin in the orbit of a sharp discrete group}\label{subsec:minnu}

In this section and the next one we do not need the rank condition \eqref{eqn:rank}.

In Lemma~\ref{lem:nonzeroprecise} we assumed that $\Gamma\!\cdot\!x_0\cap X_c\subset Z(G_s)\!\cdot\!x_0$, where $X_c=K\cdot x_0$ is the maximal compact subsymmetric space of~$X$ consisting of the points~$x$ whose ``pseudo-distance to the origin'' $\Vert\nu(x)\Vert$ is zero and $Z(G_s)$ is the center of the commutator subgroup of~$G$.
We now prove the following, where $G_c$ denotes the maximal compact normal subgroup of~$G$ (as in Chapter~\ref{subsec:Lambda+}) and $G_H$ the maximal normal subgroup of~$G$ contained in~$H$.

\begin{proposition}\label{prop:sharp}
For any discrete subgroup $\Gamma$ of~$G$ acting properly discontinuously on~$X$, there is an element $g\in G$ such that $g^{-1}\gamma g\cdot x_0\notin X_c$ for all $\gamma\in\Gamma\smallsetminus G_c\,G_H$.
\end{proposition}

In Section~\ref{subsec:proofnonzero} we shall combine Proposition~\ref{prop:sharp} with Lemma~\ref{lem:nonzeroprecise} to prove Proposition~\ref{prop:nonzero}.(1).
Recall that in Proposition~\ref{prop:nonzero}.(1) we assumed that $H$ does not contain any simple factor of~$G$; it has the following consequence.

\begin{remark}\label{rem:GcGH}
If $H$ does not contain any simple factor of~$G$, then $G_H=Z(G)\cap H$ and $\Gamma\cap G_cG_H=\Gamma\cap G_c$ for any discrete subgroup $\Gamma$ of~$G$ acting properly discontinuously on $X=G/H$.
\end{remark}

\noindent
The assumption $\Gamma\cap G_c\subset Z(G_s)$ in Proposition~\ref{prop:nonzero}.(1) is there to ensure that if $g^{-1}\gamma g\cdot x_0\notin X_c$ for all $\gamma\in\Gamma\smallsetminus G_c$ (as given by Proposition~\ref{prop:sharp}), then $g^{-1}\Gamma g\cdot x_0\cap X_c\subset Z(G_s)\cdot x_0$ (as required to apply Lemma~\ref{lem:nonzeroprecise}).

In the rest of this section we give a proof of Proposition~\ref{prop:sharp}.

\subsection*{$\bullet$ The main lemma and its interpretation}

We first establish the following.

\begin{lemma}\label{lem:GcH}
For any $\gamma\in G\smallsetminus G_c\,G_H$, there is an element $g\in G$ such that $g^{-1}\gamma g\cdot x_0\notin X_c$, or in other words $g^{-1}\gamma g\notin KH$.
\end{lemma}

We note that $G_H$ is the set of elements of~$G$ that act trivially on~$X$.
In particular, for any $\gamma\in G\smallsetminus G_H$ there is an element $g\in G$ such that $g^{-1}\gamma g\cdot x_0\neq x_0$.
Lemma~\ref{lem:GcH} states that if $\gamma\notin G_c\,G_H$, then we can actually find $g$ such that $g^{-1}\gamma g\cdot x_0\notin X_c$.
The condition $\gamma\notin G_c\,G_H$ cannot be improved: if $\gamma\in G_c\,G_H$, then any conjugate of~$\gamma$ maps $x_0$ inside $G_c\cdot x_0\subset X_c$, since $G_c\,G_H$ is normal in~$G$.

Here is a group-theoretic interpretation.

\begin{remark}
For any subset $S$ of~$G$, let
$$G[S] := \bigcap_{g\in G} gSg^{-1}.$$
If $S$ is a group, then $G[S]$ is the maximal normal subgroup of~$G$ contained in~$S$.
In particular, $G[K]=G_c$ and $G[H]=G_H$.
Lemma~\ref{lem:GcH} states that $G[KH]=G[K]G[H]$.
We note that this equality may fail if we replace~$K$ by some noncompact symmetric subgroup of~$G$, \ie by~$H'$ such that $G/H'$ is a non-Riemannian symmetric space.
\end{remark}

\subsection*{$\bullet$ Preliminary Lie-theoretic remarks}

Before we prove Lemma~\ref{lem:GcH}, we make a few useful remarks.
For any subspaces $\mathfrak{e},\mathfrak{f}$ of~$\g$, we set
\begin{equation}\label{eqn:zeroadjoint}
\mathfrak{e}^{\mathfrak{f}} := \big\{ Y\in\mathfrak{e} : [\mathfrak{f},Y]=\{ 0\}\big\} .
\end{equation}

\begin{lemma}\label{lem:normalizerkq}
Assume that $G$ is simple.
\begin{enumerate}
  \item For any nonzero ideal $\kk'$ of~$\kk$, we have $\p^{\kk'}=\{ 0\}$.
  \item The Lie algebra spanned by $\kk\cap\q$ contains~$\kk_s$.
  \item The normalizer $N_H(\kk\cap\q):=\{ h\in H : \Ad(h)(\kk\cap\q)=\kk\cap\q\}$ of $\kk\cap\q$ in~$H$ is contained in~$K$.
\end{enumerate}
\end{lemma}

\begin{proof}[Proof of Lemma~\ref{lem:normalizerkq}]
\begin{enumerate}
  \item If $\kk'$ is an ideal of~$\kk$, then the space $\p^{\kk'}$ is globally stable under $\ad(\kk)$, or equivalently under $\Ad(K)$.
  But the adjoint action of $K$ on~$\p$ is irreducible \cite[Ch.~XI, Prop.\,7.4]{kn69}, hence $\p^{\kk'}$ is either $\{0\}$ or~$\p$.
  Since $K$ is reductive, we can write $\kk$ as the direct sum of $\kk'$ and of some other ideal~$\kk''$.
  If $\p^{\kk'}=\p$, then $\kk''+\p$ is an ideal of~$\g$, hence $\kk''+\p=\g$ since $\g$ is simple; in other words, $\kk'=\{ 0\}$.
  \item For any reductive Lie group~$L$ with Lie algebra $\mathfrak{l}$, we denote by $\mathfrak{l}_s$ the Lie algebra of the commutator subgroup (or semisimple part) of~$L$.
  Proving that $\kk_s$ is contained in the Lie algebra spanned by $\kk\cap\q$ is equivalent to proving that $(\kk_{\C})_s$ is contained in the Lie algebra spanned by $\kk_{\C}\cap\q_{\C}$.
  In turn, this is equivalent to proving that $(\h^d)_s$ is contained in the Lie algebra spanned by $\h^d\cap\p^d$, since the complexifications of $\h^d$ and $\p^d$ are $\kk_{\C}$ and~$\q_{\C}$, respectively (see Section~\ref{subsec:realforms}).
  But $(\h^d)_s$ admits the Cartan decomposition $(\h^d)_s=(\h^d)_s\cap\kk^d+(\h^d)_s\cap\p^d$, and it is well-known that if $\mathfrak{l}$ is a semisimple Lie algebra with Cartan decomposition $\mathfrak{l}=\kk_{\mathfrak{l}}+\p_{\mathfrak{l}}$, then $[\p_{\mathfrak{l}},\p_{\mathfrak{l}}]+\p_{\mathfrak{l}}=\mathfrak{l}$ (one easily checks that $[\p_{\mathfrak{l}},\p_{\mathfrak{l}}]+\p_{\mathfrak{l}}$ is an ideal of~$\mathfrak{l}$, hence equal to~$\mathfrak{l}$ if $\mathfrak{l}$ is simple; the general semisimple case follows from decomposing~$\mathfrak{l}$ into a sum of simple ideals).
  Thus $(\h^d)_s$ is contained in the Lie algebra spanned by $(\h^d)_s\cap\p^d\subset\h^d\cap\p^d$.
  \item The group $L:=N_H(\kk\cap\q)$ is stable under the Cartan involution $\theta$ of~$G$, since $\kk\cap\q$ is fixed by~$\theta$.
  Therefore $L$ is reductive and admits the Cartan decomposition $L=(K\cap L)\,\exp(\p\cap\mathfrak{l})$.
  Proving that $L$ is contained in~$K$ is equivalent to proving that $\p\cap\mathfrak{l}=\{ 0\}$.
  We have
  $$\p\cap\mathfrak{l} = \big\{ Y\in\h\cap\p : \ad(Y)(\kk\cap\q)\subset\kk\cap\q\big\} = (\h\cap\p)^{\kk\cap\q},$$
  hence $\p\cap\mathfrak{l}$ is contained in $\p^{\kk\cap\q}=\p^{\langle\kk\cap\q\rangle}$, where $\langle\kk\cap\q\rangle$ is the Lie algebra spanned by $\kk\cap\q$.
  By (1) (with $\kk'=\kk_s$) and~(2), we have $\p^{\langle\kk\cap\q\rangle}=\{ 0\}$.\qedhere
\end{enumerate}
\end{proof}

\subsection*{$\bullet$ Proof of Lemma~\ref{lem:GcH}}

Suppose that $\gamma$ satisfies
\begin{equation}\label{eqn:nu0}
g^{-1}\gamma g\in KH \quad\text{for all }g\in G.
\end{equation}
Let us prove that $\gamma\in G_c\,G_H$.
We first assume that $G$ is simple.
The idea is to work in the Riemannian symmetric space $G/K$ of~$G$, where we can use the $G$-invariant metric $d_{G/K}$.
As before, we denote by $y_0$ the image~of~$K$~in~$G/K$.

Firstly, we claim that $\gamma\in K$.
Indeed, write $\gamma\in Kh$ where $h\in H$.
Then \eqref{eqn:nu0} with $g\in K$ implies $hKh^{-1}\subset KH$, \ie $hKh^{-1}\!\cdot\!y_0\subset H\!\cdot\!y_0$.
By considering the tangent space of $G/K$ at~$x_0$, which identifies with $\g/\kk$, we see that $\Ad(h)\kk\subset\h+\kk$, or in other words $\kk\subset\h+\Ad(h^{-1})(\kk)$.
This implies $\Ad(h^{-1})(\kk\cap\q)=\kk\cap\q$.
By Lemma~\ref{lem:normalizerkq}.(3), we have $h\in K$.

Secondly, we claim that $\gamma^{-1}$ fixes pointwise the set $K\overline{B_+}\cdot y_0$.
Indeed, let $k\in K$ and $b\in\overline{B_+}$.
By \eqref{eqn:nu0}, we have $\gamma^{-1}kb\cdot y_0\in kbH\cdot y_0$.
By \eqref{eqn:mudistance}, \eqref{eqn:munuB}, and Lemma~\ref{lem:munu},
$$d_{G/K}(y_0,kb\cdot y_0) = \Vert\mu(b)\Vert = \Vert\nu(b)\Vert = \Vert\nu(bh)\Vert \leq \Vert\mu(bh)\Vert = d_{G/K}(y_0,kbh\cdot y_0)$$
for all $h\in H$, hence $kb\cdot y_0$ is the projection of~$y_0$ to the totally geodesic subspace $kbH\!\cdot\!y_0$.
Since $\gamma\in K$ fixes~$y_0$ and acts on $G/K$ by isometries,~we~have
$$d_{G/K}(y_0,\gamma^{-1}kb\cdot y_0) = d_{G/K}(y_0,kb\cdot y_0) \leq d_{G/K}(y_0,kbh\cdot y_0)$$
for all $h\in H$.
But $\gamma^{-1}kb\cdot y_0$ belongs to $kbH\cdot y_0$ by assumption, and $kb\cdot y_0$ is the projection of~$y_0$ to $kbH\cdot y_0$, so $\gamma^{-1}kb\cdot y_0=kb\cdot y_0$.
This proves the~claim.

To prove that $\gamma\in G_c\,G_H$, we assume that the simple group~$G$ is noncompact, so that $G_cG_H=Z(G)$ (otherwise $G_c=G$).
Then $\overline{B_+}\neq\{ e\}$.
We have seen that $\gamma^{-1}$ fixes pointwise the set $K\overline{B_+}\cdot y_0$, which is equivalent to the fact that $\gamma\in (kb)K(kb)^{-1}$ for all $k\in K$ and $b\in\overline{B_+}$.
Thus $\gamma$ belongs to the closed normal subgroup
$$K' := \bigcap_{k\in K,\ b\in\overline{B_+}} (kb)K(kb)^{-1}$$
of~$K$.
We note that $\Ad(k')(Y)=Y$ for all $k'\in K'$ and $Y\in\overline{\bb_+}$.
Indeed, $\Ad(k')(Y)-Y\in\p$ since $K'\subset K$, and $\Ad(k')(Y)-Y\in\kk$ since $b^{-1}K'b\subset K$.
In particular, the Lie algebra $\kk'$ of~$K'$ satisfies $\p^{\kk'}\neq\{ 0\}$ with the notation \eqref{eqn:zeroadjoint}.
But $\kk'$ is an ideal of~$\kk$, hence $\kk'=\{ 0\}$ by Lemma~\ref{lem:normalizerkq}.(1).
In other words, $K'$ is contained in the center $Z(K)$ of~$K$.
We claim that in fact $K'\subset Z(G)$.
Indeed, for any $k'\in K'$ the set $\g^{\Ad(k')}$ of fixed points of~$\g$ under $\Ad(k')$ is a Lie subalgebra that contains both~$\kk$ and $\overline{\bb_+}\neq\{ 0\}$.
But the Lie algebra~$\g$ is generated by~$\kk$ and any nontrivial element of~$\p$ (because the adjoint action of $K$ on~$\p$ is irreducible \cite[Ch.~XI, Prop.\,7.4]{kn69}), hence $\g^{\Ad(k')}=\g$, which means that $k'\in Z(G)$.
In particular, $\gamma\in Z(G)=G_cG_H$.

In the general case where $G$ is not necessarily simple, we write~$G$ as the almost product of a split central torus $\simeq\R^a$, of~$G_c\,G_H$, and of noncompact simple factors $G_1,\dots,G_m$ with $G_i\not\subset H$ for all~$i$.
Since $\gamma$ is elliptic, we can decompose it as $\gamma=\gamma_0\gamma_1\dots\gamma_m$, where $\gamma_0\in G_c\,G_H$ and $\gamma_i\in G_i$ for all $i\geq 1$.
For $i\geq 1$, the restriction of~$\sigma$ to~$G_i$ is an involution; the polar decomposition $G_i=(K\cap G_i)(\overline{B_+}\cap G_i)(H\cap G_i)$ holds, with $\overline{B_+}\cap\nolinebreak G_i\neq\nolinebreak\{ e\}$, and the corresponding projection is the restriction of~$\nu$.
By the previous paragraph, $\gamma_i\in Z(G_i)$ for all $i\geq 1$.
Therefore $\gamma\in G_c\,G_H$ since $Z(G_i)\subset G_c\,G_H$.
This completes the proof of Lemma~\ref{lem:GcH}.

\subsection*{$\bullet$ Proof of Proposition~\ref{prop:sharp}}

Let $\Gamma$ be a discrete subgroup of~$G$ acting properly discontinuously~$X$.
Consider the set
$$F := \{ \gamma\in\Gamma : d_{\aaa}(\mu(\gamma),\mu(H)) < 1\} .$$
For any $\gamma\in F$ we have $\gamma\cdot\mathcal{C}\cap\mathcal{C}\neq\emptyset$, where $\mathcal{C}$ is the compact subset of $X=G/H$ obtained as the image of $\mu^{-1}([0,1])\subset G$; therefore $F$ is finite.
For $\gamma\in F$, the map $f_{\gamma} : G\rightarrow G$ sending $g\in G$ to $g^{-1}\gamma g$ is real analytic, hence $f_{\gamma}^{-1}(KH)$ is an analytic submanifold of~$G$.
By Lemma~\ref{lem:GcH}, if $\gamma\notin G_c\,G_H$, then $f_{\gamma}^{-1}(KH)$ is strictly contained in~$G$, hence it has positive codimension.
In particular, there is an element $g\in G$ with $\Vert\mu(g)\Vert<1/2$ such that $g^{-1}\gamma g\notin KH$ (\ie $g^{-1}\gamma g\cdot x_0\notin X_c$) for all $\gamma\in F\smallsetminus G_c\,G_H$.
By Lemmas \ref{lem:prmu} and~\ref{lem:munu}, for all $\gamma\in\Gamma\smallsetminus F$,
$$
\Vert\nu(g^{-1}\gamma g)\Vert \geq d_{\aaa}\big(\mu(g^{-1}\gamma g),\mu(H)\big) \geq d_{\aaa}(\mu(\gamma),\mu(H)) - 2\Vert\mu(g)\Vert > 0.$$
In particular, $g^{-1}\gamma g\cdot x_0\notin X_c$ for all $\gamma\in\Gamma\smallsetminus G_c\,G_H$.
This completes the proof of Proposition~\ref{prop:sharp}.

\section{Uniformity for standard Clifford--Klein forms}\label{subsec:uniformnonzero}

In Section~\ref{subsec:proofnonzero}, we shall prove Proposition~\ref{prop:nonzero}.(2) by combining Lemma~\ref{lem:nonzeroprecise} with the following consequence of the Kazhdan--Margulis theorem, applied to some conjugate of~$L$ instead of~$G$.

\begin{proposition}\label{prop:KM}
Assume that the reductive group~$G$ has a compact center.
There is a constant $r_G>0$ (depending only on~$G$) with the following property: for any discrete subgroup~$\Gamma$ of~$G$, there is an element $g\in G$ such that
$$\Vert\mu(g^{-1}\gamma g)\Vert \geq r_G \quad\text{for all }\gamma\in\Gamma\smallsetminus G_c.$$
\end{proposition}

As before, $G_c$ denotes the largest compact normal subgroup of~$G$.
The condition $\gamma\in\Gamma\smallsetminus\nolinebreak G_c$ cannot be improved: if $\gamma\in G_c$, then $\mu(g^{-1}\gamma g)=0$ for all $g\in G$ since $g^{-1}\gamma g\in G_c\subset K$.
The condition that the center $Z(G)$ of~$G$ is compact also cannot be improved: if $\mathrm{Lie}(Z(G))\cap\aaa$ contains a nonzero vector~$Y$, then for any $t\in\R_+$ the cyclic group generated by $\gamma_t:=\exp(tY)\in G\smallsetminus G_c$ is discrete in~$G$ and $\Vert\mu(g^{-1}\gamma_tg)\Vert=t\,\Vert Y\Vert$ for all $g\in G$.

Recall that $\Vert\mu(g)\Vert=d_{G/K}(y_0,g\cdot y_0)$ for all $g\in G$, where $y_0$ is the image of~$K$ in the Riemannian symmetric space $G/K$.
Thus Proposition~\ref{prop:KM} has the following geometric interpretation: there is a constant $r_G>0$ such that any Riemannian locally symmetric space $M=\Gamma\backslash G/K$ locally modeled on $G/K$ admits a point at which the injectivity radius is~$\geq r_G$.

Proposition~\ref{prop:KM} is not new; we give a proof for the reader's convenience.
We begin with an elementary geometric lemma in the Riemannian symmetric space $G/K$, designed to treat groups~$\Gamma$ with torsion.

\begin{lemma}\label{lem:distGK}
For any $g\in G\smallsetminus G_c$ of finite order and any $R,\varepsilon>0$, there exists $r>0$ such that for any ball $B$ of radius~$R$ in $G/K$,
$$\vol_{G/K}\big(\big\{ y\in B :\, d_{G/K}(y,g\!\cdot\!y)<r\big\} \big) < \varepsilon.$$
This $r$ depends only on the conjugacy class of $g$ in~$G$ (and on $R$ and~$\varepsilon$).
\end{lemma}

\begin{proof}
For $g\in G\smallsetminus G_c$ of order $n\geq 2$, let $\mathcal{F}_g$ be the set of fixed points of $g$ in $G/K$.
We claim that the set of points $y\in G/K$ with $d_{G/K}(y,g\!\cdot\!y)<\nolinebreak r$ is contained in an $(n-1)r$-neighborhood of~$\mathcal{F}_g$.
Indeed, for $y\in G/K$, consider the ``center of gravity'' $z$ of the $g$-orbit $\{ y,g\!\cdot\!y,\dots,g^{n-1}\!\cdot\!y\}$, such that $\sum_{i=0}^{n-1} d_{G/K}(z,g^i\!\cdot\!y)^2$ is minimal.
(The existence and uniqueness of such a point were first established by \'E.~Cartan \cite{car26} to prove his fixed point~theorem.)
The point~$z$ belongs to the convex hull of $\{ y,g\!\cdot\!y,\dots,g^{n-1}\!\cdot\!y\}$, hence there exists $1\leq i_0\leq n-1$ such that $d_{G/K}(y,g^{i_0}\cdot y)\geq d_{G/K}(y,z)$.
Moreover, $z\in\mathcal{F}_g$, hence $d_{G/K}(y,z)\geq d_{G/K}(y,\mathcal{F}_g)$.
By the triangular inequality,
$$\!\!\!d_{G/K}(y,g\cdot y) = \frac{1}{i_0}\,\sum_{i=0}^{i_0-1} d_{G/K}(g^i\cdot y,g^{i+1}\cdot y) \geq \frac{1}{i_0}\,d_{G/K}(y,g^{i_0}\cdot y)\ \geq \frac{1}{i_0}\,d_{G/K}(y,\mathcal{F}_g),$$
which proves the claim.
Let $R,\varepsilon>0$.
We note that $\mathcal{F}_g$ is an analytic subvariety of $G/K$ of positive codimension since $g\notin G_c$.
Therefore, for any ball $B'$ of radius $(n+1)R$ centered at a point of~$\mathcal{F}_g$, there exists $r>0$ such that
$$\vol_{G/K}\big(\big\{ y\in B' :\, d_{G/K}(y,g\!\cdot\!y)<r\big\} \big) < \varepsilon.$$
Using the fact that the centralizer of $g$ in~$G$ acts transitively on~$\mathcal{F}_g$ (see \cite[Ch.\,IV, \S\,7]{hel01}), it is easy to see that this $r$ can actually be taken uniformly for all such balls.
We conclude the proof of Lemma~\ref{lem:distGK} by observing that any ball of radius~$R$ meeting the $(n-1)r$-neighborhood of~$\mathcal{F}_g$ is actually contained in a ball of radius $(n+1)R$ centered at a point of~$\mathcal{F}_g$, since $r\geq R$.
The fact that $r$ depends only on the conjugacy class of $g$ in~$G$ (and on $R$ and~$\varepsilon$) follows from the fact that the metric $d_{G/K}$ is $G$-invariant.
\end{proof}

\begin{proof}[Proof of Proposition~\ref{prop:KM}]
We first assume that $G$ is semisimple with no compact factor, so that $G_c=Z(G)$.
The Kazhdan--Margulis theorem (see \cite[Th.\,11.8]{rag72}) then gives the existence of a neighborhood~$\mathcal{W}$ of $e$ in~$G$ with the following property: for any discrete subgroup~$\Gamma$ of~$G$, there is an element $g\in G$ such that $g^{-1}\Gamma g\cap\mathcal{W}=\{ e\}$.
It is enough to prove Proposition~\ref{prop:KM} for discrete groups~$\Gamma$ such that $\Gamma\cap\mathcal{W}=\{ e\}$.

We note that for all $g,\gamma\in G$, we have $d_{G/K}(y_0,g^{-1}\gamma g\cdot y_0)=d_{G/K}(y,\gamma\cdot y)$ where $y:=g\cdot y_0$.
Therefore, using the interpretation \eqref{eqn:mudistance} of $\Vert\mu\Vert$ as a distance in the Riemannian symmetric space $G/K$, it is enough to prove the existence of a constant $r_G>0$ with the following property: for any discrete subgroup $\Gamma$ of~$G$ with $\Gamma\cap\mathcal{W}=\{ e\}$, there is a point $y\in G/K$ such that for any $\gamma\in\Gamma\smallsetminus Z(G)$,
\begin{equation}\label{eqn:dist>r}
d_{G/K}(y,\gamma\cdot y) \geq r_G.
\end{equation}
In order to prove this, we consider a bounded neighborhood $\mathcal{U}$ of $e$ in~$G$ such that $\mathcal{U}\mathcal{U}^{-1}\subset\mathcal{W}$, and an integer~$m$ such that
\begin{equation}\label{eqn:m}
m\cdot\vol_G(\mathcal{U}) \,>\, \vol_G\big(K_1\cdot\mathcal{U}\big),
\end{equation}
where we set
$$K_1 := \big\{ g\in G :\, d_{G/K}(y_0,g\cdot y_0)<1\big\} .$$

\noindent
$\bullet$ We claim that for any \emph{torsion-free} discrete subgroup $\Gamma$ of~$L$ with $\Gamma\cap\nolinebreak\mathcal{W}=\nolinebreak\{ e\}$,
\begin{equation}\label{eqn:disttorsionfree}
\Vert\mu(\gamma)\Vert = d_{G/K}(y_0,\gamma\cdot y_0) \geq \frac{1}{m}.
\end{equation}
Indeed, let $\Gamma$ be such a group.
Then $\gamma\,\mathcal{U}\cap\gamma'\,\mathcal{U}=\emptyset$ for all $\gamma\neq\gamma'$ in~$\Gamma$, hence
$$\vol_G\big(K_1\cdot\mathcal{U}\big) \geq \#\big(\Gamma\cap K_1\big) \cdot \vol_G(\mathcal{U}).$$
Therefore, by \eqref{eqn:m},
$$\#\big(\Gamma\cap K_1\big) < m.$$
Using the fact \eqref{eqn:triangineq} that $\Vert\mu(g^m)\Vert\leq m\,\Vert\mu(g)\Vert$ for all $g\in G$, we obtain that any element $\gamma\in\Gamma$ with $\Vert\mu(\gamma)\Vert<1/m$ has order $<m$; the number of such elements~$\gamma$ is $<m$.
In particular, since $\Gamma$ is torsion-free, the only element $\gamma\in\Gamma$ with $\Vert\mu(\gamma)\Vert<1/m$ is~$e$, proving \eqref{eqn:disttorsionfree}.

\noindent
$\bullet$ We now deal with groups~$\Gamma$ that have torsion.
By Lemma~\ref{lem:distGK}, for any $g\in G\smallsetminus G_c$ of finite order there exists $r\in (0,\frac{1}{3m}]$ such that for any ball $B$ of radius $1/3m$ in $G/K$,
\begin{equation}\label{eqn:vol}
\vol_{G/K}\big(\big\{ y\in B :\, d_{G/K}(y,g\cdot y)<r\big\} \big) < \frac{1}{m}\,\vol_{G/K}(B),
\end{equation}
and this $r$ depends only on the conjugacy class of $g$ in~$G$.
Since there are only finitely many conjugacy classes of elements of order $<m$ in~$G$ \cite[Ch.\,IX, Cor.\,4.4 \& Prop.\,4.6]{hel01}, there exists a constant $r=r_G$ such that \eqref{eqn:vol} holds for all $g\in G\smallsetminus G_c$ of order $<m$ and all balls $B$ of radius $1/3m$.
Let us prove that this constant~$r_G$ satisfies \eqref{eqn:dist>r}.
Let $\Gamma$ be a discrete subgroup of~$G$ such that $\Gamma\cap\mathcal{W}=\{ e\} $.
The same reasoning as before shows that any element $\gamma\in\Gamma$ with $\Vert\mu(\gamma)\Vert<1/m$ has order $<m$; the number of such elements~$\gamma$ is $<m$.
By \eqref{eqn:vol}, there is a point $y\in B_{G/K}(y_0,\frac{1}{3m})$ such that $d_{G/K}(y,\gamma\cdot y)\geq r_G$ for all $\gamma\in\Gamma\smallsetminus G_c$ with $\Vert\mu(\gamma)\Vert<1/m$.
For all $\gamma\in\Gamma$ with $\Vert\mu(\gamma)\Vert=d_{G/K}(y_0,\gamma\cdot y_0)\geq 1/m$, we also have
$$d_{G/K}(y,\gamma\cdot y) \geq d_{G/K}(y_0,\gamma\cdot y_0) - 2\,d_{G/K}(y,y_0) \geq \frac{1}{3m} \geq r_G,$$
which proves \eqref{eqn:dist>r} and completes the proof of Proposition~\ref{prop:KM} in the case when $G$ has no compact factor.

We now consider the general case where $G$ may have compact factors.
Let $\pi : G\rightarrow G/G_c$ be the natural projection.
The group $\pi(G)=G/G_c$ is semisimple with a trivial center and no compact factor.
It admits the Cartan decomposition
$$\pi(G) = \pi(K)\,\pi(\overline{A_+})\,\pi(K).$$
Let $\mu_{\pi(G)} : \pi(G)\rightarrow\log\pi(\overline{A_+})$ be the corresponding Cartan projection.
The restriction of $\pi$ to~$A$ is injective, hence we may identify $\log\pi(\overline{A_+})$ with~$\overline{\aaa_+}$.
With this identification,
$$\mu_{\pi(G)}(\pi(g)) = \mu(g)$$
for all $g\in G$.
Therefore, Proposition~\ref{prop:KM} for~$G$ follows from Proposition~\ref{prop:KM} for~$\pi(G)$, given that for any discrete subgroup~$\Gamma$ of~$G$ the group $\pi(\Gamma)$ is discrete in~$\pi(G)$.
\end{proof}

\begin{remark}\label{rem:KMconn}
If $G$ is disconnected, with finitely many connected components, then it still admits a Cartan decomposition $G=K\overline{A_+}K$, where $K$ is a maximal compact subgroup of~$G$ and $\overline{A_+}$ a positive Weyl chamber in a maximal split torus of~$G$, possibly smaller than the corresponding positive Weyl chamber for the identity component of~$G$.
The corresponding Cartan projection $\mu : G\rightarrow\log\overline{A_+}$ is well-defined and has the property that $\Vert\mu(g)\Vert=d_{G/K}(y_0,g\cdot y_0)$ for all $g\in G$, where $y_0$ denotes the image of $K$ in $G/K$.
Lemma~\ref{lem:distGK} and Proposition~\ref{prop:KM} hold with the same proof.
\end{remark}

\section{Proof of Proposition~8.1}\label{subsec:proofnonzero}

Recall from \eqref{eqn:explicitRX} that we may take $R_X$ to be $4\Vert\rho_a\Vert/q$ in Proposition~\ref{prop:Vlambda}.
For any subgroup $\Gamma$ of~$G$ acting properly discontinuously on~$X$, we set
$$r_{\Gamma} := \inf\big\{ \Vert\nu(x)\Vert : x\in\Gamma\!\cdot\!x_0\text{ and }x\notin X_c\big\} > 0$$
(see Remark~\ref{rem:infnu}).

We first consider Proposition~\ref{prop:nonzero}.(1).
Let $X_{\Gamma}$ be a sharp Clifford--Klein form of~$X$ with $\Gamma\cap G_c\subset Z(G_s)$.
If $\Gamma\cdot x_0\cap X_c\subset Z(G_s)\cdot x_0$, then, by Lemma~\ref{lem:nonzeroprecise} and Remark~\ref{rem:dlambdarho}, the operator $S_{\Gamma}$ is well-defined
and nonzero on~$\V_{Z,\lambda}$ for any $\lambda\in\jj^{\ast}_+\cap (2\rho_c-\rho+\Lambda^{\Gamma\cap Z(G_s)})$ with $d(\lambda)$ larger than
\begin{equation}\label{eqn:finalcst}
\max\Big(m_{\rho}\,,\,\frac{4\Vert\rho_a\Vert}{qc}\,,\,\frac{4\Vert\rho_a\Vert (r_{\Gamma}+C) + \log\big(2c_G\,\#(\Gamma\cap K)\big)}{c\,\log\cosh(q'r_{\Gamma})}\Big).
\end{equation}
Otherwise, we use Proposition~\ref{prop:sharp}, Remark~\ref{rem:GcGH}, and the assumptions that $H$ does not contain any simple factor of~$G$ and $\Gamma\cap G_c\subset Z(G_s)$ to obtain the existence of an element $g\in G$ such that $g^{-1}\Gamma g\cdot\nolinebreak x_0\cap\nolinebreak X_c\subset Z(G_s)\cdot x_0$; then $S_{g^{-1}\Gamma g}$ is well-defined
and nonzero on~$\V_{Z,\lambda}$ for any $\lambda\in\jj^{\ast}_+\cap (2\rho_c-\rho+\Lambda^{g^{-1}\Gamma g\cap Z(G_s)})$ with $d(\lambda)$ larger than
\begin{equation}\label{eqn:finalcstconj}
\max\Big(m_{\rho}\,,\,\frac{4\Vert\rho_a\Vert}{qc}\,,\,\frac{4\Vert\rho_a\Vert (r_{g^{-1}\Gamma g}+C) + \log\big(2c_G\,\#(g^{-1}\Gamma g\cap K)\big)}{c\,\log\cosh(q'r_{g^{-1}\Gamma g})}\Big).
\end{equation}
By Remark~\ref{rem:translconj} (and the fact that $g^{-1}\Gamma g\cap Z(G_s)=\Gamma\cap Z(G_s)$), the operator $S_{\Gamma}$ is well-defined and nonzero on $g\cdot\V_{Z,\lambda}$ for any $\lambda\in\jj^{\ast}_+\cap (2\rho_c-\rho+\Lambda^{\Gamma\cap Z(G_s)})$ satisfying \eqref{eqn:finalcstconj}.
This concludes the proof of Proposition~\ref{prop:nonzero}.(1).

We now consider Proposition~\ref{prop:nonzero}.(2).
Let $L$ be a reductive subgroup of~$G$ acting properly on~$X$.
Assume that the center of~$L$ is compact.
There is a conjugate $L'$ of $L$ in~$G$ that is stable under the Cartan involution~$\theta$; in particular, $L'$ is $(c,0)$-sharp for some $c>0$ (Example~\ref{ex:standardthetastable}).
By Remark~\ref{rem:translconj}, it is sufficient to prove Proposition~\ref{prop:nonzero}.(2) for~$L'$.
Let $L'_c$ be the maximal compact normal subgroup of~$L'$.
Applying Proposition~\ref{prop:KM} to $L'$ instead of~$G$, we obtain the existence of a constant $r_{L'}>0$ (depending only on~$L'$) such that any discrete subgroup $\Gamma$ of~$L'$ admits a conjugate $g^{-1}\Gamma g$, $g\in\nolinebreak L'$, with $\Vert\mu(g^{-1}\gamma g)\Vert\geq r_{L'}$ for all $\gamma\in\Gamma\smallsetminus L'_c$.
The reason why we apply Proposition~\ref{prop:KM} to $L'$ and not~$G$ is that in this way the group $g^{-1}\Gamma g\subset L'$ remains $(c,0)$-sharp.
Lemma~\ref{lem:munu} then yields $\Vert\nu(g^{-1}\gamma g)\Vert\geq c\,r_{L'}$ for all $\gamma\in\Gamma\smallsetminus L'_c$.
In particular, $g^{-1}\gamma g\cdot x_0\notin X_c$ for all $\gamma\in\Gamma\smallsetminus L'_c$ and $r_{\Gamma}\geq c\,r_{L'}$.
By Remark~\ref{rem:GcGH} and the assumptions that $H$ does not contain any simple factor of~$G$ and $\Gamma\cap L'_c\subset Z(G_s)$, we have $g^{-1}\Gamma g\cap K\subset Z(G_s)$ and $g^{-1}\Gamma g\!\cdot\!x_0\cap X_c\subset Z(G_s)\!\cdot\!x_0$, which enables us to apply Lemma~\ref{lem:nonzeroprecise}.
Using Remark~\ref{rem:dlambdarho}, we obtain that the operator~$S_{g^{-1}\Gamma g}$ is well-defined
and nonzero on~$\V_{Z,\lambda}$ for any $\lambda\in\jj^{\ast}_+\cap (2\rho_c-\rho+\Lambda^{g^{-1}\Gamma g\cap Z(G_s)})$ with $d(\lambda)$ larger than
\begin{equation}\label{eqn:finalcstL}
R := \max\Big(m_{\rho}\,,\,\frac{4\Vert\rho_a\Vert}{qc}\,,\,\frac{4\Vert\rho_a\Vert c\,r_{L'} + \log\big(2c_G\,\#Z(G_s)\big)}{c\,\log\cosh(q'c\,r_{L'})}\Big).
\end{equation}
Proposition~\ref{prop:nonzero}.(2) follows, using Remark~\ref{rem:translconj}.

We now consider Proposition~\ref{prop:nonzero}.(3).
Let $L$ be a reductive subgroup of~$G$ of real rank~$1$ and let $\Gamma$ be a convex cocompact subgroup of~$L$ with $\Gamma\cap G_c\subset Z(G_s)$.
By Proposition~\ref{prop:sharp}, Remark~\ref{rem:GcGH}, and the assumptions that $H$ does not contain any simple factor of~$G$ and $\Gamma\cap G_c\subset Z(G_s)$, there is an element $g\in G$ such that $g^{-1}\gamma g\cdot x_0\notin X_c$ for all $\gamma\in\Gamma\cap Z(G_s)$.
By Proposition~\ref{prop:sharpnessproperties}, the group $g^{-1}\Gamma g$ is $(c,C)$-sharp for some $c,C>0$ (where $c$ depends only on~$L$).
Choose $\varepsilon\in (0,r_{g^{-1}\Gamma g})$.
By Lemma~\ref{lem:munudeform} applied to $g^{-1}\Gamma g\subset g^{-1}Lg$ instead of $\Gamma\subset L$, there is a neighborhood $\mathcal{U}'\subset\Hom(\Gamma,G)$ of the natural inclusion such that for all $\varphi\in\mathcal{U}'$, the group $g^{-1}\varphi(\Gamma)g$ is discrete in~$G$ and $(c-\varepsilon,C+\varepsilon)$-sharp for~$X$, and satisfies $\Vert\nu(g^{-1}\varphi(\gamma)g)\Vert\geq\nolinebreak r_{g^{-1}\Gamma g}-\nolinebreak\varepsilon$ for all $\gamma\in\Gamma\smallsetminus\nolinebreak Z(G_s)$.
We now use the following fact, which holds because there are only finitely many conjugacy classes of elements of order $\leq\# Z(G_s)$ in~$G$ \cite[Ch.\,IX, Cor.\,4.4 \& Prop.\,4.6]{hel01} and they are all closed \cite[Th.\,9.2]{bor91}.

\begin{remark}\label{rem:deformcenter}
There is a neighborhood $\mathcal{U}\subset\mathcal{U}'\subset\Hom(\Gamma,G)$ of the natural inclusion such that $\varphi(\Gamma\cap Z(G_s))\subset Z(G_s)$ for all $\varphi\in\mathcal{U}$.
\end{remark}

\noindent
By Remark~\ref{rem:deformcenter}, we have $g^{-1}\varphi(\Gamma)g\cdot x_0\cap X_c\subset Z(G_s)\cdot x_0$ and $r_{g^{-1}\varphi(\Gamma)g}\geq r_{g^{-1}\Gamma g}-\varepsilon$, as well as $g^{-1}\varphi(\Gamma)g\cap K\subset Z(G_s)$; we can apply Lemma~\ref{lem:nonzeroprecise}.
Using Remark~\ref{rem:dlambdarho}, we obtain that for all $\varphi\in\mathcal{U}$, the operator~$S_{g^{-1}\varphi(\Gamma)g}$ is well-defined and nonzero on~$\V_{Z,\lambda}$ for any $\lambda\in\jj^{\ast}_+\cap (2\rho_c-\rho+\Lambda^{g^{-1}\varphi(\Gamma)g\cap Z(G_s)})$ with $d(\lambda)$ larger than
$$R := \max\Big(m_{\rho}\,,\,\frac{4\Vert\rho_a\Vert}{qc}\,,\,\frac{4\Vert\rho_a\Vert (r+C) + \log\big(2c_G\,\#Z(G_s)\big)}{c\,\log\cosh(q'(r-\varepsilon))}\Big).$$
Proposition~\ref{prop:nonzero}.(3) follows, using Remark~\ref{rem:translconj}.
If $\Gamma\cap L_c\subset Z(G_s)$, then we can conjugate~$\Gamma$ as in the proof of Proposition~\ref{prop:nonzero}.(2) and take $r=c\,r_{L'}$ and $C=\nolinebreak 0$.
Since the function~$d$ takes discrete values on $\jj^{\ast}_+\cap (2\rho_c-\rho+\Lambda)$, by choosing $\varepsilon$ small enough we see that we can take the same~$R$ as in Proposition~\ref{prop:nonzero}.(2).
This completes the proof.

\section{Proof of the results of Chapters 1 to~3}\label{subsec:proofs}

The bulk of the paper was the proof of Proposition~\ref{prop:nonzero}; now we briefly explain how the results of Chapters \ref{sec:intro} to~\ref{sec:theorems} follow.

Theorem~\ref{thm:precise}.(1) follows immediately from Proposition~\ref{prop:nonzero}.(1); Theorem \ref{thm:precise}.(2) from Proposition~\ref{prop:nonzero}.(2); Theorem~\ref{thm:precisedeform} from Proposition~\ref{prop:nonzero}.(3); Theorem~\ref{thm:regular} from Proposition~\ref{prop:Vlambda}.
In the case when ${}^{\backprime}G$ is connected with no compact factor, Propositions \ref{prop:deformexoticcompact} and~\ref{prop:deformexotic} follow from Lemmas \ref{lem:muexotic} and~\ref{lem:nonzeroprecise} as in the proof of Proposition~\ref{prop:nonzero}.(3) (see Section~\ref{subsec:proofnonzero}).

In order to deduce Theorems \ref{thm:universal}, \ref{thm:deform}, and~\ref{thm:infinitespec} from Theorems \ref{thm:precise} and~\ref{thm:precisedeform}, and to prove Propositions \ref{prop:deformexoticcompact} and~\ref{prop:deformexotic} in the general case, it is sufficient to deal with the following three issues:
\begin{itemize}
  \item $G$ may be disconnected,
  \item some simple factors of~$G$ may be contained in~$H$,
  \item $G$ may have compact factors.
\end{itemize}
Indeed, when $G$ has no compact factor, the condition $\Gamma\cap G_c\subset Z(G_s)$ of Theorems \ref{thm:precise} and~\ref{thm:precisedeform} is automatically satisfied (see Remark~\ref{rem:conditionsGL}.(a)).
The first issue is easily dealt with: if $G_0$ denotes the identity component of~$G$, then $G_0/(G_0 \cap H)$ is a connected component of~$X$, so $\Spec_d(G_0/H)$ is a subset of $\Spec_d(X)$ (extend eigenfunctions by~$0$ on the other connected components).
In order to deal with the second and third issues, we consider the group $\overline{G}:=G/G_cG_H$, where $G_H$ is the maximal normal subgroup of~$G$ contained in~$H$ (see Section~\ref{subsec:minnu}).
We note that $\overline{G}$ is reductive with no compact factor and that none of its simple factors is contained in $\overline{H}:=H/G_cG_H\cap H$, hence Theorems \ref{thm:precise} and~\ref{thm:precisedeform} apply to the reductive symmetric space $\overline{X}:=\overline{G}/\overline{H}$.
To relate $X$ to~$\overline{X}$, we make the following elementary observation.

\begin{observation}
The natural projection $\pi : X\rightarrow\overline{X}$ induces homomorphisms
\begin{itemize}
  \item $C^{\infty}(\overline{X}) \overset{\pi^{\ast}}{\longhookrightarrow} C^{\infty}(X)$,
  \item $\D(X) \overset{\pi_{\ast}}{\longtwoheadrightarrow} \D(\overline{X})$,
  \item $\Hom_{\C\text{-alg}}(\D(\overline{X}),\C) \overset{\pi^{\ast}}{\longhookrightarrow} \Hom_{\C\text{-alg}}(\D(X),\C)$
\end{itemize}
such that for all $D\in\D(X)$, $f\in C^{\infty}(\overline{X})$, and $\chi\in\Hom_{\C\text{-alg}}(\D(\overline{X}),\C)$,
$$(\pi_{\ast}D)f = \chi(\pi_{\ast}D)f \quad\Longleftrightarrow\quad D(\pi^{\ast}f) = (\pi^{\ast}\chi)(D)\,\pi^{\ast}f.$$
Moreover, $\pi^{\ast}(L^2(\overline{X}))\subset L^2(X)$, hence
$$\pi^{\ast}\big(\Spec_d(\overline{X})\big) \subset \Spec_d(X).$$
\end{observation}

\smallskip

Let us now consider Clifford--Klein forms.
We note that if $\Gamma$ is a discrete subgroup of~$G$ acting properly discontinuously and freely on~$X$, then the image $\overline{\Gamma}$ of~$\Gamma$ in~$\overline{G}$ is discrete and acts properly discontinuously on~$\overline{X}$, but not necessarily freely.
However, in all the previous chapters we could actually drop the assumption that $\Gamma$ acts freely, allowing $X_{\Gamma}$ to be an \emph{orbifold} (or \emph{$V$-manifold} in the sense of Satake) instead of a manifold.
Indeed, let us define $L^2(X_{\Gamma})$ to be the set of $\Gamma$-invariant functions on~$X$ that are square-integrable on some fundamental domain for the action of~$\Gamma$.
If $C^{\infty}_c(X_{\Gamma})$ denotes the space of $\Gamma$-invariant smooth functions on~$X$ with compact support \emph{modulo}~$\Gamma$, then any $D\in\D(X)$ leaves $C^{\infty}_c(X_{\Gamma})$ invariant, so that for $\chi_{\lambda} : \D(X)\rightarrow\C$ we can define the notion of weak solution $f\in L^2(X_{\Gamma})$ to the system
$$Df = \chi_{\lambda}(D)f \quad\quad\mathrm{for\ all}\ D\in\D(X) \eqno{(\M_{\lambda})}$$
with respect to integration against elements of $C^{\infty}_c(X_{\Gamma})$.
We can then define $\Spec_d(X_{\Gamma})$ to be the set of $\C$-algebra homomorphisms $\chi_{\lambda} : \D(X)\rightarrow\C$ for which the system $(\M_{\lambda})$ admits a nonzero weak solution $f\in L^2(X_{\Gamma})$.
Since our construction of joint eigenfunctions is obtained by the summation operator~$S_{\Gamma}$, Propositions \ref{prop:Vlambda} and~\ref{prop:nonzero}, as well as Theorems \ref{thm:precise} and~\ref{thm:precisedeform}, hold in this more general setting.
We conclude the proof of Theorems \ref{thm:universal}, \ref{thm:deform}, and~\ref{thm:infinitespec} and Propositions \ref{prop:deformexoticcompact} and~\ref{prop:deformexotic} with the following observation.

\begin{observation}
\begin{enumerate}
  \item The rank condition \eqref{eqn:rank} for $X=G/H$ holds if and only if that for $\overline{X}=\overline{G}/\overline{H}$ holds.  
  \item For any discrete subgroup~$\Gamma$ of~$G$ acting properly discontinuously on~$X$, the image~$\overline{\Gamma}$ of~$\Gamma$ in~$\overline{G}$ is discrete and acts properly discontinuously on~$\overline{X}$.
  \item The projection $\pi : X\rightarrow\overline{X}$ induces 
$\pi^{\ast}(L^2(\overline{X}_{\overline{\Gamma}}))\subset L^2(X_{\Gamma})$,
 hence
 $$\pi^{\ast}\big(\Spec_d(\overline{X}_{\overline{\Gamma}})\big) \subset \Spec_d(X_{\Gamma}).$$
\end{enumerate}
\end{observation}

%% file: StableSpec9-AdS.tex
\chapter{Three-dimensional anti-de Sitter manifolds}\label{sec:exAdS}

In this chapter and the following one, we concentrate on a few examples to illustrate our general theory.
We first examine the case of the $3$-dimensional anti-de Sitter space $X=\AdS^3=\SO(2,2)_0/\SO(1,2)_0$.
Our purpose is $3$-fold:
\begin{itemize}
  \item recall the description of the Clifford--Klein forms of $\AdS^3$ in terms of representations of surface groups, as developed since the 1980's (Sections \ref{subsec:CKformsAdS3} to~\ref{subsec:CLip}); 
  \item use it to give an explicit infinite subset of the discrete spectrum of the Laplacian on any Clifford--Klein form $\Gamma\backslash\AdS^3$ with $\Gamma$ finitely generated, in terms of some geometric constant $C_{Lip}(\Gamma)$ (Section~\ref{subsec:discspecAdS3});
  \item understand the analytic estimates developed in Chapters~\ref{sec:Vlambda} and~\ref{sec:FJ} through concrete harmonic analysis computations on the group $\SL_2(\R)$ (Sections \ref{subsec:FJAdS3} to~\ref{subsec:proofclaimAdS3}).
\end{itemize}

As mentioned in the introduction, $X=\AdS^3$ is a Lorentzian analogue of the real hyperbolic space $\HH^3=\SO(1,3)_0/\SO(3)$: it is a model space for all Lorentzian $3$-manifolds of constant negative curvature, or \emph{anti-de Sitter} $3$-manifolds.
One way to see $X$ is as the quadric of equation $Q=1$ in~$\R^4$ with the Lorentzian metric induced by~$-Q$, where
\begin{equation}\label{eqn:defQ}
Q(x) = x_1^2 + x_2^2 - x_3^2 - x_4^2\,;
\end{equation}
the sectional curvature of~$X$ is then~$-1$ (see \cite{wol11}).
Another way to see~$X$ is as the manifold $\SL_2(\R)$, with the Lorentzian structure induced by $1/8$ times the Killing form of $\ssl_2(\R)$ and the transitive action (by isometries) of the group
$$G := \SL_2(\R)\times\SL_2(\R)$$
by left and right multiplication:
\begin{equation}\label{eqn:lrmul}
(g_1,g_2)\cdot g = g_1gg_2^{-1}.
\end{equation}
We will use both realizations of~$X$.
An explicit correspondence is given by
\begin{eqnarray}\label{eqn:correspAdS3}
\{ x\in\R^4 : Q(x)=1\} & \overset{\sim}{\longrightarrow} & \quad\quad\quad\SL_2(\R)\quad\quad\quad\quad.\nonumber\\
x\quad\quad\quad & \longmapsto & \begin{pmatrix} x_1+x_4 & -x_2+x_3\\ x_2+x_3 & x_1-x_4\end{pmatrix}
\end{eqnarray}
The stabilizer in~$G$ of the identity element $1\in\SL_2(\R)$ is the diagonal $H:=\Diag(\SL_2(\R))$, which is the set of fixed points of~$G$ under the involution $\sigma : (g_1,g_2)\mapsto (g_2,g_1)$.
Thus $X=\SO(2,2)_0/\SO(1,2)_0$ identifies with
$$G/H = (\SL_2(\R)\times\SL_2(\R))/\Diag(\SL_2(\R)).$$
We note that the action of $G$ on~$X$ factors through $G/\{ \pm(1,1)\}\simeq\SO(2,2)_0$; we have $H/\{ \pm(1,1)\}\simeq\SO(1,2)_0$.
By \cite{kli96} and~\cite{kr85}, all compact anti-de Sitter $3$-manifolds are Clifford--Klein forms $X_{\Gamma}=\Gamma\backslash X$ of~$X$, up to finite covering.
We now recall how these Clifford--Klein forms (compact or not) can be described in terms of representations of surface groups.

\section{Description of the Clifford--Klein forms of $\AdS^3$}\label{subsec:CKformsAdS3}

As in Section~\ref{subsec:introex}, let $-\mathrm{I}\in\SO(2,2)_0$ be the diagonal matrix with all entries equal to~$-1$; it identifies with $\overline{(1,-1)}\in G/\{ \pm(1,1)\}$ and acts on $X=\AdS^3$ by $x\mapsto -x$.
Describing the Clifford--Klein forms of~$X$ reduces to describing those of its quotient of order two
\begin{eqnarray*}
\overline{X} & := & \SO(2,2)_0/\big(\SO(1,2)_0\times\{ \pm\mathrm{I}\} \big)\\
& \ \simeq & \big(\PSL_2(\R)\times\PSL_2(\R)\big)/\Diag(\PSL_2(\R)).
\end{eqnarray*}
If $\Gamma$ is a discrete subgroup of~$G$ acting properly discontinuously and freely on~$X$, then its projection $\overline{\Gamma}$ to $\PSL_2(\R)\times\PSL_2(\R)$ acts properly discontinuously and freely on~$\overline{X}$; the natural projection $X_{\Gamma}\rightarrow\overline{X}_{\overline{\Gamma}}$ between Clifford--Klein forms is an isomorphism if $-\mathrm{I}$ belongs to the image of $\Gamma$ in $\SO(2,2)_0$, and a double covering otherwise.

A fundamental result of Kulkarni--Raymond \cite{kr85} states that if a torsion-free discrete subgroup~$\overline{\Gamma}$ of $\PSL_2(\R)\times\PSL_2(\R)$ acts properly discontinuously on~$\overline{X}$, then it is of the form
\begin{equation}\label{eqn:Gammabar}
\overline{\Gamma} = \{ (j(\boldsymbol\gamma),\rho(\boldsymbol\gamma)) : \boldsymbol\gamma\in\pi_1(S)\} ,
\end{equation}
where $S$ is a hyperbolic surface and $j,\rho\in\Hom(\pi_1(S),\PSL_2(\R))$ are two representations of the surface group $\pi_1(S)$, with one of them Fuchsian, \ie injective and discrete.
The Clifford--Klein form $\overline{X}_{\overline{\Gamma}}={\overline{\Gamma}}\backslash\overline{X}$ is compact if and only if $S$ is.
Pairs $(j,\rho)\in\Hom(\pi_1(S),\PSL_2(\R))^2$ such that the group $(j,\rho)(\pi_1(S))$ acts properly discontinuously on~$\overline{X}$ are said to be \emph{admissible} (terminology of \cite{salPhD}).
We note that not all pairs $(j,\rho)$ are admissible: for instance, if $j$ and~$\rho$ are conjugate, then the infinite group $(j,\rho)(\pi_1(S))$ does not act properly discontinuously on~$\overline{X}$ since it fixes a point.
The question is to determine which pairs are admissible.

Easy examples of admissible pairs are obtained by taking $j$ Fuchsian and $\rho$ constant, or more generally $\rho$ with values in a compact subgroup of $\PSL_2(\R)$: the group $\overline{\Gamma}:=(j,\rho)(\pi_1(S))$ and the Clifford--Klein form $\overline{X}_{\overline{\Gamma}}=\overline{\Gamma}\backslash\overline{X}$ are then \emph{standard} in the sense of Definition~\ref{def:standard}.
When $\rho$ is constant, $\overline{X}_{\overline{\Gamma}}$ identifies with ${}^{\backprime}\Gamma\backslash{}^{\backprime}G$, where ${}^{\backprime}G=\PSL_2(\R)$ and ${}^{\backprime}\Gamma=j(\pi_1(S))$ is a discrete subgroup of~$\,{}^{\backprime}G$; in other words, it is the unit tangent bundle to the hyperbolic surface ${}^{\backprime}\Gamma\backslash\HH^2$ (where $\HH^2$ denotes the hyperbolic plane).
The first \emph{nonstandard} examples of compact anti-de Sitter $3$-manifolds were obtained by deforming standard ones, \ie proving that for fixed Fuchsian~$j$, the pair $(j,\rho)$ is admissible for any $\rho$ close enough to the constant homomorphism: this was done by Goldman \cite{gol85} when $\rho(\pi_1(S))$ is abelian, then by \cite{kob98} in general.
Salein \cite{sal00} constructed the first nonstandard compact Clifford--Klein forms that are \emph{not} deformations of standard ones.
It is also easy to construct nonstandard Clifford--Klein forms~$\overline{X}_{\overline{\Gamma}}$ that are not compact but \emph{convex cocompact}, in the following sense.
We refer to \cite[Ch.\,5]{kasPhD} and \cite{gk12} for more details.

\begin{definition}\label{def:ccAdS3}
A Clifford--Klein form $\overline{X}_{\overline{\Gamma}}$ is \emph{convex cocompact} if, up to finite index and switching the two factors, $\overline{\Gamma}$ is of the form \eqref{eqn:Gammabar} with $j$ injective and $j(\pi_1(S))$ convex cocompact in $\PSL_2(\R)$ in the sense of Section~\ref{subsec:introstand}.
\end{definition}

This terminology is justified by the fact that the convex cocompact Clifford--Klein forms of~$\overline{X}$ are circle bundles over convex cocompact hyperbolic surfaces, up to a finite covering \cite{dgk12}.
We shall say that a Clifford--Klein form $X_{\Gamma}$ of $X=\AdS^3$ is convex cocompact if its projection $\overline{X}_{\overline{\Gamma}}$ is.

By the Selberg lemma \cite[Lem.\,8]{sel60}, any finitely generated subgroup $\overline{\Gamma}$ of $\PSL_2(\R)\times\PSL_2(\R)$ acting properly discontinuously on~$\overline{X}$ has a finite-index subgroup
that is torsion-free, hence of the form \eqref{eqn:Gammabar}.
However, in order to obtain estimates on the discrete spectrum of $\overline{X}_{\overline{\Gamma}}$ itself and not only of a finite covering, we need to understand the precise structure of $\overline{\Gamma}$ itself.
We shall use the following result, whose proof is based on \cite{kr85}.

\addtocounter{equation}{-1}

\begin{lemma}\label{lem:torsionAdS3}
Let $\overline{\Gamma}$ be a finitely generated discrete subgroup of $\PSL_2(\R)\times\nolinebreak\PSL_2(\R)$ (possibly with torsion) acting properly discontinuously on~$\overline{X}$.
Then either $\Gamma$ is standard (\ie $\Gamma$ or $\sigma(\Gamma)$ is contained in a conjugate of $\PSL_2(\R)\times\PSO(2)$) or $\overline{\Gamma}$ is of the form
\begin{equation}\label{eqn:Gammajrho}
\overline{\Gamma} = \{ (j(\boldsymbol\gamma),\rho(\boldsymbol\!\gamma)) : \boldsymbol\gamma\in\pi_1(S)\} ,
\end{equation}
where $S$ is a $2$-dimensional hyperbolic orbifold, $\pi_1(S)$ is the orbifold fundamental group of~$S$, and $(j,\rho)\in\Hom(\pi_1(S),\PSL_2(\R))^2$, with $j$ or~$\rho$ Fuchsian.
\end{lemma}

Recall that a $2$-dimensional hyperbolic \emph{orbifold}~$S$ is a hyperbolic surface with finitely many cone singularities, whose stabilizers are finite groups; the orbifold fundamental group $\pi_1(S)$ is torsion-free if and only if $S$ is an actual hyperbolic surface.
The point of Lemma~\ref{lem:torsionAdS3} is that in the nonstandard case, even if $\overline{\Gamma}$ has torsion, one of its projections to $\PSL_2(\R)$ is still discrete and \emph{injective} (not only with a finite kernel).

\begin{proof}[Proof of Lemma~\ref{lem:torsionAdS3}]
For $i\in\{ 1,2\}$, consider the restriction to~$\overline{\Gamma}$ of the $i$-th projection $\mathrm{pr}_i : \PSL_2(\R)\times\PSL_2(\R)\rightarrow\PSL_2(\R)$.
The kernels $\mathrm{Ker}(\mathrm{pr}_1|_{\overline{\Gamma}})$ and $\mathrm{Ker}(\mathrm{pr}_2|_{\overline{\Gamma}})$ are discrete.
They cannot both be infinite since $\overline{\Gamma}$ acts properly discontinuously on~$\overline{X}$ \cite[\S\,5]{kr85}.
Therefore, after possibly conjugating and replacing $\overline{\Gamma}$ by~$\sigma(\overline{\Gamma})$, we may assume that $\mathrm{Ker}(\mathrm{pr}_1|_{\overline{\Gamma}})$ is finite and contained in $\{1\} \times\mathrm{PSO}(2)$.
If $\mathrm{Ker}(\mathrm{pr}_1|_{\overline{\Gamma}})=\{ 1\}$, then $\overline{\Gamma}$ is of the form \eqref{eqn:Gammajrho} with $j$ injective, and $j$ is in fact discrete \cite[\S\,5]{kr85}.
If $\mathrm{Ker}(\mathrm{pr}_1|_{\overline{\Gamma}})\neq\{ 1\}$, then~it~is~easy to see that $\overline{\Gamma}$ is contained in $\PSL_2(\R)\times\mathrm{PSO}(2)$ since it normalizes $\mathrm{Ker}(\mathrm{pr}_1|_{\overline{\Gamma}})$.
\end{proof}

\section{Deformation of convex cocompact Clifford--Klein forms of $\AdS^3$}\label{subsec:deformAdS3}

The fact that the group $\PSL_2(\R)\times\PSL_2(\R)$ is not simple allows for a rich deformation theory.

For instance, for any compact hyperbolic surface~$S$, the set of admissible pairs $(j,\rho)$ is open in $\Hom(\pi_1(S),\PSL_2(\R))^2$; the deformation space (\emph{modulo} conjugation) thus has dimension $12g-12$, where $g$ is the genus of~$S$.
In other words, for any compact Clifford--Klein form $X_{\Gamma}$ of $X=\AdS^3=G/H$, the group~$\varphi(\Gamma)$ is discrete in~$G$ and acts properly discontinuously and cocompactly on~$X$ for all $\varphi\in\Hom(\Gamma,G)$ in some neighborhood of the natural inclusion of $\Gamma$ in~$G$.
Indeed, this follows from the completeness of compact anti-de Sitter manifolds \cite{kli96} and from the Ehresmann--Thurston principle on the holonomy of geometric structures on compact manifolds (see \cite[\S\,4.2.1]{salPhD}); a quantitative proof was also given in \cite{kob98}.

More generally, proper discontinuity is preserved under small deformations for any convex cocompact Clifford--Klein form of~$X$ (in the sense of Definition~\ref{def:ccAdS3}) \cite[Cor.\,5.1.6]{kasPhD}, as a consequence of the following two facts (the first one extending Example~\ref{ex:deform}).

\begin{fact}[{\cite[Th.\,5.1.1]{kasPhD}}]\label{fact:AdS}
All convex cocompact Clifford--Klein forms of $X=\AdS^3$ are sharp.
\end{fact}

\begin{fact}[{\cite[\S\,5.7.2]{kasPhD}}]\label{fact:AdSdeform}
Let $X_{\Gamma}$ be a $(c,C)$-sharp, convex cocompact Clifford--Klein form of $X=\AdS^3=G/H$.
For any $\varepsilon>0$, there is a neighborhood $\mathcal{U}_{\varepsilon}\subset\Hom(\Gamma,G)$ of the natural inclusion such that the group $\varphi(\Gamma)$ is discrete in~$G$ and $(c-\varepsilon,C+\varepsilon)$-sharp for all $\varphi\in\mathcal{U}_{\varepsilon}$.
\end{fact}

(We refer to Definition~\ref{def:sharp} for the notion of sharpness.)

Facts \ref{fact:AdS} and~\ref{fact:AdSdeform} give the geometric estimates that we need (together with the analytic estimates of Section~\ref{subsec:FJAdS3} below) to construct an infinite stable discrete spectrum for the convex cocompact Clifford--Klein forms of $X=\AdS^3$ (Corollary~\ref{cor:AdS3deform}).
By \cite{gk12}, sharpness actually holds for all Clifford--Klein forms $X_{\Gamma}$ of~$X$ with $\Gamma$ finitely generated, which implies that the discrete spectrum is infinite for all such~$X_{\Gamma}$ (Theorem~\ref{thm:SL2precise}).

\section{The constant $C_{Lip}(\Gamma)$}\label{subsec:CLip}

The infinite subset of the spectrum that we shall give in Section~\ref{subsec:discspecAdS3} will be expressed in terms of a geometric constant $C_{Lip}(\Gamma)$.
The goal of this section is to introduce $C_{Lip}(\Gamma)$, to explain how sharpness is determined by this constant, and to provide some explanation of Facts \ref{fact:AdS} and~\ref{fact:AdSdeform}.

\subsection*{$\bullet$ A reformulation of sharpness for $X=\AdS^3$}

Let $\mu_{\PSL_2(\R)}\!: \PSL_2(\R)\!\rightarrow\nolinebreak\R_{\geq 0}$ be the Cartan projection mapping any element~$g$ to the logarithm of the highest eigenvalue of $^t\!gg$.
We will use the following geometric interpretation:
\begin{equation}\label{eqn:mudistAdS3}
\mu_{\PSL_2(\R)}(g) = d_{\HH^2}(y_0,g\cdot y_0),
\end{equation}
where $y_0$ is the point of~$\HH^2$ whose stabilizer is $\mathrm{PSO}(2)$.
Consider a $2$-dimensional hyperbolic orbifold~$S$ and a pair $(j,\rho)\in\Hom(\pi_1(S),\PSL_2(\R))^2$.
By \cite[Th.\,1.3]{kas08}, if the group $(j,\rho)(\pi_1(S))$ acts properly discontinuously on $X=\AdS^3$, then the set of points
$$\big(\,\mu_{\PSL_2(\R)}(j(\boldsymbol\gamma))\,,\,\mu_{\PSL_2(\R)}(\rho(\boldsymbol\gamma))\,\big) \,\in \R^2$$
for $\boldsymbol\gamma\in\pi_1(S)$ lies on one side only of the diagonal of~$\R^2$, up to a finite number of points.
Therefore, the group $\overline{\Gamma}:=(j,\rho)(\pi_1(S))$ is sharp for~$\overline{X}$ if and only if, up to switching $j$ and~$\rho$, there exist constants $c'<1$ and $C'\geq 0$ such that
$$\mu_{\PSL_2(\R)}(\rho(\boldsymbol\gamma)) \leq c'\,\mu_{\PSL_2(\R)}(j(\boldsymbol\gamma)) + C'$$
for all $\boldsymbol\gamma\in\pi_1(S)$; in this case, $\overline{\Gamma}$ is $(c,C)$-sharp for
\begin{equation}\label{eqn:cCAdS3}
c := \sin\Big(\frac{\pi}{4} - \arctan(c')\Big) = \frac{(1-c')}{\sqrt{2(1+{c'}^2)}} \quad\ \text{and}\ \quad C := \frac{C'}{\sqrt{2}}
\end{equation}
and $j$ is Fuchsian.

\subsection*{$\bullet$ The constants $C_{Lip}(j,\rho)$ and $C_{Lip}(\Gamma)$}

We denote by $C_{Lip}(j,\rho)$ the infimum of Lipschitz constants
$$\mathrm{Lip}(f) = \sup_{y\neq y'\text{ in }\HH^2}\ \frac{d_{\HH^2}(f(y),f(y'))}{d_{\HH^2}(y,y')}$$
of maps $f : \HH^2\rightarrow\HH^2$ that are $(j,\rho)$-equivariant, \textit{i.e.}\ that satisfy $f\big(j(\boldsymbol\gamma)\cdot y\big)=\rho(\boldsymbol\gamma)\cdot f(y)$ for all $\boldsymbol\gamma\in\pi_1(S)$ and $y\in\HH^2$.
By the Ascoli theorem, this infimum is a minimum if $j$ is Fuchsian and the Zariski closure of $(j,\rho)(\pi_1(S))$ is reductive (\ie the image of~$\rho$ does not fix a unique point on the boundary at infinity of~$\HH^2$).
The constant $C_{Lip}(j,\rho)$ is clearly invariant under conjugation of $j$ or~$\rho$ by $\PSL_2(\R)$.
The logarithm of~$C_{Lip}$ can be seen as a generalization of Thurston's ``asymmetric metric'' (or ``Lipschitz metric'') on Teichm\"uller space: see \cite[Ch.\,5]{kasPhD} and \cite{gk12}.

Let $\Gamma$ be a discrete subgroup of~$G$ acting properly discontinuously on~$X$.
By Lemma~\ref{lem:torsionAdS3}, either $\Gamma$ is standard, or its projection to $\PSL_2(\R)\times\PSL_2(\R)$ is of the form \eqref{eqn:Gammajrho}.
In the first case, we set $C_{Lip}(\Gamma):=0$.
In the second case, we set
$$C_{Lip}(\Gamma) := \min\big(C_{Lip}(j,\rho)\,,\,C_{Lip}(\rho,j)\big).$$

\subsection*{$\bullet$ Link between sharpness and the constant $C_{Lip}$}

Consider a $2$-dimensional hyperbolic orbifold~$S$ and $(j,\rho)\in\Hom(\pi_1(S),\PSL_2(\R))^2$ with $j$ Fuchsian.
Using the geometric interpretation \eqref{eqn:mudistAdS3}, we make the following easy but useful observation.

\begin{remark}\label{rem:CLipsharp}
\begin{itemize}
  \item If the Zariski closure of $(j,\rho)(\pi_1(S))$ is reductive, then there is an element $g_0\in\PSL_2(\R)$ such that for all $\boldsymbol\gamma\in\pi_1(S)$,
  $$\mu_{\PSL_2(\R)}\big(g_0^{-1}\,\rho(\boldsymbol\gamma)\,g_0\big) \leq C_{Lip}(j,\rho)\,\mu_{\PSL_2(\R)}(j(\boldsymbol\gamma)).$$
  \item In general, for any $\varepsilon>0$ there is an element $g_{\varepsilon}\in\PSL_2(\R)$ such that for all $\boldsymbol\gamma\in\pi_1(S)$,
  $$\mu_{\PSL_2(\R)}\big(g_{\varepsilon}^{-1}\,\rho(\boldsymbol\gamma)\,g_{\varepsilon}\big) \leq \big(C_{Lip}(j,\rho)+\varepsilon\big)\,\mu_{\PSL_2(\R)}(j(\boldsymbol\gamma)).$$
\end{itemize}
\end{remark}

\noindent
Indeed, for $\varepsilon\geq 0$, let $f_{\varepsilon} : \HH^2\rightarrow\HH^2$ be a $(j,\rho)$-equivariant map with $\mathrm{Lip}(f_{\varepsilon})\leq C_{Lip}(j,\rho)+\nolinebreak\varepsilon$.
We can take any $g_{\varepsilon}\in\PSL_2(\R)$ such that $f_{\varepsilon}(y_0)=g_{\varepsilon}\!\cdot\!y_0$,\linebreak using the fact that the metric $d_{\HH^2}$ is invariant under $\PSL_2(\R)$.

\medskip

Let $\Gamma$ be a discrete subgroup of~$G$.
Proposition~\ref{prop:sharpnessproperties}.(1) and Remark~\ref{rem:CLipsharp} (together with the above reformulation of sharpness) imply that if $C_{Lip}(\Gamma)<\nolinebreak 1$, then $\Gamma$ is sharp for~$X$; in particular, $\Gamma$ acts properly discontinuously on~$X$.
The converse is nontrivial but true in the finitely generated case (based on the existence of a ``maximally stretched line'' for $(j,\rho)$-equivariant maps of minimal Lipschitz constant $C_{Lip}(j,\rho)\geq 1$ \cite{kasPhD,gk12}).

\begin{fact}[\cite{kasPhD,gk12}]\label{fact:CLip}
A finitely generated discrete subgroup $\Gamma$ of~$G$ acts properly discontinuously on $X=\AdS^3$ if and only if $C_{Lip}(\Gamma)<1$, in which case $\Gamma$ is sharp for~$X$.
\end{fact}

This is how Fact~\ref{fact:AdS} and its generalization \cite{gk12} to Clifford--Klein forms~$X_{\Gamma}$ with $\Gamma$ finitely generated were obtained.
Fact~\ref{fact:AdSdeform} is a consequence of Fact~\ref{fact:CLip} and of the following continuity result.

\begin{fact}[\cite{gk12}]\label{fact:CLipcont}
The function $(j,\rho)\mapsto C_{Lip}(j,\rho)$ is continuous over the set~of pairs $(j,\rho)\in\Hom(\pi_1(S),\PSL_2(\R))^2$ with $j$ injective and $j(\pi_1(S))$ convex co\-compact in $\PSL_2(\R)$.
\end{fact}

\section{The discrete spectrum of the Laplacian}\label{subsec:discspecAdS3}

We note that here
$$\q := \g^{-\mathrm{d}\sigma} = \{ (Y,-Y) : Y\in\ssl_2(\R)\} \ \subset\ \ssl_2(\R) + \ssl_2(\R) = \g.$$
Therefore, the symmetric space $X=\AdS^3$ has rank~$1$ and the $\C$-algebra $\D(X)$ is generated by the Laplacian~$\square_X$ (Fact~\ref{fact:3.1}).
Let us identify~$X$ with the quadric of equation $Q=1$ in~$\R^4$, where the Lorentzian structure is induced by $-Q$.
As mentioned in the introduction, if we set $r(x):=\sqrt{Q(x)}$ for $Q(x)>0$, then the Laplacian~$\square_X$ is explicitly given by
$$\square_X f = \frac{1}{2}\,\square_{\R^{2,2}}\,\Big(x\longmapsto f\Big(\frac{x}{r(x)}\Big)\Big)$$
for all $f\in C^{\infty}(X)$, where
$$\square_{\R^{2,2}} = \frac{\partial^2}{\partial x_1^2} + \frac{\partial^2}{\partial x_2^2} - \frac{\partial^2}{\partial x_3^2} - \frac{\partial^2}{\partial x_4^2}$$
and where $f(x/r(x))$ is defined on the neighborhood $\{ Q>0\}$ of~$X$ in~$\R^4$.
The invariant measure~$\omega$ on~$X$ is given by
$$\omega = x_1\,\mathrm{d}x_2\,\mathrm{d}x_3\,\mathrm{d}x_4 - x_2\,\mathrm{d}x_1\,\mathrm{d}x_3\,\mathrm{d}x_4 + x_3\,\mathrm{d}x_1\,\mathrm{d}x_2\,\mathrm{d}x_4 - x_4\,\mathrm{d}x_1\,\mathrm{d}x_2\,\mathrm{d}x_3\,;$$
in other words, $\frac{1}{r}\mathrm{d}r\wedge\omega$ is the Lebesgue measure on a neighborhood of~$X$ in~$\R^4$.
The full discrete spectrum of~$\square_X$ is well-known (see \cite{far79}).
It is  a special case of the general theory stated in Fact~\ref{fact:decompVlambda}, and it also follows from Claim~\ref{claim:FJAdS3} below.

\begin{fact}
The discrete spectrum of the Laplacian on $X=\AdS^3$ is
$$\Spec_d(\square_X) = \big\{ \ell(\ell-2):\ \ell\in\N\big\} .$$
\end{fact}

We now consider Clifford--Klein forms $X_{\Gamma}$.
Here is a more precise version (and generalization) of Theorem~\ref{thm:SL2}, using the constant $C_{Lip}(\Gamma)$ of Section~\ref{subsec:CLip}.

\begin{theorem}\label{thm:SL2precise}
There is a constant $R'_X>0$ depending only on $X=\AdS^3$ such that for any Clifford--Klein form~$X_{\Gamma}$ with finitely generated $\Gamma\in\SO(2,2)_0\linebreak\simeq (\SL_2(\R)\times\nolinebreak\SL_2(\R))/\{ \pm(1,1)\}$,
\begin{itemize}
  \item if $-\mathrm{I}\notin\Gamma$, then
  $$\Spec_d(\square_{X_{\Gamma}}) \supset \bigg\{ \ell(\ell-2):\ \ell\in\N,\ \ell >\frac{R'_X}{(1-C_{Lip}(\Gamma))^3}\bigg\} \,;$$
  \item if $-\mathrm{I}\in\Gamma$, then
  $$\Spec_d(\square_{X_{\Gamma}}) \supset \bigg\{ \ell(\ell-2):\ \ell\in 2\N,\ \ell >\frac{R'_X}{(1-C_{Lip}(\Gamma))^3}\bigg\} .$$
\end{itemize}
In particular, the discrete spectrum of any Clifford--Klein form~$X_{\Gamma}$ with $\Gamma$ finitely generated is infinite.
\end{theorem}

Using Fact~\ref{fact:CLipcont}, we obtain the existence of an infinite stable discrete spectrum in the convex cocompact case.

\begin{corollary}\label{cor:AdS3deform}
For any convex cocompact Clifford--Klein form $X_{\Gamma}$ of $X=\AdS^3$ (in the sense of Definition~\ref{def:ccAdS3}), there is an infinite subset of $\Spec_d(\square_{X_{\Gamma}})$ that is stable under any small deformation of~$\Gamma$.
\end{corollary}

We note that Corollary~\ref{cor:AdS3deform} is stronger, in the case of $X=\AdS^3$, than the general Theorem~\ref{thm:deform}, because it treats small deformations of Clifford--Klein forms that may be nonstandard to start with.

For \emph{standard} Clifford--Klein forms~$X_{\Gamma}$, we have $C_{Lip}(\Gamma)=0$ and Theorem~\ref{thm:SL2precise} follows from the general Theorem~\ref{thm:precisedeform}.
We now explain how to prove Theorem~\ref{thm:SL2precise} for \emph{nonstandard} Clifford--Klein forms, using the precise version \eqref{eqn:finalcst} of Proposition~\ref{prop:nonzero}.(1) together with the theory of Sections \ref{subsec:CKformsAdS3} to~\ref{subsec:CLip} (in particular Lemma~\ref{lem:torsionAdS3}, Remark~\ref{rem:CLipsharp}, and Fact~\ref{fact:CLip}).
We first note that we can identify the closed positive Weyl chamber~$\overline{\bb_+}$ of Section~\ref{subsec:munu} with~$\R_+$ so that the polar projection
$$\nu :\ G = \SL_2(\R)\times\SL_2(\R) \,\longrightarrow\, \R_{\geq 0}$$
of \eqref{eqn:nu} is given by
\begin{equation}\label{eqn:nuAdS3}
\nu(\underline{g}) = \mu_{\SL_2(\R)}(g_1g_2^{-1})
\end{equation}
for all $\underline{g}=(g_1,g_2)\in G=\SL_2(\R)\times\SL_2(\R)$.
Here $\mu_{\SL_2(\R)} : \SL_2(\R)\!\rightarrow\R_{\geq 0}$ is the Cartan projection of $\SL_2(\R)$ obtained from the Cartan projection $\mu_{\PSL_2(\R)}$ of Section~\ref{subsec:CLip} by projecting $\SL_2(\R)$ onto $\PSL_2(\R)$.

\begin{proof}[Proof of Theorem~\ref{thm:SL2precise} for nonstandard Clifford--Klein forms]
Let $\Gamma$ be a fini\-tely generated discrete subgroup of~$G$ acting properly discontinuously on $X=\AdS^3$.
Assume that $\Gamma$ is nonstandard.
By Lemma~\ref{lem:torsionAdS3} and Fact~\ref{fact:CLip}, after possibly applying~$\sigma$, we may assume that the projection of~$\Gamma$ to $\PSL_2(\R)\times\PSL_2(\R)$ is of the form $\overline{\Gamma}=(j,\rho)(\pi_1(S))$ with $(j,\rho)\in\Hom(\pi_1(S),\PSL_2(\R))^2$ and $j$ Fuchsian, satisfying $C_{Lip}(j,\rho)<1$.
By Proposition~\ref{prop:KM}, after replacing~$j$ by some conjugate under $\PSL_2(\R)$, we may assume that $\mu_{\PSL_2(\R)}(j(\boldsymbol\gamma))\geq r_{\PSL_2(\R)}>0$ for all $\boldsymbol\gamma\in\pi_1(S)\smallsetminus\{ e\}$, where $r_{\PSL_2(\R)}$ is the constant given by Proposition~\ref{prop:KM}, which depends only on the group~$\PSL_2(\R)$.
In particular, $\Gamma\cap K=\{ e\}$.
Consider $\varepsilon>0$ such that $C_{Lip}(j,\rho)+\nolinebreak\varepsilon<\nolinebreak1$.
By Remark~\ref{rem:CLipsharp} and \eqref{eqn:cCAdS3}, after replacing~$\rho$ by some conjugate under $\PSL_2(\R)$, we may assume that $\overline{\Gamma}$ is $(c,0)$-sharp for
$$c \,:=\, \frac{1-(C_{Lip}(j,\rho)+\varepsilon)}{\sqrt{2\big(1+\big(C_{Lip}(j,\rho)+\varepsilon\big)^2\big)}} \,\geq\, \frac{1}{2} \big(1-C_{Lip}(j,\rho)-\varepsilon\big)$$
and, using \eqref{eqn:nuAdS3} and \eqref{eqn:triangineq}, that
\begin{eqnarray*}
r_{\Gamma} := \inf_{\gamma\in\Gamma\smallsetminus\{ e\}} \nu(\gamma) & \geq & \inf_{\boldsymbol\gamma\in\pi_1(S)\smallsetminus\{ e\}} \mu_{\PSL_2(\R)}(j(\boldsymbol\gamma)) - \mu_{\PSL_2(\R)}(\rho(\boldsymbol\gamma))\\
& \geq & r_{\PSL_2(\R)} \big(1-C_{Lip}(j,\rho)-\varepsilon\big) > 0.
\end{eqnarray*}
We note that the function $t\mapsto\log(\cosh(t))\,t^{-2}$ extends by continuity in~$0$ and is bounded on any bounded interval of~$\R$.
We conclude by using Proposition~\ref{prop:nonzero}.(1) with the explicit constant \eqref{eqn:finalcst}, together with Remark~\ref{rem:translconj}, and by letting $\varepsilon$ tend to zero.
\end{proof}

We note that the infinite subset of $\Spec_d(\square_{X_{\Gamma}})$ given by Theorem~\ref{thm:SL2precise} is largest when $C_{Lip}(\Gamma)=0$; this condition is realized when $\Gamma$ is standard, but also when the projection of $\Gamma$ to $\PSL_2(\R)\times\PSL_2(\R)$ is of the form \eqref{eqn:Gammajrho} with $\rho(\pi_1(S))$ unipotent.

\begin{remark}\label{rem:negativespec}
Assume that $X_{\Gamma}$ is a standard compact Clifford--Klein form with $\Gamma={}^{\backprime}\Gamma\times\{ e\}$ for some uniform lattice ${}^{\backprime}\Gamma$ of $\SL_2(\R)$.
Then the Laplacian~$\square_{X_{\Gamma}}$  has not only infinitely many positive eigenvalues that remain constant under small deformations (given by Theorem~\ref{thm:SL2precise}), but also infinitely many negative eigenvalues that vary.
\end{remark}

\noindent
Indeed, $L^2({}^{\backprime}\Gamma\backslash\HH^2)$ embeds into $L^2(X_{\Gamma})=L^2({}^{\backprime}\Gamma\backslash\SL_2(\R))$ and the restriction to $L^2({}^{\backprime}\Gamma\backslash\HH^2)$ of the Laplacian $\square_{X_{\Gamma}}$ corresponds to $-2$ times the usual Laplacian~$\Delta_{\,{}^{\backprime}\Gamma\backslash\HH^2}$ on the hyperbolic surface ${}^{\backprime}\Gamma\backslash\HH^2$ (see \cite[Ch.\,X]{lan85}).
Therefore $\square_{X_{\Gamma}}$ is essentially self-adjoint and admits infinitely many negative eigenvalues coming from eigenvalues of~$\Delta_{\,{}^{\backprime}\Gamma\backslash\HH^2}$.
All these eigenvalues vary under small deformations of ${}^{\backprime}\Gamma$ inside $\SL_2(\R)$ (Fact~\ref{fact:Teich}).

\section{Flensted-Jensen eigenfunctions and analytic estimates for $\AdS^3$}\label{subsec:FJAdS3}

In Section~\ref{subsec:discspecAdS3} we have given an explicit infinite set of eigenvalues of the Laplacian on Clifford--Klein forms of $X=\AdS^3$ (Theorem~\ref{thm:SL2precise}), based on a geometric discussion of properly discontinuous actions on $\AdS^3$ (Sections \ref{subsec:CKformsAdS3} to~\ref{subsec:CLip}).
We now make the analytic aspects of the paper more concrete by expliciting the general estimates of Chapters \ref{sec:Vlambda} and~\ref{sec:FJ} in our example $X=\AdS^3$.
We first give an explicit formula for the Flensted-Jensen eigenfunctions~$\psi_{\lambda}$.

\subsection*{$\bullet$ Flensted-Jensen functions}

It is known that, in general, the radial part of the $K$-invariant eigenfunctions on a rank-one reductive symmetric space~$X$ satisfies the Gauss hypergeometric differential equation \cite[Ch.\,III, Cor.\,2.8]{hs94}.
However, it is another thing to find an explicit global formula on the whole of~$X$ for \emph{$K$-finite} eigenfunctions such as the Flensted-Jensen functions.
We now give such a formula for $X=\AdS^3$.

We now switch to the quadric realization of~$X$: we identify $X$ with the quadric of equation $Q=1$ in~$\R^4$, where $Q$ is given by \eqref{eqn:defQ}.
We use the same letter~$Q$ to denote the corresponding complex quadratic form on~$\C^4$.
Let $\ell$ be an integer.
For any $a=(a_i)\in\C^4$ with $Q(a)=0$, the restriction of the function $x\mapsto(\sum_{i=1}^4 a_i x_i)^{-\ell}$ to~$X$ is well-defined.
It is an eigenfunction of~$\square_X$ with eigenvalue $\ell(\ell-2)$, as one sees from the formulas
$$\square_{\R^{2,2}}\bigg(\sum_{i=1}^4 a_i x_i\bigg)^{-\ell} = 0$$
for $Q(a)=0$ and
$$-r^2\,\square_{\R^{2,2}} = - \Big(r\frac{\partial}{\partial r}\Big)^2 - 2r\,\frac{\partial}{\partial r} + \square_X$$
(where, as above, we set $r(x):=\sqrt{Q(x)}$ for $Q(x)>0$).
Let $\psi^+_{\ell} : X\rightarrow\C$ and  $\psi^-_{\ell} : X\rightarrow\C$ be given by
\begin{equation}\label{eqn:FJSL2}
\psi^+_{\ell}(x) = \big(x_1 + \sqrt{-1}\,x_2\big)^{-\ell} \quad\mathrm{and}\quad \psi^-_{\ell}(x) = \big(x_1 - \sqrt{-1}\,x_2\big)^{-\ell}.
\end{equation}
Then $\square_X\,\psi^\pm_{\ell}=\ell(\ell-2)\psi^\pm_{\ell}$ and the following holds.

\begin{claim}\label{claim:FJAdS3}
For any integer $\ell \geq 2$, the functions $\psi^{\pm}_{\ell} : X\rightarrow\C$ are Flensted-Jensen functions for the parameter $\lambda=2\ell-2\in\R_+\simeq\jj^{\ast}_+$.\,The $(\g,K)$-modules generated by~$\psi^+_{\ell}$ and by $\psi^-_{\ell}$ ($\ell =2,3,...$) form the complete set of discrete series representations for~$X$.
\end{claim}

A proof of Claim~\ref{claim:FJAdS3} will be given in Section~\ref{subsec:proofclaimAdS3}, after we explicit the Flensted-Jensen duality, the Poisson transform, and the complexified Iwasawa projection $G_{\C}=K_{\C}(\exp\jj_{\C})N_{\C}$ in Sections \ref{subsec:FJdualAdS3} to~\ref{subsec:IwasawaAdS3}.

\begin{remark}
It is known that for the rank-one symmetric spaces $G/H=O(p,q)/O(p-1,q)$, the radial part of the $K$-finite eigenfunctions is given by hypergeometric functions with respect to the polar decomposition $G=K(\exp\overline{\bb_+})H$, while the spherical part is given by spherical harmonics (see \cite{far79} or \cite{sch87} for instance).
Combining the radial and spherical parts in the case $p=q=2$, we could obtain Claim~\ref{claim:FJAdS3} from some nontrivial relation between special functions \cite[Lem.\,8.1]{kor03}.
Instead, we will take an alternative approach, using the explicit realization of $X_{\C}=G_{\C}/H_{\C}$ as a complex quadric in~$\C^4$.
\end{remark}

\subsection*{$\bullet$ Analytic estimates}

Here are the estimates of Propositions \ref{prop:asym} and~\ref{prop:psilambda} for the Flensted-Jensen functions $\psi_{\ell}^{\pm}$ of \eqref{eqn:FJSL2}.
As before, we denote by $x_0$ the image of $H$ in $X=G/H$; in our quadric realization, $x_0=(1,0,0,0)\in\R^4$.

\begin{lemma}\label{lem:estimAdS}
For any $x\in X=\AdS^3$,
\begin{equation}\label{eqn:estim1AdS}
|\psi^{\pm}_{\ell}(x)| \,\leq\, \Big(\frac{\cosh\nu(x)}{2}\Big)^{-\ell/2} \,\leq\, 2^{\ell}\,e^{-\ell\,\nu(x)/2},
\end{equation}
and
\begin{equation}\label{eqn:estim2AdS}
|\psi^{\pm}_{\ell}(x)| \,\leq\, \cosh\Big(\frac{\nu(x)}{2}\Big)^{-\ell/2} \,\leq\, \cosh\Big(\frac{\nu(x)}{4}\Big)^{-\ell} \,\leq\, |\psi^{\pm}_{\ell}(x_0)| = 1.
\end{equation}
\end{lemma}

We give a direct, elementary proof of these inequalities.

\begin{proof}
By \eqref{eqn:nuAdS3}, in the realization of $X=\AdS^3$ as the group manifold $\SL_2(\R)$, the polar projection $\nu : X\rightarrow\R_{\geq 0}$ coincides with the Cartan projection $\mu_{\SL_2(\R)} :\nolinebreak \SL_2(\R)\rightarrow\nolinebreak\R_{\geq 0}$, which maps $g\in\SL_2(\R)$ to the logarithm of the highest eigenvalue of $^t\!gg$, or in other words to $\mathrm{arcosh}(\mathrm{tr}(^t\!gg)/2)$.
Using the explicit correspondence \eqref{eqn:correspAdS3}, we obtain
\begin{equation}\label{eqn:nuAdS3quadric}
\nu(x) = \operatorname{arcosh}(x_1^2 + x_2^2 + x_3^2 + x_4^2) = \operatorname{arcosh}(2x_1^2 + 2x_2^2 - 1)
\end{equation}
for all $x=(x_1,x_2,x_3,x_4)\in X$ in the quadric realization.
By definition \eqref{eqn:FJSL2} of~$\psi_{\ell}^{\pm}$, we have $|\psi^{\pm}_{\ell}(x)|=(x_1^2+x_2^2)^{-\ell/2}$ for all $x\in X$.
Thus \eqref{eqn:estim1AdS} follows directly from \eqref{eqn:nuAdS3quadric}.
To obtain \eqref{eqn:estim2AdS}, we use the general inequality $1+\cosh(2s)\geq 2\cosh(s)$ with $2s=\nu(x)$.
\end{proof}

The rest of the chapter is devoted to explaining Claim~\ref{claim:FJAdS3}.
For this purpose we explicit, in the particular case of $X=\AdS^3$, some of the notation that was introduced in Chapters \ref{sec:theorems} to~\ref{sec:nonzero}.

\section{The Flensted-Jensen duality for $\AdS^3$}\label{subsec:FJdualAdS3}

We now realize $X$ again as $(\SL_2(\R)\times\SL_2(\R))/\Diag(\SL_2(\R))$.
Then the set of inclusions \eqref{eqn:KGH} is given~by
\begin{changemargin}{-0.1cm}{0cm}
\begin{center}
\small
\begin{tabular}{c c c c c}
$K=\SO(2)\times\SO(2)$ & $\subset$ & $G=\SL_2(\R)\times\SL_2(\R)$ & $\supset$ & $H=\Diag(\SL_2(\R))$\\
\smallskip
\rotatebox{90}{$\supset$} & & \rotatebox{90}{$\supset$} & & \rotatebox{90}{$\supset$}\\
\smallskip
$K_{\C}=\SO(2,\C)\times\SO(2,\C)$ & $\subset$ & $G_{\C}=\SL_2(\C)\times\SL_2(\C)$ & $\supset$ & $H_{\C}=\Diag(\SL_2(\C))$\\
\medskip
\rotatebox{90}{$\subset$} & & \rotatebox{90}{$\subset$} & & \rotatebox{90}{$\subset$}\\
\smallskip
$H^d=\Phi(\SO(2,\C))$ & $\subset$ & $G^d=\Phi(\SL_2(\C))$ & $\supset$ & $K^d=\Phi(\SU(2)),$
\end{tabular}
\end{center}
\end{changemargin}
where $\Phi$ is the embedding of $\SL_2(\C)$ into $\SL_2(\C)\times\SL_2(\C)$ defined by
\begin{equation}\label{eqn:GdAdS3}
\Phi(g) = \big(g,\overline{^t\!g^{-1}}\big)
\end{equation}
for all $g\in\SL_2(\C)$.
We can see the complexified symmetric space~$X_{\C}$ either as the $3$-dimensional complex sphere of equation $Q=1$ in~$\C^4$ or as the group $\SL_2(\C)$ with the transitive action \eqref{eqn:lrmul} of $\SL_2(\C)\times\SL_2(\C)$ by left and right multiplication; the correspondence is given by the complex linear extension of \eqref{eqn:correspAdS3}.
The dual space~$X^d$ can be realized either as
\begin{equation}\label{eqn:Xdquadric}
X^d = \big\{ \big(x_1,\sqrt{-1}\,x_2,x_3,x_4\big) : x_i\in\R,\; x_1^2-x_2^2-x_3^2-x_4^2=1,\ x_1>0\big\}
\end{equation}
or as the set $\mathrm{Herm}(2,\C)_+\cap\SL_2(\C)$ of positive definite Hermitian matrices in $\SL_2(\C)$; it identifies with the $3$-dimensional hyperbolic space $\HH^3$.
The compact form $X_U$ of~$X_{\C}$ can be realized either as
$$X_U = \big\{ \big(x_1,x_2,\sqrt{-1}\,x_3,\sqrt{-1}\,x_4\big) : x_i\in\R,\; x_1^2+x_2^2+x_3^2+x_4^2=1\big\} $$
or as the subgroup $\SU(2)$ of $\SL_2(\C)$; it identifies with the $3$-dimensional real sphere $\mathbb{S}^3$.
The following diagram summarizes the different realizations of $X$, $X_{\C}$, and~$X^d$.
\smallskip
\begin{center}
\small
\begin{tabular}{c c c c c}
$X=G/H$ & $\simeq$ & $\SL_2(\R)$ & $\longhookrightarrow$ & $\R^4$\\
\smallskip
\rotatebox{-90}{$\subset$} & & \rotatebox{-90}{$\subset$} & & \rotatebox{-90}{$\subset$}\\
\smallskip
$X_{\C}=G_{\C}/H_{\C}$ & $\underset{\Phi'}{\overset{\sim}{\longrightarrow}}$ & $\SL_2(\C)$ & $\overset{\eqref{eqn:correspAdS3}}{\longhookrightarrow}$ & $\C^4$\\
\medskip
\rotatebox{90}{$\subset$} & & \rotatebox{90}{$\subset$} & & \rotatebox{90}{$\subset$}\\
\smallskip
$X^d=G^d/H^d$ & $\simeq$ & $\mathrm{Herm}(2,\C)_+\cap\SL_2(\C)$ & $\longhookrightarrow$ & $\R\times\sqrt{-1}\,\R\times\R\times\R$\\
\medskip
\rotatebox{90}{$\underset{\Phi}{\overset{\sim}{\longrightarrow}}$} & \rotatebox{20}{$\underset{\Phi'\circ\Phi}{\overset{\sim}{\longlongrightarrow}}$} & & & \\
\smallskip
$\SL_2(\C)/\SU(2),$ & & & & \\
\end{tabular}
\end{center}
Here we set
$$\Phi'(\underline{g}H_{\C}) := g_1 g_2^{-1}$$
for all $\underline{g}=(g_1,g_2)\in G_{\C}=\SL_2(\C)\times\SL_2(\C)$.
\emph{In the rest of the chapter, we always identify $G^d$ with $\SL_2(\C)$ using the isomorphism~$\Phi$ of \eqref{eqn:GdAdS3}.}

\section{Eigenfunctions on $X^d\simeq\HH^3$ and the Poisson transform}\label{subsec:PoissonH3}

Let $P^d$ be any Borel subgroup of $G^d=\SL_2(\C)$, let $N^d$ be the unipotent radical of~$P^d$, and let $\jj$ be any maximal split abelian subalgebra of~$\g^d$ with $\exp\jj\subset P^d$.
For instance, we could take $P^d$ to be the group of upper triangular matrices of determinant~$1$, so that $N^d$ is the group of unipotent upper triangular matrices, and take $\jj$ to be the set of real diagonal matrices of trace~$0$ (in the next section we are going to make another choice).

The boundary at infinity $\partial_{\infty}X^d\simeq\PP^1\C$ of $X^d\simeq\HH^3$ identifies with $G^d/P^d$; we denote the image of $P^d$ by~$z_0$.
Let $y_0^d$ be the image of $K^d$ in $X^d=G^d/K^d$ and let $\mathcal{L}$ be the geodesic line $(\exp\jj)\cdot y_0^d$.
The Iwasawa decomposition $G^d=K^d(\exp\jj)N^d$ holds; this means that any point $y\in X^d$ can be reached from~$y_0^d$ by first applying some translation along the line~$\mathcal{L}$, then traveling along some horosphere centered at $z_0\in\partial_{\infty}X^d$.
The Iwasawa projection $\zeta^d : G^d\rightarrow\jj$ measures this translation: we can identify $\jj$ with~$\R$ so that $\zeta^d(g)$ is the signed distance between $y_0^d$ and the horosphere through $g^{-1}\!\cdot y_0^d$ centered at~$z_0$ for any $g\in G^d$; the sign of $\zeta^d(g)$ is negative if the horosphere intersects the geodesic ray $\mathcal{R}:=(\exp\jj_+)\cdot y_0^d$ and nonnegative otherwise.
For all $k\in K^d$ and $g\in G^d$,
$$\zeta^d(g^{-1}k) = B_{k\cdot\mathcal{R}}(g\cdot y_0^d),$$
where $B_{k\cdot\mathcal{R}} : X^d\rightarrow\R$ is the Busemann function associated with the geodesic ray $k\!\cdot\!\mathcal{R}$.
Recall that by definition
$$B_{k\cdot\mathcal{R}}(x) = \lim_{t\rightarrow +\infty} \big(d_{X^d}\big(x,k\!\cdot\!\mathcal{R}(t)\big)-t\big),$$
where $d_{X^d}$ is the metric on the Riemannian symmetric space $X^d=G^d/K^d$.

We note that the group~$K^d$ acts transitively on $\partial_{\infty}X^d$.
The classical Poisson transform, defined by
$$(\mathcal{P}f)(y) = \int_{k\in K^d/K^d\cap P^d} f(k\!\cdot\!z_0)\, e^{-2B_{k\cdot\mathcal{R}}(y)} \,\mathrm{d}k$$
for all $f\in C(\partial_{\infty}X^d)$ and $y\in X^d=G^d/K^d$, induces a bijection between the continuous functions on $\partial_{\infty}X^d$ and the harmonic functions on~$X^d$ that extend continuously to~$\partial_{\infty}X^d$; the function $\mathcal{P}f$ is the unique solution to the Dirichlet problem on $X^d\simeq\HH^3$ with boundary condition~$f$ (see \cite[Ch.\,II, \S\,4]{hel00}).
If we extend the domain of definition of~$\mathcal{P}$ to the space of all hyperfunctions on~$\partial_{\infty}X^d$, then we obtain all harmonic functions on~$X^d$ in a unique way.
For $\lambda\in\jj_{\C}^{\ast}\simeq\C$ (where $\rho\in\jj_{\C}^{\ast}$ corresponds to $2\in\C$), the ``twisted Poisson transform''
$$\mathcal{P}_{\lambda} :\ \mathcal{B}(K^d/K^d\cap P^d) \overset{\sim}{\longrightarrow} \mathcal{A}(X^d,\M_{\lambda})$$
of Section~\ref{subsec:Poisson} is given by
$$(\mathcal{P}_{\lambda}f)(y) = \int_{k\in K^d/K^d\cap P^d} f(k\!\cdot\!z_0)\,e^{-(\lambda+2)B_{k\cdot\mathcal{R}}(y)}\,\mathrm{d}k$$
for $y\in X^d$; its image consists of eigenfunctions of the Laplacian on~$X^d$ with eigenvalue $\lambda(\lambda+2)/4$.

The action of $H^d=\SO(2,\C)$ on $\partial_{\infty}X^d$
corresponds to the action of~$\C^{\ast}$ by multiplication on $\PP^1\C$, hence there are three $H^d$-orbits: two closed ones $Z_0=\{ z_0\}$ and $Z_{\infty}=\{ w\cdot z_0\}$ (where $w$ is the nontrivial element of the Weyl group $W=W(\g_{\C},\jj_{\C})\simeq\Z/2\Z$), corresponding respectively to $\{ 0\}$ and $\{ \infty\}$, and an open one, corresponding to $\C^{\ast}$.

\section{Meromorphic continuation of the Iwasawa projection}\label{subsec:IwasawaAdS3}

We now assume that $\jj$ is a maximal semisimple abelian subspace of $\sqrt{-1}(\kk\cap\q)$, as in Section~\ref{subsec:Lambda+}.
If we still identify $G^d$ with $\SL_2(\C)$ by \eqref{eqn:GdAdS3}, this means that
$$\jj = \sqrt{-1}\,\R \begin{pmatrix} 0 & 1\\ -1 & 0\end{pmatrix}.$$
Thus $\jj$ is a maximal split abelian subalgebra of~$\g^d$ as in Section~\ref{subsec:PoissonH3}.
It is readily seen that
$$\n^d := \C \begin{pmatrix} 1 & \sqrt{-1}\\ \sqrt{-1} & -1\end{pmatrix}$$
is a root space for~$\jj$, hence the Iwasawa decomposition $G^d=K^d(\exp\jj)N^d$ holds for $N^d:=\exp\n^d$.
This Iwasawa decomposition can be recovered from the usual decomposition
\begin{equation}\label{eqn:usualIwasawaSL2C}
G^d = K^d\,\exp\left(\R \begin{pmatrix} 1 & 0\\ 0 & -1\end{pmatrix}\right)\,\exp\left(\C \begin{pmatrix} 0 & 1\\ 0 & 0\end{pmatrix}\right)
\end{equation}
by conjugating by
\begin{equation}\label{eqn:CayleyAdS}
k := \frac{1}{2} \begin{pmatrix}\ \ 1+\sqrt{-1} & 1+\sqrt{-1}\\ -1+\sqrt{-1} & 1-\sqrt{-1}\end{pmatrix} \in K^d.
\end{equation}
We note that
$$k\,\SL_2(\R)\,k^{-1} = \SU(1,1) = \bigg\{ g\in\SL_2(\C) :\ ^t\!\overline{g} \begin{pmatrix} 0 & 1\\ 1 & 0\end{pmatrix} g = \begin{pmatrix} 0 & 1\\ 1 & 0\end{pmatrix}\bigg\} $$
and that $\Ad(k)$ induces an identification (``Cayley transform'') between the upper half-plane model $\SL_2(\R)/\SO(2)$ of~$\HH^2$ and the unit disk model $\SU(1,1)/\linebreak\mathrm{S}(\U(1)\times\U(1))$.
An elementary computation shows that the Iwasawa projection corresponding to \eqref{eqn:usualIwasawaSL2C} is given by
\begin{equation}\label{eqn:otherIwasawaAdS3}
g\in G^d \longmapsto \frac{1}{2}\log(^t\overline{g}g)_{1,1} \in \R,
\end{equation}
where $(^t\overline{g}g)_{1,1}$ denotes the upper left entry of $^t\overline{g}g\in\SL_2(\C)$.
We now go back to the quadric realization \eqref{eqn:Xdquadric} of~$X^d$.
Using \eqref{eqn:otherIwasawaAdS3} and the explicit correspondence \eqref{eqn:correspAdS3}, we see that if $\zeta^d : G^d\rightarrow\R$ is the Iwasawa projection corresponding to $G^d=K^d(\exp\jj)N^d$, then for $\lambda\in\jj^{\ast}\simeq\R$ the map $\xi_{\lambda}^{\vee} :\nolinebreak X^d\rightarrow\nolinebreak\R$ induced by $g\mapsto e^{\langle\lambda,\zeta^d(g^{-1})\rangle}$ is given by
\begin{equation}\label{eqn:xiAdS3}
\xi_{\lambda}^{\vee}(z) = \big(z_1+\sqrt{-1}\,z_2\big)^{\lambda/2}
\end{equation}
for all $z=(z_1,z_2,z_3,z_4)\in X^d\subset\C^4$.
When $\lambda\in 2\Z$, the map $\xi_{\lambda}^{\vee}$ extends meromorphically to $X_{\C}=\{ z\in\C^4 : Q(z)=1\} $ and restricts to an analytic function on~$X$.

\section{Proof of Claim~9.12}\label{subsec:proofclaimAdS3}

We now combine the elementary computations of Sections \ref{subsec:FJdualAdS3} to~\ref{subsec:IwasawaAdS3} to obtain an explicit formula of the Flensted-Jensen functions~$\psi_{\lambda}$ for $X=\AdS^3$.

We choose $\jj$ and~$N^d$ as in Section~\ref{subsec:IwasawaAdS3} and let $P^d$ be the Borel subgroup of $G^d\simeq\SL_2(\C)$ containing $\exp\jj$ and~$N^d$.
By Section~\ref{subsec:PoissonH3}, the two closed $H^d$-orbits in $G^d/P^d$ are $Z_0=H^dP^d$ and $Z_{\infty}=H^dwP^d$.
If we identify $G^d/P^d$ with $K^d/K^d\cap P^d \simeq \SU(2)/\SO(2)$, then
$$Z_0 = \{ K^d\cap P^d\} \quad\quad\mathrm{and}\quad\quad Z_{\infty} = \{ w(K^d\cap P^d)\} .$$
For $\lambda\in\jj^{\ast}\simeq\R$, the Flensted-Jensen function $\psi_{\lambda}^0 : X^d\rightarrow\C$ associated with~$Z_0$ is the Poisson transform $\mathcal{P}_{\lambda}(\delta_{Z_0})$ of the Dirac delta function~$\delta_{Z_0}$,
hence
$$\psi_{\lambda}^0(gK^d) = e^{\langle -\lambda-\rho,\zeta^d(g^{-1})\rangle} = \xi_{-\lambda-\rho}^{\vee}(g)$$
for all $g\in G^d$.
Similarly, the Flensted-Jensen function $\psi_{\lambda}^{\infty} : X^d\rightarrow\C$ asso\-ciated with~$Z_{\infty}$ is given by
$$\psi_{\lambda}^{\infty}(gK^d) = e^{\langle -\lambda-\rho,\zeta^d(g^{-1}w)\rangle} = \xi_{-\lambda-\rho}^{\vee}(w^{-1}g).$$
Therefore, by \eqref{eqn:xiAdS3},
$$\psi_{\lambda}^0(z) = \big(z_1+\sqrt{-1}\,z_2\big)^{-(\lambda+2)/2} \quad\mathrm{and}\quad \psi_{\lambda}^{\infty}(z) = \big(z_1-\sqrt{-1}\,z_2\big)^{-(\lambda+2)/2}$$
for all $z\in X^d$, in the quadric realization \eqref{eqn:Xdquadric}.
As observed at the end of Section~\ref{subsec:IwasawaAdS3}, the functions $\psi_{\lambda}^0$ and~$\psi_{\lambda}^{\infty}$ on~$X^d$ induce analytic functions on~$X$ as soon as $(\lambda+2)/2\in\Z$, \ie as soon as $\lambda\in 2\Z$; this corresponds to the integrality condition~\eqref{eqn:lmdint} (we have $\mu_{\lambda}^e=\mu_{\lambda}^w=\lambda+2$).
The proof of Claim~\ref{claim:FJAdS3} is now complete.

%% file: StableSpec10-OtherEx.tex
\chapter{Some other illustrative examples}\label{sec:ex}

In this chapter we present some higher-dimensional examples of non-Riem\-annian locally symmetric spaces to which our theorems apply, namely higher-dimensional anti-de Sitter manifolds and group manifolds, as well as certain indefinite K\"ahler manifolds.

\section{Anti-de Sitter manifolds of arbitrary dimension}\label{subsec:AdS}

As a generalization of Chapter~\ref{sec:exAdS}, we consider the discrete spectrum of complete anti-de Sitter manifolds of arbitrary dimension $\geq 3$.

For $m\geq 2$, the anti-de Sitter space $X=\AdS^{m+1}:=\SO(2,m)_0/\SO(1,m)_0$ is a model space for all Lorentzian manifolds of dimension $m+1$ and constant negative curvature.
It can be realized as the quadric of~$\R^{m+2}$ of equation $Q=1$, endowed with the Lorentzian structure induced by $-Q$, where
$$Q(x) = x_1^2 + x_2^2 - x_3^2 - \dots - x_{m+2}^2\,;$$
the sectional curvature is then~$-1$ (see \cite{wol11}).

By the general construction of \cite{kob89}, we see that $\AdS^{m+1}$ admits proper actions by reductive subgroups~$L$ of $G:=\SO(2,m)_0$ of real rank~$1$ such as:
\begin{itemize}
  \item $L=\U(1,[\frac{m}{2}])$, where $[\frac{m}{2}]$ denotes the largest integer $\leq\frac{m}{2}$;
  \item $L =\PSL_2(\R)$, \emph{via} a real $5$-dimensional irreducible representation~$\tau_5$ of $\PSL_2(\R)$ when $m\geq 3$.
\end{itemize}
Standard Clifford--Klein forms~$X_{\Gamma}$ of~$X$ can be obtained by taking $\Gamma$ to~be~any torsion-free discrete subgroup inside~$L$ (for instance an infinite cyclic group, a nonabelian free group, a lattice of~$L$, an embedded surface group, etc.).

In particular, since $\U(1,\frac{m}{2})$ acts transitively on~$X$ for $m$ even, we can obtain compact (\resp noncompact but finite-volume) standard Clifford--Klein forms of $\AdS^{m+1}$ for $m$ even by taking $\Gamma$ to be any torsion-free uniform (\resp nonuniform) lattice in~$\U(1,\frac{m}{2})$.
This construction of compact Clifford--Klein forms of $\AdS^{m+1}$ is (conjecturally) the only one for $m>2$ since
\begin{itemize}
  \item compact Clifford--Klein forms do not exist when $m$ is odd \cite{kul81},
  \item Zeghib \cite{zeg98} has conjectured that for $m$ even $>2$, all compact Clifford--Klein should be standard, with $\Gamma\subset\U(1,\frac{m}{2})$ up to conjugation (this conjecture is still open).
\end{itemize}
We recall from Chapter~\ref{sec:exAdS} that the case $m=2$ is different, as $\AdS^3$ admits many nonstandard compact Clifford--Klein forms.

Since all compact anti-de Sitter manifolds are complete \cite{kli96}, small deformations of the anti-de Sitter structure on a compact Clifford--Klein form $\Gamma\backslash\AdS^{m+1}$ correspond to small deformations of $\Gamma$ inside $G=\SO(2,m)_0$.
When $\Gamma\subset L$ is standard, nontrivial deformations exist as soon as the first Betti number of~$\Gamma$ is nonzero \cite{kob98}, which can happen by work of Kazhdan~\cite{kaz77}. 
For $m>2$, small deformations of standard compact Clifford--Klein forms of $\AdS^{m+1}$ can never give rise to nonstandard forms (see Section~\ref{subsec:exdeform}).
However, standard \emph{noncompact} Clifford--Klein forms $\Gamma\backslash\AdS^{m+1}$ can, typically if $\Gamma$ is a convex cocompact subgroup of~$L$ that is a free group (Schottky group).
By \cite{kas12}, if $\Gamma$ is an arbitrary convex cocompact subgroup of~$L$, then it keeps acting properly discontinuously on $\AdS^{m+1}$ after any small (possibly nonstandard) deformation inside~$G$.
Nonstandard noncompact Clifford--Klein forms of $\AdS^{m+1}$ were also constructed by Benoist \cite{ben96} without using any deformation.

As a symmetric space, $X=\AdS^{m+1}$ has rank~one, hence the algebra $\D(X)$ of $G$-invariant differential operators on~$X$ is generated by the Laplacian~$\square_X$.
For standard Clifford--Klein forms of~$X$, Theorem~\ref{thm:precise}.(2) yields the following (explicit eigenfunctions can be constructed as in Chapter~\ref{sec:exAdS}).

\begin{proposition}\label{prop:higherAdS}
There is an integer~$\ell_0$ such that for any standard Clifford--Klein form $X_{\Gamma}$ of $X=\AdS^{m+1}$ with $\Gamma\subset L=\U(1,[\frac{m}{2}])$ and $\Gamma\cap Z(L)=\{ e\}$,
\begin{equation}\label{eqn:specAdS}
\Spec_d(\square_{X_{\Gamma}}) \supset \big\{ \ell(\ell-m):\ \ell\in\N,\ \ell\geq\ell_0\big\} ,
\end{equation}
and \eqref{eqn:specAdS} still holds after a small deformation of~$\Gamma$ inside~$G$.
A similar statement holds for $L=\PSL_2(\R)$, embedded in $\SO(2,m)_0$ \emph{via}~$\tau_5$.
\end{proposition}

For the reader who would not be very familiar with reductive symmetric spaces, we now explicit the notation of the previous chapters for $X=\AdS^{m+1}$.
We see $H:=\SO(1,m)_0$ as $\SO(2,m)_0\cap\SL_{m+1}(\R)$, where $\SL_{m+1}(\R)$ is embedded in the lower right corner of $\SL_{m+2}(\R)$; the involution~$\sigma$ defining~$H$ is thus given by
$$\sigma(g) = \left(\begin{BMAT}{cccc}{cccc} 1 & & & \\ & -1 & & \\ & & \ddots & \\ & & & -1\end{BMAT}\right) g \left(\begin{BMAT}{cccc}{cccc} 1 & & & \\ & -1 & & \\ & & \ddots & \\ & & & -1\end{BMAT}\right)$$
for $g\in G=\SO(2,m)_0$.

\subsection*{$\bullet$ Cartan and generalized Cartan decompositions}

The Cartan decomposition $G=KAK$ holds, where $K=\SO(2)\times\SO(m)$
and the Lie algebra $\aaa$ of~$A$ is the set of block matrices of the form
$$a_{s,t} := \left(\begin{BMAT}{c.c}{c.c}
E_{s,t} & 0\\
0 & \begin{BMAT}{ccccc}{ccccc} & & & & \\ & & & & \\ & & 0 & & \\ & & & & \\ & & & & \end{BMAT}
\end{BMAT}\right)$$
for $s,t\in\R$, where
$$E_{s,t} := \left(\begin{BMAT}{c.cc.c}{c.cc.c} 0 & & & s\\ & 0 & t & \\ & -t & 0 & \\ -s & & & 0\end{BMAT}\right)\ \in\so(4).$$
The generalized Cartan decomposition $G=KBH$ holds, where the Lie algebra $\bb$ of~$B$ is the set of elements $a_{s,0}$ for $s\in\R$.

\subsection*{$\bullet$ The Flensted-Jensen duality}

The set of inclusions \eqref{eqn:KGH} is given by
\begin{changemargin}{-0.3cm}{0cm}
\begin{center}
\small
\begin{tabular}{c c c c c}
$K=\SO(2)\times\SO(m)$ & $\subset$ & $G=\SO(2,m)_0$ & $\supset$ & $H=\SO(1,m)_0$\\
\smallskip
\rotatebox{90}{$\supset$} & & \rotatebox{90}{$\supset$} & & \rotatebox{90}{$\supset$}\\
\smallskip
$K_{\C}=\SO(2,\C)\times\SO(m,\C)$ & $\subset$ & $G_{\C}=\SO(m+2,\C)$ & $\supset$ & $H_{\C}=\SO(m+1,\C)$\\
\medskip
\rotatebox{90}{$\subset$} & & \rotatebox{90}{$\subset$} & & \rotatebox{90}{$\subset$}\\
\smallskip
$H^d=\SO(1,1)_0\times\SO(m)$ & $\subset$ & $G^d=\SO(1,m+1)_0$ & $\supset$ & $K^d=\SO(m+1).$
\end{tabular}
\end{center}
\end{changemargin}

\smallskip
\noindent
In particular, $X^d=G^d/K^d=\SO(1,m+1)_0/\SO(m+1)$ is the real hyperbolic space~$\HH^m$.

\subsection*{$\bullet$ Closed $H^d$-orbits~$Z$ and the parameter~$\lambda$ of discrete series representations}

A maximal abelian subspace of $\sqrt{-1}\,(\kk\cap\q)$ is given by $\jj:=\sqrt{-1}\,\so(2)$, where $\so(2)$ is the first factor of $\kk=\so(2)\oplus\so(m)$.
We note that $\jj$ is also maximal abelian in $\sqrt{-1}\,\q$, hence
$$\rank G/H = \rank K/H\cap K = 1 = \dim\jj.$$
Since $\jj$ is centralized by~$\kk$, the restricted root system $\Sigma(\kk_{\C},\jj_{\C})$ is empty.
Let $Y$ be the generator $\sqrt{-1}\begin{pmatrix} 0 & 1\\ -1 & 0\end{pmatrix}$ of $\jj=\sqrt{-1}\,\so(2)$ and let $e_1\in\jj^{\ast}$ be defined by $\langle e_1,Y\rangle=1$.
There are two possible choices of positive systems $\Sigma^+(\g_{\C},\jj_{\C})$, namely $\{ e_1\}$ and $\{ -e_1\}$.
By \eqref{eqn:Z1to1}, the set~$\ZZ$ of closed $H^d$-orbits in the real flag variety $G^d/P^d$ has exactly two elements.
They are actually singletons, the ``North and South poles'' of $G^d/P^d\simeq\mathbb{S}^m$.  
Take $\Sigma^+(\g_{\C},\jj_{\C})$ to be $\{ e_1\}$ (\resp $\{ -e_1\}$).
If we identify $\jj$ with~$\R$ by sending $e_1$ (\resp $-e_1$) to~$1$, then $\jj^{\ast}_+$ identifies with~$\R_+$ and we have $\rho=\frac{m}{2}$ and $\rho_c=0$, hence
$$\mu_{\lambda} = \lambda + \rho - 2\rho_c = \lambda + \frac{m}{2}.$$
Condition~\eqref{eqn:lmdint} on~$\mu_{\lambda}$ amounts to $\lambda\in\Z$.
The two discrete series representations with parameter $\pm\lambda$ are dual to each other.

\subsection*{$\bullet$ Eigenvalues of the Laplacian}

By Fact~\ref{fact:eigenvalues}, the action of the Laplacian~$\square_X$ on $L^2(X,\M_{\lambda})$ is given by multiplication by the scalar
$$(\lambda,\lambda) - (\rho,\rho) = \lambda^2 - \frac{m^2}{4},$$
which can be written as $\ell(\ell-m)$ if we set $\ell:=\lambda+\frac{m}{2}$.
This explains Proposition~\ref{prop:higherAdS}.

\section{Group manifolds}\label{subsec:groupmfd}

In this section we consider symmetric spaces of the form $X=({}^{\backprime}G\times\!{}^{\backprime}G)/\Diag({}^{\backprime}G)$ where ${}^{\backprime}G$ is any reductive linear Lie group.
As mentioned in Section~\ref{subsec:listexgroup}, the rank condition~\eqref{eqn:rank} is here equivalent to $\rank{}^{\backprime}G = \rank{}^{\backprime}K$, where ${}^{\backprime}K$ is a maximal compact subgroup of~${}^{\backprime}G$.
This condition is satisfied for ${}^{\backprime}G=\SL_2(\R)$, in which case $X$ is the $3$-dimensional anti-de Sitter space $\AdS^3$ examined in Chapter~\ref{sec:exAdS}.
More generally, it is satisfied for all simple groups~${}^{\backprime}G$ with Lie algebra in the list~\eqref{eqn:listrank}.
It is equivalent to the fact that the Cartan involution of~$^{\backprime}G$ is an inner automorphism.

\subsection*{$\bullet$ Infinite stable spectrum in real rank one.}

Assume that ${}^{\backprime}G$ has real rank~$1$.
Then the structural results of Section~\ref{subsec:CKformsAdS3} generalize: by \cite[Th.\,1.3]{kas08} (improving an earlier result of \cite{kob93}), if a torsion-free discrete subgroup~$\Gamma$ of\; ${}^{\backprime}G\times\!{}^{\backprime}G$ acts properly discontinuously on~$X$, then it is of the form
\begin{equation}\label{eqn:jrho}
\Gamma = \big\{ (j(\boldsymbol\gamma),\rho(\boldsymbol\gamma)) : \boldsymbol\gamma\in{}^{\backprime}\Gamma\big\} ,
\end{equation}
where ${}^{\backprime}\Gamma$ is a discrete subgroup of~${}^{\backprime}G$ and $j,\rho\in\Hom({}^{\backprime}\Gamma,{}^{\backprime}G)$ are two representations with $j$ injective and discrete (up to switching the two factors).
Moreover, the Clifford--Klein form~$X_{\Gamma}$ is compact if and only if $j({}^{\backprime}\Gamma)\backslash G$ is.
Standard Clifford--Klein forms correspond to the case when $\rho({}^{\backprime}\Gamma)$ is bounded.

There exist standard \emph{compact} Clifford--Klein forms~$X_{\Gamma}$ that can be deformed into nonstandard ones if and only if ${}^{\backprime}G$ has a simple factor that is locally isomorphic to $\SO(1,2n)$ or $\SU(1,n)$ \cite[Th.\,A]{kob98}.
On the other hand, for \emph{convex cocompact} Clifford--Klein forms~$X_{\Gamma}$, \ie for $\Gamma$ of the form \eqref{eqn:jrho} with $j$ injective and $j({}^{\backprime}\Gamma)$ convex cocompact in~${}^{\backprime}G$ up to switching the two factors (see Definition~\ref{def:ccAdS3}), there is much more room for deformation: for instance, $\Gamma$ could be a free group of any rank~$m$, in which case the deformation space has dimension $m\cdot 2\dim({}^{\backprime}G)$.
Similarly to Corollary~\ref{cor:AdS3deform}, we can extend Theorem~\ref{thm:deform} to \emph{nonstandard} convex cocompact Clifford--Klein forms (in particular that do not identify with ${}^{\backprime}\Gamma\backslash{}^{\backprime}G$).

\begin{theorem}\label{thm:rank1}
Let ${}^{\backprime}G$ be a semisimple linear Lie group of real rank~$1$ satisfying $\rank{}^{\backprime}G = \rank{}^{\backprime}K$.
All convex cocompact Clifford--Klein forms~$X_{\Gamma}$ have an infinite stable discrete spectrum.
\end{theorem}

We note that most semisimple groups ${}^{\backprime}G$ of real rank~$1$ satisfy the condition $\rank{}^{\backprime}G = \rank{}^{\backprime}K$: the only exception is if the Lie algebra ${}^{\backprime}{\g}$ is $\so(1, n)$ for some odd~$n$ up to a compact factor.
Theorem~\ref{thm:rank1} relies on the following two properties, which generalize Facts \ref{fact:AdS} and~\ref{fact:AdSdeform} and corroborate Conjecture~\ref{conj:sharp}.

\begin{fact}[\cite{ggkw}]\label{fact:rank1}
Let ${}^{\backprime}G$ be a semisimple linear Lie group of~real rank~$1$.
All convex cocompact Clifford--Klein forms of $X=({}^{\backprime}G\times\!{}^{\backprime}G)/\Diag({}^{\backprime}G)$ are sharp.
\end{fact}

\begin{fact}[\cite{ggkw}]\label{fact:rank1deform}
Let ${}^{\backprime}G$ be a semisimple linear Lie group of real rank~$1$ and let $X_{\Gamma}$ be a $(c,C)$-sharp, convex cocompact Clifford--Klein form of $X=({}^{\backprime}G\times\!{}^{\backprime}G)/\Diag({}^{\backprime}G)$.
For any $\varepsilon>0$, there is a neighborhood $\mathcal{U}_{\varepsilon}\subset\Hom(\Gamma,{}^{\backprime}G\times\nolinebreak{}^{\backprime}G)$ of the natural inclusion such that $\varphi(\Gamma)$ is discrete in ${}^{\backprime}G\times\!{}^{\backprime}G$ and $(c-\varepsilon,C+\varepsilon)$-sharp for all $\varphi\in\mathcal{U}_{\varepsilon}$.
\end{fact}

For ${}^{\backprime}G=\SO(1,n)$, Facts \ref{fact:rank1} and~\ref{fact:rank1deform} were first established in \cite{gk12}, using the Lipschitz approach of Section~\ref{subsec:CLip}.
In this case, Fact~\ref{fact:rank1} actually holds for a larger class of Clifford--Klein forms~$X_{\Gamma}$, namely all those that are \emph{geometrically finite} (in the sense that the hyperbolic manifold $j({}^{\backprime}\Gamma)\backslash\HH^n$ is geometrically finite, allowing for cusps) \cite{gk12}.
This implies that \emph{the discrete spectrum of any geometrically finite Clifford--Klein form of $X=(\SO(1,n)\times\nolinebreak\SO(1,n))/\Diag(\SO(1,n))$ is infinite} for $n$ even.

\subsection*{$\bullet$ ``Exotic'' Clifford--Klein forms in higher real rank}

As we have seen in Section~\ref{subsec:listexgroup}, for several families of groups~${}^{\backprime}G$ of higher real rank, the space $X=({}^{\backprime}G\times\!{}^{\backprime}G)/\Diag({}^{\backprime}G)$ admits standard compact Clifford--Klein forms~$X_{\Gamma}$ of a more general form than ${}^{\backprime}\Gamma\backslash{}^{\backprime}G$.
More precisely, let ${}^{\backprime}G_1$ and ${}^{\backprime}G_2$ be two reductive subgroups of~${}^{\backprime}G$ such that ${}^{\backprime}G_1$ acts properly and cocompactly on ${}^{\backprime}G/{}^{\backprime}G_2$: we can then take $\Gamma$ of the form $\Gamma={}^{\backprime}\Gamma_1\times{}^{\backprime}\Gamma_2$, where ${}^{\backprime}\Gamma_1$ (resp.~${}^{\backprime}\Gamma_2$) is a uniform lattice of~${}^{\backprime}G_1$ (\resp of~${}^{\backprime}G_2$).
Theorem~\ref{thm:universal} and Proposition~\ref{prop:deformexoticcompact} apply to the discrete spectrum of these ``exotic'' standard compact Clifford--Klein forms~$X_{\Gamma}\simeq{}^{\backprime}\Gamma_1\backslash{}^{\backprime}G/{}^{\backprime}\Gamma_2$ when $\rank{}^{\backprime}G=\rank{}^{\backprime}K$.

A list of examples is given in Table~2.2 of Chapter~\ref{sec:listex}.
Among them, the example $({}^{\backprime}G,{}^{\backprime}G_1,{}^{\backprime}G_2)=(\SO(2,2n)_0,\SO(1,2n)_0,\U(1,n))$ has the property that certain uniform lattices~${}^{\backprime}\Gamma_1$ of~${}^{\backprime}G_1$ admit nonstandard deformations inside~${}^{\backprime}G$, for which there exists an infinite stable discrete spectrum by Proposition~\ref{prop:deformexoticcompact}.
For $n=1$, manifolds of the form $X_{\Gamma}={}^{\backprime}\Gamma_1\backslash{}^{\backprime}G/{}^{\backprime}\Gamma_2$ have dimension~$6$ and are locally modeled on $\AdS^3\times\AdS^3$; the ring $\D(X_{\Gamma})$ is generated by the Laplacians of the two factors.
The following table, for general~$n$, shows that these Clifford--Klein forms $X_{\Gamma}={}^{\backprime}\Gamma_1\backslash{}^{\backprime}G/{}^{\backprime}\Gamma_2$ are very different from the anti-de Sitter manifolds ${}^{\backprime}G_1\backslash{}^{\backprime}G/{}^{\backprime}\Gamma_2 \simeq {}^{\backprime}\Gamma_2\backslash{}^{\backprime}G/{}^{\backprime}G_1={}^{\backprime}\Gamma_2\backslash\AdS^{2n+1}$ which we examined in Section~\ref{subsec:AdS} and from the indefinite K\"ahler manifolds ${}^{\backprime}\Gamma_1\backslash{}^{\backprime}G/{}^{\backprime}G_2={}^{\backprime}\Gamma_1\backslash\SO(2,2n)_0/\U(1,n)$ which we shall examine in Section~\ref{subsec:3cpx}.

\bigskip

\begin{center}
\begin{changemargin}{-0.9cm}{0cm}
\begin{tabular}{|c||c|c|c|} 
\hline
Type of Clifford--Klein form & ${}^{\backprime}\Gamma_1\backslash{}^{\backprime}G/{}^{\backprime}\Gamma_2$ & ${}^{\backprime}G_1\backslash{}^{\backprime}G/{}^{\backprime}\Gamma_2$ & ${}^{\backprime}\Gamma_1\backslash{}^{\backprime}G/{}^{\backprime}G_2$\\
\hline\hline
Model space~$X$ & $\SO(2,2n)_0$ & $\AdS^{2n+1}$ & $\SO(2,2n)_0/\U(1,n)$\\
\hline
Dimension & $2n^2+3n+1$ & $2n+1$ & $n(n+1)$\\
\hline
Signature & $(4n,2n^2-n+1)$ & $(2n,1)$ & $(2n,n^2-n)$\\
\hline
$\rank(X)$ & $n+1$ & $1$ & $n$\\
\hline
Degrees of generators of~$\D(X)$ & $2,4,\dots,2n,n+1$ & $2$ & $2,4,\dots,2n$\\
\hline
$\#\ZZ$ & $2(n+1)$ & $2$ & $1$\\
\hline
\end{tabular}
\end{changemargin}
\end{center}

\bigskip

More generally, whenever ${}^{\backprime}G$ has real rank $>1$, there always exist two nontrivial reductive subgroups ${}^{\backprime}G_1$ and ${}^{\backprime}G_2$ of~${}^{\backprime}G$ such that ${}^{\backprime}G_1$ acts properly (but not necessarily cocompactly) on ${}^{\backprime}G/{}^{\backprime}G_2$ \cite[Th.\,3.3]{kob93}.
When $\rank{}^{\backprime}G=\rank{}^{\backprime}K$, Theorem~\ref{thm:universal} and Propositions \ref{prop:deformexoticcompact} and~\ref{prop:deformexotic} apply to the standard Clifford--Klein forms (possibly of infinite volume) $X_{\Gamma}={}^{\backprime}\Gamma_1\backslash{}^{\backprime}G/{}^{\backprime}\Gamma_2$, where $\Gamma={}^{\backprime}\Gamma_1\times{}^{\backprime}\Gamma_2$ is the product of any discrete subgroup~${}^{\backprime}\Gamma_1$ of~${}^{\backprime}G_1$ with any discrete subgroup~${}^{\backprime}\Gamma_2$ of~${}^{\backprime}G_2$.

\subsection*{$\bullet$ Link between the discrete series representations of $X$ and~${}^{\backprime}G$}

We now assume that ${}^{\backprime}G$ is connected and that $\rank{}^{\backprime}G=\rank{}^{\backprime}K$.
Flensted-Jensen's construction of discrete series representations~$\V_{Z,\lambda}$ for $X=({}^{\backprime}G\times\!{}^{\backprime}G)/\Delta({}^{\backprime}G)$ (as described in Section~\ref{subsec:discreteseries}) yields all of Harish-Chandra's discrete series representations~$\pi_{{}^{\backprime}\!\lambda}$ for~${}^{\backprime}G$.
This is well-known, but for the reader's convenience we briefly recall the Harish-Chandra discrete series and make the link with our previous notation.

Let ${}^{\backprime}\theta$ be a Cartan involution of~${}^{\backprime}G$ and let ${}^{\backprime}K=({}^{\backprime}G)^{{}^{\backprime}\theta}$ be the corresponding maximal compact subgroup of~${}^{\backprime}G$.
For simplicity, suppose that ${}^{\backprime}\theta$ extends to a holomorphic involution of some complexification ${}^{\backprime}G_{\C}$ of~${}^{\backprime}G$.
As in Section~\ref{subsec:FJdualAdS3}, we define a holomorphic embedding $\Phi: {}^{\backprime}G_{\C} \to {}^{\backprime}G_{\C}\times\!{}^{\backprime}G_{\C}$~by
$$\Phi(g) := \big(g,{}^{\backprime}\theta(g)\big).$$ 
Then the set of inclusions \eqref{eqn:KGH} is given by
\smallskip
\begin{changemargin}{-2cm}{0cm}
\begin{center}
\begin{tabular}{c c c c c}
$K={}^{\backprime}K\times\!{}^{\backprime}K$ & $\subset$ & $G={}^{\backprime}G\times\!{}^{\backprime}G$ & $\supset$ & $H=\Diag({}^{\backprime}G)$\\
\smallskip
\rotatebox{90}{$\supset$} & & \rotatebox{90}{$\supset$} & & \rotatebox{90}{$\supset$}\\
\smallskip
$K_{\C}={}^{\backprime}K_{\C}\times\!{}^{\backprime}K_{\C}$ & $\subset$ & $G_{\C}={}^{\backprime}G_{\C}\times\!{}^{\backprime}G_{\C}$ & $\supset$ & $H_{\C}=\Diag({}^{\backprime}G_{\C})$\\
\medskip
\rotatebox{90}{$\subset$} & & \rotatebox{90}{$\subset$} & & \rotatebox{90}{$\subset$}\\
\smallskip
$H^d=\Phi({}^{\backprime}K_{\C})$ & $\subset$ & $G^d=\Phi({}^{\backprime}G_{\C})$ & $\supset$ & $K^d=\Phi({}^{\backprime}G_U),$
\end{tabular}
\end{center}
\end{changemargin}

\smallskip
\noindent
where ${}^{\backprime}G_U$ is the compact real form of~${}^{\backprime}G_{\C}$ defined similarly to Section~\ref{subsec:realforms}.
As in Section~\ref{subsec:FJdualAdS3}, the group~$H^d$ identifies with~${}^{\backprime}K_{\C}$ and $G^d/P^d$ with the full complex flag variety ${}^{\backprime}G_{\C}/{}^{\backprime}B_{\C}$, where ${}^{\backprime}B_{\C}$ is a Borel subgroup of~${}^{\backprime}G_{\C}$.
Fix a Cartan subalgebra ${}^{\backprime}\ttt$ of~${}^{\backprime}\kk$ and a positive system $\Delta^+({}^{\backprime}\kk_{\C},{}^{\backprime}\ttt_{\C})$.
We note that ${}^{\backprime}\ttt$ is also a Cartan subalgebra of~${}^{\backprime}\g$ since $\rank{}^{\backprime}G=\rank{}^{\backprime}K$.
The set~$\ZZ$ of closed $H^d$-orbits in $G^d/P^d$ identifies with the set of positive systems $\Delta^+({}^{\backprime}\g_{\C},{}^{\backprime}\ttt_{\C})$ containing the fixed positive system $\Delta^+({}^{\backprime}\kk_{\C},{}^{\backprime}\ttt_{\C})$.
In particular, the cardinal of~$\ZZ$ is easily computable as the quotient of the cardinals of two Weyl groups.
For instance, for ${}^{\backprime}G=\SO(1,2n)_0$, we have
$$\#\ZZ = \frac{\# W(B_n)}{\# W(D_n)} = 2.$$

Let ${}^{\backprime}\rho_c$ be half the sum of the elements of $\Delta^+({}^{\backprime}\kk_{\C},{}^{\backprime}\ttt_{\C})$.
Any choice of a positive system $\Delta^+({}^{\backprime}\g_{\C},{}^{\backprime}\ttt_{\C})$ containing $\Delta^+({}^{\backprime}\kk_{\C},{}^{\backprime}\ttt_{\C})$ determines a positive Weyl chamber ${}^{\backprime}\ttt^{\ast}_+$ in~${}^{\backprime}\ttt^{\ast}$, an element ${}^{\backprime}\rho\in{}^{\backprime}\ttt^{\ast}_+$, defined as half the sum of the elements of $\Delta^+({}^{\backprime}\g_{\C},{}^{\backprime}\ttt_{\C})$, and an element $Z\in\ZZ$.
For any ${}^{\backprime}\lambda\in{}^{\backprime}\ttt^{\ast}_+$ such that
$$\mu_{{}^{\backprime}\lambda} := {}^{\backprime}\lambda + {}^{\backprime}\rho - 2{}^{\backprime}\rho_c$$
lifts to the torus ${}^{\backprime}T\subset {}^{\backprime}K$ with Lie algebra~${}^{\backprime}\ttt$, Harish-Chandra proved the existence of an irreducible unitary representation $\pi_{{}^{\backprime}\!\lambda}$ of~${}^{\backprime}G$ with square-integrable matrix coefficients, with infinitesimal character~${}^{\backprime}\lambda$ (Harish-Chandra parameter) and minimal ${}^{\backprime}K$-type $\mu_{{}^{\backprime}\lambda}$ (Blattner parameter).
With the notation of the previous chapters, we can take
$$\jj = \{ ({}^{\backprime}Y,-{}^{\backprime}Y) : {}^{\backprime}Y\in{}^{\backprime}\ttt\} .$$
For $\lambda=({}^{\backprime}\lambda,-{}^{\backprime}\lambda)\in\jj^{\ast}$ and $Y=({}^{\backprime}Y,-{}^{\backprime}Y)\in\jj$, we have
$$\langle\lambda,Y\rangle = 2\,\langle{}^{\backprime}\lambda,{}^{\backprime}Y\rangle,$$
and if ${}^{\backprime}d : {}^{\backprime}\ttt^{\ast}\rightarrow\R_+$ denotes the ``weighted distance to the walls'' defined as in Section~\ref{subsec:Lambda+}, then
$$d(\lambda) = {}^{\backprime}d({}^{\backprime}\lambda).$$
Since
$K/H\cap K = ({}^{\backprime}K\times\!{}^{\backprime}K)/\Diag({}^{\backprime}K) \simeq {}^{\backprime}K$,
the set $\Lambda_+ = \Lambda_+(K/H\cap K)$ of \eqref{eqn:Lambda+} is here equal to $\{ ({}^{\backprime}\lambda,-{}^{\backprime}\lambda) : {}^{\backprime}\lambda\in\widehat{{}^{\backprime}K}\}$,
which naturally identifies with the set $\widehat{{}^{\backprime}K}$ of irreducible representations of ${}^{\backprime}K$.
For $\lambda=({}^{\backprime}\lambda,-{}^{\backprime}\lambda)\in\jj^{\ast}_+$, we have an isomorphism of $({}^{\backprime}\g,{}^{\backprime}K)\times ({}^{\backprime}\g,{}^{\backprime}K)$-modules:
$$\V_{Z,\lambda} \simeq (\pi_{{}^{\backprime}\!\lambda})_{{}^{\backprime}K} \boxtimes (\pi_{{}^{\backprime}\!\lambda}^{\vee})_{{}^{\backprime}K}.$$

\subsection*{$\bullet$ Regular representation on $L^2({}^{\backprime}\Gamma\backslash{}^{\backprime}G)$}

Let ${}^{\backprime}\Gamma$ be a discrete subgroup of~${}^{\backprime}G$.
The action of ${}^{\backprime}G$ on ${}^{\backprime}\Gamma\backslash{}^{\backprime}G$ from the right defines
a unitary representation of ${}^{\backprime}G$ on $L^2({}^{\backprime}\Gamma\backslash{}^{\backprime}G)$.
With the previous notation, here is a consequence of Proposition~\ref{prop:nonzero}.(2) applied to the special case
$$G={}^{\backprime}G\times\!{}^{\backprime}G, \quad H=\Diag({}^{\backprime}G), \quad \Gamma={}^{\backprime}\Gamma\times\{ e\},$$
where the Clifford--Klein form $X_{\Gamma}=\Gamma\backslash G/H$ identifies with ${}^{\backprime}\Gamma\backslash{}^{\backprime}G$.

\begin{proposition}\label{prop:intertwining}
Let ${}^{\backprime}G$ be a reductive linear group with $\rank{}^{\backprime}G = \rank{}^{\backprime}K$.
\begin{enumerate}
  \item There is a constant $R>0$ (depending only on~${}^{\backprime}G$) such that for any torsion-free discrete subgroup~${}^{\backprime}\Gamma$ of~${}^{\backprime}G$ and any discrete series representation~$\pi_{{}^{\backprime}\!\lambda}$ of~${}^{\backprime}G$ with ${}^{\backprime}d({}^{\backprime}\lambda)>R$,
  $$\Hom_{{}^{\backprime}G}\left(\pi_{{}^{\backprime}\!\lambda},L^2({}^{\backprime}\Gamma\backslash{}^{\backprime}G)\right) \neq \{ 0\} .$$
  \item The same statement holds without the ``torsion-free'' assumption on~$\Gamma$ if ${}^{\backprime}G$ has no compact factor.
\end{enumerate}
\end{proposition}

\begin{proof}
Consider ${}^{\backprime}\lambda\in{}^{\backprime}\ttt^{\ast}_+$
such that $\mu_{{}^{\backprime}\lambda}$ lifts to a maximal torus in~${}^{\backprime}K$.
Then $\lambda:=({}^{\backprime}\lambda,-{}^{\backprime}\lambda)\in\jj^{\ast}_+$ belongs to $2\rho_c-\rho+\Lambda_+$ and $d(\lambda)={}^{\backprime}d({}^{\backprime}\lambda)$.
Applying Proposition~\ref{prop:nonzero}.(2), together with \eqref{eqn:Kfiniteconj} and \eqref{eqn:transL2}, to
$$ G={}^{\backprime}G\times\!{}^{\backprime}G, \quad H=\Diag({}^{\backprime}G), \quad \Gamma={}^{\backprime}\Gamma\times\{ e\} ,$$
we obtain the existence of a constant $R>0$ such that if  ${}^{\backprime}d({}^{\backprime}\lambda)>R$ and ${}^{\backprime}G$ has no compact factor (\resp ${}^{\backprime}G$ has compact factors), then for any discrete (\resp torsion-free discrete) subgroup ${}^{\backprime}\Gamma$ of~${}^{\backprime}G$, the summation operator
$$S_{\Gamma} : L^2({}^{\backprime}G,\M_{\lambda})_{{}^{\backprime}K_1\times {}^{\backprime}K_2} \longrightarrow L^2({}^{\backprime}\Gamma\backslash{}^{\backprime}G,\M_{\lambda})$$
is well-defined and nonzero for some conjugates ${}^{\backprime}K_1=g_1{}^{\backprime}Kg_1^{-1}$ and ${}^{\backprime}K_2=g_2{}^{\backprime}Kg_2^{-1}$ of~${}^{\backprime}K$ (where $g_i\in{}^{\backprime}G$).
In our specific setting, for $\varphi\in L^2({}^{\backprime}G,\M_{\lambda})_{{}^{\backprime}K_1\times {}^{\backprime}K_2}$, the function~$S_{\Gamma}(\varphi)$ is nothing but the classical Poincar\'e series
$$\sum_{\gamma\in{}^{\backprime}\Gamma} \varphi(\gamma\,\cdot)\ \in L^2({}^{\backprime}\Gamma\backslash{}^{\backprime}G,\M_{\lambda})_{{}^{\backprime}K_2},$$
and $S_{\Gamma}$ respects the action of $({}^{\backprime}\g,{}^{\backprime}K_2)$ from the right.  
Therefore,
$$\Hom_{({}^{\backprime}\g,{}^{\backprime}K_2)}\left((\pi_{{}^{\backprime}\!\lambda})_{{}^{\backprime}K_2},L^2({}^{\backprime}\Gamma\backslash{}^{\backprime}G)\right)_{{}^{\backprime}K_2} \neq \{ 0\} $$
if ${}^{\backprime}d({}^{\backprime}\lambda)>R$.
Since $\pi_{{}^{\backprime}\!\lambda}$ is an irreducible unitary representation of~${}^{\backprime}G$, this is equivalent to
$$\Hom_{{}^{\backprime}G}\left(\pi_{{}^{\backprime}\!\lambda},L^2({}^{\backprime}\Gamma\backslash{}^{\backprime}G)\right) \neq \{ 0\}.\qedhere$$
\end{proof}

\begin{remark}\label{rem:discseries}
For arithmetic~${}^{\backprime}\Gamma$, we may consider a tower of congruence subgroups ${}^{\backprime}\Gamma\supset {}^{\backprime}\Gamma_1\supset {}^{\backprime}\Gamma_2\supset\cdots$.
In the work of DeGeorge--Wallach \cite{dw78} (cocompact case), Clozel \cite{clo86}, Rohlfs--Speh \cite{rs87}, and Savin \cite{sav89} (finite covolume case), the asymptotic behavior of the multiplicities $\Hom_{{}^{\backprime}G}\left(\pi_{{}^{\backprime}\!\lambda},L^2({}^{\backprime}\Gamma_j\backslash{}^{\backprime}G)\right)$ for a discrete series representation~$\pi_{{}^{\backprime}\!\lambda}$ was studied as $j$ goes to infinity, under the condition $\rank{}^{\backprime}G = \rank{}^{\backprime}K$. 
Then one could deduce from their result that any discrete series representation~$\pi_{{}^{\backprime}\!\lambda}$ with ${}^{\backprime}d({}^{\backprime}\lambda)$ large enough occurs in $L^2({}^{\backprime}{}^{\backprime}\Gamma\backslash{}^{\backprime}G)$ for some congruence subgroup ${}^{\backprime}{}^{\backprime}\Gamma$ of~${}^{\backprime}\Gamma$, where~${}^{\backprime}{}^{\backprime}\Gamma$ possibly depends on~$\pi_{{}^{\backprime}\!\lambda}$.
The approach of \cite{dw78, clo86, sav89} uses the Arthur--Selberg trace formula.
There is another approach for classical groups~${}^{\backprime}G$ and arithmetic subgroups~${}^{\backprime}\Gamma$ using the theta-lifting, see \cite{bw00, kaz77, li92}.
Proposition~\ref{prop:intertwining} is stronger in three respects:
\begin{enumerate}
   \item ${}^{\backprime}\Gamma$ is not necessarily arithmetic and ${}^{\backprime}\Gamma\backslash{}^{\backprime}G$ can have infinite volume,
\item we do not need to replace~${}^{\backprime}\Gamma$ by some finite-index subgroup~${}^{\backprime}{}^{\backprime}\Gamma$,
  \item the constant~$R$ is independent of the discrete group~${}^{\backprime}\Gamma$.
\end{enumerate}
\end{remark}

\section{Indefinite K\"ahler manifolds}\label{subsec:3cpx}

We now consider the symmetric space $X=\SO(2,2m)_0/\U(1,m)$ for $m\geq 2$.
Later we will assume $m$ to be even for the rank condition \eqref{eqn:rank} to be satisfied.
We see the group $\OO(2,2m)$ as the set of linear transformations of~$\R^{2m+2}$ preserving the quadratic form
$$x_1^2 + y_1^2 - x_2^2 - y_2^2 - \dots - x_{m+1}^2 - y_{m+1}^2,$$
and the subgroup $H:=\U(1,m)$ of $G:=\SO(2,2m)_0$ as the set of linear trans\-formations of $\C^{m+1}$ preserving the Hermitian form $|z_1|^2-|z_2|^2-\dots-~|z_{m+1}|^2$.
The involution~$\sigma$ of~$G$ defining~$H$ is given by $\sigma(g)=JgJ^{-1}$, where $J$ is the diagonal block matrix with all diagonal blocks equal to $\small\begin{pmatrix} 0 & -1\\ 1 & 0\end{pmatrix}$.

The natural $G$-invariant pseudo-Riemannian metric~$g$ on~$X$ has signature $(2m,m(m-1))$.
We note that here $X$ carries some additional structures, due to the fact that $H$ is the centralizer of a one-dimensional compact torus (namely its center $Z(H)\simeq\U(1)$):  
\begin{enumerate}
  \item $X$ can be identified with an adjoint orbit (namely $\Ad(G)v$ where $v$ is any generator of the Lie algebra of $Z(H)$), hence also with a coadjoint orbit \emph{via} the isomorphism $\g^{\ast}\simeq\g$ induced by the Killing form; thus, $X$ carries a Kostant--Souriau symplectic form~$\omega$ (see \cite[Ch.\,1, Th.\,1]{kir04});
  \item $X$ can be realized as an open subset of the flag variety $G_{\C}/P_{\C}$ for some maximal proper parabolic subgroup~$P_{\C}$ of $G_{\C}=\SO(2m+\nolinebreak 2,\C)$, as a generalized Borel embedding (see \cite{ko90} for instance); in particular, $X$ has a $G$-invariant complex structure and $g+\sqrt{-1}\,\omega$ is a $G$-invariant indefinite K\"ahler form on~$X$ if $g$ is normalized by the Killing form.
\end{enumerate}
The existence of the complex structure can easily be seen for $m=\nolinebreak 2$, since $\SO(2,4)_0/\U(1,2)$ identifies with $\SU(2,2)/\U(1,2)$, which can be realized as an open subset of~$\PP^3\C$ (see Section~\ref{subsec:introex}).

Standard Clifford--Klein forms~$X_{\Gamma}$ of~$X$ that are compact (\resp noncompact but of finite volume) were constructed in \cite{kob89}.
They can be obtained by taking torsion-free uniform (\resp nonuniform) lattices $\Gamma$ inside $L:=\SO(1,2m)_0$.
We note that the group $L$ acts properly and transitively on $X$.
An elementary explanation for this is to observe that $\U(m+1)$ acts transitively on the sphere $\mathbb{S}^{2m+1}=\SO(2m+2)/\SO(2m+1)$;
by duality, so does $\SO(2m+1)$ on $\SO(2m+2)/\U(m+1)$; in turn, $L$ acts properly and transitively on $X=\SO(2,2m)_0/\U(1,m)$.
(For a general argument, we refer to \cite[Lem.\,5.1]{kob94}.)

If $\Gamma$ is a free discrete subgroup of~$L$, then the noncompact standard Clifford--Klein form~$X_{\Gamma}$ has a large deformation space.
There are also examples of compact standard Clifford--Klein forms that admit interesting small deformations.
Indeed, certain arithmetic uniform lattices~$\Gamma$ of $L=\SO(1,m)_0$ have the following property: there is a continuous $1$-parameter group $(\varphi_t)_{t\in\R}$ of homomorphisms from $\Gamma$ to~$G$ such that for any $t\neq 0$ small enough, the group $\varphi_t(\Gamma)$ is discrete in~$G$ and \emph{Zariski-dense} in~$G$; this $1$-parameter group can be obtained by a \emph{bending} construction due to Johnson--Millson (see \cite[\S\,6]{kas12}).
As we have seen in Example~\ref{ex:standardthetastable}, any discrete subgroup~$\Gamma$ of~$L$ is $(\frac{\sqrt{2}}{2},0)$-sharp for~$X$; by \cite{kas12}, if $\Gamma$ is cocompact or convex cocompact in~$L$, then for any $\varepsilon>0$ there is a neighborhood $\mathcal{U}_{\varepsilon}\subset\Hom(\Gamma,G)$ of the natural inclusion such that for any $\varphi\in\mathcal{U}_{\varepsilon}$, the group $\varphi(\Gamma)$ is discrete in~$G$ and $(\frac{\sqrt{2}}{2}-\varepsilon,\varepsilon)$-sharp for~$X$ (see Lemma~\ref{lem:munudeform}).

We now assume that $m=2n$ is even, so that the rank condition \eqref{eqn:rank} is satisfied.
We start by examining the case $n=1$, in which we give explicit formulas for the Flensted-Jensen eigenfunctions of Section~\ref{subsec:FJfunction}; we then explicit the notation of the previous chapters for general~$n$.

\subsection*{$\bullet$ The case $n=1$}

The group $G=\SO(2,4)_0$ admits $\SU(2,2)$ as a double covering, and the preimage of $H=\U(1,2)$ in $\SU(2,2)$ is $\mathrm{S}(\U(1)\times\U(1,2))\simeq\U(1,2)$.
For an actual computation, in this paragraph we set $G:=\SU(2,2)$ and $H:=\mathrm{S}(\U(1)\times\U(1,2))\simeq\U(1,2)$, and we consider the maximal compact subgroup $K:=\mathrm{S}(\U(2)\times\U(2))$.
The symmetric space $X\simeq\SU(2,2)/\U(1,2)$ identifies with the open subset of~$\PP^3\C$ of equation $h>0$, where
$$h(z) = |z_1|^2 + |z_2|^2 - |z_3|^2 - |z_4|^2$$
for $z=(z_i)_{1\leq i\leq 4}\in\C^4$.
The Laplacian~$\square_X$ has been made explicit in Section~\ref{subsec:introex}.
For any $\ell\in\N$, we consider the following harmonic polynomial of degree $(\ell,\ell)$ on~$\C^2$:
$$P_{\ell}(z_1,z_2) := \sum_{i=0}^{\ell} \binom{\ell}{i}^2 (-1)^i \, |z_1|^{2\ell-2i} \, |z_2|^{2i}.$$
Up to a multiplicative scalar, it is the unique harmonic polynomial of degree $(\ell,\ell)$ that is fixed by $\U(1)\times\U(1)\simeq H\cap K$; we normalize it so that $P_{\ell}(1,0)=\nolinebreak 1$.
The function
\begin{equation}\label{eqn:FJpk}
\psi_{\ell} :\ z=(z_i)_{1\leq i\leq 4}\ \longmapsto\ P_{\ell}(z_1,z_2) \, h(z)^{\ell+1} \, \big(|z_1|^2 + |z_2|^2\big)^{-2\ell -1}
\end{equation}
on $\C^4\smallsetminus\{ 0\}$ satisfies the following differential equation:
$$h(z)\ \square_{\C^{2,2}}\,\psi_{\ell} = (\ell+1)(\ell-2)\,\psi_{\ell}.$$
Since $\psi_{\ell}$ is homogeneous of degree~$0$, we may regard it as a function on $X=\{ h>0\}\subset\PP^3\C$.
Using these properties, we obtain the following (we omit the details).

\begin{claim}\label{claim:FJP3}
For any $\ell\in\N_+$, the function $\psi_{\ell} : X\rightarrow\C$ is a Flensted-Jensen function on $X=\SU(2,2)/\U(1,2)$, with parameter $\lambda=2\ell-1\in\R_+\simeq\jj^{\ast}_+$ and with
$$\square_X\,\psi_{\ell} = 2(\ell+1)(\ell-2)\,\psi_{\ell}.$$
The $(\g,K)$-modules $\V_{\ell}$ generated by~$\psi_{\ell}$ for $\ell\in\N_+$ form the complete set of discrete series representations for~$X$.
\end{claim}

\noindent
We note that the $(\g,K)$-module~$\V_{\ell}$ is irreducible and isomorphic to the Zuckerman--Vogan derived functor module $V_0(2\ell-1,1)$ in algebraic representation theory, with notation as in \cite[\S\,4]{kob94}; in particular, $\V_{\ell}$ has infinitesimal character $\frac{1}{2}(2\ell-1,1,-1,-2\ell+1)$ in the Harish-Chandra parameterization and minimal $K$-type parameter $(\ell,-\ell,0,0)$.

For the symmetric pair $(G,H)\simeq(\SU(2,2),\U(1,2))$, the polar decomposition $G=KBH$ holds, where the Lie algebra $\bb$ of~$B$ is generated by
$$Y_0 := \left(\begin{BMAT}{c.c}{c.c}
  \begin{BMAT}{ccc}{cc}
  0 & & \\ & & 0
  \end{BMAT}
  &
  \begin{BMAT}{ccc}{cc}
  1 & & \\ & & 0
  \end{BMAT}\\
  \begin{BMAT}{ccc}{cc}
  1 & & \\ & & 0
  \end{BMAT}
  &
  \begin{BMAT}{ccc}{cc}
  0 & & \\ & & 0
  \end{BMAT}
\end{BMAT}\right) \,\in\su(2,2) \simeq \g.$$
If we identify $\bb$ with~$\R$ by sending $Y_0$ to~$1$, then
$$\nu(z) = \operatorname{arccosh}\sqrt{\frac{|z_1|^2 + |z_2|^2}{h(z)}} \,\in\R_{\geq 0}$$
for all $z=[z_1:z_2:z_3:z_4]\in X$.
Here are the analytic estimates of Propositions \ref{prop:asym} and~\ref{prop:psilambda} for the Flensted-Jensen functions $\psi_{\ell}$ of \eqref{eqn:FJpk}.

\begin{lemma}\label{lem:estimP3}
For any $z\in X=\SU(2,2)/\U(1,2)$,
$$|\psi_{\ell}(z)| \leq \big(\!\cosh\nu(z)\big)^{-2(\ell+1)} \leq 2^{2(\ell+1)}\,e^{-2(\ell+1)\nu(z)}.$$
\end{lemma}

\noindent
This estimate follows immediately from the definition \eqref{eqn:FJpk} of~$\psi_{\ell}$, in light of the inequality $|P_{\ell}(z_1,z_2)|\leq (|z_1|^2+|z_2|^2)^{\ell}$ for all $(z_1,z_2)\in\C^2$.
Using \eqref{eqn:weightKBH}, one can show that the function~$\psi_{\ell}$ is square integrable on~$X$ if and only if $\ell>1/2$.

\subsection*{$\bullet$ The general case}

We now consider $G=\SO(2,4n)_0$ and $H=\U(1,2n)$ for an arbitrary integer $n\geq 1$.
The Cartan decomposition $G=KAK$ holds, where $K=\SO(2)\times\SO(4n)$ and $A$ is the maximal split abelian subgroup of~$G$ whose Lie algebra~$\aaa$ is the set of elements
$$a_{s,t} :=
\left(\begin{BMAT}{c.c}{c.c}
\begin{BMAT}{c.c}{c.c} 0 & \begin{BMAT}{cc}{cc} s & 0\\ 0 & t\end{BMAT}\\ \begin{BMAT}{cc}{cc} s & 0\\ 0 & t\end{BMAT} & 0\end{BMAT} & 0\\
0 & \begin{BMAT}{ccccccc}{ccccccc} & & & & & & \\ & & & & & & \\ & & & & & & \\ & & & 0 & & & \\ & & & & & & \\ & & & & & & \\ & & & & & &\end{BMAT}
\end{BMAT}\right)$$
for $s,t\in\R$.
The generalized Cartan decomposition $G=KBH$ holds, where the Lie algebra $\bb$ of~$B$ is the set of elements~$a_{s,-s}$ with $s\in\R$.
The set of inclusions \eqref{eqn:KGH} is given by
\begin{center}
\small
\begin{tabular}{c c c c c}
$K=\SO(2)\times\SO(4n)$ & $\subset$ & $G=\SO(2,4n)_0$ & $\supset$ & $H=\U(1,2n)$\\
\smallskip
\rotatebox{90}{$\supset$} & & \rotatebox{90}{$\supset$} & & \rotatebox{90}{$\supset$}\\
\smallskip
$K_{\C}=\SO(2,\C)\times\SO(4n,\C)$ & $\subset$ & $G_{\C}=\SO(2+4n,\C)$ & $\supset$ & $H_{\C}=\GL(1+2n,\C)$\\
\medskip
\rotatebox{90}{$\subset$} & & \rotatebox{90}{$\subset$} & & \rotatebox{90}{$\subset$}\\
\smallskip
$H^d=\SO(2)\times\SO^{\ast}(4n)$ & $\subset$ & $G^d=\SO^{\ast}(2+4n)$ & $\supset$ & $K^d=\U(1+2n).$
\end{tabular}
\end{center}
We recall that for any $m\geq 1$, the group $\SO^{\ast}(2m)$ is a real form of $\SO(2m,\C)$ with maximal compact subgroup $\U(m)$.

A maximal abelian subspace~$\jj$ of $\sqrt{-1}(\kk\cap\q)$ is given by the set of block matrices
$$Y_{(s_1,\dots,s_n)} := \left(\begin{BMAT}{c.c}{c.c}
\begin{BMAT}{ccc}{ccc} & & \\ & 0 & \\ & & \end{BMAT}
& 0\\
0 & 
   \begin{BMAT}{c.c.c.c.c.c}{c.c.c.c.c.c}
   & & & & & s_n Y\\
   & & & & \iddots & \\
   & & & s_1 Y & & \\
   & & - s_1 Y & & & \\
   & \iddots & & & & \\
   - s_n Y & & & & &
   \end{BMAT}
\end{BMAT}\right)$$
for $s_1,\dots,s_n\in\R$, where
$$Y := \left(\begin{BMAT}{cc}{cc} 0 & \sqrt{-1}\\ \sqrt{-1} & 0 \end{BMAT}\right).$$
In particular, the rank of the symmetric space~$X$ is $\dim\jj=n$.

Let $\{ f_1,\dots,f_n\} $ be the basis of~$\jj^{\ast}$ that is dual to $\{ Y_{(1,0,\dots,0)},\dots,Y_{(0,\dots,0,1)}\} $.
The set
$$\Sigma^+(\kk_{\C},\jj_{\C}) := \{ f_i\pm f_j : 1\leq i<j\leq n\} \cup \{ 2f_k : 1\leq k\leq n\} $$
is a positive system of restricted roots of $\jj_{\C}$ in~$\kk_{\C}$.
There is a unique positive system $\Sigma^+(\g_{\C}, \jj_{\C})$ that contains it, namely
$$\{ f_i\pm f_j : 1\leq i<j\leq n\} \cup \{ 2f_k : 1\leq k\leq n\} \cup \{ f_k : 1\leq k\leq n\} .$$
By \eqref{eqn:Z1to1}, for any minimal parabolic subgroup~$P^d$ of~$G^d$, there is a unique closed $H^d$-orbit in $G^d/P^d$, \ie the set~$\ZZ$ has only one element.
The multiplicities of the restricted roots $\pm f_i \pm f_j$ and~$\pm f_k$ are four, and those of~$\pm 2 f_k$ are one.
Identifying $\jj^{\ast}$ with~$\R^n$ \emph{via} the basis $\{ f_1,\dots,f_n\}$, we obtain
\begin{align*}
\jj^{\ast}_+ & = \big\{ \lambda = (\lambda_1,\dots,\lambda_n) : \lambda_1 > \lambda_2 > \dots > \lambda_n > 0\} ,\\
d(\lambda) & = \frac{1}{2} \min\big\{ \lambda_1-\lambda_2, \lambda_2-\lambda_3, \dots, \lambda_{n-1}-\lambda_n, 2\lambda_n\big\} ,\\
\rho & = \big(4n-1, 4n-5, \dots, 7, 3\big),\\
\rho_c & = \big(4n-3, 4n-7, \dots, 5, 1\big),\\
\mu_{\lambda} & = \lambda + \rho - 2\rho_c = \big(\lambda_1 - 4n+5, \lambda_2-4n+9, \dots, \lambda_{n-1}-3, \lambda_n+1\big).
\end{align*}
The integrality condition~\eqref{eqn:lmdint} on~$\mu_{\lambda}$ amounts to
\begin{eqnarray*}
& & \lambda_j + 1 \in 2\N  \quad\quad \mathrm{for\ all}\ 1\leq j\leq n\\
& \mathrm{and} & \lambda_j - \lambda_{j+1} \geq 4  \quad \mathrm{for\ all}\ 1\leq j\leq n-1.
\end{eqnarray*}
Since the restricted root system $\Sigma(\g_{\C},\jj_{\C})$ is of type $BC_n$, the Weyl group~$W$ is isomorphic to the semidirect product $\mathcal{S}_n\ltimes(\Z/2\Z)^n$ and we have $\C$-algebra isomorphisms
\begin{equation*}
\D(X) \,\simeq\, \C[x_1,\dots,x_n]^{\mathcal{S}_n \ltimes (\Z/2\Z)^n} \,\simeq\, \C [D_1,\dots,D_n],
\end{equation*}
where $D_1,D_2,\dots,D_n$ are algebraically independent invariant polynomials of homogeneous degrees $2,4,\dots,2n$.
If we normalize the pseudo-Riemannian metric $g$ on $X$ by $g(Y,Y)=1$ for $Y:=\frac{d}{ds}|_{s=0}\,\exp(a_{s,-s})\cdot x_0\in T_{x_0}X$ (where $x_0$ denotes the image of $H$ in $X=G/H$, as usual), then the Laplacian~$\square_X$ is $16n$ times the Casimir operator defined by the Killing form (for $n=1$, this is twice the Laplacian that we defined in Section~\ref{subsec:introex} with respect to the ``indefinite Fubini--Study metric''~$h$).
By Fact~\ref{fact:eigenvalues}, the action of the Laplacian~$\square_X$ on $L^2(X,\M_{\lambda})$ is given by multiplication by the scalar
$$(\lambda,\lambda) - (\rho,\rho) = \lambda_1^2 + \dots + \lambda_n^2 - \frac{1}{3}(16n^3 + 12n^2 - n).$$
We note that the center $Z(\SO(2,4n)_0)$ is contained in $\U(1,2n)$, hence\linebreak $\Lambda^{\Gamma\cap Z(G_s)}=\Lambda$ for all~$\Gamma$ by Remark~\ref{rem:LambdaJ}; this shows that the choice of~$\Gamma$ does not impose any additional integrality condition on the discrete spectrum for $X=\SO(2,4n)_0/\U(1,2n)$ when we apply Theorems \ref{thm:precise} and~\ref{thm:precisedeform}.

\begin{remark}
In Sections \ref{subsec:AdS} and~\ref{subsec:3cpx}, the isometry group of~$X$ is in the same family $\OO(2,2m)$, with $m\in\N$ in Section~\ref{subsec:AdS} and $m\in 2\N$ in Section~\ref{subsec:3cpx}.
However, the representations $\V_{Z,\lambda}$ of $G=\SO(2,2m)_0$ that are involved are different: they are all highest-weight modules if $X=\AdS^{2m+1}$, and never highest-weight modules if $X$ is the indefinite K\"ahler manifold $\SO(2,4n)_0/\U(1,2n)$.
\end{remark}